\documentclass[11pt]{amsart}

\newtheorem{thm}{Theorem}[section]
\newtheorem*{thm*}{Theorem}

\newtheorem{cor}[thm]{Corollary}

\newtheorem{lem}[thm]{Lemma}
\newtheorem*{lem*}{Lemma}
\newtheorem{mainthm}{Theorem}
\newtheorem*{mainthm*}{Theorem}
\newtheorem{maincor}[mainthm]{Corollary}
\newtheorem{prop}[thm]{Proposition}

\theoremstyle{definition}

\newtheorem*{case*}{Case}

\newtheorem{defn}[thm]{Definition}
\newtheorem*{defn*}{Definition}
\newtheorem{exmp}[thm]{Example}
\newtheorem*{exmp*}{Example}

\newtheorem{hyp}[thm]{Hypothesis}

\newtheorem{ques}[thm]{Question}
\newtheorem{step}{Step}\renewcommand{\thestep}{}

\theoremstyle{remark}

\newtheorem{case}{Case}\renewcommand{\thecase}{}

\newtheorem{rmk}[thm]{Remark}
\newtheorem*{rmk*}{Remark}

\makeatletter
\def\alphenumi{
  \def\theenumi{\alph{enumi}}
  \def\p@enumi{\theenumi}
  \def\labelenumi{(\@alph\c@enumi)}}
\makeatother

\makeatletter
\def\thecase{\@arabic\c@case}
\makeatother

\makeatletter
\def\thestep{\@arabic\c@step}
\makeatother

\makeatletter
\newcommand{\tpitchfork}{%
  \vbox{
    \baselineskip\z@skip
    \lineskip-.52ex
    \lineskiplimit\maxdimen
    \m@th
    \ialign{##\crcr\hidewidth\smash{$-$}\hidewidth\crcr$\pitchfork$\crcr}
  }%
}
\makeatother

\DeclareFontFamily{U}{mathx}{\hyphenchar\font45}
\DeclareFontShape{U}{mathx}{m}{n}{
      <5> <6> <7> <8> <9> <10>
      <10.95> <12> <14.4> <17.28> <20.74> <24.88>
      mathx10
      }{}
\DeclareSymbolFont{mathx}{U}{mathx}{m}{n}
\DeclareFontSubstitution{U}{mathx}{m}{n}
\DeclareMathAccent{\widecheck}{0}{mathx}{"71}
\DeclareMathAccent{\wideparen}{0}{mathx}{"75}

\newcount\hh
\newcount\mm
\mm=\time
\hh=\time
\divide\hh by 60
\divide\mm by 60
\multiply\mm by 60
\mm=-\mm
\advance\mm by \time
\def\hhmm{\number\hh:\ifnum\mm<10{}0\fi\number\mm}

\let\oldmarginpar\marginpar
\renewcommand\marginpar[1]{\-\oldmarginpar[\raggedleft\footnotesize #1]%
{\raggedright\footnotesize #1}}

\newcommand\CC{\mathbb{C}}

\newcommand\NN{\mathbb{N}}

\newcommand\RR{\mathbb{R}}

\newcommand\TT{\mathbb{T}}

\newcommand\ZZ{\mathbb{Z}}

\newcommand\cA{{\mathcal{A}}}
\newcommand\cB{{\mathcal{B}}}

\newcommand\cF{{\mathcal{F}}}
\newcommand\cG{{\mathcal{G}}}

\newcommand\cV{{\mathcal{V}}}
\newcommand\cW{{\mathcal{W}}}
\newcommand\cX{{\mathcal{X}}}

\newcommand\fF{{\mathfrak{F}}}
\newcommand\fg{{\mathfrak{g}}}

\newcommand\fu{{\mathfrak{u}}}

\newcommand\sA{{\mathscr{A}}}
\newcommand\sB{{\mathscr{B}}}

\newcommand\sD{{\mathscr{D}}}
\newcommand\sE{{\mathscr{E}}}
\newcommand\sF{{\mathscr{F}}}
\newcommand\sG{{\mathscr{G}}}
\newcommand\sH{{\mathscr{H}}}
\newcommand\sI{{\mathscr{I}}}

\newcommand\sL{{\mathscr{L}}}
\newcommand\sM{{\mathscr{M}}}

\newcommand\sO{{\mathscr{O}}}

\newcommand\sS{{\mathscr{S}}}

\newcommand\sU{{\mathscr{U}}}

\newcommand\sX{{\mathscr{X}}}
\newcommand\sY{{\mathscr{Y}}}

\newcommand\bB{{\mathbf{B}}}

\newcommand\bH{{\mathbf{H}}}

\newcommand\bM{{\mathbf{M}}}

\newcommand\bx{{\mathbf{x}}}

\newcommand\by{{\mathbf{y}}}

\newcommand\bzero{{\mathbf{0}}}

\newcommand\eps{\varepsilon}

\newcommand\su{{\mathfrak{s}\mathfrak{u}}}

\newcommand\GL{\operatorname{GL}}
\newcommand\Or{\operatorname{O}}

\DeclareMathOperator{\PSL}{PSL}

\newcommand\SO{\operatorname{SO}}

\newcommand\Spin{\operatorname{Spin}}
\newcommand\SU{\operatorname{SU}}
\newcommand\U{\operatorname{U}}

\newcommand\less{\setminus}

\newcommand\ad{{\operatorname{ad}}}
\newcommand\Ad{{\operatorname{Ad}}}

\DeclareMathOperator{\Aut}{Aut}

\newcommand\Center{\operatorname{Center}}

\DeclareMathOperator{\Crit}{Crit}

\newcommand\diam{\operatorname{diam}}

\newcommand\dist{\operatorname{dist}}
\newcommand\divg{\operatorname{div}}
\newcommand\Dom{\operatorname{Dom}}

\newcommand{\esssup}{\operatornamewithlimits{ess\ sup}}

\newcommand\Exp{\operatorname{Exp}}

\newcommand\Hol{\operatorname{Hol}}
\newcommand\Hom{\operatorname{Hom}}

\DeclareMathOperator{\Ind}{Index}
\DeclareMathOperator{\Inj}{Inj}

\newcommand\Ker{\operatorname{Ker}}

\newcommand\Length{\operatorname{Length}}

\newcommand\Ran{\operatorname{Ran}}

\newcommand\Riem{\operatorname{Riem}}

\newcommand\Stab{\operatorname{Stab}}

\newcommand\tr{\operatorname{tr}}

\newcommand\vol{\operatorname{vol}}
\newcommand\Vol{\operatorname{Vol}}

\DeclareMathOperator{\YM}{\mathscr{Y}\!\!\mathscr{M}}
\DeclareMathOperator{\Zero}{Zero}

\newcommand\apriori{{\emph{a priori }}}
\newcommand\Apriori{{\emph{A priori }}}

\newcommand\dR{{\mathrm{dR}}}

\newcommand\id{{\mathrm{id}}}

\newcommand\irr{{\mathrm{irr}}}
\newcommand\loc{{\mathrm{loc}}}

\numberwithin{equation}{section}

\usepackage{amssymb}
\usepackage{bbm}
\usepackage{footmisc}
\usepackage{graphicx}
\usepackage{hyperref}
\usepackage{mathrsfs}
\usepackage[usenames]{color}
\usepackage{paralist}
\usepackage{slashed}
\usepackage{url}
\usepackage{verbatim}
\usepackage{xr}

\hypersetup{pdftitle={Morse theory for the Yang--Mills energy function near flat connections}}
\hypersetup{pdfauthor={Paul M. N. Feehan}}

\begin{document}

\title[Yang--Mills Morse theory near flat connections]{Morse theory for the Yang--Mills energy function near flat connections}

\author[Paul M. N. Feehan]{Paul M. N. Feehan}
\address{Department of Mathematics, Rutgers, The State University of New Jersey, 110 Frelinghuysen Road, Piscataway, NJ 08854-8019, United States}
\email{feehan@math.rutgers.edu}

\date{This version: January 4, 2024}

\begin{abstract}
A result \cite[Corollary 4.3]{UhlChern} due to Uhlenbeck (which we had attempted to reprove as \cite[Theorem 5.1]{Feehan_yangmillsenergygapflat}) asserts that the $W^{1,p}$-distance between the gauge-equivalence class of a connection $A$ and the moduli subspace of flat connections $M(P)$ on a principal $G$-bundle $P$ over a closed Riemannian manifold $X$ of dimension $d\geq 2$ is bounded by a constant times the $L^p$ norm of the curvature, $\|F_A\|_{L^p(X)}$, when $G$ is a compact Lie group, $F_A$ is $L^p$-small, and $p>d/2$. We prove that this estimate holds as stated when the Yang--Mills energy function on the space of Sobolev connections is Morse--Bott along the moduli subspace $M(P)$ of flat connections. However, we show that it does not hold when the Yang--Mills energy function fails to be Morse--Bott, such as at the product connection in the moduli space of flat $\mathrm{SU}(2)$ connections over a real two-dimensional torus. Nevertheless, we prove that a useful modification of Uhlenbeck's estimate always holds provided one replaces $\|F_A\|_{L^p(X)}$ by a suitable power $\|F_A\|_{L^p(X)}^\lambda$, where the positive exponent $\lambda$ reflects the structure of non-regular points in $M(P)$. The proof of our refinement involves gradient flow and Morse theory for the Yang--Mills energy function on the quotient space of Sobolev connections and a {\L}ojasiewicz distance inequality for the Yang--Mills energy function. Moreover, our method shows that $M(P)$ is (essentially) a strong deformation retract of a neighborhood in the quotient space of $W^{1,p}$ connections, generalizing a proof by {\L}ojasiewicz of a conjecture of Whitney that the zero set of a real analytic function on Euclidean space is a deformation retract of a neighborhood. A special case of our estimate, when $X$ has dimension four and the connection $A$ is anti-self-dual, was proved by Fukaya \cite{Fukaya_1998} by entirely different methods, apparently unaware of Uhlenbeck's earlier work. Lastly, we prove that if $A$ is a smooth Yang--Mills connection with small enough $L^2$ energy, then $A$ is necessarily flat.
\end{abstract}




\thanks{Paul Feehan was partially supported by National Science Foundation grant DMS-1510064.}

\maketitle
\tableofcontents

\section{Introduction}
\label{sec:Introduction}
Let $G$ be a compact Lie group and $P$ be a smooth principal $G$-bundle over a closed, smooth Riemannian manifold $(X,g)$ of dimension $d \geq 2$. In this monograph, we shall explore properties of the \emph{Yang--Mills energy function} \eqref{defn:Yang-Mills_energy_function}, 
\[
\YM(A) := \frac{1}{2}\int_X |F_A|^2\,d\vol_g,
\]
on the quotient $\sB^{1,p}(P)$ of the affine space $\sA^{1,p}(P)$ of Sobolev $W^{1,p}$ connections $A$ on $P$ by the Banach Lie group $\Aut^{2,p}(P)$ of $W^{2,p}$ automorphisms (or gauge transformations) of $P$, for admissible $p\in(1,\infty)$ (see \eqref{eq:Affine_space_W1p_connections} and \eqref{eq:Quotient_space_W1p_connections} and Definition \ref{defn:Admissible_Sobolev_exponent_for_Yang-Mills_energy}), where $F_A$ is the curvature of $A$.

The quotient $\sB^{1,p}(P)$ has the structure of a Banach stratified space, although the open subspace $\sB^{*;1,p}(P)$, of points $[A]$ where $A$ has stabilizer in $\Aut^{2,p}(P)$ isomorphic to the center of $G$, is a smooth Banach manifold (see Corollary \ref{cor:Slice}). The function $\YM$ typically fails to be Morse--Bott, due to non-regular points in the critical set of $\YM$ or excess dimension of the kernel of the Hessian of $\YM$ at a critical point. Our goal in this monograph is to derive certain useful properties of $\YM$ on $\sB^{1,p}(P)$ and their applications, despite its failure in general to be Morse--Bott.

To motivate one such application (see the forthcoming Theorem \ref{mainthm:Uhlenbeck_Chern_corollary_4-3}), we recall Uhlenbeck's Theorem on Existence of Local Coulomb Gauges \cite{UhlLp}: If $A$ is a connection on the product bundle $P=B\times G$ over the unit ball $B\subset\RR^d$ and $F_A$ obeys\footnote{If $d=2$, we require $\|F_A\|_{L^s(B)} \leq \eps$ for some $s>1$.} 
\begin{equation}
\label{eq:Intro_curvature_Lp_small}
\|F_A\|_{L^{d/2}(B)} \leq \eps,
\end{equation}
for small enough $\eps=\eps(d,G)\in(0,1]$, then there is a $W^{2,p}$ gauge transformation $u$ of $P$ such that
\begin{equation}
\label{eq:Intro_Uhlenbeck_local_Coulomb_gauge_W1p_estimate}
  \|u(A)-\Theta\|_{W^{1,p}(B)} \leq C\|F_A\|_{L^p(B)},
\end{equation}
where $C=C(d,G,p)$ and $\Theta$ is the product connection on $P$ and $d_\Theta^*(u(A)-\Theta)=0$. (See Section \ref{subsec:Existence_local_Coulomb_gauges} for a more complete statement and discussion.) It is natural to ask whether a global analogue of the local estimate \eqref{eq:Intro_Uhlenbeck_local_Coulomb_gauge_W1p_estimate} holds for a connection $A$ on an arbitrary smooth principal $G$-bundle $P$ over a closed, smooth Riemannian manifold $X$ of dimension $d\geq 2$,
\begin{equation}
\label{eq:Intro_Uhlenbeck_corollary_4-3_corrected_W1p_estimate}
\|u(A)-\Gamma\|_{W_\Gamma^{1,p}(X)} \leq C\|F_A\|_{L^p(X)}^\lambda, 
\end{equation}
where $\Gamma$ is a flat connection on $P$ and $d_\Gamma^*(u(A)-\Gamma)=0$ and $\lambda \in (0,1]$ is a constant.

Donaldson and Kronheimer \cite[Proposition 4.4.11]{DK} employ the local Coulomb gauge estimate \eqref{eq:Intro_Uhlenbeck_local_Coulomb_gauge_W1p_estimate} and a patching argument to prove that \eqref{eq:Intro_Uhlenbeck_corollary_4-3_corrected_W1p_estimate} holds with $\lambda=1$ when $X$ is \emph{strongly simply connected} and $p=2$ and $d=2,3$ and $\Gamma=\Theta$ but remark \cite[p. 163]{DK} that their result extends to $d = 4$ and $p>2$. In \cite[Proposition 4.4.11]{DK}, it is not claimed that $d_\Theta^*(u(A)-\Theta)=0$. Recall that $X$ is strongly simply connected \cite[p. 161]{DK} if it can be covered by smoothly embedded balls $B_1,\ldots,B_m$ such that for any $2 \leq r \leq m$, the intersection $B_r \cap (B_1\cup\cdots\cup B_{r-1})$ is connected; the condition implies that $X$ is simply connected.

Fukaya \cite[Proposition 3.1]{Fukaya_1998} proved that a version\footnote{Fukaya uses a different system of norms.} of \eqref{eq:Intro_Uhlenbeck_corollary_4-3_corrected_W1p_estimate} holds when $d=4$, and $X$ is a compact manifold with boundary, and $A$ is anti-self-dual, and $\lambda$ is determined by the character variety $\Hom(\pi_1(X),G)/G$ or, equivalently, the moduli space $M(X,G)$ of flat $G$-connections over $X$, each of which is known to be a real analytic variety (see Appendix \ref{sec:Counterexample_corollary_4-3_Uhlenbeck_1985}). Fukaya's proof of \cite[Proposition 3.1]{Fukaya_1998} uses the local Coulomb gauge estimate \eqref{eq:Intro_Uhlenbeck_local_Coulomb_gauge_W1p_estimate} and difficult patching argument. In \cite[Proposition 3.1]{Fukaya_1998}, it is not claimed that $d_\Gamma^*(u(A)-\Gamma)=0$. It is likely \cite{Fukaya_9-25-2018} that his argument extends to allow arbitrary dimensions $d\geq 2$, connections $A \in \sA^{1,p}(P)$, and the system of Sobolev norms in \eqref{eq:Intro_Uhlenbeck_corollary_4-3_corrected_W1p_estimate}. If $X$ is a compact manifold without boundary and $A$ is anti-self-dual and $\|F_A\|_{L^2(X)}$ is smaller than a constant that depends at most on $G$, then the Chern--Weil Theorem (see Milnor and Stasheff \cite[Appendix C]{MilnorStasheff}) would imply that $A$ is necessarily flat.

Nishinou \cite{Nishinou_2007} proved that a version\footnote{Nishinou uses a different system of norms.} of \eqref{eq:Intro_Uhlenbeck_corollary_4-3_corrected_W1p_estimate} holds when $X=\TT^2$ (the real two-dimensional torus) and $P = \TT^2\times\SU(2)$ and and $\Gamma$ is the product connection and $\lambda = 1/2$; in Appendix \ref{sec:Counterexample_corollary_4-3_Uhlenbeck_1985}, we describe an example due to Mrowka \cite{Mrowka_7-30-2018} which shows that \eqref{eq:Intro_Uhlenbeck_corollary_4-3_corrected_W1p_estimate} cannot hold in this setting for $\lambda>1/2$.

In \cite[Corollary 4.3]{UhlChern}, Uhlenbeck asserted that \eqref{eq:Intro_Uhlenbeck_corollary_4-3_corrected_W1p_estimate} holds when $\lambda=1$. However, as we noted above, that conclusion is contradicted by examples when $X$ is not simply connected \cite{Uhlenbeck_3-1-2019}. The incomplete statement of \cite[Corollary 4.3]{UhlChern} was a minor oversight and Uhlenbeck's celebrated work has served as a source of inspiration for us to try to better understand the Morse theory for $\YM$ near the moduli space of flat connections.

The estimate in \cite[Corollary 4.3]{UhlChern} is not used anywhere else in \cite{UhlChern}. However, Uhlenbeck compactness, especially in higher dimensions, remains an active area of exploration with numerous applications to a broad array of research directions, as evidenced by the Fall 2022 MSRI/SLMath program \emph{Analytic and geometric aspects of gauge theory} \cite{Fredrickson_et_al_analytic_geometric_aspects_gauge_theory} and the \emph{Simons Collaboration on Special Holonomy in Geometry, Analysis, and Physics} \cite{Bryant_simons_collaboration_special_holonomy}, building on work of Chen and Sun \cite{Chen_Sun_2020}, Chen and Wentworth \cite{Chen_Wentworth_2021arxiv}, Hong and Tian \cite{Hong_Tian_2004}, Jacob \cite{Jacob_2016}, Naber and Valtorta \cite{Naber_Valtorta_2019}, Nakajima \cite{Nakajima_1987, Nakajima_1988}, Sibley \cite{Sibley_2015}, Sibley and Wentworth \cite{Sibley_Wentworth_2015}, Smith and Uhlenbeck \cite{Smith_Uhlenbeck_2022}, \cite{TianGTCalGeom}, Tao and Tian \cite{Tao_Tian_2004}, and Waldron \cite{Waldron_2023}. A few examples include
\begin{inparaenum}[(\itshape i\upshape)]
\item work of Feehan and Leness \cite{Feehan_Leness_introduction_virtual_morse_theory_so3_monopoles} and Tian and Yang \cite{Tian_Yang_2002},  on the moduli space of \emph{(projective) vortices} over $2n$-dimensional almost Hermitian manifolds;   
\item work of Gagliardo and Uhlenbeck \cite{Gagliardo_Uhlenbeck_2012}, Mazzeo and Witten \cite{Mazzeo_Witten_2020} and others on solutions to the \emph{Kapustin--Witten equations} over four-dimensional manifolds;
\item work of Lotay and Oliveira \cite{Lotay_Oliveira_2022}, Walpuski \cite{WalpuskiThesis, Walpuski_2017} and others on $G_2$-\emph{instantons} over seven-dimensional manifolds;
\item work of Tanaka \cite{Tanaka_2012} and others on $\Spin(7)$-\emph{instantons} on eight-dimensional manifolds;
\item work of Tanaka \cite{Tanaka_2013, Tanaka_2014, Tanaka_2016} and others on \emph{Donaldson--Thomas instantons} over six-dimensional symplectic manifolds;
\item work of Tanaka \cite{Tanaka_2019} and others on solutions to the \emph{Vafa--Witten equations} over four-dimensional manifolds; and 
\item work of Taubes on $\PSL(2,\CC)$ connections on three-dimensional manifolds with $L^2$ bounds on curvature \cite{Taubes_2013cjm}.
\end{inparaenum}
For these reasons, it is important to provide a correction to the statement of \cite[Corollary 4.3]{UhlChern} and give a complete proof of the corrected statement that may have applications to higher-dimensional gauge theory in addition to those that we discuss in this monograph.

Morse theory for the Chern--Simons function is used to define instanton Floer homology groups \cite{DonFloer,Floer} of three-dimensional manifolds and plays an essential role in pulling-apart arguments used by Kronheimer and Mrowka in their proof of the structure of Donaldson invariants for four-dimensional manifolds \cite[Theorem 1.7]{KMStructure}. The case $\lambda=1$ in \eqref{eq:Intro_Uhlenbeck_corollary_4-3_corrected_W1p_estimate} is equivalent\footnote{In fact, $\lambda=(1-\theta)/\theta$.} to the {\L}ojasiewicz gradient inequality for $\YM$ holding with its optimal exponent $\theta=1/2$ and exponential decay of solutions to Yang--Mills gradient flow, while the case $\lambda<1$ is equivalent to the {\L}ojasiewicz gradient inequality for $\YM$ holding with suboptimal exponent $\theta\in(1/2,1)$ and negative power-law decay of solutions to Yang--Mills gradient flow (see \cite{Feehan_yang_mills_gradient_flow_v4, Feehan_Maridakis_Lojasiewicz-Simon_coupled_Yang-Mills, Rade_1992, Simon_1983}). This dichotomy holds in general and in particular for Chern--Simons gradient flow, as noted by Kronheimer and Mrowka in \cite[p. 617]{KMStructure} and based on work of Morgan, Mrowka, and Ruberman \cite{MMR}. 

The argument provided by Uhlenbeck in \cite{UhlChern} was very brief and that prompted us to attempt to reprove it in more detail as \cite[Theorem 5.1]{Feehan_yangmillsenergygapflat} using \eqref{eq:Intro_Uhlenbeck_local_Coulomb_gauge_W1p_estimate} and a patching argument, which is incorrect as explained in \cite{Feehan_yangmillsenergygapflat_corrigendum}. Fortunately, our main result \cite[Theorem 1]{Feehan_yangmillsenergygapflat} in that article (quoted here as Theorem \ref{thm:Ldover2_energy_gap_Yang-Mills_connections}) still follows from \eqref{eq:Intro_Uhlenbeck_corollary_4-3_corrected_W1p_estimate} when $\lambda<1$. In this monograph, we prove the correction \eqref{eq:Intro_Uhlenbeck_corollary_4-3_corrected_W1p_estimate} to the estimate in \cite[Corollary 4.3]{UhlChern}: see the forthcoming Theorems \ref{mainthm:Uhlenbeck_Chern_corollary_4-3_prelim} and \ref{mainthm:Uhlenbeck_Chern_corollary_4-3}. We prove Theorem \ref{mainthm:Uhlenbeck_Chern_corollary_4-3_prelim} using the traditional ingredients that were implicit in Uhlenbeck's approach to her \cite[Corollary 4.3]{UhlChern}, namely her Weak Compactness Theorem \cite[Theorem 1.5 or 3.6]{UhlLp} and \apriori $L^p$ estimates for the first-order elliptic operator
\[
  d_\Gamma+d_\Gamma^*:\Omega^1(X;\ad P)\to \Omega^2(X;\ad P)\oplus \Omega^1(X;\ad P),
\]
where $\Gamma$ is a smooth flat connection on $P$. To extend Theorem \ref{mainthm:Uhlenbeck_Chern_corollary_4-3_prelim} to the far stronger Theorem \ref{mainthm:Uhlenbeck_Chern_corollary_4-3}, we use
\begin{inparaenum}[(\itshape i\upshape)]
\item a {\L}ojasiewicz Gradient Inequality for the Yang--Mills energy function (see Theorem \ref{thm:Lojasiewicz-Simon_W-12_gradient_inequality_Yang-Mills_energy_function}),
\item global existence and convergence for the Yang--Mills gradient flow on a Coulomb-gauge slice near a local minimum (see the forthcoming Theorem \ref{mainthm:Yang-Mills_gradient_flow_global_existence_and_convergence_started_near_local_minimum}),  
\item and the resulting {\L}ojasiewicz Distance Inequality for the Yang--Mills energy function \cite{Feehan_yang_mills_gradient_flow_v4, Feehan_Maridakis_Lojasiewicz-Simon_coupled_Yang-Mills} near a local minimum (see Corollary \ref{maincor:Lojasiewicz_distance_inequality_Yang-Mills_energy_function}).
\end{inparaenum}

\subsection{Main results}
\label{subsec:Main_results}
We summarize our main results here, but defer a detailed discussion of notation and technical background to Section \ref{sec:Preliminaries}.

\subsubsection{Existence of a flat connection on a principal bundle supporting a connection with $L^p$-small curvature}
\label{subsubsec:Existence_flat_connection_principal_bundle_supporting_connection_Lp-small_curvature}
We begin with a weaker version of \cite[Corollary 4.3]{UhlChern} and one that can be proved using methods in \cite{UhlLp} and methods implicit in \cite{UhlChern}. For a principal $G$-bundle $P$ over a closed manifold $X$, we let $[P]$ denote its homotopy equivalence class.

\begin{mainthm}[Existence of a flat connection on a principal bundle supporting a connection with $L^p$-small curvature]
\label{mainthm:Uhlenbeck_Chern_corollary_4-3_prelim}
Let $G$ be a compact Lie group, $P$ be a smooth principal $G$-bundle over a closed, smooth Riemannian manifold $(X,g)$ of dimension $d \geq 2$, and $p \in (d/2,\infty)$, and $A_1$ be a $C^\infty$ reference connection on $P$. Then there is a constant $\eps=\eps(A_1,g,G,p,[P]) \in (0,1]$ with the following significance. If $A$ is a $W^{1,p}$ connection on $P$ with
\begin{equation}
\label{eq:Lp_norm_FA_lessthan_epsilon}
\|F_A\|_{L^p(X)} < \eps,
\end{equation}
then there are a $W^{1,p}$ flat connection $\Gamma$ on $P$ and a $W^{2,p}$ gauge transformation $v \in \Aut^{2,p}(P)$ such that
  \begin{align}
    \label{eq:Uhlenbeck_Chern_corollary_4-3_uA-Gamma_global_Coulomb_gauge_prelim}
    d_\Gamma^*(v(A) - \Gamma) &= 0 \quad\text{a.e. on } X,
    \\
    \label{eq:Uhlenbeck_Chern_corollary_4-3_uA-Gamma_W1p_bound_prelim}
    \|v(A)-\Gamma\|_{W_{A_1}^{1,r}(X)} &\leq C\|F_A\|_{L^r(X)} + C\|v(A)-\Gamma\|_{L^r(X)},
  \end{align}
for any $r \in (1,p]$ and corresponding constant $C=C(A_1,g,G,r) \in [1,\infty)$. If $d\geq 3$ or $d=2$ and
$p>4/3$, then we may assume that $\Gamma$ is $C^\infty$. If $\sigma \in (0,1]$ is a constant then, for a possibly smaller $\eps=\eps(A_1,g,G,p,[P],\sigma) \in (0,1]$, one has
    \begin{equation}
    \label{eq:Uhlenbeck_compactness_bound}
    \|v(A)-\Gamma\|_{W_{A_1}^{1,p}(X)} < \sigma.
  \end{equation}
\end{mainthm}

Our statement of Theorem \ref{mainthm:Uhlenbeck_Chern_corollary_4-3_prelim}, specifically the inequality \eqref{eq:Uhlenbeck_Chern_corollary_4-3_uA-Gamma_W1p_bound_prelim}, corrects one due to Uhlenbeck \cite[Corollary 4.3]{UhlChern} and the author in \cite[Theorem 5.1]{Feehan_yangmillsenergygapflat}. (In \cite{Feehan_yangmillsenergygapflat}, we had attempted to supply a detailed proof of \cite[Corollary 4.3]{UhlChern}, which was omitted in \cite{UhlChern}, but our argument was incorrect as we explain in \cite{Feehan_yangmillsenergygapflat_corrigendum}.)

One of the conclusions in Theorem \ref{mainthm:Uhlenbeck_Chern_corollary_4-3_prelim} --- that there exists a flat connection $\Gamma$ on $P$, given a $W^{1,p}$ connection on $P$ with
$\|F_A\|_{L^p(X)} < \eps$
--- was asserted by \cite[Corollary 4.3]{UhlChern}. However, the dependencies of the constant $\eps$ and how they could be determined were unclear from the proof outlined in \cite{UhlChern}.

However, \emph{unlike} \cite[Corollary 4.3]{UhlChern}, while we do assert that the estimates \eqref{eq:Uhlenbeck_Chern_corollary_4-3_uA-Gamma_W1p_bound_prelim} and \eqref{eq:Uhlenbeck_compactness_bound} hold (those inequalities are not included in \cite[Corollary 4.3]{UhlChern}), we do \emph{not} assert the stronger estimate
\[
  \|v(A)-\Gamma\|_{W_{A_1}^{1,p}(X)} \leq C\|F_A\|_{L^p(X)},
\]
stated in \cite[Corollary 4.3]{UhlChern}. Indeed, the preceding estimate
is false in general without additional hypotheses, as we explain in Appendix \ref{sec:Counterexample_corollary_4-3_Uhlenbeck_1985}. The Coulomb-gauge conclusion \eqref{eq:Uhlenbeck_Chern_corollary_4-3_uA-Gamma_global_Coulomb_gauge_prelim} is stated in the proof outlined by Uhlenbeck for \cite[Corollary 4.3]{UhlChern}. An example of where the preceding estimate does hold with an additional hypothesis is given by the forthcoming corollary and possibly this is what Uhlenbeck originally had in mind:

\begin{maincor}[Refined estimate near a flat connection with zero-dimensional Zariski tangent space]
\label{maincor:Uhlenbeck_Chern_corollary_4-3_H1Gamma_zero}
Continue the hypotheses of Theorem \ref{mainthm:Uhlenbeck_Chern_corollary_4-3_prelim}. If the Zariski tangent space, $H_\Gamma^1(X;\ad P)$ in \eqref{eq:H_Gamma^i_adP_W1p} or \eqref{eq:DeRham_cohomology_group_flat_connection} (for a $W^{1,p}$ or $C^\infty$ flat connection $\Gamma$, respectively), at the point $[\Gamma]$ in the moduli space $M(P)$ in \eqref{eq:Moduli_space_flat_connections} of flat connections on $P$ has dimension zero, then
\begin{equation}
\label{eq:Uhlenbeck_Chern_corollary_4-3_uA-Gamma_W1p_bound_refined}
    \|v(A)-\Gamma\|_{W_{A_1}^{1,r}(X)} \leq C\|F_A\|_{L^r(X)}.
\end{equation}  
\end{maincor}

Under the hypotheses of Theorem \ref{mainthm:Uhlenbeck_Chern_corollary_4-3_prelim}, our forthcoming Theorem \ref{mainthm:Uhlenbeck_Chern_corollary_4-3} provides the estimate \eqref{eq:Uhlenbeck_Chern_corollary_4-3_uA-Gamma_W1p_bound}, obtained by replacing $\|F_A\|_{L^r(X)}$ in \eqref{eq:Uhlenbeck_Chern_corollary_4-3_uA-Gamma_W1p_bound_refined} with $\|F_A\|_{L^r(X)}^\lambda$ for $\lambda\in(0,1]$ depending on $[\Gamma]$. Theorem \ref{mainthm:Uhlenbeck_Chern_corollary_4-3} also provides the estimate \eqref{eq:Uhlenbeck_Chern_corollary_4-3_uA-Gamma_W1p_bound_refined} (equivalent to \eqref{eq:Uhlenbeck_Chern_corollary_4-3_uA-Gamma_W1p_bound} with $\lambda=1$) while allowing $H_\Gamma^1(X;\ad P)$ to be non-zero and instead requiring that $H_\Gamma^2(X;\ad P) = (0)$, so $[\Gamma]$ is a regular point of $M(P)$. We recall from Ho, Wilkin, and Wu \cite[Lemma 2.1, p. 4]{Ho_Wilkin_Wu_2019} that if $X$ is orientable, then $H_\Gamma^i(X;\ad P)^* \cong H_\Gamma^{d-i}(X;\ad P)$ for $i=0,1,\ldots,d$. Thus, if $d =3$ and $H_\Gamma^1(X;\ad P) = (0)$, we also have $H_\Gamma^2(X;\ad P) = (0)$ by Poincar\'e duality.

Unlike the proof of Theorem \ref{mainthm:Uhlenbeck_Chern_corollary_4-3}, the proof of Corollary \ref{maincor:Uhlenbeck_Chern_corollary_4-3_H1Gamma_zero} is relatively straightforward (see Section \ref{subsec:Existence_flat_connection_principal_bundle_supporting_connection_Lp-small_curvature} for its proof and that of Theorem \ref{mainthm:Uhlenbeck_Chern_corollary_4-3_prelim}). See Example \ref{exmp:Optimal_estimate_over_simply_connected_manifolds} for a discussion of an elementary example of where the hypotheses of Corollary \ref{maincor:Uhlenbeck_Chern_corollary_4-3_H1Gamma_zero} hold, namely when $X$ is simply connected and $\Gamma$ is the product connection on $P = X\times G$. The real value of Corollary \ref{maincor:Uhlenbeck_Chern_corollary_4-3_H1Gamma_zero} is that it clearly illustrates that is far from obvious that an estimate such as \eqref{eq:Uhlenbeck_Chern_corollary_4-3_uA-Gamma_W1p_bound} in Theorem \ref{mainthm:Uhlenbeck_Chern_corollary_4-3} would hold, at least until one examines its proof based on the {\L}ojasiewicz gradient inequality.

\subsubsection{{\L}ojasiewicz distance inequality for functions on Banach spaces}
\label{subsec:Lojasiewicz_distance_inequality_hilbert_space}
We shall use a version of the {\L}ojasiewicz gradient inequality, called the \emph{{\L}ojasiewicz distance inequality}, for the Yang--Mills energy function (see Corollary \ref{maincor:Lojasiewicz_distance_inequality_Yang-Mills_energy_function}) in our proof of the forthcoming Theorem \ref{mainthm:Uhlenbeck_Chern_corollary_4-3}. Simon proved a {\L}ojasiewicz gradient inequality \cite[Theorem 3, Equation (2.2]{Simon_1983}, like the forthcoming \eqref{eq:Lojasiewicz_gradient_inequality}, and a gradient-distance inequality \cite[Theorem 3, Equation (2.1]{Simon_1983} for a certain analytic function on a Banach space of $C^{2,\alpha}$ sections of a Riemannian vector bundle over a closed, smooth Riemannian manifold using a Lyapunov--Schmidt reduction to the case of the finite-dimensional {\L}ojasiewicz gradient inequality \cite{Lojasiewicz_1965}. However, our forthcoming {\L}ojasiewicz distance inequalities \eqref{eq:Lojasiewicz_distance_inequality_critical_set} and \eqref{eq:Lojasiewicz_distance_inequality_zero_set} do not appear to be directly accessible to such a Lyapunov--Schmidt reduction.

\begin{thm}[{\L}ojasiewicz distance inequality for analytic functions on Euclidean space]
\label{thm:Lojasiewicz_distance_inequality}
(See {\L}ojasiewicz\footnote{The first page number refers to the version of {\L}ojasiewicz's original manuscript mimeographed by IHES while the page number in parentheses refers to the cited LaTeX version of his manuscript prepared by M. Coste and available on the Internet.} \cite[Theorem 2, p. 85 (62)]{Lojasiewicz_1965}.)
Let $n \geq 1$ be an integer, $U \subset \RR^n$ be an open neighborhood of the origin, and $f:U\to[0,\infty)$ be an analytic function. If $f(0) = 0$ and $f'(0) = 0$, then there are constants $C \in [1, \infty)$, and $\sigma \in (0,1]$, and $\delta\in(0,\sigma/4]$, and $\beta \in [1,\infty)$ such that
\begin{equation}
\label{eq:Lojasiewicz_distance_inequality}
f(x) \geq C\dist(x, B_\sigma\cap \Zero f)^\beta, \quad\text{for all } x \in B_\delta,
\end{equation}
where $\Zero f := f^{-1}(0)$ and $\dist(x,Z) := \inf\{\|x-z\|: z\in Z\}$ for any $Z\subset\RR^n$ and $B_r := \{x \in \RR^n: \|x\| < r\}$ for $r\in(0,\infty)$. 
\end{thm}

We refer the reader to Feehan \cite[Theorem 1 and Corollary 4]{Feehan_lojasiewicz_inequality_all_dimensions} for a simpler proof of Theorem \ref{thm:Lojasiewicz_distance_inequality} which is partly inspired by that of Bierstone and Milman\footnote{Bierstone and Milman omit a hypothesis in \cite[Theorem 2.8]{Bierstone_Milman_1997} that their analytic function is non-negative, although that assumption appears to be implicit in their proof of \cite[Theorem 2.8]{Bierstone_Milman_1997}.} \cite[Theorem 2.8]{Bierstone_Milman_1997}.

We refer to Feehan \cite[Section 1.1.3]{Feehan_lojasiewicz_inequality_all_dimensions} for a comparison of several types of {\L}ojasiewicz inequalities and a discussion of when one can be derived from another:
\medskip

\noindent\emph{Gradient} inequality with $\theta\in[1/2,1)$:
\begin{equation}
\label{eq:Gradient_inequality}
    \|f'(x)\| \geq C|f(x)|^\theta, \quad\text{for all } x\in B_\sigma.
\end{equation}
\emph{Distance-to-critical-set} inequality with $\alpha=1/(1-\theta)\in[2,\infty)$:
\begin{equation}
\label{eq:Distance-to-critical-set_inequality} 
    |f(x)| \geq C\dist(x,B_\sigma\cap\Crit f)^\alpha, \quad\text{for all } x\in B_\delta.
\end{equation}
\emph{Distance-to-zero-set} inequality with $\beta=\alpha/2\in[1,\infty)$:
\begin{equation}
\label{eq:Distance-to-zero-set_inequality}
    |f(x)| \geq C\dist(x,B_\sigma\cap\Zero f)^\beta, \quad\text{for all } x\in B_\delta.
\end{equation}
\emph{Gradient-distance} inequality with $\gamma=\theta/(1-\theta)\in[1,\infty)$:
\begin{equation}
\label{eq:Gradient_distance_inequality}
    \|f'(x)\| \geq C\dist(x,B_\sigma\cap\Crit f)^\gamma, \quad\text{for all } x\in B_\delta.
\end{equation}
The gradient-distance inequality \eqref{eq:Gradient_distance_inequality} is stated by Simon in \cite[Equation (2.3)]{Simon_1983} and attributed by him to {\L}ojasiewicz, but it is not stated in \cite{Lojasiewicz_1965} and it appears to be weaker than the distance-to-critical set inequality \eqref{eq:Distance-to-critical-set_inequality}.  

More generally, when $\sX$ and $\sH$ are infinite-dimensional Banach and Hilbert spaces, respectively, the existence (even for short time) of a solution to the gradient flow defined by a function, $\sF$, in the proof of the forthcoming Theorem \ref{mainthm:Lojasiewicz_distance_inequality_hilbert_space} is not automatic and one must explicitly include that as a hypothesis for an abstract gradient flow. When we specialize to the case of the Yang--Mills energy function \eqref{eq:Yang-Mills_energy_function}, we may appeal to the existence of solutions to its gradient flow provided by Feehan \cite[Theorem 6]{Feehan_yang_mills_gradient_flow_v4}, a variant of which we prove here as the forthcoming Theorem \ref{mainthm:Yang-Mills_gradient_flow_global_existence_and_convergence_started_near_local_minimum}.

\begin{mainthm}[{\L}ojasiewicz distance inequalities for functions on Banach spaces]
\label{mainthm:Lojasiewicz_distance_inequality_hilbert_space}
Let $\sX$ be a Banach space and $\sH$ be a Hilbert space such that $\sX \subset \sH$ is a continuous embedding and its adjoint $\sH^* \subset \sX^*$ is a continuous embedding of continuous dual spaces, and $\sU \subset \sX$ be an open subset. Let $\sF : \sU \to [0,\infty)$ be a function with continuous gradient map $\sF' : \sU \to \sH$ such that $\sF(0)=0$ and the following conditions hold:
\begin{enumerate}
\item[(a)] \emph{({\L}ojasiewicz gradient inequality.)}
\label{item:Lojasiewicz_gradient_inequality}
There are constants $C \in (0, \infty)$ and $\sigma \in (0,1]$ and $\theta \in [1/2,1)$ such that
\begin{equation}
\label{eq:Lojasiewicz_gradient_inequality}
\|\sF'(x)\|_\sH \geq C\sF(x)^\theta, \quad\text{for all } x \in B_\sigma,
\end{equation}
where $B_r := \{x \in \sX: \|x\|_\sX < r\}$ for $r >0$.
\item[(b)] \emph{(Solution to gradient flow.)}
  \label{item:Gradient_flow}
   There are a constant $\delta \in (0,\sigma/4]$ and, for each $x \in B_\delta$, a solution, $\bx \in C([0,\infty);\sX)\cap C^1((0,\infty);\sH)$, to
\begin{equation}
\label{eq:Gradient_flow}
\frac{d\bx}{dt} = -\sF'(\bx(t)) \quad \text{(in $\sH$) with } \bx(0)=x,
\end{equation}
such that $\bx(t) \in B_{\sigma/2}$ for all $t\in[0,\infty)$ and $\bx(t) \to \bx_\infty$ in $\sX$ as $t\to\infty$, where $\bx_\infty \in B_\sigma\cap\Crit\sF$ and $\Crit\sF := \{x\in\sU:\sF'(x)=0\}$.
\end{enumerate}
Let $\sE : \sU \to \RR$ be a function with continuous gradient map $\sE' : \sU \to \sH$.
\begin{enumerate}
\item \emph{(Distance to the critical and zero sets.)}
\label{item:Distance_critical_set}  
If $\sE\geq 0$ on $\sU$ and $\sE(0)=0$ and $\sF := \sE$ obeys \eqref{eq:Lojasiewicz_gradient_inequality} and \eqref{eq:Gradient_flow}, then there are constants $C_1 \in (0, \infty)$ and $\alpha = 1/(1-\theta) \in [2,\infty)$ such that
\begin{equation}
\label{eq:Lojasiewicz_distance_inequality_critical_set}
\sE(x) \geq C_1\dist_\sH(x, B_\sigma\cap \Crit\sE)^\alpha, \quad\text{for all } x \in B_\delta,
\end{equation}
where
\begin{equation}
\label{eq:Hilbert_distance_point_subset_Banach_space}
\dist_\sH(x,S) := \inf\{\|x-a\|_\sH: a\in S\},  
\end{equation}
for any point $x\in\sX$ and subset $S \subset \sX$. If in addition $B_\sigma\cap \Crit\sE \subset B_\sigma\cap \Zero\sE$, where $\Zero\sE := \{x\in\sU:\sE(x)=0\}$, then
\begin{equation}
\label{eq:Lojasiewicz_distance_inequality_zero_set}
\sE(x) \geq C_1\dist_\sH(x, B_\sigma\cap \Zero\sE)^\alpha, \quad\text{for all } x \in B_\delta.
\end{equation}

\item \emph{(Distance to the zero set.)} 
\label{item:Distance_zero_noncritical_set}
If $\sE(0)=0$ and $\sF := \sE^2$ obeys \eqref{eq:Lojasiewicz_gradient_inequality} and \eqref{eq:Gradient_flow} and $B_\sigma\cap \Crit\sF=B_\sigma\cap \Zero\sF$, then there are constants $C_2 \in (0, \infty)$ and $\beta = \alpha/2 \in [1,\infty)$ such that
\begin{equation}
\label{eq:Lojasiewicz_distance_inequality_zero_noncritical_set}
|\sE(x)| \geq C_2\dist_\sH(x, B_\sigma\cap \Zero\sE)^\beta, \quad\text{for all } x \in B_\delta.
\end{equation}
\end{enumerate}
\end{mainthm}

\begin{rmk}[On the hypothesis that $\Crit\sE^2=\Zero\sE^2$ in Item \eqref{item:Distance_zero_noncritical_set} of Theorem \ref{mainthm:Lojasiewicz_distance_inequality_hilbert_space}]
\label{rmk:Local_arc_connectivity}
Assume the notation of Theorem \ref{mainthm:Lojasiewicz_distance_inequality_hilbert_space}. For $\sF = \sE^2$, we have $\sF'(x)=2\sE(x)\sE'(x)$ for all $x\in\sU$. If $\sE(0)=0$, then $\sF(0)=0$ and $\sF'(0)=0$. For $x \in \sU$, we have $\sF'(x)=0 \iff \sE(x)=0$ or $\sE'(x)=0$, so $\Crit\sF = \Zero\sE \cup \Crit\sE$. In particular, we have that $\Crit\sF \supset \Zero\sF$.

If $\sE'(0)\neq 0$ then, after possibly shrinking $\sigma$, we may assume without loss of generality that $\sE'(x)\neq 0$ for all $x\in B_\sigma$ and thus $B_\sigma\cap \Crit\sF = B_\sigma\cap \Zero\sE$. However, if $\sE'(0)= 0$ and $\Crit\sE$ is locally arc-connected by continuous piecewise-$C^1$ arcs then, after possibly decreasing $\sigma$, we obtain that $B_\sigma\cap \Crit\sE$ is arc connected by continuous piecewise-$C^1$ arcs and hence $B_\sigma\cap \Crit\sE \subset B_\sigma\cap \Zero\sE$. Thus, $B_\sigma\cap \Crit\sF = B_\sigma\cap \Zero\sE$ when $\Crit\sE$ is locally arc-connected by continuous piecewise-$C^1$ arcs.
\end{rmk}

When $\sX=\sH=\RR^d$, Theorem \ref{mainthm:Lojasiewicz_distance_inequality_hilbert_space} was stated by {\L}ojasiewicz in \cite[Corollary, p. 88]{Lojasiewicz_1963} and proved by him in \cite{Lojasiewicz_1959, Lojasiewicz_1961} and in \cite{Lojasiewicz_1965}, with simplified proofs provided by Bierstone and Milman as
\cite[Theorem 6.4 and Remark 6.5]{BierstoneMilman} and
\cite[Theorem 2.8]{Bierstone_Milman_1997}. When $\sE$ is a polynomial on $\RR^d$, Theorem \ref{mainthm:Lojasiewicz_distance_inequality_hilbert_space} is due to H\"ormander \cite[Lemma 1]{Hormander_1958}. We reproved \cite[Theorem 2.8]{Bierstone_Milman_1997} as \cite[Corollary 4]{Feehan_lojasiewicz_inequality_all_dimensions} and provided details supplementing those given by Bierstone and Milman; see also {\L}ojasiewicz \cite{Lojasiewicz_1984}.

\subsubsection{Global existence and convergence of Yang--Mills gradient flow on a Coulomb-gauge slice around a local minimum}
\label{subsubsec:Global_existence_convergence_Yang-Mills_gradient_flow_slice_local_minimum}
We have the following variant of \cite[Theorem 6]{Feehan_yang_mills_gradient_flow_v4}, which does not require that the Yang--Mills gradient flow be restricted to a Coulomb-gauge slice through the local minimum but uses more restrictive Sobolev norms.

\begin{mainthm}[Global existence and convergence of Yang--Mills gradient flow on a Coulomb-gauge slice around a local minimum]
\label{mainthm:Yang-Mills_gradient_flow_global_existence_and_convergence_started_near_local_minimum}
Let $G$ be a compact Lie group and $A_1$ and
$A_{\min}$ be $C^\infty$ connections on a smooth principal $G$-bundle $P$ over a closed, connected, oriented, smooth Riemannian manifold $(X,g)$ of dimension $d\geq 2$, where $A_1$ serves as a reference connection in the definition of Sobolev and H\"older norms and $A_{\min}$ is a local minimum of the Yang--Mills energy function \eqref{eq:Yang-Mills_energy_function}, and let $p \in [2,\infty)$ obey $p > d/2$. Then there are constants $C \in (0,\infty)$, and $\sigma \in (0,1]$, and $\theta \in [1/2,1)$, depending on $A_1,A_{\min},g,p$ with the following significance.
\begin{enumerate}
\item \emph{Global existence and uniqueness:}
\label{item:Global_existence_uniqueness} There is a constant $\eps = \eps(A_1,A_{\min},g,p) \in (0,\sigma/4)$ with the following significance. If $A_0$ is a $W^{1,p}$ connection on $P$ such that
\begin{equation}
\label{eq:Initial_comnnection_W1p_close_local_minimum}
\|A_0 - A_{\min}\|_{W_{A_1}^{1,p}(X)} < \eps,
\end{equation}
then there is a unique solution, $A(t) = A_{\min} + a(t)$ for $t\in [0,\infty)$, with
\begin{align*}
  a &\in C([0,\infty); \Ker d_{A_{\min}}^*\cap W_{A_1}^{1,p}(T^*X\otimes\ad P))
  \\
  &\qquad \cap C^\infty((0,\infty); \Ker d_{A_{\min}}^*\cap\Omega^1(X;\ad P)),
\end{align*}
to the equation \eqref{eq:Yang-Mills_gradient_flow_slice} for Yang--Mills gradient flow on a Coulomb-gauge slice through $A_{\min}$ with initial data $A(0) = A_0$ and obeying
\begin{equation}
\label{eq:A_near_Amin_all_time}
\|A(t) - A_{\min}\|_{W_{A_1}^{1,p}(X)} < \sigma/2, \quad\text{for all } t \in [0,\infty).
\end{equation}

\item \emph{Continuity with respect to initial data:}
\label{item:Continuity_respect_initial_data} The solution $A(\cdot)$ varies continuously with respect to $A_0$ in the $C([0,\infty); \Ker d_{A_{\min}}^*\cap W_{A_1}^{1,p}(X;T^*X\otimes\ad P))$ topology.

\item \emph{Convergence:}
\label{item:Convergence} As $t\to\infty$, the solution $A(t)$ converges strongly with respect to the norm on $W_{A_1}^{1,p}(X;T^*X\otimes\ad P)$ to a Yang--Mills connection $A_\infty$ of class $W^{1,p}$ on $P$, and the gradient flowline has finite length in the sense that
$$
\int_0^\infty \left\|\frac{\partial A}{\partial t}\right\|_{W_{A_1}^{1,p}(X)}\,dt < \infty.
$$
If $A_\infty$ is a cluster point of the orbit $O(A) = \{A(t):t\geq 0\}$, then $A_\infty = A_{\min}$.

\item \emph{Convergence rate:} 
\label{item:Convergence_rate} For all $t \geq 1$, 
\begin{multline}
\label{eq:Convergence_rate}
\|A(t) - A_\infty\|_{W_{A_1}^{1,p}(X)}
\\
\leq
\begin{cases}
\displaystyle
\frac{1}{C(1-\theta)}\left(C^2(2\theta-1)(t-1) + (\YM(A_0)-\YM(A_\infty))^{1-2\theta}\right)^{-(1-\theta)/(2\theta-1)},&
\\
\qquad\text{if } 1/2 < \theta < 1,&
\\
\displaystyle
\frac{2}{C}(\YM(A_0)-\YM(A_\infty))^{1/2}\exp(-C^2(t-1)/2),&
\\
\qquad\text{if } \theta = 1/2.&
\end{cases}
\end{multline}

\item \emph{Stability:}
\label{item:Stability}
As an equilibrium of the Yang--Mills gradient flow \eqref{eq:Yang-Mills_gradient_flow_slice} on a Coulomb-gauge slice through $A_{\min}$, the point $A_\infty$ is \emph{Lyapunov stable} and if $A_\infty$ is isolated or a cluster point of the orbit $O(A)$, then $A_\infty$ is \emph{uniformly asymptotically stable}.\footnote{See Sell and You \cite[Section 2.3.3]{Sell_You_2002} for definitions of these concepts of stability.}
\end{enumerate}
\end{mainthm}

\begin{rmk}[Comparison with results on Yang--Mills gradient flow due to R\r{a}de, Kozono, Maeda, and Naito, and Schlatter and Struwe]
\label{rmk:Results_Yang-Mills_gradient_flow_2_leq_d_leq_4}  
R\r{a}de \cite[Theorems 1 and 2]{Rade_1992} established a version of Theorem \ref{mainthm:Yang-Mills_gradient_flow_global_existence_and_convergence_started_near_local_minimum} for Yang--Mills gradient flow \eqref{eq:Yang-Mills_gradient_flow} when $d=2$ or $3$, without any constraint on the initial energy $\YM(A_0)$ or norm constraint on the initial data $A_0$ like \eqref{eq:Initial_comnnection_W1p_close_local_minimum} or a restriction of the flow to a Coulomb-gauge slice. His proof of local existence is very different from ours and does not use the theory of analytic semigroups or the Donaldson--DeTurck trick \cite{DonASD, DK, DeTurck_1983} used by Donaldson and later by Struwe \cite{Struwe_1994}. When $d=4$, Kozono, Maeda, and Naito \cite[Corollary 5.7]{Kozono_Maeda_Naito_1995} established the global existence and convergence of Yang--Mills gradient flow \eqref{eq:Yang-Mills_gradient_flow} when the norm constraint \eqref{eq:Initial_comnnection_W1p_close_local_minimum} on $A_0$ is replaced by one that the initial energy $\YM(A_0)$ is sufficiently small.

When $d=4$, the constant $\eps$ in the forthcoming equation \eqref{eq:Initial_connection_small_energy} may be chosen small enough to exclude the possibility of energy bubbling. Consequently, according to Schlatter \cite[Theorems 1.2 and 1.3]{Schlatter_1997} and Struwe \cite[Theorems 2.3 and 2.4]{Struwe_1994} (see also Schlatter \cite[Theorems 1.2 and 1.3]{Schlatter_1997} for the completion of the proof of \cite[Theorem 2.4]{Struwe_1994}) or Kozono, Maeda, and Naito \cite[Theorem 5.6 and Corollary 5.7]{Kozono_Maeda_Naito_1995}), there is a global solution $A=A_0+a$ to the Yang--Mills gradient flow equation \eqref{eq:Yang-Mills_gradient_flow}, where
\[
a \in C([0,\infty); W_{A_1}^{1,2}(X;T^*X\otimes\ad P))\cap C^1((0,\infty); L^2(X;T^*X\otimes\ad P)),
\]
when the norm constraint \eqref{eq:Initial_comnnection_W1p_close_local_minimum} on $A_0$ is replaced by one that the initial energy $\YM(A_0)$ is sufficiently small. These authors also prove subsequential convergence modulo gauge transformations to a limiting Yang--Mills connection $A_\infty$ on $P$. However, Theorem \ref{mainthm:Yang-Mills_gradient_flow_global_existence_and_convergence_started_near_local_minimum} (like the corresponding result of R\r{a}de \cite{Rade_1992} when $d=2$ or $3$) provides a far stronger global existence and convergence result than those of \cite{Kozono_Maeda_Naito_1995, Schlatter_1997, Struwe_1994}.

The results of Kozono, Maeda, and Naito \cite[Theorem 5.6 and Corollary 5.7]{Kozono_Maeda_Naito_1995}, Schlatter \cite[Lemma 2.4]{Schlatter_1997}, and Struwe \cite[See Lemma 3.6]{Struwe_1994} that can be used to prove global existence for Yang--Mills gradient flow \eqref{eq:Yang-Mills_gradient_flow} using a careful local analysis of the flow do not obviously extend to Yang--Mills gradient flow \eqref{eq:Yang-Mills_gradient_flow_slice} on a Coulomb-gauge slice due to the presence of the $L^2$-orthogonal projection $\Pi_{A_\infty}$ in the definition of \eqref{eq:Yang-Mills_gradient_flow_slice}.
\end{rmk}

It is worth noting that Theorem \ref{mainthm:Yang-Mills_gradient_flow_global_existence_and_convergence_started_near_local_minimum} does not contradict an important result due to Naito:

\begin{thm}[Finite-time blow-up for Yang-Mills gradient flow over a sphere of dimension greater than or equal to five]
\label{thm:Naito_1-3}
(See \cite[Theorem 1.3]{Naito_1994}.)
Let $G \subset\SO(n)$ be a compact Lie group and $P$ be a non-trivial principal $G$-bundle over the sphere $S^d$ of dimension $d\geq 5$ and with its standard round Riemannian metric of radius one. Then there is a constant $\eps \in (0,1]$ with the following significance. If $A_0$ is a $C^\infty$ connection on $P$ such that $\|F_{A_0}\|_{L^2(S^d)} < \eps$, then the solution $A(t)$ to Yang-Mills gradient flow \eqref{eq:Yang-Mills_gradient_flow} with initial data, $A(0) = A_0$, blows up in finite time.
\end{thm}

When $d=2$ or $3$, the hypothesis \eqref{eq:Initial_comnnection_W1p_close_local_minimum} in Theorem \ref{mainthm:Yang-Mills_gradient_flow_global_existence_and_convergence_started_near_local_minimum} can be relaxed.

\begin{maincor}[Global existence and convergence of Yang--Mills gradient flow on a slice for initial connections with small energy over low-dimensional manifolds]
\label{maincor:Yang-Mills_gradient_flow_global_existence_and_convergence_started_small_energy}
Assume the hypotheses of Theorem \ref{mainthm:Yang-Mills_gradient_flow_global_existence_and_convergence_started_near_local_minimum} but restrict to $d=2$ or $3$ and replace the hypothesis \eqref{eq:Initial_comnnection_W1p_close_local_minimum} by
\begin{equation}
\label{eq:Initial_connection_small_energy}
\|F_{A_0}\|_{L^2(X)} \leq \eps.
\end{equation}
Then after excluding \eqref{eq:A_near_Amin_all_time}, the conclusions of Theorem \ref{mainthm:Yang-Mills_gradient_flow_global_existence_and_convergence_started_near_local_minimum} continue to hold with $A_\infty$ replaced by a $C^\infty$ flat connection $\Gamma$ and $A_0$ replaced by $u(A_0)$, where $u\in \Aut^{2,p}(P)$ with $p>d/2$.
\end{maincor}

\begin{rmk}[Extension of Corollary \ref{maincor:Yang-Mills_gradient_flow_global_existence_and_convergence_started_small_energy} to higher dimensions and critical Sobolev exponent]
\label{rmk:Yang-Mills_gradient_flow_global_existence_and_convergence_started_small_Ldover2_curvature}  
We expect Corollary \ref{maincor:Yang-Mills_gradient_flow_global_existence_and_convergence_started_small_energy} to continue to hold when $d\geq 3$ and the small initial energy condition \eqref{eq:Initial_connection_small_energy} is replaced by a condition that
\[
  \|F_{A_0}\|_{L^{d/2}(X)} \leq \eps.
\]
Indeed, by applying Feehan \cite[Theorem 2.14]{Feehan_lojasiewicz_inequality_ground_state}, both the Uhlenbeck Compactness Theorem \ref{thm:Metric_Uhlenbeck_compactness} and Lemma \ref{lem:Finite_open_covering} should hold as stated for $d\geq 3$ and $p=d/2$ if $b \in (0,\infty)$ is replaced by a small enough constant $\eps = \eps(g,G) \in (0,1]$. This is discussed in Feehan \cite{Feehan_lojasiewicz_inequality_ground_state}.
\end{rmk}

\subsubsection{Almost strong deformation retraction of a neighborhood in the quotient space of Sobolev connections onto the moduli subspace of flat connections}
\label{subsubsec:Retraction_nbhd_onto_moduli_space_flat_connections}
For $b\in (0,\infty)$ and $p\in (d/2,\infty)$ and $r\in [1,p]$, we set
\begin{equation}
  \label{eq:Quotient_space_W1p_connections_Lr_bound_curvature}
  \sB_b^{1,p}(P,g,r) := \left\{[A] \in \sB^{1,p}(P): \|F_A\|_{L^r(X)} \leq b\right\}.
\end{equation}
Recall that $W^{1,p}(X) \subset L^{2p}(X)$ is a continuous Sobolev embedding by \cite[Theorem 4.12]{AdamsFournier} if $p^*=dp/(d-p)\geq 2p$, that is, $d\geq 2(d-p)$ or equivalently, $p\geq d/2$. Hence, if $A$ is a connection one-form on a product $G$-bundle and $F_A=dA+\frac{1}{2}[A,A]$, then $|F_A| \in L^p(X)$ and the subspace $\sB_b^{1,p}(P,g,r)$ is well-defined for $r\leq p$. We now apply Theorem \ref{mainthm:Yang-Mills_gradient_flow_global_existence_and_convergence_started_near_local_minimum} when $d\geq 4$ and Corollary \ref{maincor:Yang-Mills_gradient_flow_global_existence_and_convergence_started_small_energy} when $d=2$ or $3$ to obtain

\begin{maincor}[Almost strong deformation retraction of a neighborhood in the quotient space of Sobolev connections onto the moduli subspace of flat connections]
\label{maincor:Retraction_open_nbhd_onto_moduli_space_flat_connections}
Let $G$ be a compact Lie group and $P$ be a smooth principal $G$-bundle over a closed, smooth Riemannian manifold $(X,g)$ of dimension $d \geq 2$, and $p\in (d/2,\infty)$, and
\begin{inparaenum}[(\itshape a\upshape)]
\item $d/2<r\leq p$ if $d\geq 4$, or
\item $r=2$ if $d=2$ or $3$.
\end{inparaenum}  
Then there is a constant $\eps = \eps(g,G,P,p,r) \in (0, 1]$ and a continuous map
\begin{equation}
\label{eq:Almost_strong_deformation_retraction}
  H:\sB_\eps^{1,p}(P,g,r) \times [0,1] \to \sB^{1,p}(P,g,r),
\end{equation}
where $\sB_\eps^{1,p}(P,g,r)$ is as in \eqref{eq:Quotient_space_W1p_connections_Lr_bound_curvature}, such that for all $[A]\in\sB_\eps^{1,p}(P,g,r)$,
\begin{multline*}
  H([A],0) = [A], \quad H([A],1) \in M(P), \quad\text{and}\quad H([\Gamma],t) = [\Gamma],
  \\
  \quad\text{for all } [\Gamma]\in M(P)\text{ and } t \in [0,1].
\end{multline*}
When $d=2$ or $3$, then $r=2$ and $\sB^{1,p}(P,g,r)$ in \eqref{eq:Almost_strong_deformation_retraction} can be replaced by $\sB_\eps^{1,p}(P,g,2)$ and the map $H$ is a strong deformation retraction (as in Definition \ref{defn:Strong_deformation_retraction}).
\end{maincor}

See R\r{a}de \cite[Corollary, p. 128]{Rade_1992} for a more general version of Corollary \ref{maincor:Retraction_open_nbhd_onto_moduli_space_flat_connections} when $d=2$ or $3$. It is very likely that the map $H$ in  \eqref{eq:Almost_strong_deformation_retraction} is also a strong deformation retraction when $d=4$, by applying methods of \cite{Feehan_lojasiewicz_inequality_ground_state}, so $\sB^{1,p}(P,g,r)$ in \eqref{eq:Almost_strong_deformation_retraction} can be replaced by $\sB_\eps^{1,p}(P,g,2)$. When $d\geq 5$, while we know that the $L^2$ norm of $F_{A(t)}$ is non-increasing with respect to $t\in[0,\infty)$ for Yang--Mills gradient flow $A(t)$, we do not know that the $L^r$ norm of $F_{A(t)}$ is non-increasing with respect to $t\in[0,\infty)$ when $r>2$.

\subsubsection{{\L}ojasiewicz distance inequality for the Yang--Mills energy function}
\label{subsubsec:Lojasiewicz_distance_inequalities_Yang-Mills_energy_function}
In Section \ref{subsec:Yang-Mills_heat_equation}, we introduce the restriction $\widehat\YM$ in \eqref{eq:Yang-Mills_energy_function_slice} of the Yang--Mills energy function $\YM$ in \eqref{eq:Yang-Mills_energy_function} to a Coulomb-gauge slice through a smooth connection $A_\infty$,
\[
  \widehat\YM: A_\infty + \Ker d_{A_\infty}^*\cap W_{A_1}^{1,p}(X;T^*X\otimes\ad P) \to \RR,
\]
and note that its gradient is $\widehat{\YM'}(A) = \Pi_{A_\infty}d_A^*F_A$ in \eqref{eq:Yang-Mills_energy_function_gradient_slice}, where $\Pi_{A\infty}$ is $L^2$-orthogonal projection onto the Coulomb-gauge slice. By applying Theorems \ref{mainthm:Lojasiewicz_distance_inequality_hilbert_space} and \ref{mainthm:Yang-Mills_gradient_flow_global_existence_and_convergence_started_near_local_minimum}, we obtain 

\begin{maincor}[{\L}ojasiewicz distance inequalities for the Yang--Mills energy function on a Coulomb-gauge slice]
\label{maincor:Lojasiewicz_distance_inequality_Yang-Mills_energy_function_slice}
Let $G$ be a compact Lie group, $P$ be a smooth principal $G$-bundle over a closed, smooth Riemannian manifold $(X,g)$ of dimension $d \geq 2$, and $p \in [2,\infty)$ obey $p>d/2$, and $A_1$ be a $C^\infty$ reference connection on $P$ for the definition of Sobolev and H\"older norms. If $A_{\min}$ is a $C^\infty$ connection on $P$ that is a local minimum for the Yang--Mills energy function $\YM$ in \eqref{eq:Yang-Mills_energy_function} (equivalently, for $\widehat\YM$ in \eqref{eq:Yang-Mills_energy_function_slice}), then
there is a constant $C=C(A_1,A_{\min},G,g,p) \in (0, \infty)$ such that
\begin{multline}
\label{eq:Lojasiewicz_distance_inequality_critical_set_Yang-Mills_energy_function_slice}
\widehat\YM(A)-\widehat\YM(A_{\min}) \geq C\,\mathbf{dist}_{L^2(X)}\left(A, \bB_\sigma(A_{\min})\cap \Crit\widehat\YM\right)^\alpha,
\\
\text{for all } A \in \bB_\delta(A_{\min}),
\end{multline}
where $\alpha = 1/(1-\theta) \in [2,\infty)$ and $\theta=\theta(A_1,A_{\min},G,g,p)\in[1/2,1)$ and $\sigma=\sigma(A_1,A_{\min},G,g,p)\in(0,1]$ are as in Theorem \ref{thm:Lojasiewicz-Simon_W-12_gradient_inequality_Yang-Mills_energy_function_slice} and $\delta\in(0,\sigma/4]$ and
\begin{equation}
\label{eq:Distance_point_subset_affine_space_slice}
\mathbf{dist}_{L^2(X)}(A,S) := \inf_{A' \in S}\|A-A'\|_{L^2(X)},
\end{equation}
for any subset $S \subset A_{\min}+\Ker d_{A_{\min}}^*\cap W_{A_1}^{1,p}(X;T^*X\otimes\ad P)$ and for any $R\in(0,\infty)$,
\begin{equation}
\label{eq:W1p_connections_Coulomb_gauge_ball}
\bB_R(A_{\min}) := A_{\min}+\left\{a \in \Ker d_{A_{\min}}^*\cap W_{A_1}^{1,p}(X;T^*X\otimes\ad P): \|a\|_{W_{A_1}^{1,p}(X)} < R\right\}.
\end{equation}
If $\widehat\YM$ is \emph{Morse--Bott} at $A_{\min}$ in the sense of Definition \ref{defn:Definition_Morse-Bott_quotient_space}, then $\alpha=2$ in \eqref{eq:Lojasiewicz_distance_inequality_critical_set_Yang-Mills_energy_function_slice}. If $A_{\min}$ is a \emph{flat} connection, then
\begin{equation}
\label{eq:Lojasiewicz_distance_inequality_zero_set_Yang-Mills_energy_function_slice}
\widehat\YM(A) \geq C\,\mathbf{dist}_{L^2(X)}\left(A, \bB_\sigma(A_{\min})\cap \Zero\widehat\YM\right)^\alpha, \quad\text{for all } A \in \bB_\delta(A_{\min}).
\end{equation}
\end{maincor}

\begin{maincor}[{\L}ojasiewicz distance inequalities for the Yang--Mills energy function]
\label{maincor:Lojasiewicz_distance_inequality_Yang-Mills_energy_function}
Continue the hypotheses and notation of Corollary \ref{maincor:Lojasiewicz_distance_inequality_Yang-Mills_energy_function_slice}. Then for constants $\alpha = 1/(1-\theta) \in [2,\infty)$ and $\theta=\theta(A_1,A_{\min},G,g,p)\in[1/2,1)$ and $\sigma=\sigma(A_1,A_{\min},G,g,p)$ as in Theorem \ref{thm:Lojasiewicz-Simon_W-12_gradient_inequality_Yang-Mills_energy_function_slice} and $\zeta = \zeta(A_1,A_{\min},G,g,p)\in(0,1]$ as in Theorem \ref{thm:Feehan_proposition_3-4-4_Lp}, and small enough $\eta=\eta(\zeta,\sigma)\in(0,1]$, we have
\begin{equation}
\label{eq:Lojasiewicz_distance_inequality_critical_set_Yang-Mills_energy_function}
\YM(A)-\YM(A_{\min}) \geq C\dist_{L^2(X)}\left(A, B_\sigma(A_{\min})\cap \Crit\YM\right)^\alpha, \quad\text{for all } A \in B_\eta(A_{\min}),
\end{equation}
where 
\begin{equation}
\label{eq:Distance_point_subset_affine_space}
\dist_{L^2(X)}(A,S) := \inf_{\begin{subarray}{c}A'\in S,\\ u\in\Aut^{2,p}(P)\end{subarray}}\|u(A)-A'\|_{L^2(X)},
\end{equation}
for any subset $S \subset \sA^{1,p}(P)$ and for any $R\in(0,\infty)$,
\begin{equation}
\label{eq:W1p_connections_ball}
B_R(A_{\min}) := A_{\min}+\left\{a \in W_{A_1}^{1,p}(X;T^*X\otimes\ad P): \|a\|_{W_{A_1}^{1,p}(X)} < R\right\}.
\end{equation}
If $\YM$ is \emph{Morse--Bott} at $A_{\min}$ in the sense of Definition \ref{defn:Definition_Morse-Bott_affine_space}, then $\alpha=2$ in \eqref{eq:Lojasiewicz_distance_inequality_critical_set_Yang-Mills_energy_function}. If $A_{\min}$ is a \emph{flat} connection, then
\begin{equation}
\label{eq:Lojasiewicz_distance_inequality_zero_set_Yang-Mills_energy_function}
\YM(A) \geq C\dist_{L^2(X)}\left(A, B_\sigma(A_{\min})\cap \Zero\YM\right)^\alpha, \quad\text{for all } A \in B_\eta(A_{\min}).
\end{equation}
\end{maincor}

\subsubsection{$L^{d/2}$ and $L^2$ energy gaps for Yang--Mills connections}
\label{subsubsec:Energy_gap_Yang-Mills_connections}
We recall our

\begin{thm}[$L^{d/2}$-energy gap for Yang--Mills connections]
\label{thm:Ldover2_energy_gap_Yang-Mills_connections}
(See Feehan \cite[Theorem 1]{Feehan_yangmillsenergygapflat} and \cite{Feehan_yangmillsenergygapflat_corrigendum}.)  
Let $G$ be a compact Lie group and $P$ be a smooth principal $G$-bundle over a closed, smooth Riemannian manifold $(X,g)$ of dimension $d \geq 2$. Then there is a positive constant, $\eps = \eps(g,G,[P]) \in (0, 1]$, with the following significance. If $A$ is a $C^\infty$ Yang--Mills connection on $P$ and its curvature, $F_A$, obeys
\begin{equation}
\label{eq:Curvature_Ldover2_small}
\|F_A\|_{L^{d/2}(X)} \leq \eps,
\end{equation}
then $A$ is a flat connection.
\end{thm}

Our corrigendum \cite{Feehan_yangmillsenergygapflat_corrigendum} includes the slight modification of our proof of \cite[Theorem 1]{Feehan_yangmillsenergygapflat} required when we replace our appeal to Uhlenbeck's \cite[Corollary 4.3]{UhlChern} with one to our forthcoming Theorem \ref{mainthm:Uhlenbeck_Chern_corollary_4-3}.

\begin{rmk}[Alternative approaches to the proof of Theorem \ref{thm:Ldover2_energy_gap_Yang-Mills_connections}]
\label{rmk:Huang_proof_Ldover2_energy_gap_Yang-Mills_connections}  
In \cite[Section 3]{Huang_2017}, Huang gives a proof of Theorem \ref{thm:Ldover2_energy_gap_Yang-Mills_connections} (see his \cite[Theorem 1.1]{Huang_2017}) that is valid when Theorem \ref{mainthm:Uhlenbeck_Chern_corollary_4-3} holds with exponent $\lambda=1$, but the argument breaks down if $\lambda < 1$ is small enough. Indeed, rather than the inequality stated on \cite[p. 913]{Huang_2017}, 
  \[
    \|F_A\|_{L^2(X)}^2 \leq C\|F_A\|_{L^2(X)}^3,
  \]
which leads to a contradiction when $\|F_A\|_{L^2(X)} < 1/C$, where $C=C(g,P)$, one instead only obtains
  \[
    \|F_A\|_{L^2(X)}^2 \leq C\|F_A\|_{L^2(X)}^{1+2\lambda},
  \]
and there is no contradiction when $\lambda \leq 1/2$ and $\|F_A\|_{L^2(X)} \leq \delta$, regardless how small one chooses $\delta \in (0,1]$.
\end{rmk}  

In Section \ref{sec:Energy_gap_Yang-Mills_connections}, we shall give a proof of the forthcoming more general Theorem \ref{mainthm:L2_energy_gap_Yang-Mills_connections} that is quite different from the proof of Theorem \ref{thm:Ldover2_energy_gap_Yang-Mills_connections} that we gave in \cite{Feehan_yangmillsenergygapflat, Feehan_yangmillsenergygapflat_corrigendum}. Our proof in Section \ref{sec:Energy_gap_Yang-Mills_connections} does not use the {\L}ojasiewicz gradient inequality but rather instead relies on analyticity of the Yang--Mills energy function in a more fundamental way. 

\begin{mainthm}[$L^2$-energy gap for Yang--Mills connections]
\label{mainthm:L2_energy_gap_Yang-Mills_connections}
Let $G$ be a compact Lie group and $P$ be a smooth principal $G$-bundle over a closed, smooth Riemannian manifold $(X,g)$ of dimension $d \geq 2$. Then there is a positive constant, $\eps = \eps(g,G,[P]) \in (0, 1]$, with the following significance. If $A$ is a $C^\infty$ Yang--Mills connection on $P$ and its curvature, $F_A$, obeys
\begin{equation}
\label{eq:Curvature_L2_small}
\|F_A\|_{L^2(X)} \leq \eps,
\end{equation}
then $A$ is a flat connection.
\end{mainthm}

Our proof of Theorem \ref{mainthm:L2_energy_gap_Yang-Mills_connections} in Section \ref{subsec:Proof_energy_gap_Yang-Mills_connections} relies on the local piecewise-$C^1$ arc-connectivity of critical sets of the Yang--Mills energy function $\YM$ and that in turn is implied by their analyticity, which we establish in the forthcoming Proposition \ref{prop:Semianalyticity_set_Yang-Mills_connections_uniform_Lp_bound_curvature}. It is well-known (see for example, Simpson \cite{Simpson_1994part2} or Ho, Wilkin, and Wu \cite{Ho_Wilkin_Wu_2019} together with references cited therein) that the representation variety, $\Hom(\pi_1(X);G)/G$, is an analytic set and that this can be identified with the moduli space of flat connections, $M(X)$. For the moduli space of anti-self-dual connections $M(P,g)$ on a principal $G$-bundle $P$ over a closed, four-dimensional, Riemannian manifold $(X,g)$, Donaldson and Kronheimer \cite[p. 126, 139]{DK} observe that analyticity is a consequence of the Method of Kuranishi \cite{Kuranishi} while Koiso \cite{Koiso_1987} (for arbitrary $d\geq 2$) and Taubes \cite{TauFrame} (for $d=4$) observe more generally that the Method of Kuranishi implies that moduli spaces Yang--Mills connections are analytic sets. We shall apply a similar argument to prove Proposition \ref{prop:Semianalyticity_set_Yang-Mills_connections_uniform_Lp_bound_curvature}.

Previous $L^p$ curvature gap results due to Bourguignon, Lawson, and Simons \cite{Bourguignon_Lawson_Simons_1979, Bourguignon_Lawson_1981} required that a curvature operator defined by the Riemannian metric $g$ on the base manifold $X$ obeyed certain positivity properties and that $p=\infty$. For example, in \cite[Theorem 5.3]{Bourguignon_Lawson_Simons_1979}, Bourguignon, Lawson, and Simons asserted that if $d \geq 3$ and $X$ is the $d$-dimensional sphere $S^d$ with its standard round metric of radius one, and $A$ is a Yang--Mills connection on a principal $G$-bundle $P$ over $S^d$ such that
\begin{equation}
\label{eq:Linfinity_energy_gap_sphere}
\|F_A\|_{L^\infty(S^d)} \leq \eps,
\end{equation}
where $\eps=\eps(d)\in(0,1]$, then $A$ is flat. A detailed proof of this gap result is provided by Bourguignon and Lawson in
\cite[Theorem 5.19]{Bourguignon_Lawson_1981} for $d \geq 5$, \cite[Theorem 5.20]{Bourguignon_Lawson_1981} for $d=4$ (by combining the cases of $L^\infty$-small $F_A^+$ and $F_A^-$), and
\cite[Theorem 5.25]{Bourguignon_Lawson_1981} for $d=3$. (The results for the cases $d \geq 5$, $d=4$, and $d=3$ are combined in their \cite[Theorem C]{Bourguignon_Lawson_1981}.) In the penultimate paragraph prior to the statement of their \cite[Theorem 5.26]{Bourguignon_Lawson_1981}, Bourguignon and Lawson imply that these gap results extend to the case of a closed, smooth manifold $X$ if a curvature operator defined by its Riemannian metric $g$ is positive definite. This observation of Bourguignon and Lawson was improved by Gerhardt as \cite[Theorem 1.2]{Gerhardt_2010} by replacing their small $L^\infty(X)$ norm condition on $F_A$ with a small $L^{d/2}(X)$ norm condition,
\begin{equation}
\label{eq:Ldover2_energy_gap_Gerhardt}
\|F_A\|_{L^{d/2}(X)} \leq \eps,
\end{equation}
where $\eps=\eps(g,\dim G)\in(0,1]$. Nakajima proved an $L^2$-energy gap for Yang--Mills connections over the sphere, $S^d$, but with an arbitrary Riemannian metric \cite[Corollary 1.2]{Nakajima_1987}.

\subsubsection{Nonlinear estimate for distance to moduli subspace of flat connections}
\label{subsubsec:Nonlinear_estimate_distance_moduli_space_flat_connections}
We let 
\begin{equation}
\label{eq:Moduli_space_flat_connections}
M(P) := \{\Gamma \in \sA^{1,p}(P): F_\Gamma = 0\}/\Aut^{2,p}(P),
\end{equation}
denote the moduli space of gauge-equivalence classes, $[\Gamma]$, of flat connections $\Gamma$ on $P$.

\begin{mainthm}[Nonlinear estimate for distance to moduli subspace of flat connections]
\label{mainthm:Uhlenbeck_Chern_corollary_4-3}
Continue the hypotheses and notation of Theorem \ref{mainthm:Uhlenbeck_Chern_corollary_4-3_prelim}. Then there is a constant $C = C(A_1,g,G,[P],p,r) \in [1,\infty)$ such that
\begin{equation}
\label{eq:Uhlenbeck_Chern_corollary_4-3_uA-Gamma_W1p_bound}
   \|v(A)-\Gamma\|_{W_{A_1}^{1,r}(X)} \leq C\|F_A\|_{L^r(X)}^\lambda,
\end{equation}
where $\lambda = 2/\alpha \in (0,1]$ and $\alpha=\alpha(A_1,g,G,p) \in [2,\infty)$ is as in Corollary \ref{maincor:Lojasiewicz_distance_inequality_Yang-Mills_energy_function}, and, in particular, 
\begin{equation}
\label{eq:Uhlenbeck_Chern_corollary_4-3_A_M(P)_W1p_distance_bound}
 \dist_{W_{A_1}^{1,r}(X)}([A],M(P)) \leq C\|F_A\|_{L^r(X)}^\lambda,
\end{equation}
where
\[
  \dist_{W_{A_1}^{1,r}(X)}([A],M(P)) :=  \inf_{\begin{subarray}{c}u\in\Aut^{2,p}(P)\\ [\Gamma] \in M(P)\end{subarray}}\|u(A)-\Gamma\|_{W_{A_1}^{1,r}(X)}.
\]
If $\YM$ is \emph{Morse--Bott} at $[\Gamma]$ in the sense of Definition \ref{defn:Definition_Morse-Bott_quotient_space}, then $\lambda=1$ in \eqref{eq:Uhlenbeck_Chern_corollary_4-3_uA-Gamma_W1p_bound} and \eqref{eq:Uhlenbeck_Chern_corollary_4-3_A_M(P)_W1p_distance_bound}.
\end{mainthm}

\begin{rmk}[Sufficient condition for $\YM$ to be Morse--Bott]
\label{rmk:Sufficient_condition_for_YM_to_be_Morse-Bott}
Lemma \ref{lem:Morse-Bott_property_Yang-Mills_energy_near_flat_connection} implies that if
\[
  H_\Gamma^2(X;\ad P) = (0),
\]
for $H_\Gamma^2(X;\ad P)$ as in the forthcoming definition \eqref{eq:H_Gamma^i_adP_W1p}, then $\YM$ is \emph{Morse--Bott} at $[\Gamma]$ in the sense of Definition \ref{defn:Definition_Morse-Bott_quotient_space}. See Section \ref{subsec:Structure_character_variety_closed_Rieman_surface_stratified_space} for a detailed discussion of the ``pillowcase'' $M(P)$, comprising the moduli space of flat $\SU(2)$-connections over the torus $\TT^2 = \RR^2/\ZZ^2$, and which points $[\Gamma]$ have $H_\Gamma^2(X;\ad P) = (0)$ or $H_\Gamma^2(X;\ad P) \neq (0)$.
\end{rmk}

\begin{rmk}[Nonlinear estimate for distance to moduli subspace of flat connections]
\label{rmk:Nonlinear_estimate_distance_moduli_space_flat_connections}  
As we noted earlier, Fukaya \cite[Proposition 3.1]{Fukaya_1998} proved that a version\footnote{Fukaya uses a different system of norms.} of \eqref{eq:Uhlenbeck_Chern_corollary_4-3_A_M(P)_W1p_distance_bound} holds when $d=4$, and $X$ is a compact manifold with boundary, and $A$ is anti-self-dual, and $\lambda = \lambda(g,G,\pi_1(X)) \in (0,1]$. However, Fukaya used a difficult patching argument to obtain his result and there is no overlap between the methods of proof of Theorem \ref{mainthm:Uhlenbeck_Chern_corollary_4-3} and \cite[Proposition 3.1]{Fukaya_1998}.
\end{rmk}  

\subsubsection{Optimal {\L}ojasiewicz distance inequality for a Yang--Mills energy function implies its Morse--Bott property}
\label{subsubsec:Optimal_Lojasiewicz_distance_inequality_Yang-Mills_energy_function_implies_Morse-Bott_property}
Let $\sX,\sY$ be real Banach spaces, and $\sL(\sX,\sY)$ denote the Banach space of bounded linear operators from $\sX$ to $\sY$, and $\Ker T$ and $\Ran T$ denote the kernel and range of $T \in \sL(\sX,\sY)$, and $\sX^*$ denote the continuous dual space of $\sX$. If $\sU\subset\sX$ is an open subset, $\sE:\sU\to\RR$ is a $C^2$ function, and $\Crit\sE$ is a $C^2$, connected submanifold of $\sU$, then the tangent space, $T_x\Crit \sE$, is contained in $\Ker \sE''(x)$, for each $x \in \Crit\sE$, where $\sE'(x)\in\sX^*$ and $\sE''(x)\in\sL(\sX,\sX^*)$.

\begin{defn}[Morse--Bott properties]
\label{defn:Morse-Bott_function}
(See Feehan \cite[Definition 1.5]{Feehan_lojasiewicz_inequality_all_dimensions_morse-bott}.)  
Let $\sX$ be a real Banach space, and $\sU \subset \sX$ be an open neighborhood of the origin, and $\sE:\sU\to\RR$ be a $C^2$ function such that $\Crit \sE$ is a $C^2$, connected submanifold.
\begin{enumerate}
  \item
  \label{item:Morse-Bott_point}
  If $x_0 \in \Crit \sE$ and $\Ker \sE''(x_0) \subset \sX$ has a closed complement $\sX_0$ and $\Ran \sE''(x_0) = \sX_0^*$, and $T_{x_0}\Crit \sE = \Ker \sE''(x_0)$, then $\sE$ is \emph{Morse--Bott at the point} $x_0$;
  \item
  \label{item:Morse-Bott_critical_set}
  If $\sE$ is \emph{Morse--Bott at each point} $x \in \Crit \sE$, then $\sE$ is \emph{Morse--Bott along $\Crit \sE$} or a \emph{Morse--Bott function}.
\end{enumerate}
\end{defn}

The Morse--Bott Lemma in this setting (see Feehan \cite[Theorem 2.10]{Feehan_lojasiewicz_inequality_all_dimensions_morse-bott}) implies that if $\sE$ is Morse--Bott at a point, as in Definition \ref{defn:Morse-Bott_function} \eqref{item:Morse-Bott_point}, then $\sE$ is Morse--Bott along an open neighborhood of that point in $\Crit \sE$, as in Definition \ref{defn:Morse-Bott_function} \eqref{item:Morse-Bott_critical_set}.

\begin{mainthm}[Morse--Bott property of an analytic function with {\L}ojasiewicz exponent one half]
\label{mainthm:Analytic_function_Lojasiewicz_exponent_one-half_Morse-Bott_Banach}
(See Feehan \cite[Theorem 2]{Feehan_lojasiewicz_inequality_all_dimensions_morse-bott}.)  
Let $\sX$ be a real Banach space, and $\sU \subset \sX$ be an open neighborhood of the origin, and $\sE:\sU\to\RR$ be a
non-constant analytic function such that $\sE(0) = 0$ and $\sE'(0) = 0$ and $\sE''(0) \in \sL(\sX,\sX^*)$ is a Fredholm operator with index zero. If there is a constant $C \in (0,\infty)$ such that, after possibly shrinking $\sU$,
\begin{equation}
\label{eq:Lojasiewicz_gradient_inequality_dual_space}
\|\sE'(x)\|_{\sX^*} \geq C|\sE(x)|^{1/2}, \quad\text{for all } x \in \sU,
\end{equation}
that is, the {\L}ojasiewicz gradient inequality for $\sE$ holds with optimal exponent $\theta=1/2$, then $\sE$ is a Morse--Bott function in the sense of Definition \ref{defn:Morse-Bott_function}.
\end{mainthm}

In the case of the Yang--Mills energy function on a Coulomb-gauge slice, Theorem \ref{mainthm:Analytic_function_Lojasiewicz_exponent_one-half_Morse-Bott_Banach} yields

\begin{mainthm}[Morse--Bott property of a Yang--Mills energy function with {\L}ojasiewicz exponent one half]
\label{mainthm:Yang-Mills_energy_function_Lojasiewicz_exponent_one-half_Morse-Bott}
Let $(X,g)$ be a closed, smooth Riemannian manifold of dimension $d = 2, 3$, or $4$, and $G$ be a compact Lie group, and $P$ be a smooth principal $G$-bundle over $X$, and $p>d/2$ be a constant. Let $A_1$ be a $C^\infty$ reference connection on $P$, and $\Gamma$ be a $C^\infty$ flat connection on $P$.
If there is a constant $C_0 \in (0,\infty)$ such that, after possibly decreasing $\sigma$,
\begin{equation}
\label{eq:Yang-Mills_energy_function_optimal_Lojasiewicz_distance_inequality}
 \|A-\Gamma\|_{W_\Gamma^{1,2}(X)} \leq C_0\|F_A\|_{L^2(X)}, \quad\text{for all } A \in \bB_\sigma(\Gamma),
\end{equation}
where $\bB_\sigma(\Gamma)$ is as in \eqref{eq:W1p_connections_Coulomb_gauge_ball},
\[
\bB_\sigma(\Gamma) = \Gamma + \left\{a \in \Ker d_\Gamma^*\cap W_{A_1}^{1,p}(X;T^*X\otimes\ad P): \|a\|_{W_{A_1}^{1,p}(X)} < \sigma\right\},
\]
then there is a constant $C\in(1,\infty)$ such that $\widehat\YM$ obeys the {\L}ojasiewicz gradient inequality with optimal exponent one half,
\begin{equation}
\label{eq:Yang-Mills_energy_function_optimal_Lojasiewicz_gradient_inequality}
 \|\widehat{\YM'}(A)\|_{W_{A_1}^{-1,2}(X)} \geq C\widehat\YM(A)^{1/2}, \quad\text{for all } A \in \bB_\sigma(\Gamma),
\end{equation}
where $\widehat{\YM'}(A) = \Pi_\Gamma d_A^*F_A$, and $\widehat\YM$ is Morse--Bott at $\Gamma$ in the sense of Definition \ref{defn:Definition_Morse-Bott_quotient_space}.
\end{mainthm}

\begin{rmk}[Extension to higher-dimensional manifolds]
\label{rmk:Yang-Mills_energy_function_Lojasiewicz_exponent_one-half_Morse-Bott}  
It is highly likely that a version of Theorem \ref{mainthm:Yang-Mills_energy_function_Lojasiewicz_exponent_one-half_Morse-Bott} would hold for base manifolds $X$ of dimension $d \geq 5$, but the proof would require a more general version of Theorem \ref{mainthm:Analytic_function_Lojasiewicz_exponent_one-half_Morse-Bott_Banach}.
\end{rmk}  

Theorem \ref{mainthm:Yang-Mills_energy_function_Lojasiewicz_exponent_one-half_Morse-Bott} can be used to generate counterexamples to the estimate in \cite[Corollary 4.3]{UhlChern} since it suffices to identify non-regular flat connections. See Appendix \ref{sec:Counterexample_corollary_4-3_Uhlenbeck_1985} for one such example and further discussion and references.

\subsection{Outline}
\label{subsec:Outline}
In Section \ref{sec:Preliminaries}, we review our gauge theory conventions and notation. In Section \ref{sec:Uhlenbeck_weak_compactness_theorem_revisited}, we discuss the construction of bundle automorphisms that bring a nearby connection into Coulomb gauge with respect to a given reference connection. In particular, we reinterpret the usual sequential compactness conclusion in Uhlenbeck's Weak Compactness \cite{UhlLp} in terms of compactness with respect to a metric topology and prove Theorem \ref{mainthm:Uhlenbeck_Chern_corollary_4-3_prelim}.

In order to prove local well-posedness for the Yang--Mills gradient flow on a Coulomb-gauge slice \eqref{eq:Yang-Mills_gradient_flow_slice}, we shall apply the general theory for abstract nonlinear evolution equations in Banach spaces described by Sell and You \cite{Sell_You_2002}. Our approach in Section \ref{sec:Local_well-posedness_nonlinear_evolution_equations_Banach_spaces} is broadly similar to the one we take in \cite[Section 17]{Feehan_yang_mills_gradient_flow_v4} but differs in one important respect. Rather than consider the usual Yang--Mills gradient flow \eqref{eq:Yang-Mills_gradient_flow} and apply the Donaldson--DeTurck trick \cite{DonASD, DeTurck_1983} to obtain a gauge-equivalent quasilinear parabolic equation, we instead consider Yang--Mills gradient flow restricted to a Coulomb-gauge slice through a Yang--Mills connection, by analogy with Chern--Simons gradient flow restricted to a Coulomb-gauge slice through a flat connection, as discussed by Morgan, Mrowka, and Ruberman \cite[Section 2.6]{MMR}. As we shall explain in Section \ref{sec:Local_well-posedness_Yang-Mills_gradient_flow}, the regularity properties for solutions to Yang--Mills gradient flow on a Coulomb-gauge slice are much better than those obtained through an application of the Donaldson--DeTurck trick. We make explicit use of this improvement in our proof of Corollary \ref{maincor:Retraction_open_nbhd_onto_moduli_space_flat_connections}, giving an `almost’ strong deformation retraction of a neighborhood in the quotient space of Sobolev connections onto the moduli subspace of flat connections.

In Section \ref{sec:Global_existence_convergence_rate_Lojasiewicz-Simon_gradient_flow_near_local_minimum}, we summarize our key results from \cite[Section 2]{Feehan_yang_mills_gradient_flow_v4} on global existence, convergence, and convergence rate for gradient flow defined by a smooth function near a critical point when the function obeys a {\L}ojasiewicz--Simon gradient inequality.

Our results in Section \ref{sec:Local_well-posedness_nonlinear_evolution_equations_Banach_spaces} assume (as do Sell and You \cite{Sell_You_2002}) that the nonlinear evolution equation is defined by a \emph{positive sectorial} unbounded linear operator on a Banach space --- such as the covariant Laplace operator \cite{FU} on a space of $L^p$ sections of a vector bundle over a closed manifold --- plus a nonlinear term. While the Chern--Simons \cite{DonFloer} or Chern--Simons--Dirac operators \cite{KMBook} are not positive sectorial operators --- they are first-order elliptic operators upon restriction to a Coulomb-gauge slice with spectra unbounded from above or below --- they can nonetheless be expressed as the \emph{difference} of two positive sectorial operators. Hence, the theory in Section \ref{sec:Local_well-posedness_nonlinear_evolution_equations_Banach_spaces} can be applied, with suitable modifications, to yield properties of gradient flows for the Chern--Simons or Chern--Simons--Dirac functions used in the definitions of instanton \cite{DonFloer, Floer} or monopole Floer homologies \cite{KMBook} of three-dimensional manifolds. Selected results from Section \ref{sec:Global_existence_convergence_rate_Lojasiewicz-Simon_gradient_flow_near_local_minimum} on convergence of gradient flows for analytic functions on Banach spaces can also be applied to Chern--Simons or Chern--Simons--Dirac gradient flows.

In Section \ref{sec:Lojasiewicz_distance_inequality_functions_Banach_spaces}, we prove Theorem \ref{mainthm:Lojasiewicz_distance_inequality_hilbert_space}, giving our {{\L}ojasiewicz distance inequality for functions on Banach spaces.

Section \ref{sec:Local_well-posedness_Yang-Mills_gradient_flow} includes proofs of our results on local well-posedness, \apriori estimates, and minimal lifetimes for solutions to Yang--Mills gradient flow on a Coulomb-gauge slice by applying our general results on evolution equations in Banach spaces from Section \ref{sec:Local_well-posedness_nonlinear_evolution_equations_Banach_spaces}.

  In Section \ref{sec:Global_existence_convergence_rate_Yang-Mills_gradient_flow_near_local_minimum}, we  apply the main results of Sections \ref{sec:Local_well-posedness_nonlinear_evolution_equations_Banach_spaces}, \ref{sec:Global_existence_convergence_rate_Lojasiewicz-Simon_gradient_flow_near_local_minimum}, \ref{sec:Lojasiewicz_distance_inequality_functions_Banach_spaces}, and \ref{sec:Local_well-posedness_Yang-Mills_gradient_flow} to prove Theorem \ref{mainthm:Yang-Mills_gradient_flow_global_existence_and_convergence_started_near_local_minimum} and Corollary \ref{maincor:Yang-Mills_gradient_flow_global_existence_and_convergence_started_small_energy}, on the global existence, convergence, and convergence rate for Yang--Mills gradient flow on a Coulomb-gauge slice near a local minimum. We apply these results to prove Corollary \ref{maincor:Retraction_open_nbhd_onto_moduli_space_flat_connections}, on the existence of an `almost' strong deformation retract from a neighborhood in the quotient space of Sobolev connections onto the moduli subspace of flat connections, and Corollaries \ref{maincor:Lojasiewicz_distance_inequality_Yang-Mills_energy_function_slice} and \ref{maincor:Lojasiewicz_distance_inequality_Yang-Mills_energy_function}, on the {\L}ojasiewicz distance inequality for the Yang--Mills energy function.

In Section \ref{sec:Estimates_distance_moduli_space_flat_connections}, we prove Theorems \ref{mainthm:Uhlenbeck_Chern_corollary_4-3} and \ref{mainthm:Yang-Mills_energy_function_Lojasiewicz_exponent_one-half_Morse-Bott}. In Section \ref{sec:Energy_gap_Yang-Mills_connections}, we prove Theorem \ref{mainthm:L2_energy_gap_Yang-Mills_connections} by directly exploiting the analyticity of the Yang--Mills energy function rather than its {\L}ojasiewicz gradient inequality. In Appendix \ref{sec:Counterexample_corollary_4-3_Uhlenbeck_1985}, we describe a counterexample to the estimates stated in \cite[Corollary 4.3]{UhlChern} and \cite[Theorem 5.1]{Feehan_yangmillsenergygapflat}, based on an observation by Mrowka \cite{Mrowka_7-30-2018}.

\subsection{Acknowledgments}
\label{subsec:Acknowledgments}
This monograph owes its existence to Tom Mrowka and his description of a counterexample \cite{Mrowka_7-30-2018} to exponential convergence (implied by our
\cite[Theorem 2 and Remark 1.8]{Feehan_lojasiewicz_inequality_ground_state_v1}) for Yang--Mills gradient flow near the moduli space of flat connections. I am indebted to him and very grateful to Toby Colding, Mariano Echeverria, Kenji Fukaya, Chris Herald, Peter Kronheimer, Aaron Naber, Takeo Nishinou, Tom Parker, Tristan Rivi{\`e}re, Nikolai Saveliev, Penny Smith, Yuuji Tanaka, Karen Uhlenbeck, Katrin Wehrheim, and Graeme Wilkin for helpful communications during the preparation of this monograph. I am especially grateful to Tom Parker and Nikolai Saveliev for perceptive comments, questions, references, and help with Appendix \ref{sec:Counterexample_corollary_4-3_Uhlenbeck_1985}. Versions of this monograph were presented at seminars at Boston College, the Dublin Institute for Advanced Studies, Harvard University, MIT, and Stanford University during Fall 2018 and Spring 2019: I am grateful to those seminar organizers and participants for their comments and questions. I thank the National Science Foundation for their support and the Dublin Institute for Advanced Studies for their hospitality.

\section{Preliminaries}
\label{sec:Preliminaries}
Throughout our monograph, $G$ denotes a compact Lie group and $P$ a smooth principal $G$-bundle over a closed, smooth manifold, $X$, of dimension $d \geq 2$ and endowed with Riemannian metric, $g$. We let\footnote{We follow the notational conventions of Friedman and Morgan \cite[p. 230]{FrM}, where they define $\ad P$ as we do here and define $\Ad P$ to be the group of automorphisms of the principal $G$-bundle, $P$.}
$\ad P := P\times_{\ad}\fg$ denote the real vector bundle associated to $P$ by the adjoint representation of $G$ on its Lie algebra,
$\Ad:G \ni u \to \Ad_u \in \Aut\fg$. We fix an inner product on the Lie algebra $\fg$ that is invariant under the adjoint action of $G$ and thus define a fiber metric on $\ad P$. When $\fg$ is semisimple, one may use a negative multiple of the Cartan--Killing form $\kappa:\fg\times\fg\to\RR$ to define such an inner product on $\fg$ --- for example, see Hilgert and Neeb \cite[Definition 5.5.3 and Theorem 5.5.9]{Hilgert_Neeb_structure_geometry_lie_groups}. More generally, because $G$ is compact it has a faithful representation $\rho:G\to \Aut(V)$ for some complex vector space $V$ as a consequence of the Peter--Weyl Theorem and so $G$ is isomorphic to a closed subgroup of $\U(n)$ or $\Or(n)$ for some integer $n$ (see Br\"ocker and tom Dieck \cite[Theorem III.4.1 and Exercise III.4.7.1]{BrockertomDieck} or Knapp \cite[Corollary 4.22]{Knapp_1986}). We can then obtain the desired inner product on $\fg$ by applying \cite[Proposition 4.24]{Knapp_1986} or by restricting the inner product $\langle\xi,\eta\rangle := \tr(\xi^*\eta)$, for all $\xi,\eta\in\fu(n)$.

Because choices of conventions in Yang--Mills gauge theory vary among authors and as such choices will matter here, we summarize them here. We follow the mathematical conventions of Kobayashi and Nomizu \cite[Chapters II and III]{Kobayashi_Nomizu_v1}, with amplifications described by Bleecker \cite[Chapters 1--3]{Bleecker_1981} that are useful in gauge theory, though we adopt the notation employed by Donaldson and Kronheimer \cite[Chapters 2--4]{DK} and Uhlenbeck \cite{UhlLp}. Bourguignon and Lawson \cite[Section 2]{Bourguignon_Lawson_1981} provide a useful summary of Yang--Mills gauge theory that overlaps with our development here.

We assume that $G$ acts on $P$ on the right \cite[Definition 1.1.1]{Bleecker_1981}, \cite[Section I.1.5]{Kobayashi_Nomizu_v1}. We let $A$ denote a smooth connection on $P$ through any one of its three standard equivalent definitions, namely \cite[Definitions 1.2.1, 1.2.2, and 1.2.3 and Theorems 1.2.4 and 1.2.5]{Bleecker_1981}, \cite[Section II.1]{Kobayashi_Nomizu_v1}:
\begin{inparaenum}[(\itshape i\upshape)]
\item a connection one-form $A \in \Omega^1(P;\fg)$,
\item a family of horizontal subspaces $H_p\subset T_pP$ smoothly varying with $p\in P$, or
\item a set of smooth local connection one-forms $A_\alpha \in \Omega^1(U_\alpha;\fg)$ with respect to an open cover $\{U_\alpha\}_{\alpha\in\sI}$ of $X$ and smooth local sections $\sigma_\alpha:U_\alpha \to P$.
\end{inparaenum}
In particular, if $g_{\alpha\beta}:U_\alpha\cap U_\beta\to G$ is a smooth transition function \cite[Definition 1.1.3]{Bleecker_1981}, \cite[Section I.1.5]{Kobayashi_Nomizu_v1} defined by $\sigma_\beta = \sigma_\alpha g_{\alpha\beta}$, then \cite[Definition 1.2.3]{Bleecker_1981}, \cite[Proposition II.1.4]{Kobayashi_Nomizu_v1}
\begin{equation}
\label{eq:Kobayashi_Nomizu_proposition_II-1-4}
A_\beta = \Ad(g_{\alpha\beta}^{-1})A_\alpha + g_{\alpha\beta}^*\theta \quad\text{on } U_\alpha\cap U_\beta,  
\end{equation}
where $\theta \in \Omega^1(G;\fg)$ is the \emph{Maurer--Cartan form} (or \emph{canonical one-form}); when $G\subset\GL(n,\CC)$, then \eqref{eq:Kobayashi_Nomizu_proposition_II-1-4} simplifies to give
\[
  A_\beta = g_{\alpha\beta}^{-1}A_\alpha g_{\alpha\beta} + g_{\alpha\beta}^{-1}dg_{\alpha\beta} \quad\text{on } U_\alpha\cap U_\beta.
\]
In particular, if $B$ is any other smooth connection on $P$, then $A-B \in \Omega^1(X;\ad P)$ \cite[Theorem 3.2.8]{Bleecker_1981}, where we let
\[
  \Omega^l(X; \ad P) :=  C^\infty(X;\wedge^l(T^*X)\otimes\ad P)
\]
the Fr{\'e}chet space of $C^\infty$ sections of $\wedge^l(T^*X)\otimes\ad P$, for an integer $l\geq 0$.

Given a connection $A$ on $P$, one obtains the \emph{exterior covariant derivative} \cite[Definitions 2.2.2 and 3.1.3]{Bleecker_1981}, \cite[Proposition II.5.1]{Kobayashi_Nomizu_v1}
\[
  d_A: \bar\Omega^l(P;\fg) \to \bar\Omega^{l+1}(P;\fg),
\]
where $\l\geq 0$ is an integer and $\bar\Omega^l(P;\fg) \subset \Omega^l(P;\fg)$ is the subspace of \emph{tensorial $l$-forms $\varphi$ of type $\ad\, G$} such that \cite[Definition 3.1.2]{Bleecker_1981}, \cite[p. 75]{Kobayashi_Nomizu_v1}
\begin{inparaenum}[(\itshape i\upshape)]
\item
\label{item:tensorial_invariance}
$R_g^*\varphi = \Ad(g^{-1})\varphi$ for all $g\in G$, where $R_g:P\to P$ denotes right multiplication by $g$, and
\item
\label{item:tensorial_vanish_vertical_fibers}
$\varphi_p(\xi_1,\ldots,\xi_l)=0$ if any one of $\xi_i\in T_pP$ is vertical, for $p\in P$.
\end{inparaenum}
If $\varphi \in \Omega^l(P;\fg)$ obeys condition \eqref{item:tensorial_invariance} but not \eqref{item:tensorial_vanish_vertical_fibers}, then $\varphi$ is a \emph{pseudotensorial $l$-form of type $\ad\,G$}. In particular, $A \in \Omega^1(P;\fg)$ is a pseudotensorial $1$-form of type $\ad\, G$ by \cite[Proposition II.1.1]{Kobayashi_Nomizu_v1}. As customary \cite[Equation (2.1.12)]{DK}, we also let
\begin{equation}
\label{eq:Exterior_covariant_derivative}
  d_A: \Omega^l(X;\ad P) \to \Omega^{l+1}(X;\ad P),
\end{equation}
denote the equivalent expression for exterior covariant derivative and let
\begin{equation}
\label{eq:Exterior_covariant_derivative_L2_adjoint}
d_A^*:\Omega^{l+1}(X; \ad P) \to \Omega^l(X; \ad P),
\end{equation}
and denote its $L^2$-adjoint with respect to the Riemannian metric \cite[Equation (2.1.24)]{DK}.

If $\varphi \in \bar\Omega^l(P;\fg)$, then \cite[Corollary 3.1.6]{Bleecker_1981}
\begin{equation}
\label{eq:Bleecker_corollary_3-1-6}
  d_A\varphi = d\varphi + [A,\varphi] \in \bar\Omega^{l+1}(P;\fg).
\end{equation}
If $\varphi \in \Omega^l(X;\ad P)$, then we have the corresponding local expressions,
\begin{equation}
\label{eq:Bleecker_corollary_3-1-6_local}
  d_A\varphi\restriction_{U_\alpha} = d\varphi + [A_\alpha,\varphi] \in \Omega^{l+1}(U_\alpha;\fg),
\end{equation}
or in the case of $G\subset\GL(n,\CC)$ \cite[Theorem 2.2.12]{Bleecker_1981},
\[
  d_A\varphi\restriction_{U_\alpha} = d\varphi + A_\alpha\wedge\varphi - (-1)^l\varphi\wedge A_\alpha  \in \Omega^{l+1}(U_\alpha;\fg).
\]
The \emph{curvature} of $A\in \Omega^1(P;\fg)$ is defined by \cite[Definition 2.2.3]{Bleecker_1981}, \cite[p. 77]{Kobayashi_Nomizu_v1}
\begin{equation}
\label{eq:Kobayashi_Nomizu_page_77}
  F_A = d_A A \in \bar\Omega^2(P;\fg),
\end{equation}
and by virtue of the \emph{structure equation} \cite[Theorem 2.2.4]{Bleecker_1981}, \cite[Theorem II.5.2]{Kobayashi_Nomizu_v1}, one has
\begin{equation}
\label{eq:Kobayashi_Nomizu_theorem_2-5-2}
F_A = dA + \frac{1}{2}[A,A] \in \bar\Omega^2(P;\fg).
\end{equation}
(Note that $d_A\varphi \in \bar\Omega^{l+1}(P;\fg)$ even if $\varphi \in \Omega^l(P;\fg)$ is only pseudotensorial by \cite[Proposition  II.5.1 (c)]{Kobayashi_Nomizu_v1}.) We also write $F_A \in \Omega^2(X;\ad P)$ for the curvature equivalently defined by the corresponding set of local expressions \cite[Theorem 2.2.11]{Bleecker_1981}
\begin{equation}
\label{eq:Bleecker_theorem_2-2-11}
  F_A\restriction_{U_\alpha} = dA_\alpha + \frac{1}{2}[A_\alpha,A_\alpha] \in \Omega^2(U_\alpha;\fg),
\end{equation}
or in the case of $G\subset\GL(n,\CC)$ \cite[Corollary 2.2.13]{Bleecker_1981},
\[
  F_A\restriction_{U_\alpha} = dA_\alpha + A_\alpha\wedge A_\alpha \in \Omega^2(U_\alpha;\fg).
\]
If $a \in \bar\Omega^1(P;\fg)$, then \eqref{eq:Kobayashi_Nomizu_theorem_2-5-2} yields
\[
  F_{A+a} = dA + \frac{1}{2}[A,A] + da + \frac{1}{2}[A,a] + \frac{1}{2}[a,A] + \frac{1}{2}[a,a],
\]
that is, using \eqref{eq:Bleecker_corollary_3-1-6_local} and $[a,A]=[A,a]$ by the forthcoming equation \eqref{eq:Bleecker_definition_2-1-1_one_forms},
\begin{equation}
\label{eq:Donaldson_Kronheimer_2-1-14}
F_{A+a} = F_A + d_Aa + \frac{1}{2}[a,a].
\end{equation}
or in the case of $G\subset\GL(n,\CC)$ \cite[Equation (2.1.14)]{DK},
\[
  F_{A+a} = F_A + d_Aa + a\wedge a.
\]
We note that if $a, b \in \Omega^1(X;\ad P)$ and $\xi,\eta\in C^\infty(TX)$, then \cite[Definition 2.1.1]{Bleecker_1981}
\begin{equation}
  \label{eq:Bleecker_definition_2-1-1_one_forms}
  [a,b](\xi,\eta) = [a(\xi),b(\eta)] - [a(\eta),b(\xi)]
\end{equation}
or in the case of $G\subset\GL(n,\CC)$ \cite[Theorem 2.2.12]{Bleecker_1981},
\[
  [a,b] = a\wedge b + b\wedge a.
\]
We let $\Aut(P)$ denote the Fr{\'e}chet space of all smooth automorphisms of $P$ \cite[Definition 3.2.1]{Bleecker_1981}, or \emph{gauge transformations}. We recall that $\Aut(P) \cong \Omega^0(X;\Ad P)$ by \cite[Theorem 3.2.2]{Bleecker_1981}, where $\Ad P := P\times_G G$ and $g\in G$ acts on $G$ on the left by conjugation via $h \mapsto ghg^{-1}$ for all $h\in G$ \cite[Definition 3.1.1]{Bleecker_1981}. If $\sA(P)$ denotes the Fr{\'e}chet space of all connections on $P$, then one obtains a right action \cite[Theorem 3.2.5]{Bleecker_1981}, \cite[Theorem II.6.1]{Kobayashi_Nomizu_v1},
\begin{equation}
\label{eq:Right_action_gauge_transformations_on_connections}
  \sA(P) \times \Aut(P) \ni (A,u) \mapsto u(A) = u^*A \in \sA(P).
\end{equation}
If $u \in \Aut(P)$ is represented locally by $u(\sigma_\alpha) = \sigma_\alpha s_\alpha$ on $U_\alpha \subset X$, where $\sigma_\alpha:U_\alpha\to P$ is a local section and $s_\alpha:U_\alpha\to G$ is a smooth map, then \cite[Theorem 3.2.14]{Bleecker_1981}
\begin{equation}
\label{eq:Bleecker_theorem_3-2-14}
u(A)\restriction_{U_\alpha} = \Ad(s_\alpha^{-1})A_\alpha + s_\alpha^*\theta \in \Omega^1(U_\alpha;\fg),  
\end{equation}  
or in the case of $G\subset\GL(n,\CC)$,
\[
  u(A)\restriction_{U_\alpha} = s_\alpha^{-1}A_\alpha s_\alpha + s_\alpha^{-1}ds_\alpha \in \Omega^1(U_\alpha;\fg).
\]
If $B$ is any other smooth connection on $P$ and $G\subset\GL(n,\CC)$, then
\begin{align*}
  (u(A) - B)_\alpha
  &=
    s_\alpha^{-1}A_\alpha s_\alpha + s_\alpha^{-1}ds_\alpha - B_\alpha
  \\
  &= s_\alpha^{-1}(A_\alpha-B_\alpha) s_\alpha + s_\alpha^{-1}(ds_\alpha + [B_\alpha,s_\alpha])
  \\
  &= s_\alpha^{-1}(A-B)_\alpha s_\alpha + s_\alpha^{-1}d_Bs_\alpha \quad\text{on } U_\alpha.
\end{align*}
If $s \in \Omega^0(X;\Ad P)$ is represented locally by the collection $\{s_\alpha\}_{\alpha\in\sI}$, then (as in \cite[p. 32]{UhlLp}) the corresponding global expression for the action of $u\in\Aut(P)$ is given by 
\begin{equation}
\label{eq:Uhlenbeck_1982_Lp_gauge_action_page_32}
u(A)-B = s^{-1}(A-B)s + s^{-1}d_Bs.
\end{equation}
In order to construct Sobolev spaces of connections and gauge transformations, extending the usual definitions of Sobolev spaces of functions on open subsets of Euclidean space in Adams and Fournier \cite[Chapter 3]{AdamsFournier}, we shall need suitable covariant derivatives. If $E$ is a smooth vector bundle over $X$ with covariant derivative \cite[Section III.1]{Kobayashi_Nomizu_v1}
\[
  \nabla:C^\infty(X;E) \to C^\infty(X;T^*X\otimes E),
\]
and $A$ is smooth connection on $P$ with induced covariant derivative (see \cite[Equation (2.1.12) (ii)]{DK} or Kobayashi \cite[Equation (1.1.1)]{Kobayashi})
\begin{equation}
\label{eq:Covariant_is_exterior_covariant_derivative_on_sections}
  \nabla_A = d_A:C^\infty(X;\ad P) \to C^\infty(X;T^*X\otimes \ad P),
\end{equation}
we let $\nabla_A$ denote the induced covariant derivative on the tensor product bundle $E\otimes\ad P$,
\[
  \nabla_A:C^\infty(X;E\otimes\ad P) \to C^\infty(X;T^*X\otimes E\otimes \ad P).
\]
The covariant derivative on $E=\wedge^l(T^*X)$ is induced by the Levi--Civita connection on $T^*X$.

We denote the Banach space of sections of $\wedge^l(T^*X)\otimes\ad P$ of Sobolev class $W^{k,p}$, for any $k\in \NN$ and $p \in [1,\infty]$, by $W_A^{k,p}(X; \wedge^l(T^*X)\otimes\ad P)$, with norm,
\[
  \|\phi\|_{W_A^{k,p}(X)} := \left(\sum_{j=0}^k \int_X |\nabla_A^j\phi|^p\,d\vol_g \right)^{1/p},
  \quad\text{for } 1\leq p<\infty,
\]
where $d\vol_g = \sqrt{\det(g_{ij})}\,dx_1\wedge\cdots\wedge dx_d$ with respect to local coordinates $(x_1,\ldots,x_n)$, assuming now that $X$ is also oriented, and
\[
\|\phi\|_{W_A^{k,\infty}(X)} := \sum_{j=0}^k \esssup_X |\nabla_A^j\phi|, \quad\text{for } p=\infty,
\]
where $\phi \in W_A^{k,p}(X; \wedge^l(T^*X)\otimes\ad P)$.

For $p \geq 1$ and fixed $C^\infty$ connections $A_0$ and $A_1$ on $P$, we let
\begin{equation}
\label{eq:Affine_space_W1p_connections}  
  \sA^{1,p}(P) := A_0 + W_{A_1}^{1,p}(X;T^*X\otimes\ad P)
\end{equation}
denote the affine space of Sobolev $W^{1,p}$ connections on $P$. For $p \in (d/2,\infty)$, we let $\Aut^{2,p}(P)$ denote the Banach Lie group of Sobolev $W^{2,p}$ automorphisms of $P$ \cite[Section 2.3.1]{DK}, \cite[Appendix A and p. 32 and pp. 45--51]{FU}, \cite[Section 3.1.2]{FrM}, let
\begin{equation}
\label{eq:Quotient_space_W1p_connections}  
 \sB^{1,p}(P) := \sA^{1,p}(P)/\Aut^{2,p}(P)
\end{equation}
denote the quotient space of gauge-equivalence classes of $W^{1,p}$ connections on $P$, and let
\begin{equation}
\label{eq:Quotient_map_W1p_connections}  
 \pi: \sA^{1,p}(P) \ni A \mapsto [A] \in \sB^{1,p}(P)
\end{equation}
denote the quotient map.

Throughout this monograph, constants are generally denoted by $C$ (or $C(*)$ to indicate explicit dependencies) and may increase from one line to the next in a series of inequalities. We write $\eps \in (0,1]$ to emphasize a positive constant that is understood to be small or $K \in [1,\infty)$ to emphasize a constant that is understood to be positive but finite. We let $\Inj(X,g)$ denote the injectivity radius of a smooth Riemannian manifold $(X,g)$. Following Adams and Fournier \cite[Sections 1.26 and 1.28]{AdamsFournier}, for an open subset $U\subset\RR^n$ and integer $m\geq 0$, we let $C^m(U)$ (respectively, $C^m(\bar U)$) denote the vector space of (real or complex-valued) functions on $U$ which, together with their derivatives up to order $m$, are continuous (respectively, bounded and uniformly continuous) on $U$. The H\"older spaces $C^{m,\lambda}(\bar U)$ for $\lambda\in(0,1]$ are defined as in \cite[Section 1.29]{AdamsFournier}. We write $C^{m,\lambda}(U)$ (or equivalently, $C_{\loc}^{m,\lambda}(U)$) for the vector space of functions $f$ such that $f \in C^{m,\lambda}(\bar V)$ for all $V\Subset U$. As usual, we denote the set of non-negative integers by $\NN$.

\section{Uhlenbeck's Weak Compactness Theorem Revisited}
\label{sec:Uhlenbeck_weak_compactness_theorem_revisited}
The proofs of our main results involve restriction from all Sobolev $W^{1,p}$ connections to ones that obey a Coulomb-gauge slice condition with respect to some reference connection. Therefore, in Section \ref{subsec:Transformation_Coulomb_gauge}, we recall some of our results with Maridakis from \cite{Feehan_Maridakis_Lojasiewicz-Simon_coupled_Yang-Mills} on the construction of bundle automorphisms that bring a nearby connection into Coulomb gauge with respect to a given reference connection. In Section \ref{subsec:Uhlenbeck_weak_compactness_theorem}, we reinterpret the usual sequential compactness conclusion in Uhlenbeck's Weak Compactness Theorem (see \cite[Theorem 1.5 or 3.6]{UhlLp}) in terms of compactness with respect to a metric topology. Finally, in Section \ref{subsec:Existence_flat_connection_principal_bundle_supporting_connection_Lp-small_curvature}, we prove Theorem \ref{mainthm:Uhlenbeck_Chern_corollary_4-3_prelim}.

\subsection{Transformation to Coulomb gauge}
\label{subsec:Transformation_Coulomb_gauge}
We begin with the basic

\begin{thm}[Existence of $W^{2,q}$ Coulomb gauge transformations for $W^{1,q}$ connections that are $W^{1,\frac{d}{2}}$ close to a reference connection]
\label{thm:Feehan_proposition_3-4-4_Lp}
(See  Feehan and Maridakis \cite[Theorem 14]{Feehan_Maridakis_Lojasiewicz-Simon_coupled_Yang-Mills}.)
Let $(X,g)$ be a closed, smooth Riemannian manifold of dimension $d \geq 2$, and $G$ be a compact Lie group, and $P$ be a smooth principal $G$-bundle over $X$. If $A_1$ is a $C^\infty$ connection on $P$, and $A_0$ is a Sobolev connection on $P$ of class $W^{1,q}$ with $d/2<q<\infty$, and $p \in (1,\infty)$ obeys $d/2 \leq p \leq q$, then there are constants $\zeta = \zeta(A_0,A_1,g,G,p,q) \in (0,1]$ and $C = C(A_0,A_1,g,G,p,q) \in (0,\infty)$ with the following significance. If $A$ is a $W^{1,q}$ connection on $P$ that obeys
\begin{equation}
\label{eq:Feehan_3-4-4_Lp_A_minus_A0_W1p_close}
\|A - A_0\|_{W_{A_1}^{1,p}(X)} < \zeta,
\end{equation}
then there exists $u \in \Aut^{2,q}(P)$ such that
\[
d_{A_0}^*(u(A) - A_0) = 0,
\]
and
\[
\|u(A) - A_0\|_{W_{A_1}^{1,p}(X)} < C\|A - A_0\|_{W_{A_1}^{1,p}(X)}.
\]
\end{thm}

The norm condition \eqref{thm:Feehan_proposition_3-4-4_Lp} in the hypotheses of Theorem \ref{thm:Feehan_proposition_3-4-4_Lp} is often stronger than convenient, so in this monograph we shall use the following variant.

\begin{thm}[Existence of $W^{2,q}$ Coulomb gauge transformations for $W^{1,q}$ connections that are $L^r$ close to a reference connection]
\label{thm:Feehan_proposition_3-4-4_Lp_Lr_close}
(See Feehan and Maridakis \cite[Theorem 15]{Feehan_Maridakis_Lojasiewicz-Simon_coupled_Yang-Mills}.) 
Let $(X,g)$ be a closed, smooth Riemannian manifold of dimension $d \geq 2$, and $G$ be a compact Lie group, and $P$ be a smooth principal $G$-bundle over $X$. If $A_1$ is a $C^\infty$ reference connection on $P$, used in the definition of Sobolev and H\"older norms, and $A_0$ is a Sobolev connection on $P$ of class $W^{1,q}$ with $d/2<q<\infty$, and $r$ is a constant that obeys $d < r \leq q^* = dq/(d-q)$ if $q<d$ or $d < r < \infty$ if $q\geq d$, then there are constants $\zeta = \zeta(A_0,A_1,g,G,r) \in (0,1]$ and $C = C(A_0,A_1,g,G,r) \in (0,\infty)$ with the following significance. If $A$ is a $W^{1,q}$ connection on $P$ that obeys
\begin{equation}
\label{eq:Feehan_3-4-4_Lp_A_minus_A0_Lr_close}
\|A - A_0\|_{L^r(X)} < \zeta,
\end{equation}
then there exists $u \in \Aut^{2,q}(P)$ such that
\[
d_{A_0}^*(u(A) - A_0) = 0,
\]
and
\[
\|u(A) - A_0\|_{L^r(X)} < C\|A - A_0\|_{L^r(X)}.
\]
\end{thm}

We shall also require an analogue of Freed and Uhlenbeck \cite[Theorem 3.2]{FU}, allowing for arbitrary $d\geq 2$ (instead of $d=4$), compact Lie groups $G$ (instead of $G=\SU(2)$), and $W^{1,p}$ connections (rather than $W^{2,2}$), and $W^{2,p}$ gauge transformations (rather than $W^{2,3}$) with $p\in(d/2,\infty)$, and real analytic rather than smooth embedding $\Phi$. The proof of our Theorem \ref{thm:Feehan_proposition_3-4-4_Lp_Lr_close} relies on our version of the Quantitative Implicit Function Theorem (see Feehan and Maridakis \cite[Theorem F.1]{Feehan_Maridakis_Lojasiewicz-Simon_coupled_Yang-Mills}, which is valid in the smooth ($C^k$ for integer $k\geq 1$ or $C^\infty$) or (real or complex) analytic categories, and when $p>d/2$ that also yields a proof of Theorem \ref{thm:Feehan_proposition_3-4-4_Lp}. All of the maps described as smooth in the proofs of \cite[Theorems 3.2 and 4.4]{FU} or \cite[Theorems 14 or 15]{Feehan_Maridakis_Lojasiewicz-Simon_coupled_Yang-Mills} are actually analytic and so we obtain the following corollaries. 

\begin{cor}[Analytic slice theorem for $W^{1,q}$-small open neighborhoods]
\label{cor:Freed_Uhlenbeck_3-2_W1q}
(Compare Freed and Uhlenbeck \cite[Theorems 3.2 and 4.4]{FU} when $d=4$ and Feehan and Maridakis \cite[Corollary 18]{Feehan_Maridakis_Lojasiewicz-Simon_coupled_Yang-Mills} for arbitrary $d\geq 2$.)
Assume the hypotheses of Theorem \ref{thm:Feehan_proposition_3-4-4_Lp} and that $\eps\in(0,1]$ is as in its conclusion. If 
\begin{equation}
\label{eq:W2q_gauge_transformations_orthogonal_stabilizer}
  \Aut_0^{2,q}(P) := \Exp\left(\left(\Ker\Delta_{A_0}\right)^\perp\cap W_{A_1}^{2,q}(X;\ad P)\right) \subset \Aut^{2,q}(P),
\end{equation}
where $\Delta_{A_0} = d_{A_0}^*d_{A_0}$, then there is an analytic diffeomorphism,
\begin{equation}
\label{eq:Local_splitting_W1q_connections_as_W2q_gauge_transformations_W1q_Coulomb_gauge}
\Phi:\sA^{1,q}(P) \supset B_\zeta(A_0) \ni A \mapsto (u(A)-A_0, u) \in \Ker d_{A_0}^*\cap W_{A_1}^{1,q}(X;T^*X\otimes\ad P) \times \Aut_0^{2,q}(P),
\end{equation}
from an open ball
\[
  B_\zeta(A_0) := \{A\in \sA^{1,q}(P): \|A-A_0\|_{W^{1,q}(X)} < \zeta\}
\]
onto an open neighborhood of $(0,\id_P)$.
\end{cor}


\begin{cor}[Analytic slice theorem for $L^r$-small open neighborhoods]
\label{cor:Freed_Uhlenbeck_3-2_Lr}
(Compare Freed and Uhlenbeck \cite[Theorems 3.2 and 4.4]{FU} when $d=4$ and see Feehan and Maridakis \cite[Corollary 18]{Feehan_Maridakis_Lojasiewicz-Simon_coupled_Yang-Mills} for arbitrary $d\geq 2$.)
Assume the hypotheses of Theorem \ref{thm:Feehan_proposition_3-4-4_Lp_Lr_close} and that $\zeta\in(0,1]$ is as in its conclusion. If $\Aut^{1,r}(P)$ is the Banach Lie group of all $W^{1,r}$ gauge transformations of $P$ and
\[
  \Aut_0^{1,r}(P) := \Exp\left(\left(\Ker\Delta_{A_0}\right)^\perp\cap W_{A_1}^{1,r}(X;\ad P)\right) \subset \Aut^{1,r}(P),
\]
then there is an analytic diffeomorphism,
\begin{multline*}
\Phi:\sA^r(P) \supset U_\zeta(A_0) \ni A \mapsto (u(A)-A_0, u) \in \Ker d_{A_0}^*\cap L^r(X;T^*X\otimes\ad P) \times \Aut_0^{1,r}(P),
\end{multline*}
from an open ball
\[
  U_\zeta(A_0) := \{A\in \sA^r(P): \|A-A_0\|_{L^r(X)} < \zeta\}
\]
onto an open neighborhood of $(0,\id_P)$, where $\sA^r(P) = A_0+L^r(X;T^*X\otimes\ad P)$ is the affine space of all $L^r$ connections on $P$. Moreover, the restriction of $\Phi$ to $\sA^{1,q}(P) \subset \sA^r(P)$ is an analytic diffemorphism onto an open neighborhood of $(0,\id_P)$:
\begin{multline*}
\Phi:\sA^{1,q}(P) \cap U_\zeta(A_0) \ni A \mapsto (u(A)-A_0, u) \in \Ker d_{A_0}^*\cap W^{1,q}(X;T^*X\otimes\ad P) \times \Aut_0^{2,q}(P).
\end{multline*}
\end{cor}

As customary, we denote
\begin{equation}
\label{eq:Minimal_stabilizer}  
\sB^{*;1,q}(P) := \left\{A \in \sA^{1,q}(P): \Stab(A) = \Center(G) \right\}/\Aut^{2,q}(P),
\end{equation}
an open subspace of $\sB^{1,q}(P)$ by Feehan and Leness \cite[Lemma 10.4.1 (2)]{Feehan_Leness_introduction_virtual_morse_theory_so3_monopoles}, where
\begin{equation}
\label{eq:Stabilizer}
  \Stab(A) := \{u \in \Aut(P): u(A) = A\}
\end{equation}
can be regarded as a subgroup of a fiber $G\cong P_{x_0}$, for some basepoint $x_0 \in G$ \cite[p. 132]{DK}. We shall use the following consequence of Corollary \ref{cor:Freed_Uhlenbeck_3-2_W1q}.

\begin{cor}[Real analytic Banach manifold structure on the quotient space of $W^{1,q}$ connections]
\label{cor:Slice}
(Compare Donaldson and Kronheimer \cite[Proposition 4.2.9]{DK} and Freed and Uhlenbeck \cite[Corollary, p. 50]{FU} when $d=4$ and see Feehan and Maridakis \cite[Corollary 18]{Feehan_Maridakis_Lojasiewicz-Simon_coupled_Yang-Mills} for arbitrary $d\geq 2$.)  
Let $(X,g)$ be a closed, smooth Riemannian manifold of dimension $d \geq 2$, and $G$ be a compact Lie group, and $P$ be a smooth principal $G$-bundle over $X$, and $q$ obey $d/2<q<\infty$. If $A_1$ is a $C^\infty$ reference connection on $P$ and $[A] \in \sB^{1,q}(P)$, then there is a constant $\zeta = \zeta(A_1,[A],g,G,q) \in (0,1]$ with the following significance. If
\[
\bB_\zeta(A) := A+\left\{a \in \Ker d_A^*\cap W_{A_1}^{1,q}(X;T^*X\otimes\ad P): \|a\|_{W_{A_1}^{1,q}(X)} < \zeta\right\},
\]
then the map,
\[
\pi_A: \bB_\zeta(A)/\Stab(A) \ni [A+a] \mapsto [A+a] \in \sB^{1,q}(P),
\]
is a homeomorphism onto an open neighborhood of $[A]\in \sB^{1,q}(P)$. For $A+a \in \bB_\zeta(A)$, the stabilizer of $A+a$ in $\Stab(A)$ is naturally isomorphic to that of $A+a$ in $\Aut^{2,q}(P)$. In particular, the inverse coordinate charts, $\pi_A$, determine real analytic transition functions for $\sB^{*;1,q}(P)$, giving it the structure of a real analytic Banach manifold, and each map $\pi_A$ is a real analytic diffeomorphism from the
open subset of points $[A+a] \in \bB_\zeta(A)/\Stab(A)$ where $A+a$ has stabilizer isomorphic to $\Center(G)$.
\end{cor}

\subsection{Uhlenbeck's Weak Compactness Theorem}
\label{subsec:Uhlenbeck_weak_compactness_theorem}
In order to precisely determine the dependencies of the constant $\eps$ in Theorem \ref{mainthm:Uhlenbeck_Chern_corollary_4-3_prelim}, we need to reinterpret the usual conclusion about sequential compactness in Uhlenbeck's Weak Compactness Theorem (see \cite[Theorem 1.5 or 3.6]{UhlLp}) in terms of compactness with respect to a metric topology. We first recall some related facts from functional analysis and topology. Let $\sX$ be a real Banach space and let $\sX^*$ denote its continuous dual space. The weak topology, $\tau = \tau(\sX,\sX^*)$, on $\sX$ is the coarsest (or weakest) topology (in the sense of \cite[Section 3.1]{Brezis}) on $\sX$ associated to the collection of maps $\{\varphi_\alpha\}_{\alpha\in\sX^*}$, where $\varphi_\alpha:\sX \ni x \mapsto \langle x, \alpha \rangle \in \RR$ \cite[Section 3.2]{Brezis} and $\langle \cdot, \cdot \rangle: \sX\times\sX^* \to \RR$ is the canonical pairing. This topology is Hausdorff by \cite[Proposition 3.3]{Brezis}. If $\{x_n\}_{n\in\NN} \subset \sX$, then $x_n \rightharpoonup x$ (weakly in $\sX$) as $n\to\infty$ if and only if $\langle x, \alpha_n\rangle \to \langle x, \alpha \rangle$ as $n\to\infty$ for all $x \in \sX$ \cite[Proposition 3.1 or 3.5 (i)]{Brezis}. Moreover, if $x_n \rightharpoonup x$ as $n\to\infty$, then the sequence $\{\|x_n\|\}_{n\in\NN}$ is bounded and $\|x\| \leq \liminf_{n\to\infty}\|x_n\|$ by \cite[Proposition 3.5 (iii)]{Brezis}.

If $\sX^*$ is separable, then the restriction of the weak topology, $\tau$, to the unit ball, $B_\sX \subset \sX$, is metrizable \cite[Theorem 3.29]{Brezis}. Indeed, if $\boldsymbol{\alpha} := \{\alpha_m\}_{m\in\NN} \subset \sB_{\sX^*}$ is a countable dense subset of the unit ball in $\sX^*$, then
\[
[x]_{\boldsymbol{\alpha}} := \sum_{m=1}^\infty \frac{1}{2^m}|\langle x, \alpha_m \rangle|
\]
is a norm on $\sX$ with $[x]_{\boldsymbol{\alpha}} \leq \|x\|$; the topology on $B_\sX$ defined by the corresponding metric, $D_{\boldsymbol{\alpha}}(x,y) := [x-y]_{\boldsymbol{\alpha}}$ for all $x,y \in \sX$ coincides with the restriction of $\tau$ to $B_\sX$. See the proofs of \cite[Theorems 3.28 and 3.29]{Brezis}.

We shall now apply analogues of the preceding ideas to the topology of spaces of connections, $\sA^{1,p}(P)$ and $\sB^{1,p}(P)$.

\begin{defn}[Uhlenbeck convergence]
\label{defn:Uhlenbeck_convergence_metric_compactness}
Let $G$ be a Lie group, $P$ be a smooth principal $G$-bundle over a smooth Riemannian manifold $(X,g)$ of dimension $d \geq 2$, and $p \in (d/2,\infty)$, and $A_1$ be a $C^\infty$ reference connection on $P$. If $\{A_n\}_{n\in\NN} \subset \sA^{1,p}(P)$ is a sequence of $W^{1,p}$ connections on $P$, then we say that $\{A_n\}_{n\in\NN}$ \emph{Uhlenbeck converges in $\sA^{1,p}(P)$} to a $W^{1,p}$ connection $A$ on $P$ if there is a sequence $\{u_n\}_{n\in\NN} \subset \Aut^{2,p}(P)$ of $W^{2,p}$ gauge transformations of $P$ such that
\begin{equation}
\label{eq:Uhlenbeck_convergence}
    u_n(A_n) \rightharpoonup A \quad\text{(weakly) in } W_{A_1}^{1,p}(X;T^*X\otimes\ad P) \quad\text{as } n\to\infty.
\end{equation}
If $\{[A_n]\}_{n\in\NN} \subset \sB^{1,p}(P)$ is a sequence of gauge-equivalence classes of $W^{1,p}$ connections on $P$, then we say that $\{[A_n]\}_{n\in\NN}$ \emph{Uhlenbeck converges in $\sB^{1,p}(P)$} to a gauge-equivalence class $[A]$ of a $W^{1,p}$ connection on $P$ if $\{A_n\}_{n\in\NN}$ Uhlenbeck converges in $\sA^{1,p}(P)$ to $A$.
\end{defn}

For any $p\in [1,\infty)$, the continuous dual space $(W_{A_1}^{1,p}(X;T^*X\otimes\ad P))^*$ is separable by \cite[Theorem 1.15]{AdamsFournier}, since the same is true for $W_{A_1}^{1,p}(X;T^*X\otimes\ad P)$ by \cite[Theorem 3.6]{AdamsFournier}. Thus, we can make the

\begin{defn}[Uhlenbeck metric and topology]
\label{defn:Uhlenbeck_metric_topology}
Continue the notation of Definition \ref{defn:Uhlenbeck_convergence_metric_compactness}. Let $\boldsymbol{\alpha} := \{\alpha_m\}_{m\in\NN}$ be a countable dense subset of the unit ball in $(W_{A_1}^{1,p}(X;T^*X\otimes\ad P))^*$. We define the \emph{Uhlenbeck metric} on $\sB^{1,p}(P)$ by
\begin{equation}
\label{eq:Uhlenbeck_metric}
D_{\boldsymbol{\alpha}}([A],[B]) := \inf_{u\in\Aut^{2,p}(P)} \sum_{m=1}^\infty \frac{1}{2^m}|\alpha_m(u(A)-B)|, \quad\text{for all } [A], [B] \in \sB^{1,p}(P).
\end{equation}
We define the \emph{Uhlenbeck topology} on $\sB^{1,p}(P)$ to be the topology defined by the Uhlenbeck metric.
\end{defn}

Recall that two metrics, $D_1$ and $D_2$, on a set $S$ are \emph{strongly equivalent} if there exists a constant $C \in [1,\infty)$ such that $C^{-1}D_1(x,y) \leq D_2(x,y) \leq CD_1(x,y)$, for all $x,y \in S$, while they are \emph{(topologically) equivalent}\footnote{Recall that topologically equivalent metrics, $D_1$ and $D_2$, need not be strongly equivalent. Indeed, a metric $D$ on a set $S$ is topologically equivalent to the metrics $\min\{1,D\}$ and $D/(1+D)$, but not strongly equivalent to them.} if they determine the same topology on $S$ \cite[Section C.1.5]{Ok_real_analysis_economic_applications}. Moreover, metrics $D_1$ and $D_2$ are topologically equivalent if and only if the convergent sequences in $(D_1,S)$ are the same as the convergent sequences in $(D_2,S)$ --- for example, see \cite[Problem 1.1.12]{Gamelin_Greene_introduction_topology}.
Assuming that $D_{\boldsymbol{\alpha}}$ is a metric on $\sB^{1,p}(P)$ and given $q\in[1,\infty]$, we can define two additional metrics on $\sB^{1,p}(P)$ by
\begin{align}
  \label{eq:Lq_distance}
  D_q([A],[B]) := \inf_{u\in\Aut^{2,p}(P)} \|u(A)-B\|_{L^q(X)}
  \\
  \label{eq:Uhlenbeck_metric+Lq_distance}
  D_{\boldsymbol{\alpha},q}([A],[B]) := D_{\boldsymbol{\alpha}}([A],[B]) + D_q([A],[B]),
\end{align}
See Donaldson \cite{DonCompact} for definitions of related metrics on $\sB^{1,p}(P)$. In the following Lemma \ref{lem:Equivalence_Uhlenbeck_metric_convergence}, we verify the basic properties of $D_{\boldsymbol{\alpha}}$.

\begin{lem}[Equivalence of Uhlenbeck and metric convergence]
\label{lem:Equivalence_Uhlenbeck_metric_convergence}
Continue the notation of Definition \ref{defn:Uhlenbeck_metric_topology}.  Let $q = q(d,p) \in (d,\infty)$ be a constant obeying\footnote{For $p \in (d/2,\infty)$, note that $p^* > d$.} $q < p^*=dp/(d-p)$ when $p<d$ and unconstrained when $p \geq d$. Let $\boldsymbol{\alpha} := \{\alpha_m\}_{m\in\NN}$ be a countable dense subset of the unit ball in $(W_{A_1}^{1,p}(X;T^*X\otimes\ad P))^*$. Then the following hold.
\begin{enumerate}
\item
\label{item:D_alpha_is_metric}
  The expression for $D_{\boldsymbol{\alpha}}$ in \eqref{eq:Uhlenbeck_metric} defines a metric on $\sB^{1,p}(P)$.
\item
\label{item:Uhlenbeck_and_D_alpha_convergence_equivalent}
  A sequence $\{[A_n]\}_{n\in\NN}$ Uhlenbeck converges in $\sB^{1,p}(P)$ to $[A]$ if and only if it converges to $[A]$ with respect to the metric $D_{\boldsymbol{\alpha}}$ on $\sB^{1,p}(P)$.
\item
\label{item:D_alpha_and_D_alpha_q_metrics_topologically_equivalent}
  The metrics $D_{\boldsymbol{\alpha}}$ and $D_{\boldsymbol{\alpha},q}$ on $\sB^{1,p}(P)$ are topologically equivalent.
\end{enumerate}
\end{lem}

\begin{proof}
It is straightforward to verify Item \eqref{item:D_alpha_is_metric}. Consider Item \eqref{item:Uhlenbeck_and_D_alpha_convergence_equivalent}. For the ``only if'' direction, note that if $u_n(A_n) \rightharpoonup A$ as $n \to \infty$, then $D_{\boldsymbol{\alpha}}([A_n],[A]) \to 0$ as $n \to \infty$ by definition \eqref{eq:Uhlenbeck_metric} of $D_{\boldsymbol{\alpha}}$. For the ``if'' direction, choose $\delta \in (0,1]$ and suppose $D_{\boldsymbol{\alpha}}([A_n],[A]) \to 0$ as $n \to \infty$. For each $n\in\NN$, choose $v_n \in \Aut^{2,p}(P)$ such that
\[
  \sum_{m=1}^\infty \frac{1}{2^m}|\alpha_m(v_n(A_n)-A)|
  <
  \inf_{u\in\Aut^{2,p}(P)} \sum_{m=1}^\infty \frac{1}{2^m}|\alpha_m(u(A_n)-A)| + \delta.
\]
By assumption, there is a large enough $n_\delta \in \NN$ such that $D_{\boldsymbol{\alpha}}([A_n],[A]) < \delta$ for all $n\geq n_\delta$ and so
\[
  \sum_{m=1}^\infty \frac{1}{2^m}|\alpha_m(v_n(A_n)-A)| < 2\delta, \quad\text{for all } n\geq n_\delta.
\]
Hence, $|\alpha_m(v_n(A_n)-A)| < 2^{m+1}\delta$ for all $m\in\NN$ and all $n\geq n_\delta$. Consequently, $\alpha_m(v_n(A_n)-A) \to 0$ for all $m\in\NN$ as $n\to\infty$ and thus $\alpha(v_n(A_n)-A) \to 0$ for all $\alpha$ in the unit ball of $(W_{A_1}^{1,p}(X;T^*X\otimes\ad P))^*$ as $n\to\infty$, that is, $v_n(A_n) \rightharpoonup A$ as $n \to \infty$, as claimed.

Consider Item \eqref{item:D_alpha_and_D_alpha_q_metrics_topologically_equivalent}. Suppose $\{[A_n]\}_{n\in\NN} \subset \sB^{1,p}(P)$ is a sequence. If $\{[A_n]\}_{n\in\NN}$ converges to $[A]$ with respect to $D_{\boldsymbol{\alpha},q}$, then it clearly converges with respect to $D_{\boldsymbol{\alpha}}$. Conversely, if $\{[A_n]\}_{n\in\NN}$ converges to $[A]$ with respect to $D_{\boldsymbol{\alpha}}$, then it Uhlenbeck converges by Item \eqref{item:Uhlenbeck_and_D_alpha_convergence_equivalent} and so there is a sequence $\{u_n\}_{n\in\NN} \subset \Aut^{2,p}(P)$ of $W^{2,p}$ gauge transformations such that
\[
    u_n(A_n) \rightharpoonup A \quad\text{(weakly) in } W_{A_1}^{1,p}(X;T^*X\otimes\ad P) \quad\text{as } n\to\infty.
\]
The choice of $q$ ensures that the Sobolev embedding $W^{1,p}(X) \Subset L^q(X)$ is compact by the Kondrachov--Rellich Embedding Theorem \cite[Theorem 6.3]{AdamsFournier}. By \cite[Proposition 3.5 (iii)]{Brezis}, the weakly convergent sequence $\{u_n(A_n)\}_{n\in\NN}$ is bounded in $A_1+W_{A_1}^{1,p}(X;T^*X\otimes\ad P)$ and thus precompact in $A_1+L^q(X;T^*X\otimes\ad P)$ and hence, after passing to a subsequence,
\[
  u_n(A_n) \to A  \quad\text{(strongly) in } L^q(X;T^*X\otimes\ad P) \quad\text{as } n\to\infty.
\]
Therefore, $\{[A_n]\}_{n\in\NN}$ converges to $[A]$ with respect to $D_{\boldsymbol{\alpha},q}$. Consequently, the metrics $D_{\boldsymbol{\alpha}}$ and $D_{\boldsymbol{\alpha},q}$ are topologically equivalent, as claimed.
\end{proof}

For any $b \in [0,\infty)$, Lemma \ref{lem:Equivalence_Uhlenbeck_metric_convergence} allows us to restate the Uhlenbeck Weak Compactness Theorem \cite[Theorem 1.5 or 3.6]{UhlLp} for the subspace $\sB_b^{1,p}(P,g,p)$ in \eqref{eq:Quotient_space_W1p_connections_Lr_bound_curvature} in terms of compactness with respect to metric topologies rather than sequential compactness. Note that Item \eqref{item:Uhlenbeck_and_D_alpha_convergence_equivalent} in Lemma \ref{lem:Equivalence_Uhlenbeck_metric_convergence} implies that the topology defined by the metric $D_{\boldsymbol{\alpha}}$ is independent of $\boldsymbol{\alpha}$.

\begin{thm}[Uhlenbeck compactness for the quotient space of $W^{1,p}$ connections with a uniform $L^p$ bound on curvature]
\label{thm:Metric_Uhlenbeck_compactness}
(Compare Uhlenbeck \cite[Theorem 1.5 or 3.6]{UhlLp}.)
Let $G$ be a compact Lie group, $P$ be a smooth principal $G$-bundle over a smooth Riemannian manifold $(X,g)$ of dimension $d \geq 2$, and $p \in (d/2,\infty)$, and $A_1$ be a $C^\infty$ reference connection on $P$, and $b \in [0,\infty)$ be a constant. Then the space $\sB_b^{1,p}(P,g,p)$ is compact with respect to the Uhlenbeck topology.
\end{thm}

\begin{rmk}[Uhlenbeck compactness for the quotient space of $W^{1,p}$ connections with a small uniform $L^{d/2}$ bound on curvature]
\label{rmk:Metric_Uhlenbeck_compactness_Ldover2_curvature_bound}  
Theorem \ref{thm:Metric_Uhlenbeck_compactness} should continue to hold for the space $\sB_\eps^{1,p}(P,g,d/2)$ defined in \eqref{eq:Quotient_space_W1p_connections_Lr_bound_curvature} when the constant $\eps=\eps(g,G)\in(0,1]$ is sufficiently small. For results in this direction, see Feehan \cite{Feehan_lojasiewicz_inequality_ground_state}. 
\end{rmk}  

\begin{proof}
For metrizable topological spaces we recall that compactness is equivalent to sequential compactness \cite[Theorem 28.2]{Munkres_topology_second_edition}. If $\{[A_n]\}_{n\in\NN} \subset \sB_b^{1,p}(P,g,p)$ is a sequence, then by \cite[Theorem 1.5 or 3.6]{UhlLp} (see also \cite[Theorem 7.1]{Wehrheim_2004}) there is a subsequence $\{n_k\}_{k\in\NN} \subset \NN$ and, after relabeling, a sequence $\{u_n\}_{n\in\NN} \subset \Aut^{2,p}(P)$ of $W^{2,p}$ gauge transformations such that
\[
    u_n(A_n) \rightharpoonup A \quad\text{(weakly) in } W_{A_1}^{1,p}(X;T^*X\otimes\ad P) \quad\text{as } n\to\infty,
\]
where $[A] \in \sB_b^{1,p}(P,g,p)$. Hence, $\{[A_n]\}_{n\in\NN}$ converges with respect to $D_{\boldsymbol{\alpha}}$ (for any choice of $\boldsymbol{\alpha}$) and so $\sB_b^{1,p}(P,g,p)$ is sequentially compact and thus compact.
\end{proof}

Consequently, by Theorem \ref{thm:Metric_Uhlenbeck_compactness} the moduli subspace of flat connections \eqref{eq:Moduli_space_flat_connections},
\[
  M(P) = \sB_0^{1,p}(P,g,p),
\]
is also compact with respect to the Uhlenbeck topology. For a metric $D$ on $\sB^{1,p}(P)$ and $\zeta \in (0,\infty)$, let
\begin{equation}
\label{eq:D_ball}  
  \sU_{D;\zeta}([A]) := \{[B] \in \sB^{1,p}(P): D([A],[B]) < \zeta\}
\end{equation}
denote the open ball of $D$-radius $\zeta \in (0,\infty)$ and center $[A]$. 

\begin{lem}[Finite cover by open balls centered along the moduli subspace of flat connections]
\label{lem:Finite_open_covering}
Let $G$ be a compact Lie group, $P$ be a smooth principal $G$-bundle over a smooth Riemannian manifold $(X,g)$ of dimension $d \geq 2$, and $p \in (d/2,\infty)$, and $A_1$ be a $C^\infty$ reference connection on $P$, and $\boldsymbol{\alpha}$ and $q$ be as in Lemma \ref{lem:Equivalence_Uhlenbeck_metric_convergence}. If $\zeta \in (0,1]$ is a constant, there are a constant $\eps = \eps(A_1,g,G,p,q,\boldsymbol{\alpha},\zeta) \in (0,1]$ and a positive integer $\ell = \ell(A_1,g,G,p,q,\boldsymbol{\alpha},\zeta) \in \NN$ such that the following hold:
\begin{enumerate}
\item If $D = D_{\boldsymbol{\alpha}}$, $D_{\boldsymbol{\alpha},q}$, or $D_q$, then the balls \eqref{eq:D_ball} are open in the Uhlenbeck topology and $\sB_\eps^{1,p}(P,g,p)$ has a finite cover by balls \eqref{eq:D_ball} of radius $\zeta$ and centers $[\Gamma_i]\in M(P)$ for $i=1,\ldots,\ell$, where $\eps$ and $\ell$ depend on $(A_1,g,G,p,\zeta)$ and on $\boldsymbol{\alpha}$, $(q,\boldsymbol{\alpha})$, or $q$, respectively.
\item Each $\Gamma_i$ is a $W^{1,p}$ flat connection and which may be chosen to be $C^\infty$ when $d\geq 3$ or $d=2$ and $p>4/3$.   
\end{enumerate}
\end{lem}

\begin{proof}
The space $\sB_b^{1,p}(P,g,p)$ is compact with respect to the Uhlenbeck topology by Theorem \ref{thm:Metric_Uhlenbeck_compactness}. For $D=D_q$ or $D_{\boldsymbol{\alpha},q}$ or $D_{\boldsymbol{\alpha}}$, the balls \eqref{eq:D_ball} are open in the Uhlenbeck topology and so any open cover of $\sB_b^{1,p}(P,g,p)$ by open neighborhoods,
\[
  \bigcup_{[A]\in\sB_b^{1,p}(P,g,p)}\sB_b^{1,p}(P,g,p)\cap \sU_{D;\zeta}([A]),
\]
has a finite subcover,
\[
    \bigcup_{i=1}^\ell\sB_b^{1,p}(P,g,p)\cap \sU_{D;\zeta}([A_i]).
\]
We now claim that, given $\zeta \in (0,1]$, there exists an $\eps = \eps \in (0,1]$ such that
\[
  \sB_\eps^{1,p}(P,g,p) \subset \bigcup_{[\Gamma] \in M(P)} \sU_{D;\zeta}([\Gamma]).
\]
Indeed, if the claim were false we could then choose $\eps_n = 1/n$ for $n\in\NN$ and a sequence $[A_n] \in \sB^{1,p}(P)$ such that  $\|F_{A_n}\|_{L^p(X)} \leq \eps_n$ and so $\|F_{A_n}\|_{L^p(X)} \to 0$ as $n \to \infty$ while $[A_n] \notin \sU_{D;\zeta}([\Gamma])$ for any $[\Gamma] \in M(P)$. However, this would contradict Theorem \ref{thm:Metric_Uhlenbeck_compactness}, which implies that, after passing to a subsequence and relabelling, $D([A_n],[\Gamma]) \to 0$ as $n \to \infty$, for some flat connection $\Gamma$ on $P$, and hence that $[A_n] \in \sU_{D;\zeta}([\Gamma])$ for all $n \geq n_\zeta$.

Because the space $\sB_\eps^{1,p}(P,g,p)$ is compact, the open cover
\[
  \left\{\sB_\eps^{1,p}(P,g,p)\cap\sU_{D;\zeta}([\Gamma]\right\}_{[\Gamma]\in M(P)}
\]
has a finite open subcover,
\[
  \left\{\sB_\eps^{1,p}(P,g,p)\cap\sU_{D;\zeta}([\Gamma_i])\right\}_{i=1}^\ell,
\]
for an integer $\ell \geq 1$ and $W^{1,p}$ flat connections $\Gamma_i$ on $P$, for $i=1,\ldots,\ell$, and therefore
\[
  \sB_\eps^{1,p}(P,g,p) \subset \bigcup_{i=1}^\ell \sU_{D;\zeta}([\Gamma_i]).
\]
We note from the forthcoming Theorem \ref{thm:Regularity_weakly_Yang-Mills_W1p_connection} that if $p \in (4/3,\infty)$ obeys $p > d/2$ and $A_\infty$ is a $W^{1,p}$ connection on $P$ that is a weak solution to the Yang--Mills equation, then there is a $W^{2,p}$ gauge transformation $u_\infty \in \Aut^{2,p}(P)$ such that $u_\infty(A_\infty)$ is a $C^\infty$ Yang--Mills connection on $P$. Hence, by replacing $\Gamma_i$ with $u_i(\Gamma_i)$ for such a $u_i\in\Aut^{2,p}(P)$ and relabeling, we may assume that each $\Gamma_i$ is a $C^\infty$ flat connection on $P$. 
\end{proof}

\subsection{Existence of a flat connection on a principal bundle supporting a connection with $L^p$-small curvature}
\label{subsec:Existence_flat_connection_principal_bundle_supporting_connection_Lp-small_curvature}
We conclude this section with our proofs of Theorem \ref{mainthm:Uhlenbeck_Chern_corollary_4-3_prelim} and Corollary \ref{maincor:Uhlenbeck_Chern_corollary_4-3_H1Gamma_zero}.

\begin{proof}[Proof of Theorem \ref{mainthm:Uhlenbeck_Chern_corollary_4-3_prelim}]
We begin by recalling our exposition \cite[Theorem 5.1 (1)]{Feehan_yangmillsenergygapflat} of Uhlenbeck's proof \cite[Corollary 4.3]{UhlChern} that the moduli space $M(P)$ of flat connections is non-empty. Uhlenbeck appeals to her Weak Compactness Theorem \cite[Theorem 1.5, p. 34 or Theorem 3.6, p. 41]{UhlLp} for a sequence of $W^{1,p}$ connections, $\{A_n\}_{n\in\NN}$, on $P$ with a uniform $L^p$ bound on their curvatures, $F_{A_n}$, and observes\footnote{The argument here is reminiscent of the direct minimization algorithm of Sedlacek \cite{Sedlacek} in the case $d=4$; see his statements and proofs of \cite[Theorems 4.1 and 4.3, and Proposition 4.2]{Sedlacek}.}
that this yields the existence of a $W^{1,p}$ flat connection, $\Gamma$, on $P$ and a $W^{2,p}$ gauge transformation, $u \in \Aut^{2,p}(P)$, such that $u(A)$ is weakly $W_\Gamma^{1,p}$ close to $\Gamma$ and (strongly) $L^q(X)$ close to $A_\infty$ by virtue of the Kondrachev--Rellich compact embedding $W^{1,p}(X)\Subset L^q(X)$ \cite[Theorem 6.3]{AdamsFournier} with
\[
\begin{cases}
1 \leq q < dp/(d-p), &\text{for } p<d,
\\
1 \leq q < \infty, &\text{for } p=d,
\\
1 \leq q \leq \infty, &\text{for } p>d.
\end{cases}
\]
Since $p > d/2$ (and thus $dp/(d-p) > 2p > d$) by hypothesis, we may restrict the preceding Sobolev exponent, $q$, to one obeying
\begin{equation}
\label{eq:Uhlenbeck_Chern_corollary_4-3_choice_q_page_454}
d < q < 2p.
\end{equation}
To see the existence of a flat connection on $P$, one argues by contradiction. Suppose that for every $\eps \in (0, 1]$, there exists a $W^{1,p}$ connection, $A$, on $P$ such that $\|F_A\|_{L^p(X)} \leq \eps$ but $P$ does not support a flat connection. Therefore, we may choose a sequence of $W^{1,p}$ connections, $\{A_n\}_{n\in\NN}$ on $P$, such that
\[
\eps_n := \|F_{A_n}\|_{L^p(X)} \searrow 0, \quad\text{as } n \to \infty.
\]
Uhlenbeck's Weak Compactness Theorem \cite[Theorem 1.5]{UhlLp} for sequences of $W^{1,p}$ connections yields the existence of subsequence, also denoted $\{A_n\}_{n\in\NN}$, a sequence of gauge transformations, $\{u_n\}_{n\in\NN} \subset \Aut^{2,p}(P)$, and a $W^{1,p}$ connection, $\Gamma$ on $P$, such that as $n \to \infty$,
\begin{align*}
u_n(A_n) - \Gamma &\rightharpoonup 0 \quad\text{weakly in } W_\Gamma^{1,p}(X; T^*X\otimes\ad P),
\\
u_n(A_n) - \Gamma &\to 0 \quad\text{strongly in } L^q(X; T^*X\otimes\ad P).
\end{align*}
But \cite[Theorem 1.5]{UhlLp} also implies that
\[
\|F_\Gamma\|_{L^p(X)} \leq \sup_{n\in\NN} \|F_{A_n}\|_{L^p(X)},
\]
and so
\[
\|F_\Gamma\|_{L^p(X)}
\leq
\lim_{m\to\infty}\sup_{n\geq m} \|F_{A_n}\|_{L^p(X)}
=
\limsup_{m\to\infty} \|F_{A_m}\|_{L^p(X)}
=
\lim_{m\to\infty}\eps_m
= 0.
\]
Hence, $F_\Gamma \equiv 0$ a.e. on $X$, that is, $\Gamma$ is necessarily flat, a contradiction. Thus, $\eps \in (0, 1]$ exists, as claimed. 
  
Now let $q = q(d,p) \in (d,\infty)$ be as in the hypotheses of Lemma \ref{lem:Equivalence_Uhlenbeck_metric_convergence}. For $A$ as given in the hypotheses, relabel the $\Gamma_i$ provided by Lemma \ref{lem:Finite_open_covering} as $\Gamma$ and relabel $\zeta$ as $\zeta/2$, so that
\[
D_q([A],[\Gamma]) = \inf_{\tilde{u} \in \Aut^{2,p}(P)}\|\tilde{u}(A)-\Gamma\|_{L^q(X)} < \frac{\zeta}{2}.
\]
Hence, for some $u\in \Aut^{2,p}(P)$ of class $W^{2,p}$,
\begin{equation}
  \label{eq:Uhlenbeck_compactness_bound_prelim}
  \|u(A)-\Gamma\|_{L^q(X)} < \zeta.
\end{equation}
Choose $\zeta = \zeta(A_1,g,G,q(d,p)) = \zeta(A_1,g,G,p)\in (0,1]$ small enough that the hypothesis \eqref{eq:Feehan_3-4-4_Lp_A_minus_A0_Lr_close} of Theorem \ref{thm:Feehan_proposition_3-4-4_Lp_Lr_close} is satisfied by virtue of \eqref{eq:Uhlenbeck_compactness_bound_prelim}. Consequently, there exists a $W^{2,p}$ gauge transformation $w \in \Aut^{2,p}(P)$ such that, if we set $v = uw \in \Aut^{2,p}(P)$ (so $v(A) = v^*A = (uw)^*A = w^*u^*(A) = w(u(A))$, via the right action \eqref{eq:Right_action_gauge_transformations_on_connections} of $\Aut^{2,p}(P)$ on $\sA^{1,p}(P)$), then
\[
  d_\Gamma^*(v(A)-\Gamma)=0,
\]
which yields \eqref{eq:Uhlenbeck_Chern_corollary_4-3_uA-Gamma_global_Coulomb_gauge_prelim} and, for $C=C(A_1,g,G,p,q)\in(0,1]$,
\begin{equation}
  \label{eq:Quantitative_slice}
  \|v(A)-\Gamma\|_{L^q(X)} < C\|u(A)-\Gamma\|_{L^q(X)}.
\end{equation}
By compactness of $M(P)$, we may assume that $\zeta$ is independent of $\Gamma$. But from \eqref{eq:Donaldson_Kronheimer_2-1-14}, we have
\[
  F_{v(A)} = F_\Gamma + d_\Gamma(v(A)-\Gamma) + \frac{1}{2}[v(A)-\Gamma, v(A)-\Gamma]
\]
and because $F_\Gamma=0$, we obtain from \eqref{eq:Uhlenbeck_Chern_corollary_4-3_uA-Gamma_global_Coulomb_gauge_prelim} that
\begin{equation}
  \label{eq:d+d*_vA-Gamma}
  (d_\Gamma+d_\Gamma^*)(v(A)-\Gamma) = F_{v(A)} - \frac{1}{2}[v(A)-\Gamma, v(A)-\Gamma].
\end{equation}
By hypothesis, $r \in (1,p]$. For $r < d$ and $s=r*=dr/(d-r)\in[d/(d-1),\infty)$, we have $1/r=1/s+1/d > 1/s+1/q$ for $q \in (d,\infty)$, giving a continuous Sobolev multiplication map, 
\begin{equation}
  \label{eq:Ls_times_Lq_to_Lr}
  L^s(X)\times L^q(X) \to L^r(X),
\end{equation}
and a continuous Sobolev embedding,
\begin{equation}
  \label{eq:W1r_to_Ls}
  W^{1,r}(X)\subset L^s(X),
\end{equation}
by \cite[Theorem 4.12]{AdamsFournier}. For $r\geq d$, we have the continuous Sobolev embedding \eqref{eq:W1r_to_Ls} for any $s\in[1,\infty)$ by \cite[Theorem 4.12]{AdamsFournier} and, for\footnote{Note that $p\geq r$ by hypothesis and so $p\geq d$ in this case and $q \in (d,\infty)$ is unconstrained when $r\geq d$ by our comments regarding $q$ in the statement of Lemma \ref{lem:Equivalence_Uhlenbeck_metric_convergence}.} $q \in (r,\infty)$, we may choose $s=s(d,q)\in(d,\infty)$ large enough that $1/r \geq 1/s+1/q$, again giving the continuous Sobolev multiplication map \eqref{eq:Ls_times_Lq_to_Lr}.

Thus, for $r \in (1,p]$ and $s<\infty$ as just defined and $q = q(d,p)$ as chosen initially, the $L^r$ estimate for the first-order elliptic operator $d_\Gamma+d_\Gamma^*$ (for example, Feehan \cite[Theorem 14.60]{Feehan_yang_mills_gradient_flow_v4}) gives
\begin{align*}
  {}&\|v(A)-\Gamma\|_{W_{A_1}^{1,r}(X)}
  \\
    &\quad \leq C\|(d_\Gamma+d_\Gamma^*)(v(A)-\Gamma)\|_{L^r(X)} + C\|v(A)-\Gamma\|_{L^r(X)} \quad\text{($L^r$ elliptic estimate)}
  \\
    &\quad\leq C\|F_{v(A)}\|_{L^r(X)} + C\|[v(A)-\Gamma, v(A)-\Gamma]\|_{L^r(X)} + C\|v(A)-\Gamma\|_{L^r(X)}
      \quad\text{(by \eqref{eq:d+d*_vA-Gamma})}
  \\
    &\quad\leq C\|F_A\|_{L^r(X)} + C\|v(A)-\Gamma\|_{L^q(X)}\|v(A)-\Gamma\|_{L^s(X)} + C\|v(A)-\Gamma\|_{L^r(X)}
      \quad\text{(by \eqref{eq:Ls_times_Lq_to_Lr})}
  \\
    &\quad\leq C\|F_A\|_{L^r(X)} + C\|v(A)-\Gamma\|_{L^q(X)}\|v(A)-\Gamma\|_{W_{A_1}^{1,r}(X)} + C\|v(A)-\Gamma\|_{L^r(X)}
      \quad\text{(by \eqref{eq:W1r_to_Ls}),}
\end{align*}
where we also apply the Kato Inequality \cite[Equation (6.20)]{FU} to obtain the last inequality. In summary,
\begin{multline}
  \label{eq:Uhlenbeck_Chern_corollary_4-3_uA-Gamma_W1p_bound_prelim_pre-rearrangement}
  \|v(A)-\Gamma\|_{W_{A_1}^{1,r}(X)}
  \leq C\|F_A\|_{L^r(X)} + C\|v(A)-\Gamma\|_{L^q(X)}\|v(A)-\Gamma\|_{W_{A_1}^{1,r}(X)}
  \\
  + C\|v(A)-\Gamma\|_{L^r(X)}.
\end{multline}
Using \eqref{eq:Uhlenbeck_compactness_bound_prelim}, rearrangement in the preceding inequality \eqref{eq:Uhlenbeck_Chern_corollary_4-3_uA-Gamma_W1p_bound_prelim_pre-rearrangement} for small enough $\zeta=\zeta(A_1,g,G,p\in(0,1]$ now yields \eqref{eq:Uhlenbeck_Chern_corollary_4-3_uA-Gamma_W1p_bound_prelim} with constant $C=C(A_1,g,G,r)\in[1,\infty)$.

Finally, we observe that
\begin{align*}
  \|v(A)-\Gamma\|_{W_{A_1}^{1,p}(X)} &\leq  C\|F_A\|_{L^r(X)} + C\|v(A)-\Gamma\|_{L^r(X)} \quad\text{(by \eqref{eq:Uhlenbeck_Chern_corollary_4-3_uA-Gamma_W1p_bound_prelim})}
  \\
  &\leq  C\|F_A\|_{L^p(X)} + C\|v(A)-\Gamma\|_{L^q(X)} \quad\text{(since $r\leq p$ and $r<q$)}
  \\
  &\leq  C\|F_A\|_{L^p(X)} + C\|u(A)-\Gamma\|_{L^q(X)} \quad\text{(by \eqref{eq:Quantitative_slice})}
  \\
  &\leq C(\eps+\zeta) \quad\text{(by \eqref{eq:Lp_norm_FA_lessthan_epsilon} and \eqref{eq:Uhlenbeck_compactness_bound_prelim})}.
\end{align*}
The conclusion \eqref{eq:Uhlenbeck_compactness_bound} now follows for $\eps=\eps(A_1,g,G,p,\sigma)\in (0,1]$ and $\zeta=\zeta(A_1,g,G,p,\sigma)\in (0,1]$ small enough that $C(\eps+\zeta) < \sigma$. This completes the proof of Theorem \ref{mainthm:Uhlenbeck_Chern_corollary_4-3_prelim}.
\end{proof}

We now give the 

\begin{proof}[Proof of Corollary \ref{maincor:Uhlenbeck_Chern_corollary_4-3_H1Gamma_zero}]
By hypothesis, we have $H_\Gamma^1(X;\ad P) = (0)$. The forthcoming \eqref{eq:H_Gamma^i_adP_W1p_harmonic} gives
\begin{multline*}
  \bH_\Gamma^1(X;\ad P)
  =
  \Ker\left(d_\Gamma+d_\Gamma^*:W_{A_1}^{1,p}(X;\wedge^i(T^*X)\otimes\ad P) \right.
  \\
  \to \left. L^p(X;\wedge^{i+1}(T^*X)\otimes\ad P) \oplus L^p(X;\wedge^{i-1}T^*X\otimes\ad P)\right),
\end{multline*}
and we recall that $\bH_\Gamma^1(X;\ad P) \cong H_\Gamma^1(X;\ad P)$. Thus,
\[
  \Ker \Delta_\Gamma \cap W_{A_1}^{1,p}(X;\wedge^i(T^*X)\otimes\ad P)
  =
  \bH_\Gamma^1(X;\ad P)
  =
  (0),
\]  
where $\Delta_\Gamma = d_\Gamma^*d_\Gamma + d_\Gamma^*d_\Gamma$ from \eqref{eq:Lawson_page_93_Hodge_Laplacian}. Again, the $L^r$ estimate for the first-order elliptic operator $d_\Gamma+d_\Gamma^*$ (see Feehan \cite[Theorem 14.60]{Feehan_yang_mills_gradient_flow_v4}) yields, for $C_0=C_0(A_1,g,G,\Gamma,r)\in[1,\infty)$, 
\begin{equation}
  \label{eq:Basic_Lr_elliptic_estimate_dGamma+dGamma*}
  \|a\|_{W_{A_1}^{1,r}(X)} \leq C_0\|(d_\Gamma+d_\Gamma^*)a\|_{L^r(X)} + \|a\|_{L^r(X)},
  \quad\text{for all } a \in  W_{A_1}^{1,p}(X;T^*X\otimes\ad P).
\end{equation}
The spectrum of the Laplace operator $\Delta_\Gamma$ on $W_{A_1}^{1,p}(X;T^*X\otimes\ad P)$ is countable without accumulation points, consisting of non-negative, real eigenvalues, $\lambda$, with finite multiplicities equal to $\dim\Ker(\Delta_\Gamma-\lambda\,\id)$ (see, for example, Feehan and Maridakis \cite[Proposition 2.2.3, p. 23]{Feehan_Maridakis_Lojasiewicz-Simon_coupled_Yang-Mills}. Since $\Ker \Delta_\Gamma \cap W_{A_1}^{1,p}(X;\wedge^i(T^*X)\otimes\ad P) (0)$, the least eigenvalue, $\lambda_1$, is positive and we have the eigenvalue calculation (see, for example, Reed and Simon \cite[Theorem XIII.1, p. 76]{Reed_Simon_v4})
\begin{equation}
  \label{eq:Delta_Gamma_min_max}
  \lambda_1 = \inf_{\|b\|_{L^2(X)}=1} (b, \Delta_\Gamma b)_{L^2(X)},
\end{equation}
where the infimum is over $b \in \Dom(\Delta_\Gamma) \subset L^2(X;T^*X\otimes\ad P)$ with $\|b\|_{L^2(X)} = 1$. Observe that
\begin{align*}
  \lambda_1\|a\|_{L^2(X)}^2
  &= (\lambda_1 a,a)_{L^2(X)}
  \\
  &\leq (\Delta_\Gamma a,a)_{L^2(X)} \quad\text{(by \eqref{eq:Delta_Gamma_min_max})}
  \\
  &= ((d_\Gamma^*d_\Gamma + d_\Gamma d_\Gamma^*)a,a)_{L^2(X)}
  \\
  &= (d_\Gamma a, d_\Gamma a)_{L^2(X)} + (d_\Gamma^* a, d_\Gamma^* a)_{L^2(X)}
  \\
  &= \|d_\Gamma a\|_{L^2(X)}^2 + \|d_\Gamma^* a\|_{L^2(X)}^2 = \|(d_\Gamma + d_\Gamma^*)a\|_{L^2(X)}^2,
\end{align*}
and so we obtain the eigenvalue-operator estimate
\begin{equation}
  \label{eq:Laplace_operator_eigenvalue_estimate}
  \|a\|_{L^2(X)} \leq \frac{1}{\sqrt{\lambda_1}}\|(d_\Gamma + d_\Gamma^*)a\|_{L^2(X)},
  \quad\text{for all } a \in  W_{A_1}^{1,p}(X;T^*X\otimes\ad P).
\end{equation}
We claim that interpolation and \eqref{eq:Laplace_operator_eigenvalue_estimate} yields the improved elliptic estimate:
\begin{equation}
  \label{eq:Improved_Lr_elliptic_estimate_dGamma+dGamma*}
  \|a\|_{W_{A_1}^{1,r}(X)} \leq C\|(d_\Gamma+d_\Gamma^*)a\|_{L^r(X)},
  \quad\text{for all } a \in  W_{A_1}^{1,p}(X;T^*X\otimes\ad P).
\end{equation}
The preceding inequality is immediate when $r = 2$. For $r \in (2,\infty)$, we may use $2 < r < r^*$, where $r^*$ is chosen to ensure that the continuous Sobolev embedding $W^{1,r}(X) \subset L^{r^*}(X)$ holds, so (see Adams and Fournier \cite[Theorem 8.12, p. 85]{AdamsFournier})
\begin{inparaenum}[(\itshape i\upshape)]
\item $r^* = dr/(d-r)$ when $r<d$, and
\item $r^* \in (d,\infty)$ when $r=d$, and
\item $r^*=\infty$ when $r>d$.
\end{inparaenum}
We apply the interpolation inequality implied by Young's inequality and Littlewood's inequality (see Gilbarg and Trudinger \cite[Equations (7.6), (7.9), and (7.10), pp. 145--146]{GT}) with $\theta\in (0,1)$ determined by $1/r = \theta/2 + (1-\theta)/r^*$, so $\theta = (1/r-1/r^*)/(1/2-1/r^*)$, to give
\[
  \|a\|_{L^r(X)} \leq \eps\|a\|_{L^{r^*}(X)} + \eps^{-\theta}\|a\|_{L^2(X)},
  \quad\text{for all } a \in  W_{A_1}^{1,p}(X;T^*X\otimes\ad P).
\]
Combining the preceding inequality with the Sobolev embedding, for $C= C(g,G,r)\in [1,\infty)$,
\[
  \|a\|_{L^{r^*}(X)} \leq C_1\|a\|_{W_{A_1}^{1,r}(X)},
  \quad\text{for all } a \in  W_{A_1}^{1,p}(X;T^*X\otimes\ad P).
\]
and our basic $L^r$ elliptic estimate \eqref{eq:Basic_Lr_elliptic_estimate_dGamma+dGamma*} yields
\begin{align*}
  \|a\|_{W_{A_1}^{1,r}(X)}
  &\leq C_0\|(d_\Gamma+d_\Gamma^*)a\|_{L^r(X)}
    + \eps C_1\|a\|_{W_{A_1}^{1,r}(X)} + \eps^{-\theta}\|a\|_{L^2(X)}
  \\
  &\leq C_0\|(d_\Gamma+d_\Gamma^*)a\|_{L^r(X)}
    + \eps C_1\|a\|_{W_{A_1}^{1,r}(X)} + \frac{\eps^{-\theta}}{\sqrt{\lambda_1}}\|(d_\Gamma + d_\Gamma^*)a\|_{L^2(X)}
    \textrm{(by \eqref{eq:Laplace_operator_eigenvalue_estimate})}.
\end{align*}  
Rerrangement for $\eps=\eps(g,G,r)=1/(2C_1)\in(0,1]$ and the fact that (see Gilbarg and Trudinger \cite[Equation (7.8), p. 146]{GT}) for $r \in (2,\infty)$,
\[
  \|(d_\Gamma + d_\Gamma^*)a\|_{L^2(X)} \leq \vol(X,g)^{\frac{1}{2}-\frac{1}{r}}\|(d_\Gamma + d_\Gamma^*)a\|_{L^r(X)},
  \quad\text{for all } a \in  W_{A_1}^{1,p}(X;T^*X\otimes\ad P),
\]  
now gives the claimed inequality \eqref{eq:Improved_Lr_elliptic_estimate_dGamma+dGamma*} for $r \in (2,\infty)$.

For $r \in (1,2)$, we may combine (see Gilbarg and Trudinger \cite[Equation (7.8), p. 146]{GT})
\[
  \|a\|_{L^r(X)} \leq \vol(X,g)^{\frac{1}{r}-\frac{1}{2}}\|a\|_{L^2(X)},
  \quad\text{for all } a \in  W_{A_1}^{1,p}(X;T^*X\otimes\ad P),
\]
with the basic $L^r$ elliptic estimate \eqref{eq:Basic_Lr_elliptic_estimate_dGamma+dGamma*} and the eigenvalue-operator estimate \eqref{eq:Laplace_operator_eigenvalue_estimate} to again give the claimed inequality \eqref{eq:Improved_Lr_elliptic_estimate_dGamma+dGamma*} for $r \in (1,2)$. This completes the proof of the claim \eqref{eq:Improved_Lr_elliptic_estimate_dGamma+dGamma*} for $r \in (1,\infty)$.

Our derivation of the inequality \eqref{eq:Uhlenbeck_Chern_corollary_4-3_uA-Gamma_W1p_bound_prelim_pre-rearrangement} in our proof of Theorem \ref{mainthm:Uhlenbeck_Chern_corollary_4-3_prelim} had relied in the the basic $L^r$ elliptic estimate \eqref{eq:Basic_Lr_elliptic_estimate_dGamma+dGamma*} and if we replace that application by that of $L^r$ elliptic estimate \eqref{eq:Improved_Lr_elliptic_estimate_dGamma+dGamma*} appropriate for the case $\Ker (d_\Gamma + d_\Gamma^*) = (0)$, we instead obtain
\begin{equation}
  \label{eq:Uhlenbeck_Chern_corollary_4-3_uA-Gamma_W1p_bound_refined_pre-rearrangement}
  \|v(A)-\Gamma\|_{W_{A_1}^{1,r}(X)}
  \leq C\|F_A\|_{L^r(X)} + C\|v(A)-\Gamma\|_{L^q(X)}\|v(A)-\Gamma\|_{W_{A_1}^{1,r}(X)}.
\end{equation}
If we now use \eqref{eq:Uhlenbeck_compactness_bound_prelim} and rearrangement in the preceding inequality \eqref{eq:Uhlenbeck_Chern_corollary_4-3_uA-Gamma_W1p_bound_refined_pre-rearrangement} for small enough $\zeta=\zeta(A_1,g,G,p\in(0,1]$, we obtain
\[
    \|v(A)-\Gamma\|_{W_{A_1}^{1,r}(X)} \leq C\|F_A\|_{L^r(X)},
\]
with constant $C=C(A_1,g,G,r)\in[1,\infty)$, as asserted in \eqref{eq:Uhlenbeck_Chern_corollary_4-3_uA-Gamma_W1p_bound_refined}. This completes the proof of Corollary \ref{maincor:Uhlenbeck_Chern_corollary_4-3_H1Gamma_zero}.
\end{proof}

We provide an elementary example of where Corollary \ref{maincor:Uhlenbeck_Chern_corollary_4-3_H1Gamma_zero} applies.

\begin{exmp}[Optimal estimate over simply connected manifolds]
\label{exmp:Optimal_estimate_over_simply_connected_manifolds}  
We now describe an elementary example of where the hypotheses of Corollary \ref{maincor:Uhlenbeck_Chern_corollary_4-3_H1Gamma_zero} hold, namely when $X$ is simply connected. We have $H_1(X;\ZZ) = (0)$ since $\pi_1(X) = \{1\}$ and $H_1(X;\ZZ)$ is isomorphic to the Abelianization of $\pi_1(X)$ by Lee \cite[Theorem 13.14, p. 352]{Lee_john_topological_manifolds}. Hence, $H_1(X;\RR) = (0)$ by the Universal Coefficient Theorem for homology
and thus, $H^{d-1}(X;\RR) = (0)$ by Poincar\'e duality (see, for example, Hatcher \cite[Theorem 3.30, p. 241]{Hatcher}). The de Rham Theorem (see, for example, Lee \cite[Theorem 8.14, p. 484]{Lee_john_smooth_manifolds}) asserts that $H_{\dR}^{d-1}(X) \cong H^{d-1}(X;\RR)$ and so $H_{\dR}^{d-1}(X) = (0)$. But $H_{\dR}^1(X;\RR) \cong H_{\dR}^{d-1}(X;\RR)$ by the Hodge $*_g$ isomorphism (see Warner \cite[Theorem 6.13. p. 226]{Warner}), and so $H_{\dR}^1(X;\RR) = (0)$.
  
Since $\pi_1(X) = \{1\}$, it follows from Donaldson and Kronheimer \cite[Proposition 2.2.6, p. 50]{DK} or Kobayashi \cite[Proposition 1.2.6, p. 6]{Kobayashi_differential_geometry_complex_vector_bundles}
that the flat connection $\Gamma$ is gauge-equivalent to the product connection $\Theta$ and $P\cong X\times G$ and $\ad P \cong X\times\fg$. The Hodge Theorem (for example, Warner \cite[Theorem 6.8, p. 223]{Warner}) yields
\[
  \bH^1(X) := \Ker(d+d^*)\cap W^{1,p}(X;\RR) \cong H_{\dR}^1(X) = (0),
\]
and consequently,
\[
  \bH_\Theta^1(X;\ad P) := \Ker(d_\Theta+d_\Theta^*)\cap W^{1,p}(X;T^*X\otimes \fg) \cong H_{\dR}^1(X)\otimes\fg = (0).
\]
Therefore, the elliptic estimate \eqref{eq:Improved_Lr_elliptic_estimate_dGamma+dGamma*} holds, namely
\[
  \|a\|_{W_{A_1}^{1,r}(X)} \leq C\|(d_\Theta+d_\Theta^*)a\|_{L^r(X)},
  \quad\text{for all } a \in W^{1,p}(X;T^*X\otimes \fg),
\]
matching that of Wehrheim \cite[Theorem D (b), p. 7 or Theorem 5.1 (ii), p. 77]{Wehrheim_2004}. (Wehrheim allows $X$ to have a non-empty boundary $\partial X$ in the presence of the boundary condition $*_ga\restriction \partial X = 0$; while she only imposes the milder hypothesis $H^1(X;\RR) = (0)$, she assumes that $\Gamma = \Theta$.) In particular, the hypotheses of Corollary \ref{maincor:Uhlenbeck_Chern_corollary_4-3_H1Gamma_zero} hold and we obtain the optimal estimate,
\[
  \|v(A)-\Theta\|_{W_{A_1}^{1,r}(X)} \leq C\|F_A\|_{L^r(X)},
\]
provided by \eqref{eq:Uhlenbeck_Chern_corollary_4-3_uA-Gamma_W1p_bound_refined}.
\end{exmp}

\section{Local well-posedness, a priori estimates, and minimal lifetimes for solutions to nonlinear evolution equations in Banach spaces}
\label{sec:Local_well-posedness_nonlinear_evolution_equations_Banach_spaces}
In order to prove local well-posedness for the Yang--Mills gradient flow on a Coulomb-gauge slice \eqref{eq:Yang-Mills_gradient_flow_slice}, we shall apply the general theory for abstract nonlinear evolution equations in Banach spaces described by Sell and You \cite{Sell_You_2002}. Our approach is broadly similar to the one we take in \cite[Section 17]{Feehan_yang_mills_gradient_flow_v4} but differs in one important respect, namely that rather than consider the usual Yang--Mills gradient flow \eqref{eq:Yang-Mills_gradient_flow} and apply the Donaldson--DeTurck trick \cite{DonASD, DeTurck_1983} to obtain a gauge-equivalent quasilinear parabolic equation, we instead consider Yang--Mills gradient flow restricted to a Coulomb-gauge slice through a Yang--Mills connection, by analogy with Chern--Simons gradient flow restricted to a Coulomb-gauge slice through a flat connection, as discussed by Morgan, Mrowka, and Ruberman \cite[Section 2.6]{MMR}. As we shall explain in Section \ref{sec:Local_well-posedness_Yang-Mills_gradient_flow}, the regularity properties for solutions to Yang--Mills gradient flow on a Coulomb-gauge slice are much better behaved than those obtained through an application of the Donaldson--DeTurck trick.

Our exposition in \cite[Section 13]{Feehan_yang_mills_gradient_flow_v4} added detail to that of Sell and You in \cite{Sell_You_2002} and so while we shall summarize our treatment in \cite {Feehan_yang_mills_gradient_flow_v4} in this section, many of the ideas are due to Sell and You and authors cited by them in \cite{Sell_You_2002}. However, our treatment of the higher-order spatial and temporal regularity of strong solutions to nonlinear evolution equations in Banach spaces provided by Corollaries \ref{cor:Higher-order_spatial_regularity_strong_solution_nonlinear_evolution_equation_Banach_space} and \ref{cor:First-order_temporal_regularity_strong_solution_nonlinear_evolution_equation_Banach_space} and Remark \ref{rmk:Higher-order_temporal_regularity_strong_solution_nonlinear_evolution_equation_Banach_space} has no parallel in \cite{Sell_You_2002}.

\subsection{Positive sectorial operators and nonlinear evolution equations in Banach spaces}
\label{subsec:Positive_sectorial_operators_nonlinear_evolution_equations_Banach_spaces}
In this section, we recall the definition of a positive sectorial operator, their associated linear and nonlinear evolution equations, and classes of solutions to those equations. 

\begin{defn}[Resolvent set]
\label{defn:Resolvent_set}
(See Rudin \cite[Section 13.26]{Rudin} or Yosida \cite[Section 8.1]{Yosida}.)
Let $\cA$ be an unbounded linear operator on a Banach space $\cW$ and denote its domain by $\sD(A)$. Let $\rho(\cA) \subset \CC$ denote the \emph{resolvent set} for $\cA$, that is, the set of all $\lambda \in \CC$ such that $\lambda  - \cA:\sD(\cA) \subset \cW \to \cW$ is a one-to-one map with dense range $\Ran(\lambda - \cA) \subset \cW$ and bounded inverse, $R(\lambda, \cA) := (\lambda - \cA)^{-1}$, the \emph{resolvent operator} on $\cW$.
\end{defn}

Given $a \in \RR$ and $\delta, \sigma \in (0, \pi)$, one defines sectors in the complex plane, $\CC$, by \cite[p. 77]{Sell_You_2002}
\begin{subequations}
\begin{align}
\label{eq:Sell_You_page_77_complex_plane_sector_definition_Delta_of_a}
\Delta_\delta(a) := \{z \in \CC: |\arg(z-a)| < \delta \hbox{ and } z \neq a\},
\\
\label{eq:Sell_You_page_77_complex_plane_sector_definition_Sigma_of_a}
\Sigma_\sigma(a) := \{z \in \CC: |\arg(z-a)| > \sigma \hbox{ and } z \neq a\}.
\end{align}
\end{subequations}
We recall the key

\begin{defn}[Sectorial operator]
\label{defn:Sell_You_page_78_definition_of_sectorial_operator}
(See Sell and You \cite[p. 78]{Sell_You_2002}.)
Continue the notation of Definition \ref{defn:Resolvent_set}. The operator, $\cA$, is \emph{sectorial} if it obeys the following two conditions:
\begin{enumerate}
\item $\cA$ is densely defined and closed;
\item There exist real numbers $a \in \RR$, and $\sigma \in (0, \pi/2)$, and $M \geq 1$, such that one has $\Sigma_\sigma(a) \subset \rho(\cA)$, and
\begin{equation}
\label{eq:Sell_You_36-2}
\|R(\lambda, \cA)\| \leq \frac{M}{|\lambda - a|}, \quad \text{for all } \lambda \in \Sigma_\sigma(a),
\end{equation}
or equivalently by \cite[Equation (36.1)]{Sell_You_2002}, that $\Delta_\delta(-a) \subset \rho(-\cA)$ where $\delta = \pi-\sigma$, and
\begin{equation}
\label{eq:Sell_You_36-3}
\|R(\lambda, -\cA)\| \leq \frac{M}{|\lambda + a|}, \quad \text{for all } \lambda \in \Delta_\delta(-a),
\end{equation}
\end{enumerate}
A sectorial operator, $\cA$, is said to be \emph{positive} if it satisfies \eqref{eq:Sell_You_36-2} for some $a > 0$.
\end{defn}

\begin{hyp}[Standing Hypothesis A]
\label{hyp:Sell_You_4_standing_hypothesis_A}
(See Sell and You \cite[Standing Hypothesis A, p. 141]{Sell_You_2002}.)
Let $\cA$ be a positive, sectorial operator on a Banach space $\cW$ with associated analytic semigroup $e^{-\cA t}$. Let $\cV^{2\alpha}$ be the family of interpolation spaces generated by the fractional powers of $\cA$, where $\cV^{2\alpha} = \sD(\cA^\alpha)$, for $\alpha \geq 0$. Let $\|u\|_{\cV^{2\alpha}} := \|\cA^\alpha u\|_\cW$, for $u \in \cV^{2\alpha}$, denote the norm on $\cV^{2\alpha}$. 
\end{hyp}

See \cite[Lemma 37.4]{Sell_You_2002} for an explanation of the terminology in Hypothesis \ref{hyp:Sell_You_4_standing_hypothesis_A}. We recall the

\begin{defn}[Continuous, locally bounded, spatial Lipschitz maps]
\label{defn:Sell_You_equations_46-7_and_46-8}  
(See Sell and You \cite[Equations (46.7) and (46.8)]{Sell_You_2002}.)
If $\cV$ and $\cW$ are Banach spaces, then a function\footnote{We shall assume that the nonlinearity $\cF(t,\cdot)$ is defined for all $t\in\RR$ for simplicity, but it need only be defined for $t$ in a subinterval of $\RR$.)}
\[
\cF:\RR\times \cV \to \cW
\]
belongs to $C^{0,1}(\RR\times \cV; \cW)$ if for each ball $B \subset \cV$ and compact interval $J\subset\RR$, there are positive constants $K_0=K_0(B,J)$ and $K_1=K_1(B,J)$ such that
\begin{align}
\label{eq:Sell_You_46-7}
\|\cF(t,x)\|_\cW &\leq K_0, \quad\text{for all } t \in J \quad\text{and}\quad x \in B,
\\
\label{eq:Sell_You_46-8}
\|\cF(t,x_1) - \cF(t,x_2)\|_\cW &\leq K_1\|x_1-x_2\|_{\cV^{2\beta}}, \quad\text{for all } t \in J \quad\text{and}\quad x_1, x_2 \in B.
\end{align}
\end{defn}

Our Definition \ref{defn:Sell_You_equations_46-7_and_46-8} of $C^{0,1}(\RR\times \cV^{2\beta}; \cW)$ relaxes that of \cite[Equations (46.7) and (46.8)]{Sell_You_2002}; Sell and You require that Equations \eqref{eq:Sell_You_46-7} and \eqref{eq:Sell_You_46-8} hold uniformly for \emph{all} $t\in[0,\infty)$. 

\begin{defn}[Mild solution]
\label{defn:Mild_solution}  
(See Sell and You \cite[Section 4.2, p. 146]{Sell_You_2002} (when $\cF$ depends only on $t$) and \cite[Section 4.7.1, p. 233]{Sell_You_2002} (when $\cF$ may depend on $t$ and $x$.)
Let $I = [t_0, t_0 + T)$ be an interval in $\RR$, where $t_0\in\RR$ and $T > 0$, and let $\rho\geq 0$. A pair $(u , I)$ is called a \emph{mild solution} of
\begin{equation}
\label{eq:Sell_You_47-1}
\dot u(t) + \cA u(t) = \cF(t,u(t)), \quad\text{for } u(t_0) = u_0 \in \cW \text{ and } t \geq t_0 \geq 0,
\end{equation}
\emph{in the space} $\cV^{2\rho}$ on the interval $I$ if $u \in C(I; \cV^{2\rho})$ and is a solution of the \emph{variation of constants formula} in $\cV^{2\rho}$,
\begin{equation}
\label{eq:Sell_You_47-2}
u(t) = e^{-(t-t_0)\cA}u_0 + \int_{t_0}^t e^{-(t-s)\cA} \cF(s,u(s))\,ds, \quad\text{for all } t \in I,
\end{equation}
where the integral is in the Bochner sense and represents a point in $\cV^{2\rho}$ for each $t\in I$ (see \cite[Appendix C]{Sell_You_2002}). Note that the initial condition obeys $u(t_0) = u_0 \in \cV^{2\rho}$.
\end{defn}

Recall from Sell and You \cite[Section 4.7.1, p. 232]{Sell_You_2002} that there is no loss of generality in assuming that the sectorial operator $\cA$ is \emph{positive} in \eqref{eq:Sell_You_47-1}. Because $\cA$ is sectorial, there exists $a \in \RR$ such that $\cB := \cA + a\,\id_\cW$ is a positive, sectorial operator \cite[Section 3.6]{Sell_You_2002}. Equation \eqref{eq:Sell_You_47-1} is equivalent to the equation $\dot u + \cB u = \cG(t,u(t))$,
where $\cG(t,x) = \cF(t,x) + ax$.

\begin{defn}[Strong solution]
\label{defn:Strong_solution} 
(See Sell and You \cite[Section 4.2, p. 146, and Section 4.7.1, p. 233]{Sell_You_2002}.) Continue the notation of Definition \ref{defn:Mild_solution}. A pair $(u,I)$ is called a \emph{strong solution} of \eqref{eq:Sell_You_47-1} \emph{in the space} $\cV^{2\rho}$ on $I$ if the following hold:
\begin{enumerate}
\item $u \in C(I; \cV^{2\rho})$ and $u(t_0) = u_0$;
\item $u$ is (strongly) differentiable in $\cW$ almost everywhere (a.e.) in $I$;
\item $\dot u \in L^1_{\loc}(I; \cW)$ and $u(t) = u(t_1) + \int_{t_1}^t \dot u(s)\,ds$, for all $t, t_1 \in I$;
\item $\cA u \in L^1_{\loc}(I; \cW)$; and
\item $u$ satisfies
\begin{equation}
\label{eq:Sell_You_47-3}
\dot u(t) + \cA u(t) = \cF(t,u(t)) \quad\text{in $\cW$ for a.e. } t \in I.
\end{equation}  
\end{enumerate}
\end{defn}

Note that because $\cA u \in L^1_{\loc}(I; \cW)$, or equivalently, $u \in L_{\loc}^1[0,T;\sD(\cA))$, by Definition \ref{defn:Strong_solution} then $u(t) \in \sD(\cA)$ for a.e. $t \in I$.

\begin{defn}[Classical solution]
\label{defn:Classical_solution} 
(See Sell and You \cite[Section 4.2, p. 147, and Section 4.7.1, p. 233]{Sell_You_2002}.) If in addition to Definition \ref{defn:Mild_solution}, the pair $(u,I)$ obeys the following properties, then it is called a \emph{classical solution} of \eqref{eq:Sell_You_47-1} \emph{in the space $\cV^{2\rho}$} on $I$:
\begin{enumerate}
\item $\dot u \in C((t_0, t_0 + T); \cW)$;
\item Equation \eqref{eq:Sell_You_47-3} is satisfied for all $t \in (t_0, t_0 + T)$;
\item $u(t) \in \sD(\cA)$ for all $t \in (t_0, t_0 + T)$.
\end{enumerate}  
\end{defn}

Notice that $(u,I)$ is a mild solution of \eqref{eq:Sell_You_47-1} if and only if $v(t) := u(t)$ is a mild solution of the \emph{linear} inhomogeneous problem,
$$
\dot v(t) + \cA v(t) = \cF(t,u(t)), \quad\text{for } v(t_0) = u_0 \in \cV^{2\rho} \text{ and all } t \geq t_0 \geq 0.
$$
Consequently, \cite[Lemma 42.1]{Sell_You_2002} implies that a classical solution, or a strong solution, if it exists, must be a mild solution.

\subsection{Local well-posedness for solutions to nonlinear evolution equations in Banach spaces}
\label{subsec:Local_well-posedness_solutions_nonlinear_evolution_equations_Banach_spaces}
We begin by recalling the following local existence and uniqueness result for mild solutions of \eqref{eq:Sell_You_47-1} (compare \cite[Theorem 46.1]{Sell_You_2002}).

\begin{thm}[Existence and uniqueness of mild solutions to a nonlinear evolution equation in a Banach space]
\label{thm:Sell_You_lemma_47-1}
(See Sell and You \cite[Lemma 47.1]{Sell_You_2002}.) Assume the setup of the preceding paragraphs and that, for some $\beta \in [0, 1)$,
\begin{equation}
\label{eq:Sell_You_47-4}
\cF \in C^{0,1}(\RR\times \cV^{2\beta}; \cW).
\end{equation}
Given $b > 0$, there exists a positive constant,
\[
\tau = \tau\left(b, K_0, K_1, M_0, M_\beta, \beta\right),
\]
with the following significance.  For every $u_0 \in \cV^{2\beta}$ obeying $\|u_0\|_{\cV^{2\beta}} \leq b$ and every $t_0 \geq 0$, the initial value problem \eqref{eq:Sell_You_47-1} has a unique, mild solution in $\cV^{2\beta}$ on an interval $[t_0, t_0 + \tau)$, and
\begin{equation}
\label{eq:Sell_You_47-5}
u \in C([t_0, t_0+\tau); \cV^{2\beta}) \cap C^{0,\theta_1}([t_0, t_0+\tau); \cV^{2\alpha}) \cap C^{0,\theta}((t_0, t_0+\tau); \cV^{2r}),
\end{equation}
for all $\alpha$ and $r$ with $0 \leq \alpha < r$ and $0 \leq r < 1$, where $\theta_1 > 0$ and $\theta > 0$.
\end{thm}

See Feehan \cite[Lemma 13.6]{Feehan_yang_mills_gradient_flow_v4} for an explicit formula for $\tau$ and an \apriori estimate for $u$ in Theorem \ref{thm:Sell_You_lemma_47-1} in terms of known constants.

By imposing more explicit polynomial growth and Lipschitz conditions on the nonlinearity, $\cF$, we can obtain a more precise lower bound, $\tau$, on the lifetime of the mild solution, $u$, to \eqref{eq:Sell_You_47-1} as well as a more precise \apriori estimate for $u$ on the interval $[t_0, t_0 + \tau]$ than would otherwise be possible (for example, \cite[Theorem 46.1 or Lemma 47.1]{Sell_You_2002}). In particular, we replace \eqref{eq:Sell_You_46-7} and \eqref{eq:Sell_You_46-8} with following more precise growth and Lipschitz conditions, for some $n \geq 1$ and compact interval $J\subset \RR$ and positive constants $\kappa_0, \kappa_1$ (depending on $J$):
\begin{align}
\label{eq:Sell_You_46-7_polynomial_nonlinearity}
\|\cF(t, x)\|_\cW  &\leq \kappa_0\left(1 + \|x\|_{\cV^{2\beta}}^n\right),
\quad \text{for all } t\in J \text{ and } x \in \cV^{2\beta},
\\
\label{eq:Sell_You_46-8_polynomial_nonlinearity}
\|\cF(t, x_1) - \cF(t, x_2)\|_\cW  &\leq \kappa_1\left(1 + \|x_1\|_{\cV^{2\beta}}^{n-1} + \|x_2\|_{\cV^{2\beta}}^{n-1}\right)
\|x_1 - x_2\|_{\cV^{2\beta}},
\\
&\notag\qquad \text{for all } t\in J \text{ and } x_1, x_2 \in \cV^{2\beta}.
\end{align}
In the case of the nonlinearity defined by the Yang--Mills gradient-flow equation
\eqref{eq:Yang-Mills_heat_equation_as_perturbation_rough_Laplacian_plus_one_heat_equation} restricted to a Coulomb-gauge slice,
we have $n=3$ but in general $n$ need not be an integer. This yields the following improvement to \cite[Lemma 47.1]{Sell_You_2002}.

\begin{thm}[Local well-posedness, \apriori estimates, and minimal lifetimes for mild solutions to nonlinear evolution equations in Banach spaces]
\label{thm:Sell_You_lemma_47-1_polynomial_nonlinearity}
(See Feehan \cite[Theorem 13.3]{Feehan_yang_mills_gradient_flow_v4}.)
Assume the setup of the paragraphs preceding the statement of Theorem \ref{thm:Sell_You_lemma_47-1} and that, for some $\beta \in [0, 1)$, a function $\cF \in C^{0,1}(\RR\times \cV^{2\beta}; \cW)$ obeys \eqref{eq:Sell_You_46-7_polynomial_nonlinearity} and \eqref{eq:Sell_You_46-8_polynomial_nonlinearity}. Then given $b \in (0,\infty)$, there exists a positive
constant,
$$
\tau = \tau\left(b, M_0, M_\beta, n, \beta, \kappa_0, \kappa_1\right),
$$
with the following significance.  For every $u_0 \in \cV^{2\beta}$ obeying $\|u_0\|_{\cV^{2\beta}} \leq b$ and every $t_0 \in \RR$, the initial value problem \eqref{eq:Sell_You_47-1} has a unique, mild solution in $\cV^{2\beta}$ on an interval $[t_0, t_0 + \tau)$, which obeys
\begin{equation}
\label{eq:Sell_You_47-5_polynomial_nonlinearity}
u \in C([t_0, t_0+\tau]; \cV^{2\beta}) \cap
C^{0,\theta_1}([t_0, t_0+\tau); \cV^{2\alpha}) \cap C^{0,\theta}((t_0, t_0+\tau); \cV^{2r}),
\end{equation}
for all $\alpha$ and $r$ with $0 \leq \alpha < r$ and $0 \leq r < 1$, where $\theta_1 > 0$ and $\theta > 0$. Moreover, the solution $u$ obeys the \apriori estimate,
\begin{equation}
\label{eq:Sell_You_lemma_47-1_polynomial_nonlinearity_apriori_estimate}
\|u(t)\|_{\cV^{2\beta}}
\leq
M_0\|u_0\|_{\cV^{2\beta}}
+
\frac{2M_\beta \kappa_0}{1-\beta} \left(1 + M_0 b\right)^n (t-t_0)^{1-\beta},
\quad \text{for all } t\in[t_0, t_0+\tau].
\end{equation}
If $v_0 \in \cV^{2\beta}$ obeys $\|v_0\|_{\cV^{2\beta}} \leq b$ and $v$ is the unique, mild solution to \eqref{eq:Sell_You_47-1} in $\cV^{2\beta}$ on $[t_0, t_0 + \tau]$ with $v(t_0) = v_0$ and satisfying \eqref{eq:Sell_You_47-5_polynomial_nonlinearity}, then
\begin{equation}
\label{eq:Sell_You_lemma_47-1_polynomial_nonlinearity_continuity_with_respect_to_initial_data}
\sup_{t\in [t_0, t_0+\tau]}\|u(t)-v(t)\|_{\cV^{2\beta}} \leq 2M_0\|u_0 - v_0\|_{\cV^{2\beta}}.
\end{equation}
\end{thm}

Write the solution $u(t)$ to the initial value problem \eqref{eq:Sell_You_47-1} with initial data $u_0$ provided by Theorem \ref{thm:Sell_You_lemma_47-1_polynomial_nonlinearity} as $u(t)= u_0+w(t)$. We can then rewrite \eqref{eq:Sell_You_47-1} as
\begin{equation}
\label{eq:Sell_You_47-1_zero_initial_data}
\dot w(t) + \cA w(t) = \cF_0(t,w(t)), \quad\text{for } w(t_0) = 0 \text{ and all } t \geq t_0,
\end{equation}
where
\begin{equation}
\label{eq:nonlinearity_with_initial_data}
  \cF_0(t,x) := \cF(t,u_0+x), \quad\text{ for all } (t,x) \in \RR\times\cV^{2\beta}.
\end{equation}
From Theorem \ref{thm:Sell_You_lemma_47-1_polynomial_nonlinearity} we can now deduce an estimate for $\sup_{t\in [t_0, t_0+\tau]}\|u(t)-u_0\|_{\cV^{2\beta}}$.

\begin{cor}[Estimate for the differences between mild solutions to nonlinear evolution equations in Banach spaces and their initial data]
\label{cor:Sell_You_lemma_47-1_polynomial_nonlinearity_estimate_difference_solution_initial_data}
Assume the hypotheses of Theorem \ref{thm:Sell_You_lemma_47-1_polynomial_nonlinearity}. If $\cF_0$ in \eqref{eq:nonlinearity_with_initial_data} obeys \eqref{eq:Sell_You_46-7_polynomial_nonlinearity} and \eqref{eq:Sell_You_46-8_polynomial_nonlinearity}, then
\begin{equation}
\label{eq:Sell_You_lemma_47-1_polynomial_nonlinearity_estimate_difference_solution_initial_data}
\sup_{t\in [t_0, t_0+\tau]}\|u(t)-u_0\|_{\cV^{2\beta}} \leq \frac{2M_\beta \kappa_0}{1-\beta} \left(1 + M_0 b\right)^n \tau^{1-\beta}.
\end{equation}
\end{cor}

\begin{proof}
Applying estimate \eqref{eq:Sell_You_lemma_47-1_polynomial_nonlinearity_apriori_estimate} in Theorem \ref{thm:Sell_You_lemma_47-1_polynomial_nonlinearity} (with $\cF_0$ in place of $\cF$) to the solution $w(t)=u(t)-u_0$ to \eqref{eq:Sell_You_47-1_zero_initial_data} yields
\begin{equation}
\label{eq:Sell_You_lemma_47-1_polynomial_nonlinearity_apriori_estimate_zero_initial_data}
\sup_{t\in [t_0, t_0+\tau]}\|w(t)\|_{\cV^{2\beta}}
\leq
\frac{2M_\beta \kappa_0}{1-\beta} \left(1 + M_0 b\right)^n \tau^{1-\beta},
\end{equation}
and this yields \eqref{eq:Sell_You_lemma_47-1_polynomial_nonlinearity_estimate_difference_solution_initial_data}.
\end{proof}


We recall the

\begin{defn}[Continuous, locally bounded, spatial Lipschitz and temporal H\"older continuous maps]
\label{defn:Sell_You_equations_46-6}  
(See Sell and You \cite[Equation (46.6) or p. 658]{Sell_You_2002}.)
If $\cV$ and $\cW$ are Banach spaces and $\theta \in (0, 1]$, then a function
\[
\cF:\RR\times \cV \to \cW
\]
belongs to $C^{0,1;\theta}(\RR\times \cV; \cW)$ if $\cF \in C^{0,1}(\RR\times \cV; \cW)$ and for each ball $B \subset \cV$ and compact interval $J\subset\RR$, there is a positive constant $K_2=K_2(B,J)$ such that
\begin{multline}
\label{eq:Sell_You_46-6}
\|\cF(t_1,x_1) - \cF(t_2,x_2)\|_\cW \leq K_2\left(\|x_1-x_2\|_\cV + |t_1-t_2|^\theta\right),
\\
\text{for all } t_1, t_2 \in J \text{ and } x_1, x_2 \in B.
\end{multline}
\end{defn}

By imposing the additional regularity condition \eqref{eq:Sell_You_46-6} on a nonlinearity $\cF \in C^{0,1}(\RR\times \cV; \cW)$, one can show that the mild solution provided by Theorem \ref{thm:Sell_You_lemma_47-1_polynomial_nonlinearity} is a strong solution (compare \cite[Theorem 46.2]{Sell_You_2002}). 

\begin{thm}[Strong solutions to nonlinear evolution equations in Banach spaces]
\label{thm:Sell_You_lemma_47-2}
(See Sell and You \cite[Lemma 47.2]{Sell_You_2002} and Feehan \cite[Remark 13.9]{Feehan_yang_mills_gradient_flow_v4}.)
Assume the hypotheses of Theorem \ref{thm:Sell_You_lemma_47-1} and in addition that, for some $\theta \in (0, 1]$,
\begin{equation}
\label{eq:Sell_You_lemma_47-2_F_C_Lipschitz_theta}
\cF \in C^{0,1;\theta}(\RR\times \cV^{2\beta}; \cW).
\end{equation}
If $u_0 \in \cV^{2\beta}$ and $u$ is a mild solution of Equation \eqref{eq:Sell_You_47-1} in $\cV^{2\beta}$ on an interval $[t_0, t_0+T)$ for some $t_0\in\RR$ and $T > 0$, then $u$ is a strong solution in $\cV^{2\beta}$ on the interval $[t_0, t_0+T)$, and it satisfies for all $r \in [0,1)$,
\begin{equation}
\label{eq:Sell_You_47-7}
u \in C([t_0, t_0+T); \cV^{2\beta}) \cap C^{0,1-r}((t_0, t_0+T); \cV^{2r}) \cap C((t_0, t_0+T); \cV^2) \cap W_{\loc}^{1,1}((t_0, t_0+T); \cW).
\end{equation}
\end{thm}

\subsection{Higher-order regularity for solutions to nonlinear evolution equations in Banach spaces}
\label{subsec:Higher-order_regularity_solutions_nonlinear_evolution_equations_Banach_spaces}
It is convenient to denote $\cV^\infty := \cap_{\alpha\in\RR}\cV^\alpha$ and observe that this is a Fr{\'e}chet space with translation invariant complete metric induced in a standard way from the sequence of Banach space norms on $\cV^k$ for $k \in \NN$,
\begin{equation}
\label{eq:Vinfinity_metric}
d(x,y) := \sum_{k=0}^\infty \frac{1}{2^{k+1}}\frac{\|x-y\|_{\cV^k}}{1+\|x-y\|_{\cV^k}}, \quad\text{for all } x,y \in \cV^\infty.
\end{equation}
If $U \subset \RR^m$ is an open subset, $(\fF,d)$ is a metric space, $f \in C(U,\fF)$ is a continuous map, and $\alpha \in (0,1]$, then we may define the vector space of \emph{uniformly H{\"o}lder continuous maps}, $C^{0,\alpha}(\bar U,\fF)$, in the usual way (see Gilbarg and Trudinger \cite[Equation (4.4)]{GT}) by writing $f \in C^{0,\alpha}(\bar U,\fF)$ if and only if 
\[
  [f]_\alpha := \sup_{\begin{subarray}{c}x,y\in U\\ x\neq y\end{subarray}} \frac{d(f(x),f(y))}{\|x-y\|_{\RR^m}} < \infty.
\]
We say that $f \in C^{0,\alpha}(U,\fF)$ (or equivalently, $C_{\loc}^{0,\alpha}(U,\fF)$) if $f \in C^{0,\alpha}(\bar V,\fF)$ for all $V \Subset U$. As usual, we say that $f$ is \emph{Lipschitz} if $\alpha=1$. It is insightful to recall the following extension of \emph{Rademacher's Theorem} (see Federer \cite[Theorem 3.1.6]{Federer}) from the more familiar setting of maps between Euclidean spaces. 

\begin{lem}[Rademacher's Theorem for maps from $\RR^1$ into reflexive Banach spaces]
\label{lem:Rademacher}  
(See Sell and You \cite[Lemma C.6]{Sell_You_2002}.) If $\cB$ is a reflexive Banach space and $f : [a, b] \to \cB$ is Lipschitz continuous, then $f$ is strongly differentiable almost everywhere in $[a, b]$, with strong derivative $f' \in L^1(a, b; \cB)$.
\end{lem}

\begin{cor}[Higher-order spatial regularity of strong solutions to nonlinear evolution equations in Banach spaces]
\label{cor:Higher-order_spatial_regularity_strong_solution_nonlinear_evolution_equation_Banach_space}
Assume the setup in the paragraphs preceding the statement of Theorem \ref{thm:Sell_You_lemma_47-1} and that, for some $\beta \in [0, 1)$ and $\xi > 0$, the function $\cF$ obeys  
\begin{equation}
\label{eq:Nonlinearity_map_V2beta+2eta_to_V2eta_space_Lipschitz_theta_time}
\cF \in C^{0,1;\theta}(\RR\times \cV^{2\beta+2\xi}; \cV^{2\xi}).
\end{equation}
If $u_0 \in \cV^{2\beta}$ and $u$ is a mild solution of Equation \eqref{eq:Sell_You_47-1} in $\cV^{2\beta}$ on an interval $[t_0, t_0+T)$ for some $t_0\in\RR$ and $T > 0$, then $u$ is a strong solution in $\cV^{2\beta}$ on the interval $[t_0, t_0+T)$, and it satisfies
\begin{multline}
\label{eq:Finite_higher-order_spatial_regularity_strong_solution_nonlinear_evolution_equation_Banach_space}
u \in C([t_0, t_0+T); \cV^{2\beta}) \cap C^{0,1-r}((t_0, t_0+T); \cV^{2\xi+2r}) \cap C((t_0, t_0+T); \cV^{2\xi+2})
\\
\cap W_{\loc}^{1,1}((t_0, t_0+T); \cV^{2\xi}).
\end{multline}
If \eqref{eq:Nonlinearity_map_V2beta+2eta_to_V2eta_space_Lipschitz_theta_time} holds for all $\xi>0$, then
\begin{equation}
\label{eq:Higher-order_spatial_regularity_strong_solution_nonlinear_evolution_equation_Banach_space}
u \in C([t_0, t_0+T); \cV^{2\beta}) \cap C^{0,1}((t_0, t_0+T); \cV^\infty).
\end{equation}  
\end{cor}

\begin{proof}
Define $\zeta := 1-\beta \in (0,1]$ and write $\cV^2 = \cV^{2\beta+2\zeta}$. By definition of this scale of Banach spaces, the unbounded operator $\cA:\cV^{2\zeta} \to \cV^{2\zeta}$ has domain $\cV^{2\zeta+2}$. We can now apply Theorem \ref{thm:Sell_You_lemma_47-2}, with $\cW = \cV^0$ replaced by $\cV^{2\zeta}$ and $\cV^{2\beta}$ replaced by $\cV^{2\beta+2\zeta}$ and $t_0$ replaced by any $t_1 \in (t_0,t_0+T)$ and the hypothesis \eqref{eq:Sell_You_lemma_47-2_F_C_Lipschitz_theta} replaced by \eqref{eq:Nonlinearity_map_V2beta+2eta_to_V2eta_space_Lipschitz_theta_time} and the initial data $u_0 \in \cV^{2\beta}$ replaced by $u_1 := u(t_1) \in \cV^{2\beta+2\zeta}$ to give for all $r \in [0,1)$,
\begin{multline*}
  u \in C([t_1, t_0+T); \cV^{2\beta+2\zeta}) \cap C^{0,1-r}((t_1, t_0+T); \cV^{2\zeta+2r}) \cap C((t_1, t_0+T); \cV^{2\zeta+2})
  \\
  \cap W_{\loc}^{1,1}((t_0, t_0+T); \cV^{2\zeta}).
\end{multline*}
For any sequence $\{t_k\}_{k=1}^\infty \subset (t_0,t_0+T)$ with $t_{k+1}>t_k$ for all $k\geq 1$, we may iterate the preceding step to give
\begin{multline*}
  u \in C([t_k, t_0+T); \cV^{2\beta+2k\zeta}) \cap C^{0,1-r}((t_k, t_0+T); \cV^{2k\zeta+2r}) \cap C((t_k, t_0+T); \cV^{2k\zeta+2})
  \\
  \cap W_{\loc}^{1,1}((t_0, t_0+T); \cV^{2k\zeta}),
\end{multline*}
provided $k\zeta \leq \xi$. Because the sequence $\{t_k\}_{k=1}^\infty$ is arbitrary, we obtain
\[
u \in C^{0,1-r}((t_0, t_0+T); \cV^{2\xi+2r}) \cap C((t_0, t_0+T); \cV^{2\xi+2}) \cap W_{\loc}^{1,1}((t_0, t_0+T); \cV^{2\xi}),
\]
and this gives the regularity \eqref{eq:Finite_higher-order_spatial_regularity_strong_solution_nonlinear_evolution_equation_Banach_space}.
If \eqref{eq:Nonlinearity_map_V2beta+2eta_to_V2eta_space_Lipschitz_theta_time} holds for all $\xi > 0$, then
\[
u \in C^{0,1-r}((t_0, t_0+T); \cV^\infty).
\]
Because $r\in [0,1)$ is arbitrary, this yields the regularity \eqref{eq:Higher-order_spatial_regularity_strong_solution_nonlinear_evolution_equation_Banach_space}.
\end{proof}

\begin{cor}[First-order temporal regularity of strong solutions to nonlinear evolution equations in Banach spaces]
\label{cor:First-order_temporal_regularity_strong_solution_nonlinear_evolution_equation_Banach_space}
Assume the setup in the paragraphs preceding the statement of Theorem \ref{thm:Sell_You_lemma_47-1} and that, for some $\beta \in [0, 1)$ and $\eta\geq 0$ and
$\xi = \eta$, $\beta+\eta$, the partial derivatives of a function $\cF:\RR\times \cV^{2\beta}; \cW)$ obey 
\begin{subequations}
\label{eq:First-order_derivative_nonlinearity_map_V2beta+2eta_to_V2eta_space_Lipschitz}
\begin{align}
\label{eq:Partial_derivative_t_nonlinearity_map_V2beta+2eta_to_V2eta_space_Lipschitz} 
\partial_t\cF &\in C^{0,1}\left(\RR\times \cV^{2\beta+2\xi}; \cV^{2\xi}\right), 
  \\
\label{eq:Partial_derivative_v_nonlinearity_map_V2beta+2eta_to_V2eta_space_Lipschitz}   
\partial_x\cF &\in C^{0,1}\left(\RR\times \cV^{2\beta+2\xi}; \Hom\left(\cV^{2\beta+2\xi},\cV^{2\xi}\right)\right).
\end{align}
\end{subequations}
If $u_0 \in \cV^{2\beta}$ and $u$ is a mild solution of Equation \eqref{eq:Sell_You_47-1} in $\cV^{2\beta}$ on an interval $[t_0, t_0+T)$ for some $t_0\in\RR$ and $T > 0$, then $u$ is a strong solution in $\cV^{2\beta}$ on the interval $[t_0, t_0+T)$, and it satisfies
\begin{equation}
\label{eq:First-order_temporal_regularity_strong_solution_nonlinear_evolution_equation_Banach_space}
u \in C([t_0, t_0+T); \cV^{2\beta}) \cap C^1((t_0, t_0+T); \cV^{2\beta}).
\end{equation}
If \eqref{eq:First-order_derivative_nonlinearity_map_V2beta+2eta_to_V2eta_space_Lipschitz} holds for all $\xi\geq 0$, then
\begin{equation}
\label{eq:First-order_temporal_higher-order_spatial_regularity_strong_solution_nonlinear_evolution_equation_Banach_space}
u \in C([t_0, t_0+T); \cV^{2\beta}) \cap C^1((t_0, t_0+T); \cV^\infty).
\end{equation}
\end{cor}

\begin{proof}
If $B \subset \cV^{2\beta+2\xi}$ is a ball and $J\subset[t_0,\infty)$ is a compact interval, then there are positive constants $K_0'=K_0'(B,J)$ and $K_0''=K_0''(B,J)$ such that for all $t_1,t_2 \in J$ and $x_1,x_2 \in B$ we have
\begin{align*}
  {}&\|\cF(t_1,x_1) - \cF(t_2,x_2)\|_{\cV^{2\xi}}
  \\
  &\quad \leq  \|\cF(t_1,x_1) - \cF(t_1,x_2)\|_{\cV^{2\xi}} + \|\cF(t_1,x_2) - \cF(t_2,x_2)\|_{\cV^{2\xi}}
  \\
    &\quad \leq \sup_{(t,x)\in J\times B} \|\partial_x\cF(t,x)\|_{\Hom(\cV^{2\beta+2\xi},\cV^{2\xi})}
      \|x_1-x_2\|_{\cV^{2\beta+2\xi}} + \sup_{(t,x)\in J\times B} \|\partial_t\cF(t,x)\|_{\cV^{2\xi}} |t_1-t_2|
  \\
                                                  &\qquad\text{(by the Mean Value Theorem)}
  \\
                                                    &\quad\leq K_0'\|x_1-x_2\|_{\cV^{2\beta+2\xi}} + K_0''|t_1-t_2|
  \\
  &\qquad\text{(by \eqref{eq:Sell_You_46-7} and \eqref{eq:Partial_derivative_v_nonlinearity_map_V2beta+2eta_to_V2eta_space_Lipschitz} for $\partial_x\cF$ and \eqref{eq:Sell_You_46-7} and \eqref{eq:Partial_derivative_t_nonlinearity_map_V2beta+2eta_to_V2eta_space_Lipschitz} for $\partial_t\cF$),}                                                    
\end{align*}
so that $\cF$ obeys \eqref{eq:Sell_You_46-6} with $\cV$ replaced by $\cV^{2\beta+2\xi}$ and $\cW$ replaced by $\cV^{2\xi}$. We may assume that $t_0\in J$ so that for all $t \in J$ and $x \in B$, we also have
\begin{align*}
  \|\cF(t,x)\|_{\cV^{2\xi}} &\leq \|\cF(t,x) - \cF(t_0,0)\|_{\cV^{2\xi}} + \|\cF(t_0,0)\|_{\cV^{2\xi}}
  \\  
                               &\leq K_0'\|x\|_{\cV^{2\beta+2\xi}} + K_0''|t-t_0| + \|\cF(t_0,0)\|_{\cV^{2\xi}}
  \\
  &\quad\text{(by \eqref{eq:Sell_You_46-6} with $\cV$ replaced by $\cV^{2\beta+2\xi}$ and $\cW$ replaced by $\cV^{2\xi}$),}
  \\
  &\leq K_0''',
\end{align*}
where $K_0''' = K_0'''(B,J,\cF(t_0,0))$. Therefore, $\cF$ obeys \eqref{eq:Nonlinearity_map_V2beta+2eta_to_V2eta_space_Lipschitz_theta_time} with $\theta=1$ and $\xi=\beta+\eta$, so Corollary \ref{cor:Higher-order_spatial_regularity_strong_solution_nonlinear_evolution_equation_Banach_space} implies that $u$ is a strong solution in $\cV^{2\beta}$ on the interval $[t_0, t_0+T)$ and by \eqref{eq:Finite_higher-order_spatial_regularity_strong_solution_nonlinear_evolution_equation_Banach_space} satisfies
\begin{equation}
\label{eq:u_continuous_V2beta_up_to_zero_and_V2neta+2eta+2_greaterthan_zero}
u \in C([t_0, t_0+T); \cV^{2\beta}) \cap C((t_0, t_0+T); \cV^{2\beta+2\eta+2}).
\end{equation}
Let $\eps \in (0,T]$. Because $u$ is a strong solution to Equation \eqref{eq:Sell_You_47-1}, there exists $t_1 \in (t_0,t_0+\eps)$ such that $\dot u(t_1) = \cF(t_1,u(t_1))-\cA u(t_1)$ and since $u(t_1) \in \cV^{2\beta+2\eta+2}$ by \eqref{eq:u_continuous_V2beta_up_to_zero_and_V2neta+2eta+2_greaterthan_zero}, we have $\cF(t_1,u(t_1))-\cA u(t_1) \in \cV^{2\beta+2\eta}$ and thus\footnote{Alternatively, observe that if $u \in C^{0,1}((t_0, t_0+T); \cV^{2\beta+2\eta})$ and if $\cV^{2\beta+2\eta}$ were reflexive, then Lemma \ref{lem:Rademacher} would imply that $u$ were strongly differentiable almost everywhere on $(t_0,t_0+T)$ and so there would exist $t_1 \in (t_0,t_0+\eps)$ such that $\dot u(t_1) \in \cV^{2\beta+2\eta}$.},
$\dot u(t_1) \in \cV^{2\beta+2\eta}$.

By analogy with the proof of \cite[Theorem 48.5]{Sell_You_2002}, we may formally differentiate Equation \eqref{eq:Sell_You_47-1} with respect to time to give
\[
\dot v(t) + \cA v(t) =  \partial_t\cF(t,u(t)) + \partial_x\cF(t,u(t))\dot u(t), \quad\text{for } t \in (t_0, t_0+T).
\]
Define
\[
  f(t) := \cF(t,u(t)), \quad\text{for all } t\in [t_0,t_0+T),
\]
and observe that $f \in C((t_0,t_0+T);\cV^{2\beta+2\eta})$ since $\cF \in C^1(\RR\times \cV^{4\beta+2\eta};\cV^{2\beta+2\eta})$ by hypothesis \eqref{eq:First-order_derivative_nonlinearity_map_V2beta+2eta_to_V2eta_space_Lipschitz} with $\xi=\beta+\eta$ and because $u \in C((t_0, t_0+T); \cV^{2\beta+2\eta+2})$ by \eqref{eq:u_continuous_V2beta_up_to_zero_and_V2neta+2eta+2_greaterthan_zero}. By definition of $f$, we have
\[
   \dot f(t) = \partial_t\cF(t,u(t)) + \partial_x\cF(t,u(t))\dot u(t) \quad\text{for a.e. } t \in (t_0, t_0+T),
\]
and we observe that $\dot f \in L^1((t_0,t_0+T);\cV^{2\eta})$ because $\partial_t\cF \in C(\RR\times \cV^{2\beta+2\eta};\cV^{2\eta})$ by \eqref{eq:Partial_derivative_t_nonlinearity_map_V2beta+2eta_to_V2eta_space_Lipschitz} with $\xi=\eta$ and $u \in C((t_0, t_0+T); \cV^{2\beta+2\eta+2})$ by \eqref{eq:u_continuous_V2beta_up_to_zero_and_V2neta+2eta+2_greaterthan_zero} and because $\partial_x\cF \in C(\RR\times \cV^{2\beta+2\eta};\Hom(\cV^{2\beta+2\eta},\cV^{2\eta}))$ by \eqref{eq:Partial_derivative_v_nonlinearity_map_V2beta+2eta_to_V2eta_space_Lipschitz} with $\xi=\eta$ and $\dot u \in L^1((t_0,t_0+T);\cV^{2\beta+2\eta})$ by Corollary \ref{cor:Higher-order_spatial_regularity_strong_solution_nonlinear_evolution_equation_Banach_space} with $\xi=\beta+\eta$. Define
\[
  w(t) := \cF(t,u(t))-\cA u(t), \quad\text{for all } t\in [t_0,t_0+T),
\]
and observe that $w \in C((t_0, t_0+T); \cV^{2\beta+2\eta})$ because $\cF \in C(\RR\times \cV^{4\beta+2\eta};\cV^{2\beta+2\eta})$ by hypothesis \eqref{eq:First-order_derivative_nonlinearity_map_V2beta+2eta_to_V2eta_space_Lipschitz} with $\xi=\beta+\eta$ and $u \in C((t_0, t_0+T); \cV^{2\beta+2\eta+2})$ by \eqref{eq:u_continuous_V2beta_up_to_zero_and_V2neta+2eta+2_greaterthan_zero} and because $\cA u \in C((t_0, t_0+T); \cV^{2\beta+2\eta})$ since $u \in C((t_0, t_0+T); \cV^{2\beta+2\eta+2})$ by \eqref{eq:u_continuous_V2beta_up_to_zero_and_V2neta+2eta+2_greaterthan_zero}. Define
\[
    g(t) := \partial_t\cF(t, u(t)) + \partial_x\cF(t,u(t))w(t) \quad\text{for all } t\in (t_0,t_0+T),
\]
and observe that $g \in C((t_0, t_0+T); \cV^{2\eta})$ because $\partial_t\cF \in C(\RR\times \cV^{2\beta+2\eta};\cV^{2\eta})$ by \eqref{eq:Partial_derivative_t_nonlinearity_map_V2beta+2eta_to_V2eta_space_Lipschitz} with $\xi=\eta$ and $u \in C((t_0, t_0+T); \cV^{2\beta+2\eta+2})$ and because $\partial_x\cF \in C(\RR\times \cV^{2\beta+2\eta};\Hom(\cV^{2\beta+2\eta},\cV^{2\eta}))$ by \eqref{eq:Partial_derivative_v_nonlinearity_map_V2beta+2eta_to_V2eta_space_Lipschitz} with $\xi=\eta$ and $w \in C((t_0, t_0+T); \cV^{2\beta+2\eta})$.
  
By (a special case of) Theorem \ref{thm:Sell_You_lemma_47-1}, the (linear) equation,
\begin{equation}
\label{eq:Sell_You_47-1_v_g}
\dot v(t) + \cA v(t) = g(t), \quad\text{for } v(t_1) = \dot u(t_1) \in \cV^{2\beta+2\eta} \text{ and } t \geq t_1 \geq 0,
\end{equation}
has a unique mild solution $v \in C([t_1,t_0+T);\cV^{2\beta+2\eta})$ obeying
\[
v(t) = e^{-(t-t_1)\cA}\dot u(t_1) + \int_{t_1}^t e^{-(t-s)\cA} g(s)\,ds, \quad\text{for all } t \in [t_1,t_0+T).
\]
Because $u$ is a strong solution to Equation \eqref{eq:Sell_You_47-1} by Theorem \ref{thm:Sell_You_lemma_47-2}, then $\dot u(t) = w(t)$ for a.e. $t\in (t_0, t_0+T)$ and so $\dot f(t) = g(t)$ for a.e. $t\in (t_0, t_0+T)$ and thus 
\[
  \int_{t_1}^t e^{-(t-s)\cA} g(s)\,ds = \int_{t_1}^t e^{-(t-s)\cA} \dot f(s)\,ds,
\]
which gives  
\[
v(t) = e^{-(t-t_1)\cA}\dot u(t_1) + \int_{t_1}^t e^{-(t-s)\cA} \dot f(s)\,ds, \quad\text{for all } t \in [t_1,t_0+T).
\]
Hence, $v$ is the unique mild solution in $\cV^{2\beta+2\eta}$ on $[t_1,t_0+T)$ to 
\begin{equation}
\label{eq:Sell_You_47-1_dudt}
\dot v(t) + \cA v(t) = \dot f(t), \quad\text{for } v(t_1) = \dot u(t_1) \in \cV^{2\beta+2\eta} \text{ and } t \in (t_1,t_0+T).
\end{equation}
Define $\tilde u(t) := \int_{t_1}^t v(s)\,ds$ for all $t\in (t_1, t_0+T)$ and observe that
\[
  \tilde u \in C([t_1, t_0+T); \cV^{2\beta+2\eta})\cap C^1((t_1, t_0+T); \cV^{2\beta+2\eta})
\]
and that $\tilde u$ is a classical solution to \eqref{eq:Sell_You_47-1} with right-hand side $f(t) = \cF(t,u(t))$ on $(t_1,t_0+T)$, namely
\[
\dot u(t) + \cA u(t) = f(t), \quad\text{for } t \in (t_1,t_0+T),
\]
and initial data $\tilde u(t_1) = u(t_1)$. By uniqueness of mild solutions to \eqref{eq:Sell_You_47-1}, we have $\tilde u = u$ on $[t_1,t_0+T)$.

Because $\eps \in (0,T]$ was arbitrary, we see that
\[
  u \in C([t_0, t_0+T); \cV^{2\beta})\cap C^1((t_0, t_0+T); \cV^{2\beta+2\eta})
\]
is a classical solution to \eqref{eq:Sell_You_47-1} and $u$ has the regularity \eqref{eq:First-order_temporal_regularity_strong_solution_nonlinear_evolution_equation_Banach_space}. If \eqref{eq:First-order_derivative_nonlinearity_map_V2beta+2eta_to_V2eta_space_Lipschitz} holds for all $\eta \geq 0$, we obtain
\[
  u \in C([t_0, t_0+T); \cV^{2\beta})\cap C^1((t_0, t_0+T); \cV^\infty)
\]
so $u$ has the regularity \eqref{eq:First-order_temporal_higher-order_spatial_regularity_strong_solution_nonlinear_evolution_equation_Banach_space}.
\end{proof}

\begin{rmk}[Higher-order temporal regularity of strong solutions to nonlinear evolution equations in Banach spaces]
\label{rmk:Higher-order_temporal_regularity_strong_solution_nonlinear_evolution_equation_Banach_space}
Suppose that 
\begin{equation}
\label{eq:Higher-order_derivatives_nonlinearity_map_V2beta+2eta_to_V2eta_space_Lipschitz}
\cF \in C^\infty\left(\RR\times \cV^{2\beta+2\eta}; \Hom\left(\otimes^l\cV^{2\beta+2\eta}, \cV^{2\eta}\right)\right), \quad\text{for all } \eta \geq 0.
\end{equation}
By induction and the proof of Corollary \ref{cor:First-order_temporal_regularity_strong_solution_nonlinear_evolution_equation_Banach_space}, we obtain
\begin{equation}
\label{eq:Higher-order_temporal_regularity_strong_solution_nonlinear_evolution_equation_Banach_space}
u \in C([t_0, t_0+T); \cV^{2\beta}) \cap C^\infty((t_0, t_0+T); \cV^\infty).
\end{equation}
Finally, if $u_0\in\cV^\infty$, then for any $\eta\geq 0$ we have $u_0\in \cV^{2\beta+2\eta+2}$ and thus $u\in C([t_0, t_0+T); \cV^{2\beta+2\eta+2})$ and $\cF(\cdot,u)-\cA u \in C([t_0, t_0+T); \cV^{2\beta+2\eta})$, so that $\dot u \in C([t_0, t_0+T); \cV^{2\beta+2\eta})$ since $u$ is a strong solution to Equation \eqref{eq:Sell_You_47-1} and thus $u \in C^1([t_0, t_0+T); \cV^{2\beta+2\eta})$. By induction, this argument yields
\begin{equation}
\label{eq:Classical_solution_nonlinear_evolution_equation_Banach_space}
u \in C^\infty([t_0, t_0+T); \cV^\infty),
\end{equation}
and $u$ has the optimal regularity.
\end{rmk}

\subsection{Lengths of flowlines defined by solutions to nonlinear evolution equations in Banach spaces}
\label{subsec:Lengths_flow_lines_solutions_nonlinear_evolution_equations_Banach_spaces}
It remains to recall our key estimates for the lengths of flowlines defined by solutions to nonlinear evolution equations. 

\begin{lem}[A global $L^1$-in-time \apriori estimate for a mild solution to a nonlinear evolution equation]
\label{lem:Rade_7-3_abstract_L1_in_time_V2beta_space}
(See Feehan \cite[Lemma 17.8]{Feehan_yang_mills_gradient_flow_v4}.)
Assume that Hypothesis \ref{hyp:Sell_You_4_standing_hypothesis_A} holds. Let $\beta \in [0, 1)$, and let $\eps$ be a positive constant such that
$$
\eps a^{1+\beta} M_\beta \Gamma(1-\beta) \leq \frac{1}{2},
$$
where the constants $a > 0$ and $M_\beta > 0$ are as in Sell and You \cite[Theorem 37.5]{Sell_You_2002}, and $K \in [1,\infty)$. Suppose that $\cG \in C^{0,1}(\RR\times\cV^{2\beta}; \cW)$ obeys
$$
\cG(t,x) = g_0(t) + \cG_1(t,x) + \cG_2(t,x), \quad \text{for all } (t,x) \in \RR\times\cV^{2\beta},
$$
with $g_0 \in C^{0,1}(\RR; \cW)$ and $\cG_1, \cG_2 \in C^{0,1}(\RR\times\cV^{2\beta}; \cW)$ and, for some $t_0\in\RR$ and $t_0\geq 0$ and $T>0$, that
\begin{align*}
\|\cG_1(t,x)\|_\cW &\leq K\|x\|_\cW,
\\
\|\cG_2(t,x)\|_\cW &\leq \eps\|x\|_{\cV^{2\beta}}, \quad\text{for all } (t,x) \in [t_0,t_0+T)\times\cV^{2\beta}.
\end{align*}
If $v$ is a mild solution to the nonlinear evolution equation \eqref{eq:Sell_You_47-1} on $[t_0, t_0+T)$ defined by $\cA$ and the nonlinearity $\cG$, with $v \in C([t_0, t_0+T); \cW) \cap C((t_0, t_0+T); \cV^{2\beta})$, then
\begin{multline}
\label{eq:Rade_7-3_abstract_L1_in_time_V2beta_space_apriori_estimate}
\int_{t_0}^{t_0+T} \|v(t)\|_{\cV^{2\beta}} \,dt
\\
\leq a^{1+\beta} M_\beta \Gamma(1-\beta) \left(\|v_0\|_\cW
+ \int_{t_0}^{t_0+T} \|g_0(t)\|_\cW\,dt + K\int_{t_0}^{t_0+T} \|v(t)\|_\cW \,dt\right).
\end{multline}
\end{lem}

\begin{cor}[An interior $L^1$-in-time \apriori estimate for a strong solution to a nonlinear evolution equation]
\label{cor:Rade_7-3_abstract_L1_in_time_V2beta_space_interior}
(See Feehan \cite[Lemma 17.11]{Feehan_yang_mills_gradient_flow_v4}.)  
Assume the hypotheses of Lemma \ref{lem:Rade_7-3_abstract_L1_in_time_V2beta_space}. If $\delta$ is a constant obeying $2\delta \leq T$ and $v$ is a strong solution to the nonlinear evolution equation \eqref{eq:Sell_You_47-1}, then
\begin{multline}
\label{eq:Rade_7-3_abstract_L1_in_time_V2beta_space_apriori_estimate_interior}
\int_{t_0+\delta}^{t_0+T} \|v(t)\|_{\cV^{2\beta}} \,dt
\\
\leq a^{1+\beta} M_\beta \Gamma(1-\beta)
\left(\int_{t_0}^{t_0+T} \|g_0(t)\|_\cW\,dt + K(1+\delta^{-1})\int_{t_0}^{t_0+T} \|v(t)\|_\cW \,dt\right).
\end{multline}
\end{cor}

\begin{rmk}[Application of Corollary \ref{cor:Rade_7-3_abstract_L1_in_time_V2beta_space_interior} to derivation of \apriori estimates for the lengths of flowlines]
\label{rmk:Abstract_apriori_interior_estimate_trajectory}  
(See Feehan \cite[Lemmas 17.10 and 17.12]{Feehan_yang_mills_gradient_flow_v4}.)
In applications to proofs of global existence (see Hypothesis \ref{hyp:Abstract_apriori_interior_estimate_trajectory_main_introduction} in Section \ref{sec:Global_existence_convergence_rate_Lojasiewicz-Simon_gradient_flow_near_local_minimum}), we shall apply Corollary \ref{cor:Rade_7-3_abstract_L1_in_time_V2beta_space_interior} to the time derivative $v = \dot u$ of a classical solution $u \in C([t_0,t_0+T); \cV^{2\beta})\cap C^2((t_0,t_0+T); \cV^{2\beta})$ to equation \eqref{eq:Sell_You_47-1}, namely
\[
  \dot v(t) + \cA v(t) = \cG(t,v(t)), \quad\text{for } v(t_1) = \dot u(t_1) \text{ and all } \, t \in (t_1,t_0+T),
\]
where $t_1 \in (t_0,t_0+T)$ and 
\[
  \cG(t,v(t)) := \partial_t\cF(t,u(t)) + \partial_x\cF(t,u(t))\dot u(t), \quad\text{for all } t \in (t_1,t_0+T),
\]
is the induced nonlinearity.
\end{rmk}

\section{Global existence,  convergence, and convergence rate  for  abstract {\L}ojasiewicz--Simon gradient flows near a critical point}
\label{sec:Global_existence_convergence_rate_Lojasiewicz-Simon_gradient_flow_near_local_minimum}
We summarize our key results from \cite[Section 2]{Feehan_yang_mills_gradient_flow_v4} on global existence, convergence, and convergence rate for gradient flow defined by a $C^1$ function $\sE:\sX\supset\sU\to\RR$ near a critical point when $\sE$ obeys a {\L}ojasiewicz--Simon gradient inequality. We begin by stating two key hypotheses.

\begin{hyp}[{\L}ojasiewicz--Simon gradient inequality]
\label{hyp:Lojasiewicz-Simon_gradient_inequality}
Let $\sX$ be a Banach space and $\sH$ be a Hilbert space such that $\sX\subset\sH$ is a continuous embedding and\footnote{If the embedding $\eps:\sX\hookrightarrow\sH$ has dense range, then its adjoint, $\eps^*:\sH^*\hookrightarrow\sX^*$, is an injective bounded operator (that is, an embedding) by \cite[Corollary 4.12 (b)]{Rudin} and the composition $\jmath\circ\eps^*:\sH\hookrightarrow\sX^*$ is a continuous embedding, where $\jmath:\sH \ni h\mapsto (\cdot,h)_\sH \in \sH^*$ is the isometric isomorphism defined by the Riesz map.} its adjoint $\sH^*\subset\sX^*$ is a continuous embedding, and $\sU \subset \sX$ be an open subset, and $\sE:\sU\to\RR$ be a function with\footnote{We use the canonical isometric isomorphism to identify $\sH$ and $\sH^*$.} continuous gradient map $\sE':\sU\to \sH$. If $\varphi\in\sU$ is a critical point, so $\sE'(\varphi) = 0$, then there are constants, $c \in [1,\infty)$, and $\sigma \in (0,1]$, and $\theta \in [1/2,1)$ such that $\sE$ obeys the \emph{{\L}ojasiewicz--Simon gradient inequality},
\begin{equation}
\label{eq:Lojasiewicz-Simon_gradient_inequality}
\|\sE'(x)\|_{\sH} \geq c|\sE(x) - \sE(\varphi)|^\theta, \quad \text{for all } x \in \sU \text{ such that } \|x-\varphi\|_\sX < \sigma.
\end{equation}
\end{hyp}

\begin{hyp}[\Apriori interior estimate for a trajectory]
\label{hyp:Abstract_apriori_interior_estimate_trajectory_main_introduction}
(See Feehan \cite[Hypothesis 2.1 = Hypothesis 24.10]{Feehan_yang_mills_gradient_flow_v4}.)
Let $\sX$ be a Banach space that is continuously embedded in a Hilbert space $\sH$. If $\delta \in (0,\infty)$ is a constant, then there is a constant $C_1 = C_1(\delta) \in [1,\infty)$ with the following significance. If $S, T \in \RR$ are constants obeying $S+\delta \leq T$ and $u \in C([S,T); \sX)\cap C^1((S,T); \sX)$, then $\dot u \in C((S,T); \sX)$ obeys an \apriori \emph{interior estimate on $(S, T]$} in the sense that
\begin{equation}
\label{eq:Abstract_apriori_interior_estimate_trajectory_main_introduction}
\int_{S+\delta}^T \|\dot u(t)\|_\sX\,dt \leq C_1\int_S^T \|\dot u(t)\|_\sH\,dt.
\end{equation}
\end{hyp}

The following elementary \emph{energy (in-)equality} plays an important role in the analysis of the asymptotic behavior of solutions to gradient systems. It is derived for Yang--Mills gradient flow in Chen and Shen \cite[Lemma 2.1]{Chen_Shen_1994} (see also Chen and Shen \cite[Lemma 1]{Chen_Shen_1993}, \cite{Chen_Shen_1995}, and Chen, Shen, and Zhou \cite[Theorem 2.1]{Chen_Shen_Zhou_2002}), Kozono, Maeda, and Naito \cite[Lemma 4.1]{Kozono_Maeda_Naito_1995}, R\r{a}de \cite[Equation (2.7)]{Rade_1992}. and Struwe \cite[Equation (12)]{Struwe_1994} and applied by them in contexts similar to Corollary \ref{cor:Asymptotic_limit_solution_gradient_system}, but these principles holds for any gradient system.

\begin{lem}[Energy equality for solution to a gradient system]
\label{lem:Energy_equality}
Let $\sX$ be a Banach space and $\sH$ be a Hilbert space such that $\sX\subset\sH$ is a continuous embedding and its adjoint $\sH^*\subset\sX^*$ is a continuous embedding, $\sU \subset \sX$ be an open subset, and $\sE:\sU\to\RR$ be a function with continuous gradient map $\sE':\sU\to \sH$. If $T \in (0,\infty)$ and $u \in C([0,T];\sX) \cap C^1((0,T);\sH)$ is a solution to the gradient system,
\begin{equation}
\label{eq:Gradient_system_energy_equality}
\dot u(t) = -\sE'(u(t)), \quad\text{for all } t \in (0,T),
\end{equation}
then $\sE(u(t))$ is a non-increasing function of $t\in[0,T]$ and
\begin{equation}
\label{eq:Energy_equality_interval}
\int_0^T \|\sE'(u(t))\|_\sH^2\,dt = \sE(u(0)) - \sE(u(T)).  
\end{equation}
\end{lem}

\begin{proof}
Observe that
\[
  \frac{d}{dt}\sE(u(t)) = \sE'(u(t))\dot u(t) = \langle \dot u(t),\sE'(u(t))\rangle_{\sX\times\sX^*} = (\dot u(t),\sE'(u(t)))_\sH = -\|\sE'(u(t))\|_\sH^2,
\]
where we apply \eqref{eq:Gradient_system_energy_equality} to obtain the last equality. Equation \eqref{eq:Energy_equality_interval} now follows by integration.
\end{proof}  

Lemma \ref{lem:Energy_equality} has the following useful corollary.

\begin{cor}[Asymptotic behavior of a global solution to a gradient system]
\label{cor:Asymptotic_limit_solution_gradient_system}
Assume the hypotheses of Lemma \ref{lem:Energy_equality} and that $u \in C([0,\infty);\sX) \cap C^1((0,\infty);\sH)$ is a solution to \eqref{eq:Gradient_system_energy_equality}. Then there is a sequence $\{t_n\}_{n\in\NN} \subset (0,\infty)$ such that $t_n\to\infty$ and $\sE'(u(t_n)) \to 0$ in $\sH$ as $n\to\infty$. If in addition there are a point $u_\infty\in\sX$ and a subsequence $\{t_{n_k}\}_{k\in\NN} \subset \{t_n\}_{n\in\NN}$ such that $t_{n_k}\to\infty$ and $u(t_{n_k}) \rightharpoonup u_\infty$ as $k\to\infty$ in the sense that 
\begin{equation}
\label{eq:Weak_convergence_sequence_global_solution_gradient_system}
\lim_{n\to\infty} \langle v, \sE'(u(t_{n_k})) \rangle_{\sX\times\sX^*} = \langle v, \sE'(u_\infty)\rangle_{\sX\times\sX^*}, \quad\text{for all } v \in \sX,
\end{equation}  
then $\sE'(u_\infty)=0$.
\end{cor}

\begin{proof}
By \eqref{eq:Energy_equality_interval}, we have
\[
  \int_0^\infty \|\sE'(u(t))\|_\sH^2\,dt = \sE(u(0))-\lim_{T\to\infty}\sE(u(T)) \leq \sE(u(0)) < \infty.
\]
Hence, there is a sequence $\{t_n\}_{n\in\NN} \subset (0,\infty)$ such that $t_n\to\infty$ and $\|\sE'(u(t_n))\|_\sH \to 0$, that is, $\sE'(u(t_n)) \to 0$ in $\sH$ as $n\to\infty$. Consequently, for the subsequence $\{t_{n_k}\}_{k\in\NN}$ and any $v\in\sX$,
\[
  \sE'(u_\infty)v = \langle v, \sE'(u_\infty)\rangle_{\sX\times\sX^*} = \lim_{k\to\infty} \langle v, \sE'(u(t_{n_k})) \rangle_{\sX\times\sX^*} = 0,
\]
by applying \eqref{eq:Weak_convergence_sequence_global_solution_gradient_system} in the penultimate equality and thus $\sE'(u_\infty) = 0$, as desired.
\end{proof}  

Let $B_\sigma(x_0) := \{x\in\sX: \|x-x_0\|_\sX < r\}$ denote the open ball in $\sX$ with center $x_0\in\sX$ and radius $r\in(0,\infty)$. Given a function $u:[0,\infty)\to\sX$, we let $O(u) := \{u(t): t\geq 0\}$ denote its \emph{orbit}. We have the following analogue of Huang \cite[Theorems 3.3.3 and 3.3.6]{Huang_2006} and abstract analogue of Simon \cite[Corollary 2]{Simon_1983}.

\begin{thm}[Convergence of a subsequence implies convergence for a smooth solution to a gradient system]
\label{thm:Simon_corollary_2_introduction}
(See Feehan \cite[Theorem 1 = Theorem 24.14]{Feehan_yang_mills_gradient_flow_v4} and 
compare Huang \cite[Theorems 3.3.3 and 3.3.6]{Huang_2006} and Simon \cite[Corollary 2]{Simon_1983}.)
Let $\sX$ be a Banach space and $\sH$ be a Hilbert space such that $\sX\subset\sH$ is a continuous embedding and its adjoint $\sH^*\subset\sX^*$ is a continuous embedding, $\sU \subset \sX$ be an open subset, and $\sE:\sU\to\RR$ be a function with continuous gradient map $\sE':\sU\to \sH$, and $\varphi\in\sU$ be a point with $\sE'(\varphi) = 0$, and assume that $\sE$ obeys Hypothesis \ref{hyp:Lojasiewicz-Simon_gradient_inequality}. If $u \in C([0,T];\sX) \cap C^1((0,T);\sH)$ is a solution to the gradient system,
\begin{equation}
\label{eq:gradient_system}
\dot u(t) = -\sE'(u(t)), \quad t \in (0,\infty),
\end{equation}
and the orbit $O(u) = \{u(t): t\geq 0\} \subset \sX$ is precompact\footnote{Recall that \emph{precompact} (or \emph{relatively compact}) subspace $Y$ of a topological space $X$ is a subset whose closure is compact. If the topology on $X$ is metrizable, then a subspace $Z \subset X$ is compact if and only if $Z$ is \emph{sequentially compact} \cite[Theorem 28.2]{Munkres_topology_second_edition}, that is, every infinite sequence in $Z$ has a convergent subsequence in $Z$ \cite[Definition, p. 179]{Munkres_topology_second_edition}.}, and $\varphi$ is a cluster point of $O(u)$, then $u(t)$ converges to $\varphi$ as $t\to\infty$ in the sense that
\[
\lim_{t\to\infty}\|u(t)-\varphi\|_\sX = 0
\quad\hbox{and}\quad
\int_0^\infty \|\dot u\|_\sH\,dt < \infty.
\]
Furthermore, if $u$ satisfies Hypothesis \ref{hyp:Abstract_apriori_interior_estimate_trajectory_main_introduction} on $(0, \infty)$, then
\[
\int_1^\infty \|\dot u\|_\sX\,dt < \infty.
\]
\end{thm}

We next have the following abstract analogue of R\r{a}de's \cite[Proposition 7.4]{Rade_1992}, in turn a variant the \emph{Simon Alternative}, namely \cite[Theorem 2]{Simon_1983}.

\begin{thm}[Simon Alternative for convergence for a smooth solution to a gradient system]
\label{thm:Huang_3-3-6_introduction}
(See Feehan \cite[Theorem 2 = Theorem 24.17]{Feehan_yang_mills_gradient_flow_v4} and compare R\r{a}de \cite[Proposition 7.4]{Rade_1992} and Simon \cite[Theorem 2]{Simon_1983}.)
Let $\sX$ be a Banach space and $\sH$ be a Hilbert space such that $\sX\subset\sH$ is a continuous embedding and its adjoint $\sH^*\subset\sX^*$ is a continuous embedding, $\sU \subset \sX$ be an open subset, and $\sE:\sU\to\RR$ be a function with continuous gradient map $\sE':\sU\to \sH$. Assume that
\begin{enumerate}
\item $\varphi \in \sU$ is a critical point of $\sE$, that is $\sE'(\varphi)=0$; and

\item Given positive constants $b$, $\eta$, and $\tau$, there is a constant $\delta = \delta(\eta, \tau, b) \in (0, \tau]$ such that if $v\in C([t_0,t_0+\tau);\sX) \cap C^1((t_0,t_0+\tau);\sH)$ is a solution to the gradient system \eqref{eq:gradient_system} on $[t_0, t_0 + \tau)$ with $t_0 \in \RR$ and $\|v(t_0)\|_\sX \leq b$, then
\begin{equation}
\label{eq:Gradient_solution_near_initial_data_at_t0_for_short_enough_time_introduction}
\sup_{t\in [t_0, t_0+\delta)}\|v(t) - v(t_0)\|_\sX < \eta.
\end{equation}
\end{enumerate}
If $\sE$ obeys Hypothesis \ref{hyp:Lojasiewicz-Simon_gradient_inequality} and $(c,\sigma,\theta)$ are the {\L}ojasiewicz--Simon constants for $(\sE,\varphi)$, then there is a constant
\[
\eps = \eps(c, C_1, \delta, \theta, \rho, \sigma, \tau, \varphi) \in (0, \sigma/4)
\]
with the following significance.  If $u \in C([0,\infty);\sX) \cap C^1((0,\infty);\sH)$ is a solution to \eqref{eq:gradient_system} that satisfies Hypothesis \ref{hyp:Abstract_apriori_interior_estimate_trajectory_main_introduction} on $(0, \infty)$ and there is a constant $T \geq 0$ such that
\begin{equation}
\label{eq:Rade_7-2_banach_introduction}
\|u(T) - \varphi\|_\sX < \eps,
\end{equation}
then either
\begin{enumerate}
\item
\label{item:Theorem_3-3-6_energy_u_at_time_t_below_energy_critical_point_introduction}
$\sE(u(t)) < \sE(\varphi)$ for some $t>T$, or
\item
\label{item:Theorem_3-3-6_u_converges_to_limit_u_at_infty_introduction}
$u(t)$ converges in $\sX$ to a limit $u_\infty \in \sX$ as $t\to\infty$ in the sense that
\[
\lim_{t\to\infty}\|u(t)-u_\infty\|_\sX =0
\quad\hbox{and}\quad
\int_1^\infty \|\dot u\|_\sX\,dt < \infty.
\]
If $\varphi$ is a cluster point of the orbit $O(u) = \{u(t): t\geq 0\}$, then $u_\infty = \varphi$.
\end{enumerate}
\end{thm}

We have the following enhancement of Huang \cite[Theorem 3.4.8]{Huang_2006}.

\begin{thm}[Convergence rate under the validity of a {\L}ojasiewicz--Simon gradient inequality]
\label{thm:Huang_3-4-8_introduction}
(See Feehan \cite[Theorem 3 = Theorem 24.21]{Feehan_yang_mills_gradient_flow_v4}.)
Let $\sX$ be a Banach space and $\sH$ be a Hilbert space such that $\sX\subset\sH$ is a continuous embedding and its adjoint $\sH^*\subset\sX^*$ is a continuous embedding, and $\sU \subset \sX$ be an open subset, and $\sE:\sU\to\RR$ be a function with continuous gradient map $\sE':\sU \to \sH$, and $\varphi\in\sU$ be a point with $\sE'(\varphi) = 0$, and assume that $\sE$ obeys Hypothesis \ref{hyp:Lojasiewicz-Simon_gradient_inequality}. If $u$ belongs to
\begin{equation}
\label{eq:gradient_system_solution}
C([T_0,T); \sX)\cap C^1((T_0,T);\sH)
\end{equation}
with $T_0=0$ and $T=\infty$ and is a solution to the gradient system \eqref{eq:gradient_system},
\[
\dot u(t) = -\sE'(u(t)), \quad t \in (0,\infty),
\]
such that $O(u) \subset B_\sigma(\varphi) \subset \sU$ for all $t\in[0,\infty)$, then there exists $u_\infty \in \sH$ such that
\begin{equation}
\label{eq:Huang_3-45_H_introduction}
\|u(t) - u_\infty\|_\sH \leq \Psi(t), \quad t\geq 0,
\end{equation}
where
\begin{equation}
\label{eq:Huang_3-45_growth_rate_introduction}
\Psi(t)
:=
\begin{cases}
\displaystyle
\frac{1}{c(1-\theta)}\left(c^2(2\theta-1)t + (\gamma-a)^{1-2\theta}\right)^{-(1-\theta)/(2\theta-1)},
& 1/2 < \theta < 1,
\\
\displaystyle
\frac{2}{c}\sqrt{\gamma-a}\,\exp(-c^2t/2),
&\theta = 1/2,
\end{cases}
\end{equation}
and $a, \gamma$ are constants such that $\gamma > a$ and
\[
a \leq \sE(v) \leq \gamma, \quad\text{for all } v \in \sU.
\]
If in addition $u$ obeys Hypothesis \ref{hyp:Abstract_apriori_interior_estimate_trajectory_main_introduction}, then $u_\infty \in \sX$ and
\begin{equation}
\label{eq:Huang_3-45_X_introduction}
\|u(t+1) - u_\infty\|_\sX \leq 2C_1\Psi(t), \quad t\geq 0,
\end{equation}
where $C_1 \in [1,\infty)$ is the constant in Hypothesis \ref{hyp:Abstract_apriori_interior_estimate_trajectory_main_introduction} for $\delta=1$.
\end{thm}

We have the following analogue of Huang \cite[Theorem 5.1.1]{Huang_2006}.

\begin{thm}[Existence and convergence of a global solution to a gradient system near a local minimum]
\label{thm:Huang_5-1-1_introduction}
(See Feehan \cite[Theorem 4 = Theorem 24.22]{Feehan_yang_mills_gradient_flow_v4}.)
Let $\sX$ be a Banach space and $\sH$ be a Hilbert space such that $\sX\subset\sH$ is a continuous embedding and its adjoint $\sH^*\subset\sX^*$ is a continuous embedding, and $\sU \subset \sX$ be an open subset, $\sE:\sU\to\RR$ be a function with continuous gradient map $\sE':\sU \to \sH$, and $\varphi\in\sU$ be a local minimum of $\sE$, and assume that $\sE$ obeys Hypothesis \ref{hyp:Lojasiewicz-Simon_gradient_inequality} with constants $c \in (0, \infty)$ and $\sigma \in (0,1]$ and $\theta \in [1/2,1)$. Assume further that
\begin{enumerate}  
\item\label{item:Huang_5-1-1_local_existence}
\emph{(Local existence)} For each $u_0 \in \sU$, there exists a solution $u$ in \eqref{eq:gradient_system_solution} to the gradient system \eqref{eq:gradient_system} on $[T_0,T) = [0, \tau)$ with $u(0)=u_0$ and some positive constant $\tau=\tau(\sE,u_0)$;

\item\label{item:Huang_5-1-1_interior_estimate_length_flowline}
\emph{(\Apriori interior estimate for lengths of gradient flowlines)} Hypothesis \ref{hyp:Abstract_apriori_interior_estimate_trajectory_main_introduction} holds for solutions $u$ in \eqref{eq:gradient_system_solution} to the gradient system \eqref{eq:gradient_system}; and

\item\label{item:Huang_5-1-1_deviation_from_initial_data}
\emph{(Deviation from initial data)} Given positive constants $b$ and $\eta$, there is a constant $\delta = \delta(\eta, \tau, b) \in (0, \tau]$ such that if $v$ is a solution in \eqref{eq:gradient_system_solution} to the gradient system \eqref{eq:gradient_system} on $[T_0,T) = [0, \tau)$ with $\|v(0)\|_\sX \leq b$, then
\begin{equation}
\label{eq:Gradient_solution_near_initial_data_at_time_zero_for_short_enough_time_introduction}
\sup_{t\in [0, \delta]}\|v(t) - v(0)\|_\sX < \eta.
\end{equation}
\end{enumerate}
Then there is a constant $\eps = \eps(c,C_1,\delta, \theta, \rho, \sigma, \tau, \varphi) \in (0, \sigma/4)$ with the following significance. For each $u_0 \in B_\eps(\varphi)$, the gradient system \eqref{eq:gradient_system} with $u(0)=u_0$ admits a solution $u$ in \eqref{eq:gradient_system_solution} with $T_0=0$ and $T=\infty$ and $O(u) \subset B_{\sigma/2}(\varphi)$ and that converges to a limit $u_\infty \in B_\sigma(\varphi)$ as $t\to\infty$ with respect to the norm on $\sX$ in the sense that
$$
\lim_{t \to \infty} \|u(t) - u_\infty\|_\sX = 0 \quad\text{and}\quad \int_1^\infty\|\dot u(t)\|_\sX\,dt < \infty.
$$
Moreover, $\sE(u_\infty) = \sE(\varphi)$.
\end{thm}

Finally, we have the following analogue of Huang \cite[Theorem 5.1.2]{Huang_2006}.

\begin{thm}[Dynamical stability of a local minimum]
\label{thm:Huang_5-1-2_introduction}
(See Feehan \cite[Theorem 1 = Theorem 24.14 and Theorem 5 = Theorem 24.30]{Feehan_yang_mills_gradient_flow_v4}.)  
Assume the hypotheses of Theorem \ref{thm:Huang_5-1-1_introduction} and let $u$ be as in its conclusion. Then as an equilibrium of \eqref{eq:gradient_system}, the point $\varphi$ is \emph{Lyapunov stable} (see \cite[Definition, p. 32]{Sell_You_2002}). If $\varphi$ is isolated or a cluster point of the orbit $O(u)$, then $u_\infty=\varphi$ and $\varphi$ is \emph{uniformly asymptotically stable} (see \cite[Definition, p.32]{Sell_You_2002}).
\end{thm}

See Knopf and Sesum \cite{Knopf_Sesum_2019} for a related discussion of different concepts of stability in the context of Ricci flow for Riemannian metrics on $d$-manifolds near Ricci solitons (for example, Ricci-flat metrics).

\section{{\L}ojasiewicz distance inequality for functions on Banach spaces}
\label{sec:Lojasiewicz_distance_inequality_functions_Banach_spaces}
We now turn to the proof of Theorem \ref{mainthm:Lojasiewicz_distance_inequality_hilbert_space}: we provide the modifications to our proof of \cite[Corollary 4]{Feehan_lojasiewicz_inequality_all_dimensions} required for the infinite-dimensional setting considered here. 

\begin{proof}[Proof of Theorem \ref{mainthm:Lojasiewicz_distance_inequality_hilbert_space}]
  Consider Item \eqref{item:Distance_critical_set}, so $\sF=\sE$ in \eqref{eq:Lojasiewicz_gradient_inequality} and \eqref{eq:Gradient_flow}. Let $\sigma \in (0,1]$ and $\delta \in (0,\sigma_1/4]$ denote the constants for $\sE$ corresponding to those for $\sF$ in \eqref{eq:Lojasiewicz_gradient_inequality} and \eqref{eq:Gradient_flow}. Consider a point $x \in B_\delta$ such that $\sE(x) > 0$ and thus $\sE'(x)\neq 0$ by \eqref{eq:Lojasiewicz_gradient_inequality}. Let $T_0 \in (0,\infty]$ be the smallest time such that $\sE'(\bx(T_0))=0$ (and thus $\bx(T_0) \in B_{\sigma_1}\cap\Crit\sE$), where $\bx \in C([0,\infty);\sX)\cap C^1((0,\infty);\sH)$ is the given solution to \eqref{eq:Gradient_flow}, and define the $C^1$ arc length parameterization function by
\[
  s(t) := \int_0^t \|\dot\bx(t)\|_\sH\,dt, \quad\text{for all } t \in [0,T_0),
\]
so that $ds/dt = \|\dot\bx(t)\|_\sH = \|\sE'(\bx(t))\|_\sH$ by \eqref{eq:Gradient_flow}, denoting $\dot\bx(t) = d\bx/dt$ for convenience. Set $S_0 := s(T_0) \in (0,\infty]$ and write $t = t(s)$ for $s \in [0,S_0)$. Define $\by(s) := \bx(t(s))$ and observe that
\[
  \frac{d\by}{ds} = \frac{d\bx}{dt}\frac{dt}{ds} = \frac{d\bx}{dt}\left(\frac{ds}{dt}\right)^{-1} = \frac{d\bx}{dt}\frac{1}{\|\sE'(\bx(t))\|_\sH} = -\frac{\sE'(\bx(t))}{\|\sE'(\bx(t))\|_\sH}, \quad\text{for all } t \in (0,T_0),
\]
where we again apply \eqref{eq:Gradient_flow} to obtain the final equality. Hence, $\by \in C([0,S_0);\sX)\cap C^1((0,S_0);\sH)$ is a solution to the ordinary differential equation,
\begin{equation}
  \label{eq:Gradient_flow_arclength}
  \frac{d\by}{ds} = -\frac{\sE'(\by(s))}{\|\sE'(\by(s))\|_\sH} \quad \text{(in $\sH$) with } \by(0)=x.
\end{equation}
Write $Q(s):=\sE(\by(s))$ and observe that
\begin{align*}
  Q'(s) &= \sE'(\by(s))\by'(s)
  \\
  &= \left(\by'(s),\sE'(\by(s))\right)_\sH \quad\text{(by Riesz isomorphism)}
  \\
  &= -\frac{\left(\sE'(\by(s)), \sE'(\by(s))\right)_\sH}{\|\sE'(\by(s))\|_\sH}, \quad\text{for all } s\in [0,S_0) \quad\text{(by \eqref{eq:Gradient_flow_arclength})}
\end{align*}
using the Riesz isomorphism, $\sH \ni h \mapsto (\cdot,h)_\sH \in \sH^*$, to view $\sE'(\by(s))$ as an element of $\sH$ or $\sH^*$ according to the context. In particular, we obtain
\begin{equation}
\label{eq:dQds_negative}
Q'(s) = - \|\sE'(\by(s))\|_\sH < 0, \quad\text{for all } s\in [0,S_0).  
\end{equation}
Now $Q(0) = \sE(x)>0$ and $Q(s)\leq Q(0)$ for all $s \in [0,S_0)$ by \eqref{eq:dQds_negative}. But then we have\footnote{Our argument generalizes the proof of Bierstone and Milman \cite[Theorem 2.8]{Bierstone_Milman_1997} from the case of $\sX=\sH=\RR^n$. We also added a hypothesis that $\sE\geq 0$ on $\sU$, which is used in the inequality \eqref{eq:Bierstone_Milman_correction} and is implicit for the same reason in the proof of \cite[Theorem 2.8]{Bierstone_Milman_1997}, but omitted from that statement.}
\begin{align}
\label{eq:Bierstone_Milman_correction}  
\frac{\sE(x)^{1-\theta}}{1-\theta} &\geq \frac{Q(0)^{1-\theta} - Q(s)^{1-\theta}}{1-\theta}
\\
\notag  
&= -\frac{1}{1-\theta} \int_{0}^{s} \frac{d}{du}Q(u)^{1-\theta}\,du
= - \int_{0}^{s} Q(u)^{-\theta}Q'(u)\,du
\\
\notag  
&= \int_{0}^{s} \sE(\by(u))^{-\theta}\|\sE'(\by(u))\|_\sH\,du
\\
\notag  
&\geq \int_{0}^{s} C\,du = Cs, \quad 0 \leq s < S_0 \quad\text{(by \eqref{eq:Lojasiewicz_gradient_inequality})}.
\end{align}
In applying the {\L}ojasiewicz gradient inequality \eqref{eq:Lojasiewicz_gradient_inequality} to obtain the last line above, we relied on the fact that $\by(s) = \bx(t) \in B_{\sigma/2}$ by \eqref{eq:Gradient_flow} for all $t\in [0,T_0)$ or, equivalently, $s \in [0,S_0)$. Therefore,
\begin{equation}
\label{eq:Energy_power_lower_bound}
\frac{\sE(x)^{1-\theta}}{1-\theta} \geq CS_0. 
\end{equation}
It follows that $S_0<\infty$ and thus as $s\uparrow S_0$, the solution $\by(s)$ converges (in $\sX$) to a point $\by(S_0) = \bx(T_0) \in \Crit\sE$ in a finite time $S_0$. Moreover, by \eqref{eq:Gradient_flow} we also have $\bx(T_0) \in \bar{B}_{\sigma/2} \subset B_\sigma$. Since $\|\by'(s)\|_\sH = 1$, then $\by(s)$ is parameterized by arc length and
\begin{align*}
  S_0 &= \Length_\sH\{\by(s): s\in [0,S_0]\} = \int_0^{S_0}\|\dot\by(s)\|_\sH\,ds
  \\
      &\geq \|\by(S_0)-x\|_\sH = \|\bx(T_0)-x\|_\sH
  \\
      &\geq \inf_{z\in B_\sigma\cap\Crit\sE}\|z-x\|_\sH \quad\text{(since $\bx(T_0) \in B_\sigma\cap\Crit\sE$)}
  \\
      &\equiv \dist_\sH(x, B_\sigma\cap\Crit\sE).
\end{align*}
From \eqref{eq:Energy_power_lower_bound}, we thus obtain
\[
\sE(x)^{1-\theta} \geq (1-\theta)C\,\dist_\sH(x, B_\sigma\cap\Crit\sE),
\]
and this is \eqref{eq:Lojasiewicz_distance_inequality_critical_set} (noting that $\sE(x)>0$ by the reduction described earlier), with exponent $\alpha = 1/(1-\theta) \in [2,\infty)$ and positive constant $C_1 = ((1-\theta)C)^{1/(1-\theta)}$. 

If $B_\sigma\cap\Crit\sE \subset B_\sigma\cap\Zero\sE$, then
\begin{align*}
  \dist(x, B_\sigma\cap\Crit\sE) &\equiv \inf_{z\in B_\sigma\cap\Crit\sE}\|z-x\|_\sH
\\
  &\geq \inf_{z\in B_\sigma\cap\Zero\sE}\|z-x\|_\sH = \dist(x, B_\sigma\cap\Zero\sE).
\end{align*}
Therefore, \eqref{eq:Lojasiewicz_distance_inequality_zero_set} follows from \eqref{eq:Lojasiewicz_distance_inequality_critical_set}. This proves Item \eqref{item:Distance_critical_set}.

Consider Item \eqref{item:Distance_zero_noncritical_set}. We can apply \eqref{eq:Lojasiewicz_distance_inequality_zero_set} to $\sF=\sE^2$ with constants $C_1\in (0,\infty)$ and $\alpha \in [2,\infty)$ and $\sigma \in (0,1]$ and $\delta\in(0,\sigma/4]$ determined by $\sF$ to give
\[
\sF(x) \geq C_1\,\dist(x, B_\sigma\cap\Zero\sF)^\alpha, \quad\text{for all } x \in B_\delta.
\]
Clearly, $\Zero\sE = \Zero\sF$ and therefore,
\[
\sE(x)^2 \geq C_1\,\dist(x, B_\sigma\cap\Zero\sE)^\alpha, \quad\text{for all } x \in B_\delta.
\]
But this is \eqref{eq:Lojasiewicz_distance_inequality_zero_noncritical_set}, as desired, with exponent $\beta = \alpha/2 \in [1,\infty)$ and positive constant $C_2 = \sqrt{C_1}$. This completes the proof of Item \eqref{item:Distance_zero_noncritical_set} and hence Theorem \ref{mainthm:Lojasiewicz_distance_inequality_hilbert_space}.
\end{proof}

\section{Local well-posedness, a priori estimates, and minimal lifetimes for solutions to Yang--Mills gradient flow on a Coulomb-gauge slice}
\label{sec:Local_well-posedness_Yang-Mills_gradient_flow}
After establishing a few basic definitions and conventions in Section \ref{subsec:Basic_definitions}, we proceed in Section \ref{subsec:Yang-Mills_heat_equation} discuss several kinds of evolution equations defined by the gradient of the Yang--Mills energy function \eqref{defn:Yang-Mills_energy_function}:
\begin{inparaenum}[(\itshape a\upshape)]
\item traditional Yang--Mills gradient flow \eqref{eq:Yang-Mills_gradient_flow},
\item Yang--Mills heat flow \eqref{eq:Yang-Mills_heat_equation}, and
\item Yang--Mills gradient flow on a Coulomb-gauge slice, \eqref{eq:Yang-Mills_heat_equation_with_projection} or \eqref{eq:Yang-Mills_gradient_flow_slice}.
\end{inparaenum}  
In Section \ref{subsec:Estimate_Yang-Mills_heat_equation_nonlinearity}, we develop the required estimates for the Yang--Mills nonlinearity \eqref{eq:Yang-Mills_heat_equation_nonlinearity_relative_rough_Laplacian_plus_one}. In Section \ref{subsec:Rade_Lemma_7-3_generalization}, we derive \apriori estimates for the lengths of flowlines for Yang--Mills gradient flow on a Coulomb-gauge slice. Our goal in Section \ref{subsec:Local_well-posedness_priori_estimates_minimal_lifetimes_Yang-Mills_gradient_flow} is to establish local well-posedness, \apriori estimates, and minimal lifetimes for solutions to Yang--Mills gradient flow on a Coulomb-gauge slice. In particular, we prove Theorem \ref{thm:Local_well-posedness_priori_estimates_minimal_lifetimes_Yang-Mills_gradient_flow} by applying our general results on evolution equations in Banach spaces from Section \ref{sec:Local_well-posedness_nonlinear_evolution_equations_Banach_spaces}.

\subsection{Basic definitions}
\label{subsec:Basic_definitions}
 We begin with a formal

\begin{defn}[Yang--Mills energy function]
\label{defn:Yang-Mills_energy_function}
Let $G$ be a Lie group, $P$ be a smooth principal $G$-bundle over a closed, smooth Riemannian manifold $(X,g)$ of dimension $d \geq 2$, and $A$ be a $W^{1,p}$ connection on $P$, where $p$ is admissible in the sense of the forthcoming Definition \ref{defn:Admissible_Sobolev_exponent_for_Yang-Mills_energy}. We define the \emph{Yang--Mills energy function} by
\begin{equation}
\label{eq:Yang-Mills_energy_function}
\YM(A)  := \frac{1}{2}\int_X |F_A|^2\,d\vol_g.
\end{equation}
\end{defn}

\begin{defn}[Admissible Sobolev exponents for the Yang--Mills energy function]
\label{defn:Admissible_Sobolev_exponent_for_Yang-Mills_energy}
To ensure $W^{1,p}(X) \subset W^{1,2}(X)\cap L^4(X)$ and $|F_A| \in L^2(X)$ when $F_A=dA+\frac{1}{2}[A,A]$ and $A \in W^{1,p}(X;T^*X\otimes\fg)$ is a connection one-form on the product bundle, $X\times G$, we require that $p\geq 2$ if $d=2,3,4$ or $p\geq 4d/(d+4)$ if $d \geq 5$. We call such a Sobolev exponent $p$ \emph{admissible}\footnote{Indeed, if $d\geq 2$ and $1\leq p<d$, then $p^*:=dp/(d-p)\geq 4 \iff p \geq 4d/(d+4)$ and thus $W^{1,p}(X) \subset L^4(X)$ by \cite[Theorem 4.12]{AdamsFournier}, while if $d=2,3,4$, then $W^{1,p}(X) \subset L^4(X)$ for any $p\geq 2$ by \cite[Theorem 4.12]{AdamsFournier}.}.

We may also require that $p>d/2$ to ensure that $W^{2,p}(X)\subset C(X)$ by the Sobolev Embedding \cite[Theorem 4.12]{AdamsFournier}, so gauge transformations in $\Aut^{2,p}(P)$ are continuous and preserve the topology of $P$. Note that because $d/2 \geq 4d/(d+4) \iff d \geq 4$, the condition $p\geq d/2$ ensures that $p$ is admissible when $d\geq 4$. 
\end{defn}

Let $A_t$ for $t\in(-\eps,\eps)$ be a smooth embedded curve in $\sA^{1,p}(P)$ through $A_0=A$. The derivative of the Yang--Mills energy function $\YM$ in \eqref{eq:Yang-Mills_energy_function} at $A$ in a direction $a = dA/dt|_{t=0} \in W^{1,p}(X;T^*X\otimes\ad P)$ is 
\begin{equation}
\label{eq:Derivative_Yang-Mills_energy}  
\YM'(A)a := \left.\frac{d}{dt}\frac{1}{2}\int_X |F_{A_t}|^2\,d\vol_g\right|_{t=0} = \int_X \langle F_A, d_Aa \rangle\,d\vol_g.
\end{equation}
We recall the\footnote{Compare Price \cite[p. 138]{Price_1983} or Wehrheim \cite[Definition 9.1]{Wehrheim_2004}. Although Wehrheim instead requires $p>d/2$ for $d\geq 3$ and $p\geq 4/3$ if $d=2$, that does not ensure that $W^{1,p}(X) \subset W^{1,2}(X)$ when $d=2,3$ or $W^{1,p}(X) \subset L^4(X)$ when $d=5,6,7$; see Definition \ref{defn:Admissible_Sobolev_exponent_for_Yang-Mills_energy}.}

\begin{defn}[(Weak) Yang--Mills connection]
\label{defn:Definition_weak_Yang-Mills_W12_L4_connection}
Continue the notation of Definition \ref{defn:Yang-Mills_energy_function}. One calls a $W^{1,p}$ connection $A$ on $P$ a \emph{(weak) Yang--Mills connection} if it is a critical point of the Yang--Mills energy function \eqref{eq:Yang-Mills_energy_function},
\begin{equation}
\label{eq:Definition_weak_Yang-Mills_W12_L4_connection}  
\YM'(A)a = 0, \quad\text{for all } a \in W^{1,p}(X;T^*X\otimes\ad P).
\end{equation}
\end{defn}

When the connection $A$ in Definition \ref{defn:Definition_weak_Yang-Mills_W12_L4_connection} on $P$ is a smooth, then it obeys the \emph{Yang--Mills equation} with respect to the metric $g$,
\begin{equation}
\label{eq:Yang-Mills_equation}
d_A^*F_A = 0 \quad\text{on } X,
\end{equation}
that is, the Euler--Lagrange equation for \eqref{eq:Yang-Mills_energy_function}.

For $p>d/2$ that is admissible in the sense of Definition \ref{defn:Admissible_Sobolev_exponent_for_Yang-Mills_energy}, one can show that the Yang--Mills energy function \eqref{eq:Yang-Mills_energy_function},
\[
  \YM:\sA^{1,p}(P) \to \RR,
\]
is analytic \cite[Proposition 3.1.1]{Feehan_Maridakis_Lojasiewicz-Simon_coupled_Yang-Mills}. The quotient space $\sB^{1,p}(P)$ is a Banach stratified space \cite[p. 133]{DK}, with top smooth (in fact, analytic by Corollary \ref{cor:Slice}) stratum given by the open subspace $\sB^{*;1,p}(P)$ in \eqref{eq:Minimal_stabilizer} defined by gauge-equivalence classes of connections with trivial isotropy group, namely $\Center(G)$, and the remaining smooth (in fact, analytic) strata, $\sB^{H;1,p}(P)$, labeled by the (conjugacy classes of) isotropy groups $H \subset G$. The Yang--Mills energy function \eqref{eq:Yang-Mills_energy_function} is gauge-invariant and thus descends to an analytic function
\[
  \YM:\sB^{H;1,p}(P) \to \RR,
\]
on each stratum and a continuous function on the whole quotient space,
\[
  \YM:\sB^{1,p}(P) \to \RR.
\]
The discussion of the structure of $\sB^{1,p}(P)$ as a smoothly stratified space in Donaldson and Kronheimer \cite[p. 132--133]{DK} motivates our

\begin{defn}[Critical point of the Yang--Mills function on the quotient space of connections]
\label{defn:Definition_critical_point_quotient_space}
Continue the notation of Definition \ref{defn:Yang-Mills_energy_function} and assume further that $p>d/2$. We say that $[A] \in \sB^{1,p}(P)$ is a critical point of the Yang--Mills energy function $\YM:\sB^{1,p}(P) \to \RR$ if
\[
  \YM'[A]a=0, \quad\text{for all } a \in \Ker d_A^*\cap W_{A_1}^{1,p}(X;T^*X\otimes\ad P).
\]
\end{defn}

See Lemma \ref{lem:Critical_point_Yang-Mills_energy_function_slice} for further comparison of definitions of critical points of $\YM$. The Yang--Mills energy function $\YM:\sA^{1,p}(P)\to\RR$ in \eqref{eq:Yang-Mills_energy_function} has the \emph{Hessian operator},
\[
  \YM''(A):T_A\sA^{1,p}(P) \to T_A^*\sA^{1,p}(P),
\]
at a critical point $A \in \sA(P)$ given by \cite[Theorem 6.5]{Bourguignon_Lawson_1981}
\begin{equation}
\label{eq:Hessian_Yang-Mills_energy_function}
\YM''(A)(a)b = (d_Aa,d_Ab)_{L^2(X)} + (F_A,a\wedge b)_{L^2(X)},
\end{equation}
for all $a, b \in T_A\sA^{1,p}(P) = W_{A_1}^{1,p}(X;\Lambda^1\otimes\ad P)$.

\begin{defn}[Morse--Bott property of the Yang--Mills energy function on the affine space of connections]
\label{defn:Definition_Morse-Bott_affine_space}
We say that $\YM:\sA^{1,p}(P)\to\RR$ is \emph{Morse--Bott} at a critical point $A$ if
\begin{subequations}
 \label{eq:Definition_Morse-Bott_affine_space} 
\begin{gather}  
  \Crit^{1,p}\YM \text{is a $C^2$ submanifold, and}
\\
  \Ker\YM''(A) = T_A\Crit^{1,p}\YM,
\end{gather}
\end{subequations}
where $\Crit^{1,p}\YM := \sA^{1,p}(P)\cap\Crit\YM$.
\end{defn}

See Feehan \cite[Section 1]{Feehan_lojasiewicz_inequality_all_dimensions_morse-bott} and references therein for further discussion of the Morse--Bott property of smooth functions on Banach manifolds. Since the quotient space $\sB^{1,p;*}(P)$ is not a manifold at points $[A]$ where $A$ has nontrivial isotropy group in $\Aut^{2,p}(P)$, we make the

\begin{defn}[Morse--Bott property of the Yang--Mills energy function on the quotient space of connections]
\label{defn:Definition_Morse-Bott_quotient_space}
Continue the notation of Definition \ref{defn:Definition_critical_point_quotient_space}. We say that $\YM:\sB^{1,p}(P) \to \RR$ is \emph{Morse--Bott} at a critical point in $\sB^{1,p}(P)$ represented by a $W^{1,p}$ connection $A$ if
\begin{subequations}
 \label{eq:Definition_Morse-Bott_quotient_space} 
\begin{gather} 
  \mathbf{Crit}^{1,p}\YM \text{is a $C^2$ submanifold, and}
\\
  \mathbf{Ker}^{1,p}\YM''(A) = T_A\mathbf{Crit}^{1,p}\YM,
\end{gather}
\end{subequations}
where
\begin{align*}
  \mathbf{Crit}^{1,p}\YM &:= \left(A+\Ker d_A^*\cap W_{A_1}^{1,p}(X;T^*X\otimes\ad P)\right)\cap\Crit\YM,
  \\
  \mathbf{Ker}^{1,p}\YM''(A) &:= \Ker\YM''(A)\cap\Ker d_A^*\cap W_{A_1}^{1,p}(X;T^*X\otimes\ad P). 
\end{align*}
\end{defn}

Definition \ref{defn:Definition_Morse-Bott_quotient_space} can be equivalently stated for the restriction $\widehat\YM$ in \eqref{eq:Yang-Mills_energy_function_slice} of $\YM$ to the Coulomb-gauge slice $A+\Ker d_A^*\cap W_{A_1}^{1,p}(X;T^*X\otimes\ad P)$. Moreover, by the Slice Theorem \ref{cor:Freed_Uhlenbeck_3-2_W1q}, the function $\YM:\sA^{1,p}(P)\to\RR$ is Morse--Bott at a critical point $A$ in the sense of Definition \ref{defn:Definition_Morse-Bott_affine_space} if and only if $\YM:\sB^{1,p}(P)\to\RR$ is Morse--Bott at $[A]$ in the sense of Definition \ref{defn:Definition_Morse-Bott_quotient_space}. Elliptic regularity ensures that the Morse--Bott property is independent of the choice of Sobolev exponent $p$ or reference connection $A_1$.

When a Yang--Mills connection $\Gamma$ is \emph{flat}, then Morse--Bott property is implied by properties of the elliptic deformation complex for $\Gamma$ by analogy with the Kuranishi theory of deformation of complex structures \cite{Kuranishi}, as we discuss in Section \ref{subsec:Regular_flat_connections_and_Morse-Bott_property_Yang-Mills_energy_function}. The \emph{Zariski tangent space} to $\sB^{1,p}(P)$ at a point represented by a $W^{1,p}$ flat connection $\Gamma$ is $\bH_\Gamma^1(X;\ad P)$ and we say that $\Gamma$ defines a \emph{regular} point of $M(P)$ if $\bH_\Gamma^2(X;\ad P) = (0)$, where
\begin{multline}
  \label{eq:H_Gamma^i_adP_W1p_harmonic}
  \bH_\Gamma^i(X;\ad P)
  :=
  \Ker\left(d_\Gamma+d_\Gamma^*:W_{A_1}^{1,p}(X;\wedge^i(T^*X)\otimes\ad P) \right.
    \\
    \to \left. L^p(X;\wedge^{i+1}(T^*X)\otimes\ad P) \oplus L^p(X;\wedge^{i-1}T^*X\otimes\ad P)\right),
    \quad\text{for } i \in \ZZ.
\end{multline}
The preceding groups obey the canonical isomorphisms
\[
  \bH_\Gamma^i(X;\ad P) \cong H_\Gamma^i(X;\ad P), 
\]
where, for $i \in \ZZ$,
\begin{equation}
  \label{eq:H_Gamma^i_adP_W1p}
  H_\Gamma^i(X;\ad P)
  :=
  \frac{\Ker\left(d_\Gamma:W_{A_1}^{1,p}(X;\wedge^i(T^*X)\otimes\ad P) \to L^p(X;\wedge^{i+1}(T^*X)\otimes\ad P)\right)}
   {\Ran\left(d_\Gamma:W_{A_1}^{2,p}(X;\wedge^{i-1}(T^*X)\otimes\ad P) \to W_{A_1}^{1,p}(X;\wedge^i(T^*X)\otimes\ad P)\right)}.
 \end{equation}
(Compare the simpler definitions of $H_\Gamma^i(X;\ad P)$ in the forthcoming equation \eqref{eq:DeRham_cohomology_group_flat_connection} and $\bH_\Gamma^i(X;\ad P)$ in the forthcoming equation \eqref{eq:DeRham_cohomology_group_flat_connection_harmonic} when $\Gamma$ is a $C^\infty$ flat connection.) By Lemma \ref{lem:Morse-Bott_property_Yang-Mills_energy_near_flat_connection}, the Yang--Mills energy functions $\YM:\sA^{1,p}(P) \to \RR$ and $\YM:\sB^{1,p}(P) \to \RR$ are Morse--Bott at a flat connection $\Gamma$ and point $[\Gamma]$, respectively, if $\bH_\Gamma^2(X;\ad P)=(0)$.

\subsection{The Yang--Mills gradient flow, heat, and gradient flow equations on a Coulomb-gauge slice}
\label{subsec:Yang-Mills_heat_equation}
Continue the notation of Definition \ref{defn:Yang-Mills_energy_function}. We first recall the definition of the \emph{Yang--Mills gradient flow equation},
\begin{gather}
\label{eq:Yang-Mills_gradient_flow}
\frac{\partial A}{\partial t} + d_A^*F_A = 0,
  \\
\label{eq:Yang-Mills_flow_initial_condition}
A(0) = A_0,
\end{gather}
for solution $A = A_0+a$, where $A_0$ is a given $W^{1,p}$ connection on $P$ and
\[
  a \in C([0,\infty); W_{A_1}^{1,p}(T^*X\otimes\ad P)) \cap C^1((0,\infty); W_{A_1}^{1,p}(T^*X\otimes\ad P)).
\]
It is customary \cite{DonASD, Struwe_1994} to convert \eqref{eq:Yang-Mills_gradient_flow} into nonlinear \emph{Yang--Mills heat equation},
\begin{equation}
\label{eq:Yang-Mills_heat_equation}
  \frac{\partial A}{\partial t} + d_A^*F_A + d_Ad_A^*a = 0,
\end{equation}
where $a = A-A_\infty$ and $A_\infty$ is a smooth connection on $P$, by applying the Donaldson--DeTurck Trick \cite[Lemma 20.3]{Feehan_yang_mills_gradient_flow_v4}, but that results in the loss of regularity by two spatial derivatives and so we shall consider an alternative approach: The restriction of Yang--Mills gradient flow \eqref{eq:Yang-Mills_gradient_flow} to a Coulomb--gauge slice through a fixed, smooth connection $A_\infty$.

For any integer $l \geq 0$, the exterior covariant derivative \eqref{eq:Exterior_covariant_derivative} and its $L^2$-adjoint \eqref{eq:Exterior_covariant_derivative_L2_adjoint} define the \emph{Hodge Laplace operator} \cite[p. 93]{Lawson},
\begin{equation}
\label{eq:Lawson_page_93_Hodge_Laplacian}
\Delta_{A_\infty} := d_{A_\infty}^*d_{A_\infty} + d_{A_\infty}d_{A_\infty}^* \quad\text{on } \Omega^l(X;\ad P).
\end{equation}
By arguing exactly as in the proof of \cite[Theorem 8.37]{GT}, one sees that the eigenvalues of the $L^2$-self-adjoint, second-order, elliptic partial differential operator $\Delta_{A_\infty}$ are countable, real, non-negative, and discrete with at most a limit point at infinity (compare \cite[Theorem 1, p. 8, or Theorem B.2]{Chavel}, \cite[Exercise 6.6]{Warner}).

Given a $W^{1,p}$ initial connection $A(0)=A_0$ on $P$ and writing $a_0 = A_0-A_\infty$, we consider the nonlinear evolution equation,
\begin{gather}
\label{eq:Yang-Mills_heat_equation_with_projection}
  \frac{\partial a}{\partial t} + \Pi_{A_\infty}d_{A_\infty+a}^*F_{A_\infty+a} + d_{A_\infty}d_{A_\infty}^*a = 0,
  \\
\label{eq:Yang-Mills_heat_equation_initial_condition}  
  a(0) = a_0,
\end{gather}
for a solution $a := A - A_\infty$ on $[0,T)$, for some $T>\infty$, where
\begin{multline}
\label{eq:Yang-Mills_flow_solution_slice}
  a \in C\left([0,T);\Ker_{A_\infty}^*\cap\, W_{A_1}^{1,p}(X;T^*X\otimes\ad P)\right)
  \\
  \cap  C^1\left((0,T);\Ker_{A_\infty}^*\cap\, W_{A_1}^{1,p}(X;T^*X\otimes\ad P)\right)
\end{multline}
and $\Pi_{A_\infty}$ (defined in the forthcoming Equation \eqref{eq:L2-orthogonal_projection_onto_slice}) is the $L^2$-orthogonal projection from $ W_{A_1}^{1,p}(X;T^*X\otimes\ad P)$ on the Coulomb-gauge slice $\Ker_{A_\infty}^*\cap\, W_{A_1}^{1,p}(X;T^*X\otimes\ad P)$. Although the term $d_{A_\infty}d_{A_\infty}^*a$ is identically zero in \eqref{eq:Yang-Mills_heat_equation_with_projection}, we include it in order to make the parabolic nature of this equation transparent. Note that we explicitly include the projection operator $\Pi_{A_\infty}$ in \eqref{eq:Yang-Mills_heat_equation_with_projection}: If $a$ in \eqref{eq:Yang-Mills_flow_solution_slice} is a solution to a quasilinear second-order \emph{parabolic} equation,
\begin{equation}
\label{eq:Yang-Mills_heat_equation_without_projection}
  \frac{\partial a}{\partial t} + d_{A_\infty+a}^*F_{A_\infty+a} + d_{A_\infty}d_{A_\infty}^*a = 0, 
\end{equation}
then, because $\Pi_{A_\infty}\partial a/\partial t = \partial \Pi_{A_\infty}a/\partial t = \partial a/\partial t$, we see that $a$ is also a solution to the equation in \eqref{eq:Yang-Mills_heat_equation_with_projection}. However, the converse does not necessarily hold: If $a$ in \eqref{eq:Yang-Mills_flow_solution_slice} is a solution to \eqref{eq:Yang-Mills_heat_equation_with_projection}, then it is not necessarily a solution to \eqref{eq:Yang-Mills_heat_equation_without_projection}.

We may write \eqref{eq:Yang-Mills_heat_equation_with_projection} as a perturbation, 
\begin{equation}
\label{eq:Yang-Mills_heat_equation_as_perturbation_rough_Laplacian_plus_one_heat_equation}
\frac{\partial a}{\partial t} + \left(\Delta_{A_\infty} + 1\right)a = \cF(a), \quad a(0) = a_0,
\end{equation}
of the linear second-order parabolic equation defined by the \emph{augmented Hodge Laplacian} $\Delta_{A_\infty} + 1$, where we define the \emph{Yang--Mills nonlinearity} by
\begin{multline}
\label{eq:Yang-Mills_heat_equation_nonlinearity}
-\cF(a) := \Pi_{A_\infty}d_{A_\infty+a}^*F_{A_\infty+a} + d_{A_\infty}d_{A_\infty}^*a - \left(\Delta_{A_\infty} + 1\right)a,
\\
\text{for all } a \in W_{A_1}^{1,p}(X;T^*X\otimes\ad P).
\end{multline}
We recall from \eqref{eq:Donaldson_Kronheimer_2-1-14} that the curvatures $F_A$ and $F_{A_\infty + a}$ are related by
\begin{equation}
\label{eq:FAinfty+a_expression}
F_{A_\infty + a} = F_{A_\infty} + d_{A_\infty}a + \frac{1}{2} [a, a] \in W_{A_1}^{1,p}(X;\wedge^2(T^*X)\otimes\ad P).
\end{equation}
By \cite[Equation (6.2)]{Warner}, for any integer $1\leq l\leq d$ and $C^\infty$ connection $A$ on $P$ one has
\begin{equation}
\label{eq:Warner_6-2}
d_A^* = (-1)^{-d(l+1)+1}*d_A* \quad\text{on } \Omega^l(X;\ad P),
\end{equation}
where $*:\Omega^l(X) \to \Omega^{d-l}(X)$ is the Hodge star operator on $l$-forms. Hence, the expression \eqref{eq:Yang-Mills_heat_equation_nonlinearity} for the Yang--Mills nonlinearity becomes
\begin{align*}
  -\cF(a) &= \Pi_{A_\infty}\left(d_{A_\infty + a}^*F_{A_\infty + a} - d_{A_\infty}^*d_{A_\infty}a\right) - a
  \\
  &= \Pi_{A_\infty}\left(d_{A_\infty + a}^*\left(F_{A_\infty} + d_{A_\infty}a + \frac{1}{2} [a, a]\right) - d_{A_\infty}^*d_{A_\infty}a\right) - a,
\end{align*}
and thus
\begin{multline}
\label{eq:Yang-Mills_heat_equation_nonlinearity_expanded}
-\cF(a) = d_{A_\infty}^*F_{A_\infty} + \frac{1}{2} \Pi_{A_\infty}d_{A_\infty}^*[a, a] - a
  \\
+ (-1)^{-d(l+1)+1}\Pi_{A_\infty}\left(*[a,*F_{A_\infty}] + *[a,*d_{A_\infty}a] + \frac{1}{2}*[a,*[a,a]]\right).
\end{multline}
We recall from \cite[p. 235]{DK} or \cite[p. 577]{ParkerGauge} that $d_A^*d_A^*F_A = 0$ for any $C^\infty$ connection $A$ on $P$ and so, because $d_{A_\infty}^*F_{A_\infty} \in \Ker d_{A_\infty}^*$, we have $\Pi_{A_\infty}d_{A_\infty}^*F_{A_\infty} = d_{A_\infty}^*F_{A_\infty}$ in \eqref{eq:Yang-Mills_heat_equation_nonlinearity_expanded}. Of course, if $A_\infty$ is Yang--Mills, then $d_{A_\infty}^*F_{A_\infty}=0$. Rather than keep track of signs or factors of $1/2$, it will suffice for our application to use instead of \eqref{eq:Yang-Mills_heat_equation_nonlinearity_expanded} the following schematic expression for the nonlinearity,
\begin{multline}
\label{eq:Yang-Mills_heat_equation_nonlinearity_relative_rough_Laplacian_plus_one}
-\cF(a) = d_{A_\infty}^*F_{A_\infty} - a + \Pi_{A_\infty}\left(a\times F_{A_\infty} + a\times\nabla_{A_\infty}a + a\times a\times a\right),
\\
\text{for all } a \in W_{A_1}^{1,p}(X;T^*X\otimes\ad P),
\end{multline}
where we use the symbol $\times$ to denote universal pointwise bilinear expressions depending at most on the Lie group $G$ and the Riemannian metric $g$.

The Yang--Mills gradient flow \eqref{eq:Yang-Mills_gradient_flow} is defined for any
\begin{equation}
\label{eq:Yang-Mills_flow_solution}
a \in C\left([0,T); W_{A_1}^{1,p}(X;T^*X\otimes\ad P)\right)
\cap C^1\left((0,T); W_{A_1}^{1,p}(X;T^*X\otimes\ad P)\right)
\end{equation}
and not just $a$ as in \eqref{eq:Yang-Mills_flow_solution_slice}. When we restrict the Yang--Mills energy function \eqref{eq:Yang-Mills_energy_function},
\[
  \YM:A_\infty+W_{A_1}^{1,p}(X;T^*X\otimes\ad P) \ni A \mapsto \frac{1}{2}\int_X |F_A|^2\,d\vol_g \in \RR,
\]
to a slice,
\begin{equation}
\label{eq:Yang-Mills_energy_function_slice}
  \widehat\YM:A_\infty+\Ker_{A_\infty}^*\cap\, W_{A_1}^{1,p}(X;T^*X\otimes\ad P) \ni A \mapsto \frac{1}{2}\int_X |F_A|^2\,d\vol_g \in \RR,
\end{equation}
then for any $W^{1,p}$ connection $A$, we have the corresponding gradient, 
\[
  \widehat{\YM'}(A) \in \left(\Ker_{A_\infty}^*\cap\, W_{A_1}^{1,p}(X;T^*X\otimes\ad P)\right)^* \cong \Ker_{A_\infty}^*\cap\, W_{A_1}^{-1,p}(X;T^*X\otimes\ad P).
\]
The expression
\begin{equation}
\label{eq:Yang-Mills_energy_function_gradient_whole_space}
\YM'(A) = d_A^*F_A
\end{equation}
for $\YM'(A) \in (W_{A_1}^{1,p}(X;T^*X\otimes\ad P))^*$ yields
\begin{multline*}
\widehat{\YM'}(A)a = (\YM'(A),a)_{L^2(X)} = (d_A^*F_A,a)_{L^2(X)} = (\Pi_{A_\infty}d_A^*F_A,a)_{L^2(X)},
\\
\text{for all } a \in \Ker_{A_\infty}^*\cap\, W_{A_1}^{1,p}(X;T^*X\otimes\ad P), 
\end{multline*}
and so
\begin{equation}
\label{eq:Yang-Mills_energy_function_gradient_slice}
\widehat{\YM'}(A) = \Pi_{A_\infty}d_A^*F_A \in \Ker_{A_\infty}^*\cap\, W_{A_1}^{-1,p}(X;T^*X\otimes\ad P).
\end{equation}
Hence, when we restrict Yang--Mills gradient flow \eqref{eq:Yang-Mills_gradient_flow} from $W_{A_1}^{1,p}(X;T^*X\otimes\ad P)$ to
\[
  \Ker_{A_\infty}^*\cap\, W_{A_1}^{1,p}(X;T^*X\otimes\ad P),
\]
we obtain the equation for \emph{Yang--Mills gradient flow on a Coulomb-gauge slice},
\begin{gather}
\label{eq:Yang-Mills_gradient_flow_slice}
  \frac{\partial a}{\partial t} + \Pi_{A_\infty}d_{A_\infty+a}^*F_{A_\infty+a} = 0,
  \\
\label{eq:Yang-Mills_gradient_flow_slice_initial_condition}   
  a(0) = a_0,
\end{gather}
for a solution $a$ as in \eqref{eq:Yang-Mills_flow_solution_slice}. Equation \eqref{eq:Yang-Mills_gradient_flow_slice} coincides with the nonlinear evolution equation \eqref{eq:Yang-Mills_heat_equation_with_projection} since the term $d_{A_\infty}d_{A_\infty}^*a$ is identically zero when $a \in \Ker_{A_\infty}^*\cap\, W_{A_1}^{1,p}(X;T^*X\otimes\ad P)$.

The Yang--Mills energy function \eqref{eq:Yang-Mills_energy_function} is gauge-invariant and thus we expect the critical points of $\YM:\sB^{1,p}(P)\to\RR$ (see Definition \ref{defn:Definition_critical_point_quotient_space}) to coincide with gauge-equivalence classes of critical points of $\YM:\sA^{1,p}(P)\to\RR$. We formalize this equivalence in the following

\begin{lem}[Critical points of $\YM$ and $\widehat\YM$]
\label{lem:Critical_point_Yang-Mills_energy_function_slice}
Continue the notation of Definition \ref{defn:Yang-Mills_energy_function} and assume further that $p>d/2$. Let $A_\infty$ be a smooth connection on $P$ and $A$ be a $W^{1,p}$ connection on $P$ that is in Coulomb gauge with respect to $A_\infty$. Then $A$ is a critical point (respectively, local minimum) of $\YM:\sA^{1,p}(P)\to\RR$ if and only if it is a critical point (respectively, local minimum) of $\widehat\YM: A_\infty+d_{A_\infty}^*\cap W_{A_1}^{1,p}(X;T^*X\otimes\ad P)\to\RR$.
\end{lem}

\begin{proof}
Let $\zeta=\zeta(A_1,A_\infty,g,G,p)\in(0,1]$ be as in Corollary \ref{cor:Freed_Uhlenbeck_3-2_W1q}. Suppose that $A(t)$ for $t\in(-\zeta,\zeta)$ and small $\zeta\in(0,1]$ is an embedded smooth curve in $\sA^{1,p}(P)$ such that $A(0)=A$. Corollary \ref{cor:Freed_Uhlenbeck_3-2_W1q} yields embedded smooth curves $u(t)\in\Aut^{2,p}(P)$ and $a(t)\in \Ker d_{A_\infty}^*\cap W_{A_1}^{1,p}(X;T^*X\otimes\ad P)$ for $t\in(-\zeta,\zeta)$ such that $u(0)=\id_P$ and $a(0)=A-A_\infty$ and
\[
  A(t) = u(t)(A_\infty + a(t)), \quad\text{for all } t \in (-\zeta,\zeta).
\]
By gauge invariance, we obtain
\[
  \YM(A(t)) = \YM(u(t)(A_\infty + a(t))) = \YM(A_\infty + a(t)) = \widehat\YM(A_\infty + a(t)), \quad\text{for all } t \in (-\zeta,\zeta),
\]
and thus
\[
  \left.\frac{d}{dt}\YM(A(t))\right|_{t=0} = \left.\frac{d}{dt}\widehat\YM(A_\infty + a(t))\right|_{t=0}.
\]
Hence, $A$ is a critical point of $\YM$ if and only if it is a critical point of $\widehat\YM$. Furthermore, we see that $A$ is a local minimum for $\YM$ if and only if it is a local minimum for $\widehat\YM$.
\end{proof}

\subsection{Estimate for the Yang--Mills heat equation nonlinearity}
\label{subsec:Estimate_Yang-Mills_heat_equation_nonlinearity}
Our goal in this subsection is to verify that the Yang--Mills heat equation nonlinearity $\cF$, either the exact expression \eqref{eq:Yang-Mills_heat_equation_nonlinearity_expanded} or schematic expression \eqref{eq:Yang-Mills_heat_equation_nonlinearity_relative_rough_Laplacian_plus_one}, obeys the hypotheses of Corollaries \ref{cor:Higher-order_spatial_regularity_strong_solution_nonlinear_evolution_equation_Banach_space} and \ref{cor:First-order_temporal_regularity_strong_solution_nonlinear_evolution_equation_Banach_space}, for suitable choices of Banach spaces. We begin by adapting our proof of
\cite[Lemma 4.1.1]{Feehan_Maridakis_Lojasiewicz-Simon_coupled_Yang-Mills} to give

\begin{lem}[$W^{-s,p}$ estimate for components of the Yang--Mills heat equation nonlinearity]
\label{lem:W-sp_estimate_for_Yang-Mills_nonlinearity_components}
Let $G$ be a compact Lie group and $A_1, A_\infty$ be $C^\infty$ connections on a smooth principal $G$-bundle $P$ over a closed, smooth Riemannian manifold $(X,g)$ of dimension $d\geq 2$. For $d/2 < p < \infty$, let $s \in [0,1)$ and $t \in [0, -s+2)$ be Sobolev exponents that obey one of the following additional conditions for $d\geq 3$: 
\begin{enumerate}
\item If $p < d$, then $(d/p+1-s)/2 \leq t < d/p$, or
\item If $p > d$, then $t \geq -s+1$ and $d/p < t < 1+d/p$, or
\item If $p = d$, then $s>0$ and $t\geq 1$. 
\end{enumerate}
For $d=2$ and $p<2$, require in addition that $p>4/3$ and $2/p-1 \leq s<2-2/p$.
Then there is a constant $z = z(A_1,g,G,p,s,t) \in [1,\infty)$ such that the following bounds hold for all $a \in W_{A_1}^{1,p}(T^*X\otimes\ad P)$: 
\begin{subequations}
\label{eq:W-sp_estimate_for_Yang-Mills_nonlinearity_components}  
\begin{align}
\label{eq:W-sp_estimate_for_a_times_nabla_a}  
\|a\times \nabla_{A_\infty}a\|_{W_{A_1}^{-s,p}(X)} 
&\leq
z\|A_\infty-A_1\|_{L^\infty(X)}\|a\|_{W_{A_1}^{t,p}(X)}^2,
\\
\label{eq:W-sp_estimate_for_a_times_a_times_a}  
\|a\times a\times a\|_{W_{A_1}^{-s,p}(X)} 
&\leq
z\|a\|_{W_{A_1}^{t,p}(X)}^3,
\\
\label{eq:W-sp_estimate_for_FAinfty_times_a}  
\|F_{A_\infty}\times a\|_{W_{A_1}^{-s,p}(X)} 
&\leq
z\|F_{A_\infty}\|_{L^\infty(X)}\|a\|_{L^p(X)},                                                  
\end{align}
\end{subequations}
\end{lem}

\begin{rmk}[Feasability of the choice $t=1$ in Lemma \ref{lem:W-sp_estimate_for_Yang-Mills_nonlinearity_components}]
\label{rmk:W-sp_estimate_for_Yang-Mills_nonlinearity_t_eq_1}  
Because $p>d/2$, then $(d/p+1-s)/2 = (d/p+1)/2 < 3/2$ when $s=0$ and $(d/p+1-s)/2 \searrow d/(2p) < 1$ as $s \nearrow 1$, so in the case $d\geq 3$ and $p<d$ we can choose $t=1$. For the case $d=2$ and $4/3<p<2$, see the discussion in Step \ref{step:W-sp_estimate_for_a_times_nabla_a} of the proof of Lemma \ref{lem:W-sp_estimate_for_Yang-Mills_nonlinearity_components} prior to Case \ref{case:p_lessthan_d_a_times_nabla_a} for verification that the choice $t=1$ is feasible. Feasability of the choice $t=1$ is immediate in the cases $p>d$ and $p=d$.
\end{rmk}

\begin{proof}[Proof of Lemma \ref{lem:W-sp_estimate_for_Yang-Mills_nonlinearity_components}]
We begin with the
  
\begin{step}[$W^{-s,p}$ estimate for $a\times \nabla_{A_\infty}a$]
\label{step:W-sp_estimate_for_a_times_nabla_a}
Let $p'$ be the dual H\"older exponent obeying $1/p+1/p'=1$, so $p'=p/(p-1) \in (1,\infty)$. For $r\in(1,\infty)$ to be determined, let $r'$ be the dual H\"older exponent obeying $1/r+1/r'=1$, so $r'=r/(r-1) \in (1,\infty)$. Since\footnote{We use the fiber metrics to replace $(T^*X\otimes\ad P))^* \cong TX\otimes(\ad P)^*$ by $T^*X\otimes\ad P$ on the right-hand side.}
\[
  W_{A_1}^{-s,p}(T^*X\otimes\ad P)) \cong \left(W_{A_1}^{s,p'}(TX\otimes(\ad P)^*)\right)^* \cong \left(W_{A_1}^{s,p'}(T^*X\otimes \ad P)\right)^*
\]
by \cite[Section 40]{Feehan_yang_mills_gradient_flow_v4} (and references cited therein), for example, we have
\begin{align*}
\|a\times \nabla_{A_\infty}a\|_{W_{A_1}^{-s,p}(X)} 
&=
\|a\times \nabla_{A_\infty}a\|_{(W_{A_1}^{s,p'}(X))^*} 
\\
&=     
\sup_{\|\beta\|_{W_{A_\infty}^{s,p'}(X)}\leq 1} (a\times \nabla_{A_\infty}a,\beta)_{L^2(X)} 
\\
&\leq 
\|a\times \nabla_{A_\infty}a\|_{L^{r'}(X)} \sup_{\|\beta\|_{W_{A_1}^{s,p'}(X)}\leq 1} \|\beta\|_{L^r(X)},
\end{align*}
where $\beta \in W_{A_1}^{s,p'}(TX\otimes\ad P)$. Note that $\nabla_{A_\infty}a = \nabla_{A_1}a + [A_\infty-A_1, a]$ by \eqref{eq:Bleecker_corollary_3-1-6_local} and \eqref{eq:Covariant_is_exterior_covariant_derivative_on_sections}. Estimate \eqref{eq:W-sp_estimate_for_a_times_nabla_a} will follow by showing that we can always choose $r \in (1,\infty)$ so that
\begin{subequations} 
\label{eq:Lrprime_bound_a_times_nabla_a}
\begin{align}
\label{eq:Lrprime_bound_a_times_nabla_a1}
\|a\times \nabla_{A_\infty}a\|_{L^{r'}(X)} 
&\leq
z\|A_\infty-A_1\|_{L^\infty(X)}\|a\|_{W^{t,p}_{A_1}(X)}^2,
\\
\label{eq:Lrprime_bound_a_times_nabla_a2}
\|\beta\|_{L^r(X)} 
&\leq 
z\|\beta \|_{W^{s,p'}_{A_1}(X)}.
\end{align}
\end{subequations}
For $p > d/2$, then $d/p < 2$ and thus $d-d/p > d-2$ and so $d-d/p > 1$ for all $d\geq 3$. Hence, for any $p > d/2$ and any $s<1$, we have $s < d-d/p$. But $s < d-d/p \iff d/p < d-s \iff p>d/(d-s) \iff p(d-s) > d \iff sp<d(p-1) \iff sp'<d$. We can therefore choose
\[
  r := (p')_s^* = \frac{dp'}{d-sp'} = \frac{dp/(p-1)}{d-sp/(p-1)} = \frac{dp}{dp-d-sp} \in (1,\infty),
\]
and so
\[
  r' = \frac{r}{r-1} = \frac{dp/(dp-d-sp)}{dp/(dp-d-sp) - 1} = \frac{dp}{dp - (dp-d-sp)} = \frac{dp}{d+sp}  \in (1,\infty).
\]
By \cite[Theorem 4.12]{AdamsFournier}, we have a continuous Sobolev embedding $W^{s,p'}(X) \subset L^r(X)$ and this yields the estimate \eqref{eq:Lrprime_bound_a_times_nabla_a2}.

For $d=2$, we have $sp'<2 \iff s<2-2/p$. Hence, for $p\geq 2$ this imposes no additional restriction on $s\in[0,1)$ but for $p \in (1,2)$, we require that $s<2-2/p$. Under these conditions, we again obtain the estimate \eqref{eq:Lrprime_bound_a_times_nabla_a2}. 

To verify the estimate \eqref{eq:Lrprime_bound_a_times_nabla_a1}, we separately consider the cases $p<d$, $p>d$, and $p=d$. When $d=2$, we shall want the interval for $t$ in the forthcoming Case \ref{case:p_lessthan_d_a_times_nabla_a} (where $p<d$) to include $t=1$ when $d=2$. With this in mind, note that $(d/p+1-s)/2 = (2/p+1-s)/2 \leq 1 \iff 2/p-s \leq 1 \iff s \geq 2/p-1$, so the choice $t=1$ is feasible when $d=2$ and $p<2$ provided $s$ also obeys $s \geq 2/p-1$. The resulting interval for $s$,
\[
  2/p-1 \leq s<2-2/p,
\]
is non-empty if and only if  $4/p < 3$ and that explains the restriction $p>4/3$ when $d=2$ and $p<2$.

\setcounter{case}{0}
\begin{case}[If $p<d$, then $(d/p+1-s)/2 \leq t < d/p$]
\label{case:p_lessthan_d_a_times_nabla_a}
Provided
\[
  t<d/p,
\]
we have $p_t^* := dp/(d-tp) \in (1,\infty)$ and $p_{t-1}^* := dp/(d-(t-1)p) \in (1,\infty)$. By \cite[Theorem 4.12]{AdamsFournier}, we have continuous Sobolev embeddings $W^{t,p}(X) \subset L^{p_t^*}(X)$ and $W^{t-1,p}(X) \subset L^{p_{t-1}^*}(X)$. We have a continuous Sobolev multiplication map $L^{p_t^*}(X)\times L^{p_{t-1}^*}(X) \to L^{r'}(X)$ if $1/p_t^* + 1/p_{t-1}^* \leq 1/r'$, that is,
\[
  \frac{d-tp}{dp} + \frac{d-tp+p}{dp} \leq \frac{d+sp}{dp},
\]
or equivalently,
\[
  \frac{1}{p} - \frac{t}{d} + \frac{1}{p} - \frac{t}{d} + \frac{1}{d} \leq \frac{1}{p} + \frac{s}{d},
\]
or equivalently,
\[
  \frac{1}{p} + \frac{1-2t}{d} \leq \frac{s}{d},
\]
or equivalently,
\[
  \frac{d}{p} + 1 - s \leq 2t,
\]
or equivalently,
\[
  t \geq (d/p+1-s)/2.
\]
We also need $t<-s+2$, so for this choice of $t$ to be possible, we must have $(d/p+1-s)/2 < -s+2$, that is, $d/p+1-s < -2s+4$ or $s < 3-d/p$. But $p > d/2$ by hypothesis, so $3-d/p > 1$ and since $s < 1$ by hypothesis, then $s < 1 < 3-d/p$, as required. Therefore, 
\begin{align*}
  \|a\times \nabla_{A_\infty}a\|_{L^{r'}(X)} &\leq z\|a\|_{L^{p_t^*}(X)}\|\nabla_{A_\infty}a\|_{L^{p_{t-1}^*}(X)}
  \\
  &\leq z\|A_\infty-A_1\|_{L^\infty(X)}\|a\|_{L^{p_t^*}(X)}\|\nabla_{A_1}a\|_{L^{p_{t-1}^*}(X)}
  \\
  &\leq z\|A_\infty-A_1\|_{L^\infty(X)}\|a\|_{W_{A_1}^{t,p}(X)}\|\nabla_{A_1}a\|_{W_{A_1}^{t-1,p}(X)}
  \\
  &\leq z\|A_\infty-A_1\|_{L^\infty(X)}\|a\|_{W_{A_1}^{t,p}(X)}^2.
\end{align*}
This yields the estimate \eqref{eq:Lrprime_bound_a_times_nabla_a1} for this case, provided
 \[
   (d/p+1-s)/2 \leq t < d/p.
 \]
Note that this constraint on $t$ is feasible if $1-s < d/p$, or equivalently, $s > 1-d/p$. But $d/p>1$ by assumption for this case and so the condition $s > 1-d/p$ imposes no new restriction since $s \geq 0$ by hypothesis. 
\end{case}

\begin{case}[If $p>d$, then $t \geq -s+1$ and $d/p < t < 1+d/p$]
\label{case:p_greater_than_d_a_times_nabla_a}  
If we choose $t$ obeying $(t-1)p<d$ and $tp>d$, that is,
\[
d/p < t < 1+d/p,
\]  
then we have continuous Sobolev embeddings $W^{t-1,p}(X) \subset L^{p_{t-1}^*}(X)$ and $W^{t,p}(X) \subset L^\infty(X)$ by \cite[Theorem 4.12]{AdamsFournier}. This choice of $t$ obeying $tp>d$ is possible if $(-s+2)p > d$, that is, $s<2-d/p$, and because $p>d$ by assumption for this case, then $2-d/p > 1$ and so any $s \in [0,1)$ is possible.  We have a continuous Sobolev multiplication map $L^\infty(X)\times L^{p_{t-1}^*}(X) \to L^{r'}(X)$ if $1/p_{t-1}^* \leq 1/r'$, that is,
\[
  \frac{d-tp+p}{dp} \leq \frac{d+sp}{dp},
\]
or equivalently,
\[
  \frac{1}{p} - \frac{t}{d} + \frac{1}{d} \leq \frac{1}{p} + \frac{s}{d},
\]
or equivalently,
\[
  \frac{1-t}{d} \leq \frac{s}{d},
\]
or equivalently, $t \geq -s+1$. Because $-s+1\leq 1$ for $s\geq 0$, the constraint $t \geq -s+1$ is feasible for $t$ also obeying $t<-s+2$. This yields the estimate \eqref{eq:Lrprime_bound_a_times_nabla_a1} for this case.
\end{case}

\begin{case}[If $p=d$, then $s>0$ and $t\geq 1$]
\label{case:p_equals_d_a_times_nabla_a}    
If we choose $t$ obeying $(t-1)p<d$ and $tp\geq d$, that is,
\[
1 \leq t < 2,
\] 
then we have continuous Sobolev embeddings $W^{t-1,p}(X) \subset L^{p_{t-1}^*}(X)$ and $W^{t,p}(X) \subset L^v(X)$ for any $v\in[1,\infty)$ by \cite[Theorem 4.12]{AdamsFournier}.  A choice of $t\geq 1$ that also obeys $t<-s+2$ is feasible if $-s+2 > 1$ and that holds for any $s \in [0,1)$. We have a continuous Sobolev multiplication map $L^v(X)\times L^{p_{t-1}^*}(X) \to L^{r'}(X)$ for large enough $v$ if $1/p_{t-1}^* < 1/r'$. By our analysis in Case \ref{case:p_greater_than_d_a_times_nabla_a}, this holds when $t > -s+1$ and so a choice of $t\geq 1$ that also obeys $t<-s+2$ is feasible if $s>0$. This yields the estimate \eqref{eq:Lrprime_bound_a_times_nabla_a1} for this case.
\end{case}

This completes Step \ref{step:W-sp_estimate_for_a_times_nabla_a} and completes the proof of estimate \eqref{eq:W-sp_estimate_for_a_times_nabla_a}.
\end{step}

\begin{step}[$W^{-s,p}$ estimate for $a\times a \times a$]
\label{step:W-sp_estimate_for_a_times_a_times_a}
We claim that we can always find $r \in (1,\infty)$, with $r' = r/(r-1) \in (1,\infty)$, and $u \in (1,\infty)$ such that $1/(3u) \leq 1/r'$ and we have a continuous Sobolev multiplication map $L^u(X)\times L^u(X)\times L^u(X) \to L^{r'}(X)$, and such that we have a continuous Sobolev  embedding $W^{t,p}(X) \subset L^u(X)$. Assuming that claim, we see that
\begin{align*}
\|a\times a \times a\|_{W_{A_1}^{-s,p}(X)} 
&=
\|a\times a \times a\|_{(W_{A_1}^{s,p'}(X))^*} 
&\\    
&= \sup_{\|\beta\|_{W_{A_1}^{s,p'}(X)}\leq 1} (a\times a \times a,\beta)_{L^2(X)} 
\\
&\leq \|a\times a \times a\|_{L^{r'}(X)} \sup_{\|\beta\|_{W_{A_1}^{s,p'}(X)}\leq 1} \|\beta\|_{L^r(X)}
\\
&\leq z\|a\|_{L^u(X)}^3 \sup_{\|\beta\|_{W_{A_1}^{s,p'}(X)}\leq 1} \|\beta\|_{W_{A_1}^{s,p'}(X)}
\\
&\leq z\|a\|_{W_{A_1}^{t,p}(X)}^3, 
\end{align*}
where $\beta \in W_{A_1}^{s,p'}(TX\otimes\ad P)$. Estimate \eqref{eq:W-sp_estimate_for_a_times_a_times_a} will follow by choosing $r \in (1,\infty)$ exactly as in Step \ref{step:W-sp_estimate_for_a_times_nabla_a} so that \eqref{eq:Lrprime_bound_a_times_nabla_a2} holds and then verifying that
\begin{subequations}
\label{eq:Lrprime_bound_a_times_a_times_a}  
\begin{align}
\label{eq:Lrprime_bound_a_times_a_times_a_Lu_a_cubed}    
\|a\times a \times a\|_{L^{r'}(X)}
&\leq z\|a\|_{L^u(X)}^3  
\\  
\label{eq:Lu_bound_a}
\|a\|_{L^u(X)} 
&\leq
z\|a\|_{W^{t,p}_{A_1}(X)}.
\end{align}  
\end{subequations}
To verify the estimates \eqref{eq:Lrprime_bound_a_times_a_times_a}, we consider separately the cases $p<d$, $p>d$, and $p=d$.

\setcounter{case}{0}
\begin{case}[If $p<d$, then $t<d/p$]
\label{case:a_times_a_times_a_p_lessthan_d}
For $t<d/p$, we choose $u=p_t^*:=dp/(d-tp) \in (1,\infty)$.  By \cite[Theorem 4.12]{AdamsFournier}, we have a continuous Sobolev embedding $W^{t,p}(X) \subset L^u(X)$. Moreover, by expanding the inequality $1/u \leq 3/r'$ we obtain
\[
  \frac{1}{u} = \frac{d-tp}{dp} = \frac{1}{p} - \frac{t}{d} \leq \frac{3}{r'} = \frac{3(d+sp)}{dp} = \frac{3}{p} + \frac{3s}{d},
\]
that is,
\[
\frac{1}{p} - \frac{t}{d} \leq \frac{3}{p} + \frac{3s}{d},
\]
or $2/p + 3s/d + t/d \geq 0$, which trivially holds. This yields the estimates \eqref{eq:Lrprime_bound_a_times_a_times_a} for this case. 
\end{case}

\begin{case}[If $p>d$, then $t> d/p$]
\label{case:a_times_a_times_a_p_greaterthan_d}  
Because $s<1$ and thus $-s+2 > 1$ and because $d/p < 1$ for this case, we can choose $t$ obeying $d/p < t < -s+2$. We can choose $u=\infty$ to give a continuous Sobolev multiplication map $L^\infty(X)\times L^\infty(X)\times L^\infty(X) \to L^{r'}(X)$. Since $tp>d$, then $W^{t,p}(X) \subset L^\infty(X)$ is a continuous Sobolev embedding by \cite[Theorem 4.12]{AdamsFournier}. This yields the estimates \eqref{eq:Lrprime_bound_a_times_a_times_a} for this case. 
\end{case}

\begin{case}[$p=d$ and $t \geq 1$]
\label{case:a_times_a_times_a_p_equals_d}  
Because $s<1$ and thus $-s+2 > 1$ and because $d/p = 1$ for this case, we can choose $t$ obeying $1 \leq t < -s+2$. We can choose $u=3r' \in (1,\infty)$ to give a continuous Sobolev multiplication map $L^u(X)\times L^u(X)\times L^u(X) \to L^{r'}(X)$. Since $tp=d$ if $t=1$ or $tp>d$ if $t>1$, then $W^{t,p}(X) \subset L^uX)$ is a continuous Sobolev embedding by \cite[Theorem 4.12]{AdamsFournier}. This yields the estimates \eqref{eq:Lrprime_bound_a_times_a_times_a} for this case. 
\end{case}

This completes Step \ref{step:W-sp_estimate_for_a_times_a_times_a} and completes the proof of estimate \eqref{eq:W-sp_estimate_for_a_times_a_times_a}.
\end{step}

\begin{step}[$W^{-s,p}$ estimates for $F_{A_\infty}\times a$]
\label{step:W-sp_estimate_for_FAinfty_times_a}
We have 
\begin{align*}
\|F_{A_\infty}\times a\|_{W_{A_1}^{-s,p}(X)} 
&=
\|F_{A_\infty}\times a\|_{(W_{A_1}^{s,p'}(X))^*} 
\\
&=    
\sup_{\|\beta\|_{W_{A_1}^{s,p'}(X)}\leq 1} (F_{A_\infty}\times a, \beta)_{L^2(X)} 
\\
&\leq z\|F_{A_\infty}\|_{L^\infty(X)} \sup_{\|\beta\|_{W_{A_1}^{s,p'}(X)}\leq 1} (|a|, |\beta|)_{L^2(X)} 
\\
&= z\|F_{A_\infty}\|_{L^\infty(X)} \sup_{\|\beta\|_{W_{A_1}^{s,p'}(X)}\leq 1} (|a|, \beta)_{L^2(X)} 
\\
&= z\|F_{A_\infty}\|_{L^\infty(X)} \||a|\|_{W_{A_1}^{-s,p}(X)},
\end{align*}
which gives \eqref{eq:W-sp_estimate_for_FAinfty_times_a}, since $L^p(X) \subset W^{-s,p}(X)$ is a continuous Sobolev embedding for any $s\geq 0$. This completes Step \ref{step:W-sp_estimate_for_FAinfty_times_a}.
\end{step}

This completes the proof of Lemma \ref{lem:W-sp_estimate_for_Yang-Mills_nonlinearity_components}.
\end{proof}

\begin{cor}[$W^{-s,p}$ estimate for the Yang--Mills nonlinearity]
\label{cor:W-sp_estimate_for_Yang-Mills_nonlinearity}
Assume the hypotheses of Lemma \ref{lem:W-sp_estimate_for_Yang-Mills_nonlinearity_components} and its conditions on $p$, $s$, and $t$. Then there is a constant $z = z(A_1,g,G,p,s,t) \in [1,\infty)$ such that the following bounds hold for all $a, b \in W_{A_1}^{t,p}(X;\ad P)$:
\begin{align}
\label{eq:W-sp_estimate_for_Yang-Mills_nonlinearity}  
\|\cF(a)\|_{W_{A_1}^{-s,p}(X)} 
&\leq
z\left(\|d_{A_\infty}^*F_{A_\infty}\|_{W_{A_1}^{-s,p}(X)} + \|F_{A_\infty}\|_{L^\infty(X)}\|a\|_{L^p(X)} \right.
\\
\nonumber  
&\qquad + \left. \|A_\infty-A_1\|_{L^\infty(X)}\|a\|_{W_{A_1}^{t,p}(X)}^2 + \|a\|_{W_{A_1}^{t,p}(X)}^3\right).
\\
\label{eq:W-sp_Lipschitz_estimate_for_Yang-Mills_nonlinearity}  
\|\cF(a)-\cF(b)\|_{W_{A_1}^{-s,p}(X)} 
&\leq
z\left(\|F_{A_\infty}\|_{L^\infty(X)} \right.
\\
\nonumber    
&\qquad + \|A_\infty-A_1\|_{L^\infty(X)}\left(\|a\|_{W_{A_1}^{t,p}(X)} + \|b\|_{W_{A_1}^{t,p}(X)}\right)
\\
\nonumber    
&\qquad + \left. \|a\|_{W_{A_1}^{t,p}(X)}^2 + \|b\|_{W_{A_1}^{t,p}(X)}^2\right)\|a-b\|_{W_{A_1}^{t,p}(X)}.
\end{align}
\end{cor}

\begin{proof}
The conclusions follow from Lemma \ref{lem:W-sp_estimate_for_Yang-Mills_nonlinearity_components}, the expression \eqref{eq:Yang-Mills_heat_equation_nonlinearity_relative_rough_Laplacian_plus_one} for the Yang--Mills nonlinearity, and boundedness of the $L^2$-orthogonal projection $\Pi_{A_\infty}$ on $W_{A_1}^{-s,p}(T^*X\otimes\ad P)$, as defined in the forthcoming Equation  \eqref{eq:L2-orthogonal_projection_onto_slice}.
\end{proof}

\subsection{\Apriori estimates for lengths of trajectories for Yang--Mills gradient flow on a Coulomb-gauge slice}
\label{subsec:Rade_Lemma_7-3_generalization}
We recall one of our generalizations of R\r{a}de's key \apriori $L^1$-in-time interior estimate for a solution to Yang--Mills gradient flow \cite[Lemma 7.3]{Rade_1992} from the case where the base manifold $X$ has dimension $d = 2$ or $3$ to the case $d \geq 2$. The forthcoming Proposition \ref{prop:Rade_7-3_arbitrary_dimension_L1_time_W1p_space} follows in part from Corollary \ref{cor:Rade_7-3_abstract_L1_in_time_V2beta_space_interior}.

\begin{prop}[\Apriori $L^1$-in-time-$W^{1,p}$-in-space interior estimate for a solution to Yang--Mills gradient flow on a Coulomb-gauge slice over base manifolds of arbitrary dimension]
\label{prop:Rade_7-3_arbitrary_dimension_L1_time_W1p_space}
(See Feehan \cite[Corollary 26.10]{Feehan_yang_mills_gradient_flow_v4}.)  
Let $G$ be a compact Lie group and $P$ be a smooth principal $G$-bundle over a closed, smooth Riemannian manifold $(X,g)$ of dimension $d \geq 2$. Let $A_1$ and $A_\infty$ be $C^\infty$ connections on $P$, and $p \in (d/2,\infty)$ obey $p > 2$. Then there are positive constants, $C = C(A_1,A_\infty, g, p) \in [1,\infty)$ and $\eps_1 = \eps_1(A_1, A_\infty, g, p) \in (0, 1]$, such that if $A(t)$ is a strong solution to the Yang--Mills gradient flow on a Coulomb-gauge slice \eqref{eq:Yang-Mills_heat_equation_with_projection} (equivalently,  \eqref{eq:Yang-Mills_gradient_flow_slice}) over an interval $(S, T)$ with regularity,
\begin{multline*}
  A - A_\infty \in L^\infty(S, T; \Ker d_{A_\infty}^*\cap\, W_{A_1}^{1,p}(X; T^*X\otimes\ad P))
  \\
  \cap\, W^{1,1}_{\loc}(S, T; \Ker d_{A_\infty}^*\cap\, W_{A_1}^{2,p}(X; T^*X\otimes\ad P)),
\end{multline*}
where $S \in \RR$ and $\delta > 0$ and $T$ obey $S + 2\delta \leq T \leq \infty$, and
\begin{equation}
\label{eq:Linfinity_in_time_W1p_in_space_small_norm_At_minus_A_1_condition_lemma_7-3_corollary}
\|A(t) - A_\infty\|_{W_{A_1}^{1,p}(X)} \leq \eps_1, \quad\text{a.e. } t \in (S, T),
\end{equation}
then there is an integer $n = n(d,p) \geq 1$ such that
\begin{equation}
\label{eq:Rade_apriori_interior_estimate_lemma_7-3_arbitrary_dimension_W1p_L2_corollary}
\int_{S+\delta}^T \|\dot A(t)\|_{W_{A_1}^{1,p}(X)}\,dt
\leq
C\left(1 + \delta^{-n}\right)\int_S^T \|\dot A(t)\|_{L^2(X)}\,dt.
\end{equation}
\end{prop}

\begin{rmk}[Comparison of estimates obtained by application of the Donaldson--DeTurck Trick and restriction to a Coulomb-gauge slice]
\label{rmk:Comparison_Donaldson-DeTurck_trick_and_restriction_Coulomb-gauge_slice}  
We note that \cite[Corollary 26.10]{Feehan_yang_mills_gradient_flow_v4} is derived under the assumption that pure Yang--Mills gradient flow is converted to a Yang--Mills heat flow \eqref{eq:Yang-Mills_heat_equation} by application of the Donaldson--DeTurck Trick \cite[Lemma 20.3]{Feehan_yang_mills_gradient_flow_v4} whereas Proposition \ref{prop:Rade_7-3_arbitrary_dimension_L1_time_W1p_space} is stated here for the nonlinear evolution equation \eqref{eq:Yang-Mills_heat_equation_with_projection} obtained by restricting pure Yang--Mills gradient flow to a Coulomb-gauge slice through $A_\infty$. As we have seen earlier, the presence of the $L^2$-orthogonal projection operator $\Pi_{A_\infty}$ in the Yang--Mills nonlinearity \eqref{eq:Yang-Mills_heat_equation_nonlinearity} does not cause any new complication when estimating this nonlinearity. Proposition \ref{prop:Rade_7-3_arbitrary_dimension_L1_time_W1p_space} is proved in \cite{Feehan_yang_mills_gradient_flow_v4} by applying Corollary \ref{cor:Rade_7-3_abstract_L1_in_time_V2beta_space_interior} and Remark \ref{rmk:Abstract_apriori_interior_estimate_trajectory} .
\end{rmk}

\subsection{Local well-posedness, a priori estimates, and minimal lifetimes for solutions to Yang--Mills gradient flow on a Coulomb-gauge slice}
\label{subsec:Local_well-posedness_priori_estimates_minimal_lifetimes_Yang-Mills_gradient_flow}
We are now ready to apply the results of Section \ref{sec:Local_well-posedness_nonlinear_evolution_equations_Banach_spaces} on local well-posedness, a priori estimates, and minimal lifetimes for solutions to nonlinear evolution equations in abstract Banach spaces to the case of Yang--Mills gradient flow on a Coulomb-gauge slice.

\begin{thm}[Local well-posedness, a priori estimates, and minimal lifetimes for solutions to Yang--Mills gradient flow on a Coulomb-gauge slice]
\label{thm:Local_well-posedness_priori_estimates_minimal_lifetimes_Yang-Mills_gradient_flow}
Let $G$ be a compact Lie group and $A_1, A_\infty$ be $C^\infty$ connections on a smooth principal $G$-bundle $P$ over a closed, connected, oriented, smooth Riemannian manifold $(X,g)$ of dimension $d\geq 2$, where $A_1$ serves as a reference connection in the definition of Sobolev and H\"older norms and $A_\infty$ serves to define a Coulomb gauge condition, and let $p \in (d/2,\infty)$, and $b \in (0,\infty)$. Then there are constants $s\in (0,1]$ and $C_0=C_0(A_1,A_\infty,g,p,s) \in [1,\infty)$ and $C_1=C_1(A_1,A_\infty,b,g,p,s) \in [1,\infty)$ and $\tau=\tau(A_1,A_\infty,b,g,p,s) \in (0,\infty]$ with the following significance. If $A_0$ is a $W^{1,p}$ connection on $P$ such that $a_0 := A_0-A_\infty \in \Ker d_{A_\infty}^*\cap W_{A_1}^{1,p}(T^*X\otimes\ad P)$ and $\|a_0\|_{W_{A_1}^{1,p}(X)} \leq b$, then the following hold. There is a unique classical solution, $A(t) = A_\infty + a(t)$ for $t\in [0,\tau)$, to the equation \eqref{eq:Yang-Mills_gradient_flow_slice} for Yang--Mills gradient flow on a Coulomb-gauge slice with $a(0) = a_0$ and regularity
\begin{equation}
\label{eq:Yang-Mills_gradient_flow_solution_regularity}
a \in C\left([0,\tau); \Ker d_{A_\infty}^*\cap W_{A_1}^{1,p}(T^*X\otimes\ad P)\right)
\cap C^\infty\left((0,\tau); \Ker d_{A_\infty}^*\cap\Omega^1(X;\ad P)\right)
\end{equation}
that obeys
\begin{equation}
\label{eq:Yang-Mills_gradient_flow_estimate_difference_solution_initial_data}
\sup_{t\in [0,\tau)}\|a(t)|_{W_{A_1}^{1,p}(X)} \leq C_1\tau^{1-\beta}.
\end{equation}
If $\tilde A_0$ is another $W^{1,p}$ connection on $P$ and $\tilde a_0 := \tilde A_0-A_\infty \in \Ker d_{A_\infty}^*\cap W_{A_1}^{1,p}(X; T^*X\otimes\ad P)$ obeys $\|\tilde a_0\|_{W_{A_1}^{1,p}(X)} \leq b$, then the unique classical solution $\tilde a$ to \eqref{eq:Yang-Mills_gradient_flow_slice} with $\tilde a(0)=\tilde a_0$ obeys
\begin{equation}
\label{eq:Yang-Mills_gradient_flow_continuity_respect_initial_data}
\sup_{t\in [0, \tau)}\|\tilde a(t)-a(t)\|_{W_{A_1}^{1,p}(X)} \leq C_0\|\tilde a_0 - a_0\|_{W_{A_1}^{1,p}(X)}.
\end{equation}
\end{thm}

\begin{proof}
We begin by defining the Banach spaces required for the general theory of nonlinear evolution equations presented in Section \ref{sec:Local_well-posedness_nonlinear_evolution_equations_Banach_spaces} and verify Hypothesis \ref{hyp:Sell_You_4_standing_hypothesis_A}, the `Standing Hypothesis A' of Sell and You \cite[p. 141]{Sell_You_2002}. We refer to Feehan \cite[Section 17.2]{Feehan_yang_mills_gradient_flow_v4} for further background and discussion of this verification and to Feehan \cite[Appendix A]{Feehan_yang_mills_gradient_flow_v4} for background on fractional-order Sobolev spaces.

For $p \in (1,\infty)$, we define
\begin{subequations}
\label{eq:Standing_hypothesis_operator_A_rough_Laplacian_plus_one_on_Lp}
\begin{align}
\label{eq:W_is_Lp_X_T*X_adP}
\cW_0 &:= L^p(X; T^*X\otimes\ad P),
\\
\label{eq:Definition_A_positive_sectorial_operator}   
\cA &:= \Delta_{A_\infty}+1,        
\\
\label{eq:V_is_W2p_X_T*X_adP}
\cV_0^2 &:= \sD(\cA_p) = W_{A_1}^{2,p}(X; T^*X\otimes\ad P).
\end{align}
\end{subequations}
Here, the domain $\sD(\cA_p)$ is defined with respect to the range $L^p(X; T^*X\otimes\ad P)$ as the smallest closed extension of the realization $\cA_p$ of the linear second-order elliptic partial differential operator $\cA$ on $L^p(X; T^*X\otimes\ad P)$,
\[
\cA_p: \sD(\cA_p) \subset L^p(X; T^*X\otimes\ad P) \to L^p(X; T^*X\otimes\ad P).
\]
To simplify notation, we shall not distinguish between the partial differential operator $\Delta_{A_\infty}+1$ in \eqref{eq:Definition_A_positive_sectorial_operator} on $C^\infty(X; T^*X\otimes\ad P)$, and its realization on $L^p(X; T^*X\otimes\ad P)$. According to \cite[Theorem 14.54]{Feehan_yang_mills_gradient_flow_v4}, the realization $\cA_p$ is sectorial on $L^p(X; T^*X\otimes\ad P)$ and $-\cA_p$ is the infinitesimal generator of an analytic semigroup on $L^p(X; T^*X\otimes\ad P)$. Our choice of $\cA$ in \eqref{eq:Definition_A_positive_sectorial_operator} also defines a positive realization, $\cA_p$, on $L^p(X; T^*X\otimes\ad P)$ and hence fulfills Hypothesis \ref{hyp:Sell_You_4_standing_hypothesis_A} (the `Standing Hypothesis A') with $\cV_0 \equiv \sD(\cA_p) = W_{A_1}^{2,p}(X; T^*X\otimes\ad P)$, noting that the domain, $\sD(\cA_p)$, of the smallest closed extension of $\cA_p$ is identified by the \apriori estimate
\cite[Equation (14.145)]{Feehan_yang_mills_gradient_flow_v4} in \cite[Theorem 14.60]{Feehan_yang_mills_gradient_flow_v4}.

Following Sell and You \cite[Equation (37.8) and Lemma 37.3]{Sell_You_2002}, the fractional power $\cA^{-\alpha} \in \sL(\cW_0)$ is defined for any $\alpha>0$ and has $\Ker(\cA^{-\alpha}|\cW_0)=\{0\}$. Moreover, for all $\alpha>0$, Sell and You \cite[paragraph following Lemma 37.3]{Sell_You_2002} define
\[
  \cA^\alpha := (\cA^{-\alpha})^{-1},
\]
with domain $\sD(\cA^\alpha|\cW_0) := \Ran (\cA^{-\alpha}|\cW_0)$, and $\cA^0 := \id_{\cW_0}$. Continuing as in Sell and You \cite[p. 95]{Sell_You_2002}, one defines the scale of Banach spaces $\cV_0^{2\alpha}$ for $\alpha>0$ by
\[
  \cV_0^{2\alpha} := \sD(\cA^\alpha)|\cW_0) \quad\text{with}\quad \|v\|_{\cV_0^{2\alpha}} := \|\cA^\alpha v\|_{\cW_0},
\]
with $\cV_0^0 = \cW_0$. Following \cite[Appendix A]{Feehan_yang_mills_gradient_flow_v4}, these Banach spaces are identified as Sobolev spaces:
\[
  \cV_0^{2\alpha} = W_{A_1}^{2\alpha, p}(X; T^*X\otimes\ad P), \quad\text{for all } \alpha > 0.
\]
Although Sell and You also use Banach spaces $\cV_0^{2\alpha}$ when $\alpha<0$ (see \cite[Lemma 37.4, Items (5) and (6)]{Sell_You_2002}), they omit their definition and so we include one by here. We begin by observing that the Banach space $W_{A_1}^{-s, p}(X; T^*X\otimes\ad P)$, for any real $s>0$, may be defined by analogy with their definition in Adams and Fournier \cite[Sections 3.7--14]{AdamsFournier} for integer $s > 0$ as
\[
  W_{A_1}^{-s, p}(X; T^*X\otimes\ad P) :=  \left(W_{A_1}^{s, p'}(X; T^*X\otimes\ad P)\right)^*,
\]
where $p'\in(1,\infty)$ is the dual H\"older exponent for $p$, that is, $1/p+1/p'=1$. Hence, we set
\[
  \cV_0^{-2\alpha} := W_{A_1}^{-2\alpha, p}(X; T^*X\otimes\ad P), \quad\text{for all } \alpha > 0.
\]
In order to apply the results of Section \ref{sec:Local_well-posedness_nonlinear_evolution_equations_Banach_spaces} for nonlinear evolution equations in abstract Banach spaces, we shall choose a more convenient Banach space $\cW$ and scale $\cV^{2\alpha}$ for $\alpha\in\RR$ than the preceding $\cW_0$ and $\cV_0^{2\alpha}$. Specifically, upon choosing $s\in (0,1]$ obeying the hypotheses of Lemma \ref{lem:W-sp_estimate_for_Yang-Mills_nonlinearity_components}, we define
\begin{subequations}
\label{eq:Definition_W_V2_Banach_spaces_A_positive_sectorial_operator} 
\begin{align}
\label{eq:Definition_W_Banach_space}   
  \cW &:= \Ker d_{A_\infty}^*\cap W_{A_1}^{-s,p}(T^*X\otimes\ad P),
  \\
\label{eq:Definition_V2_Banach_space}   
  \cV^2 &:= \sD(\cA|\cW) = \Ker d_{A_\infty}^*\cap W_{A_1}^{-s+2,p}(T^*X\otimes\ad P).
\end{align}
\end{subequations}
By Remark \ref{rmk:W-sp_estimate_for_Yang-Mills_nonlinearity_t_eq_1}, we can choose $s\in (0,1]$ and $\beta\in (0,1)$ such that $\beta(-s+2) = t = 1$ and thus define the fractional powers
\begin{equation}
\label{eq:Definition_V2beta_Banach_space}
\cV^{2\beta} := \Ker d_{A_\infty}^*\cap W_{A_1}^{1,p}(T^*X\otimes\ad P).
\end{equation}
By Corollary \ref{cor:W-sp_estimate_for_Yang-Mills_nonlinearity}, the Yang--Mills nonlinearity $\cF$ in \eqref{eq:Yang-Mills_heat_equation_nonlinearity_relative_rough_Laplacian_plus_one} obeys \eqref{eq:Sell_You_46-7_polynomial_nonlinearity} and \eqref{eq:Sell_You_46-8_polynomial_nonlinearity} with $n=3$ and $\cW$ as in \eqref{eq:Definition_W_Banach_space} and $\cV^2$ as in \eqref{eq:Definition_V2_Banach_space} and $\cV^{2\beta}$ as in \eqref{eq:Definition_V2beta_Banach_space}.

Because $\cF(a)$ is a third-order polynomial in $a \in \cV^{2\beta}$, Corollary \ref{cor:W-sp_estimate_for_Yang-Mills_nonlinearity} yields
\begin{equation}
\label{eq:F_real_analytic}
\cF \in C^\omega(\cV^{2\beta}; \cW), 
\end{equation}
the vector space of analytic maps from $\cV^{2\beta}$ to $\cW$. Consequently, by Theorem \ref{thm:Sell_You_lemma_47-1_polynomial_nonlinearity} and Corollaries \ref{cor:Higher-order_spatial_regularity_strong_solution_nonlinear_evolution_equation_Banach_space} and \ref{cor:First-order_temporal_regularity_strong_solution_nonlinear_evolution_equation_Banach_space} and Remark \ref{rmk:Higher-order_temporal_regularity_strong_solution_nonlinear_evolution_equation_Banach_space}, the Yang--Mills gradient flow equation \eqref{eq:Yang-Mills_gradient_flow_slice}, with initial condition $a(0)=a_0$, has a unique classical solution
\begin{align*}
  a &\in C\left([0,\tau); W_{A_1}^{1,p}(X; T^*X\otimes\ad P)\right)
  \\
  &\qquad \cap C^\infty\left((0,\tau); \Ker d_{A_\infty}^* \cap W_{A_1}^{k,p}(X; T^*X\otimes\ad P)\right), \quad\text{for all } k \geq 1,
\end{align*}
for some $\tau = \tau(A_1,A_\infty,b,g,p,s) \in (0,\infty]$. By the Sobolev Embedding \cite[Theorem 4.12]{AdamsFournier}, this yields the regularity for $a$ asserted in \eqref{eq:Yang-Mills_gradient_flow_solution_regularity}. Theorem \ref{thm:Sell_You_lemma_47-1_polynomial_nonlinearity} also yields the continuity with respect to initial data asserted in \eqref{eq:Yang-Mills_gradient_flow_continuity_respect_initial_data}.

Corollary \ref{cor:Sell_You_lemma_47-1_polynomial_nonlinearity_estimate_difference_solution_initial_data} yields the $C([0,\tau]; W_{A_1}^{1,p}(X; T^*X\otimes\ad P))$ estimate \eqref{eq:Yang-Mills_gradient_flow_estimate_difference_solution_initial_data} for the difference between the solution $a(t)$ for $t\in [0,\tau)$ and its initial data $a_0$. This completes the proof of Theorem \ref{thm:Local_well-posedness_priori_estimates_minimal_lifetimes_Yang-Mills_gradient_flow}.
\end{proof}

\section{Global existence,  convergence, and convergence rate  for Yang--Mills gradient flow on a Coulomb-gauge slice near a local minimum}
\label{sec:Global_existence_convergence_rate_Yang-Mills_gradient_flow_near_local_minimum}
We are now in a position to apply the main results of Sections \ref{sec:Local_well-posedness_nonlinear_evolution_equations_Banach_spaces}, \ref{sec:Global_existence_convergence_rate_Lojasiewicz-Simon_gradient_flow_near_local_minimum}, \ref{sec:Lojasiewicz_distance_inequality_functions_Banach_spaces}, and \ref{sec:Local_well-posedness_Yang-Mills_gradient_flow} to prove our results described in our introductory Section \ref{subsec:Main_results} for Yang--Mills gradient flow on a Coulomb-gauge slice. In Section \ref{subsec:Lojasiewicz-Simon_gradient_inequality_Yang-Mills_energy_function}, we recall one version (Theorem \ref{thm:Lojasiewicz-Simon_W-12_gradient_inequality_Yang-Mills_energy_function}) of our {\L}ojasiewicz gradient inequality for the Yang--Mills energy function from Feehan and Maridakis \cite{Feehan_Maridakis_Lojasiewicz-Simon_coupled_Yang-Mills}, together with a refinement and corollary of its proof (Theorem \ref{thm:Lojasiewicz-Simon_W-12_gradient_inequality_Yang-Mills_energy_function_slice}) that is better suited to our application in this monograph. In Section \ref{subsec:Yang-Mills_gradient_flow_global_existence_and_convergence_started_near_local_minimum}, we prove Theorem \ref{mainthm:Yang-Mills_gradient_flow_global_existence_and_convergence_started_near_local_minimum} on global existence, convergence, and convergence rate for Yang--Mills gradient flow on a Coulomb-gauge slice near a local minimum, in the case of base manifolds of arbitrary dimension greater than or equal to two. In Section \ref{subsec:Yang-Mills_gradient_flow_global_existence_and_convergence_started_small_energy}, we prove an improved version, Corollary \ref{maincor:Yang-Mills_gradient_flow_global_existence_and_convergence_started_small_energy}, of the preceding result in the case of base manifolds of dimension two or three. In Section \ref{subsec:Retraction_open_nbhd_onto_moduli_space_flat_connections}, we conclude with the proof of Corollary \ref{maincor:Retraction_open_nbhd_onto_moduli_space_flat_connections}, on the existence of an `almost' strong deformation retract from a neighborhood in the quotient space of Sobolev connections onto the moduli subspace of flat connections. Finally, in Section \ref{subsec:Lojasiewicz_distance_inequality_Yang--Mills_energy_function}, we prove Corollaries \ref{maincor:Lojasiewicz_distance_inequality_Yang-Mills_energy_function_slice} and \ref{maincor:Lojasiewicz_distance_inequality_Yang-Mills_energy_function}, giving the {\L}ojasiewicz distance inequality for the Yang--Mills energy function.

\subsection{A {\L}ojasiewicz--Simon $W^{-1,2}$ gradient inequality for the Yang--Mills energy function}
\label{subsec:Lojasiewicz-Simon_gradient_inequality_Yang-Mills_energy_function}
By specializing \cite[Theorem 4 and Corollary 7]{Feehan_Maridakis_Lojasiewicz-Simon_coupled_Yang-Mills} from the case of a coupled to the pure Yang--Mills energy function, we have

\begin{thm}[{\L}ojasiewicz--Simon $W^{-1,2}$ gradient inequality for the Yang--Mills energy function]
\label{thm:Lojasiewicz-Simon_W-12_gradient_inequality_Yang-Mills_energy_function}
(See Feehan and Maridakis \cite[Theorem 4 and Corollary 7]{Feehan_Maridakis_Lojasiewicz-Simon_coupled_Yang-Mills}.)
Let $(X,g)$ be a closed, smooth Riemannian manifold of dimension $d\geq 2$, and $G$ be a compact Lie group, and $P$ be a smooth principal $G$-bundle over $X$. Let $A_1$ be a $C^\infty$ reference connection on $P$, and $A_\infty$ a Yang--Mills connection on $P$ of class $W^{1,q}$, with $q \in [2,\infty)$ obeying $q > d/2$. If $p \in [2,\infty)$ obeys $d/2 \leq p \leq q$, then the gradient map,
\[
\sM: A_1+W_{A_1}^{1,p}(X;T^*X\otimes\ad P) \ni A \mapsto d_A^*F_A
\in W_{A_1}^{-1,p}(X;T^*X\otimes\ad P),
\]
is \emph{real analytic} and there are constants $C \in (0, \infty)$, and $\sigma \in (0,1]$, and $\theta \in [1/2,1)$, depending on $A_1$, $A_\infty$, $g$, $G$, $p$, and $q$ with the following significance. If $A$ is a $W^{1,q}$ Sobolev connection on $P$ obeying the \emph{{\L}ojasiewicz--Simon neighborhood} condition,
\begin{equation}
\label{eq:Lojasiewicz-Simon_gradient_inequality_Yang-Mills_neighborhood}
\|A - A_\infty\|_{W^{1,p}_{A_1}(X)} < \sigma,
\end{equation}
then the Yang--Mills energy function \eqref{eq:Yang-Mills_energy_function} obeys the \emph{{\L}ojasiewicz--Simon gradient inequality}
\begin{equation}
\label{eq:Lojasiewicz-Simon_W-12gradient_inequality_Yang-Mills_energy_function}
\|d_A^*F_A\|_{W^{-1,2}_{A_1}(X)}
\geq
C|\YM(A) - \YM(A_\infty)|^\theta.
\end{equation}
\end{thm}

Our proof of Theorem \ref{thm:Lojasiewicz-Simon_W-12_gradient_inequality_Yang-Mills_energy_function} in
\cite{Feehan_Maridakis_Lojasiewicz-Simon_coupled_Yang-Mills} actually yields a stronger result and one that is more useful for our current application. Recall that if $A_\infty$ is a critical point of the Yang--Mills energy function \eqref{eq:Yang-Mills_energy_function} of class $W^{1,q}$ (for suitable $q>d/2$), then --- see, for example, the forthcoming Theorem \ref{thm:Regularity_weakly_Yang-Mills_W1p_connection} --- there is a $W^{2,q}$ gauge transformation $u_\infty \in \Aut^{2,q}(P)$ such that $u_\infty(A_\infty)$ is a $C^\infty$ connection on $P$. 

\begin{thm}[{\L}ojasiewicz--Simon $W^{-1,2}$ gradient inequality for the Yang--Mills energy function on a Coulomb-gauge slice]
\label{thm:Lojasiewicz-Simon_W-12_gradient_inequality_Yang-Mills_energy_function_slice}
(See Feehan and Maridakis \cite[Sections 3.1.6 and 4.2 for Case 1 --- $(A,\Phi)$ in Coulomb gauge relative to $(A_\infty,\Phi_\infty)$]{Feehan_Maridakis_Lojasiewicz-Simon_coupled_Yang-Mills}.)
Assume the hypotheses of Theorem \ref{thm:Lojasiewicz-Simon_W-12_gradient_inequality_Yang-Mills_energy_function} and that $A_\infty$ is $C^\infty$. For any $s \in \RR$ and $p \in (1,\infty)$, let
\begin{equation}
\label{eq:L2-orthogonal_projection_onto_slice}
\Pi_{A_\infty}:W_{A_1}^{s,p}(X;T^*X\otimes \ad P)
\to \Ker d_{A_\infty}^* \cap\, W_{A_1}^{s,p}(X;T^*X\otimes \ad P)
\end{equation}
be $L^2$-orthogonal projection onto the Coulomb-gauge slice through $A_\infty$. Then the gradient map,
\begin{multline*}
\sM: A_\infty + \Ker d_{A_\infty}^* \cap\, W_{A_1}^{1,p}(X;T^*X\otimes\ad P) \ni A \mapsto \Pi_{A_\infty} d_A^*F_A
\\
\in \Ker d_{A_\infty}^* \cap\, W_{A_1}^{-1,p}(X;T^*X\otimes\ad P),
\end{multline*}
is \emph{real analytic} and there are constants $C \in (0, \infty)$, and $\sigma \in (0,1]$, and $\theta \in [1/2,1)$, depending on $A_1$, $A_\infty$, $g$, $G$, $p$, and $q$ with the following significance. If $A$ is a $W^{1,q}$ Sobolev connection on $P$ obeying the \emph{{\L}ojasiewicz--Simon neighborhood} condition \eqref{eq:Lojasiewicz-Simon_gradient_inequality_Yang-Mills_neighborhood} and the Coulomb gauge condition,
\begin{equation}
\label{eq:A_Coulomb_gauge_relative_Ainfty}
d_{A_\infty^*}(A - A_\infty) = 0,
\end{equation}
then the restriction $\widehat\YM$ in \eqref{eq:Yang-Mills_energy_function_slice} of the Yang--Mills energy function \eqref{eq:Yang-Mills_energy_function} to the Coulomb-gauge slice,
\[
\Ker d_{A_\infty}^* \cap\, W_{A_1}^{1,p}(X;T^*X\otimes\ad P),
\]
obeys the \emph{{\L}ojasiewicz--Simon gradient inequality}
\begin{equation}
\label{eq:Lojasiewicz-Simon_W-12gradient_inequality_Yang-Mills_energy_function_slice}
\|\Pi_{A_\infty} d_A^*F_A\|_{W^{-1,2}_{A_1}(X)}
\geq
C|\YM(A) - \YM(A_\infty)|^\theta.
\end{equation}
\end{thm}

When $\YM$ is Morse--Bott at a critical point $A\in\sA^{1,p}(P)$ or $[A]\in\sB^{1,p}(P)$ in the sense of Definitions \ref{defn:Definition_Morse-Bott_affine_space} or \ref{defn:Definition_Morse-Bott_quotient_space}, respectively, then the gradient inequalities in Theorems \ref{thm:Lojasiewicz-Simon_W-12_gradient_inequality_Yang-Mills_energy_function} and \ref{thm:Lojasiewicz-Simon_W-12_gradient_inequality_Yang-Mills_energy_function_slice} can be improved. The improvements follow immediately by replacing our appeal to Feehan and Maridakis \cite[Theorem 3]{Feehan_Maridakis_Lojasiewicz-Simon_Banach} by one to \cite[Theorem 4]{Feehan_Maridakis_Lojasiewicz-Simon_Banach}.

\begin{thm}[Optimal {\L}ojasiewicz--Simon $W^{-1,2}$ gradient inequality for the Yang--Mills energy function]
\label{thm:Optimal_Lojasiewicz-Simon_W-12_gradient_inequality_Yang-Mills_energy_function}
Assume the hypotheses of Theorem \ref{thm:Lojasiewicz-Simon_W-12_gradient_inequality_Yang-Mills_energy_function}. If $\YM$ is Morse--Bott at $A_\infty$ in the sense of Definition \ref{defn:Definition_Morse-Bott_affine_space}, then \eqref{eq:Lojasiewicz-Simon_W-12gradient_inequality_Yang-Mills_energy_function} holds with $\theta=1/2$.
\end{thm}  

\begin{thm}[Optimal {\L}ojasiewicz--Simon $W^{-1,2}$ gradient inequality for the Yang--Mills energy function on a Coulomb-gauge slice]
\label{thm:Optimal_Lojasiewicz-Simon_W-12_gradient_inequality_Yang-Mills_energy_function_slice}
Assume the hypotheses of Theorem \ref{thm:Lojasiewicz-Simon_W-12_gradient_inequality_Yang-Mills_energy_function_slice}. If $\widehat\YM$ is Morse--Bott at $A_\infty$ in the sense of Definition \ref{defn:Definition_Morse-Bott_quotient_space}, then \eqref{eq:Lojasiewicz-Simon_W-12gradient_inequality_Yang-Mills_energy_function_slice} holds with $\theta=1/2$.
\end{thm}  

\subsection{Global existence and convergence of Yang--Mills gradient flow on a Coulomb-gauge slice near a local minimum}
\label{subsec:Yang-Mills_gradient_flow_global_existence_and_convergence_started_near_local_minimum}
In this section, we complete the

\begin{proof}[Proof of Theorem \ref{mainthm:Yang-Mills_gradient_flow_global_existence_and_convergence_started_near_local_minimum}]
We proceed by adapting our proof of \cite[Theorem 6]{Feehan_yang_mills_gradient_flow_v4} and verify that the hypotheses of Theorems \ref{thm:Huang_3-4-8_introduction}, \ref{thm:Huang_5-1-1_introduction}, and \ref{thm:Huang_5-1-2_introduction} for a gradient system in a Banach space are obeyed in the case of Yang--Mills gradient flow on a Coulomb-gauge slice.

\setcounter{step}{0}
\begin{step}[Local well-posedness, a priori estimates, and minimal lifetimes for solutions to Yang--Mills gradient flow]
\label{step:Local_well-posedness_priori_estimates_minimal_lifetimes_Yang-Mills_gradient_flow}
Theorem \ref{thm:Local_well-posedness_priori_estimates_minimal_lifetimes_Yang-Mills_gradient_flow} verifies Item \eqref{item:Huang_5-1-1_local_existence} in the hypotheses of Theorem \ref{thm:Huang_5-1-1_introduction} on local existence (also uniqueness) and also Item \eqref{item:Huang_5-1-1_deviation_from_initial_data} in the hypotheses of Theorem \ref{thm:Huang_5-1-1_introduction} on the $C([0,\tau]; W_{A_1}^{1,p}(T^*X\otimes\ad P))$ estimate for the deviation of the solution from the initial data.
\end{step}

\begin{step}[\Apriori estimate for lengths of gradient flowlines]
\label{step:Apriori_estimate_length_gradient_flow_line}
The crucial Hypothesis \ref{hyp:Abstract_apriori_interior_estimate_trajectory_main_introduction} for an abstract gradient flow is established for Yang--Mills gradient flow by our Proposition \ref{prop:Rade_7-3_arbitrary_dimension_L1_time_W1p_space}. This verifies Item \eqref{item:Huang_5-1-1_interior_estimate_length_flowline} in the hypotheses of Theorem \ref{thm:Huang_5-1-1_introduction}.
\end{step}

\begin{step}[{\L}ojasiewicz gradient inequality for the Yang--Mills energy function]
\label{step:Lojasiewicz_gradient_inequality_Yang-Mills_energy_function}
According to Theorem \ref{thm:Lojasiewicz-Simon_W-12_gradient_inequality_Yang-Mills_energy_function_slice} with $A_\infty = A_{\min}$,
the Yang--Mills energy function $\widehat\YM$ on a Coulomb-gauge slice in \eqref{eq:Yang-Mills_energy_function_slice} obeys \eqref{eq:Lojasiewicz-Simon_W-12gradient_inequality_Yang-Mills_energy_function_slice}, that is,
\[
\|\Pi_{A_{\min}}d_A^*F_A\|_{W_{A_1}^{-1,2}(X)} \geq C|\widehat\YM(A)-\widehat\YM(A_{\min})|^\theta,
\]
whenever $A \in A_{\min} + \Ker d_{A_{\min}}^*\cap W_{A_1}^{1,p}(T^*X\otimes\ad P))$ is a $W^{1,p}$ connection on $P$ that obeys \eqref{eq:Lojasiewicz-Simon_gradient_inequality_Yang-Mills_neighborhood}, that is,
\[
\|A - A_{\min}\|_{W^{1,p}_{A_1}(X)} < \sigma,
\]
for small enough $\sigma = \sigma(A_1,A_{\min},g,G,p) \in (0,1]$ and we recall from \eqref{eq:Yang-Mills_energy_function_gradient_slice} that
\[
\widehat{\YM'}(A) = \Pi_{A_{\min}}d_A^*F_A.
\]
This verifies the {\L}ojasiewicz gradient inequality in the Hypothesis \ref{hyp:Lojasiewicz-Simon_gradient_inequality} of Theorem \ref{thm:Huang_5-1-1_introduction} with
\begin{align*}
\sX &= \Ker d_{A_{\min}}^*\cap W_{A_1}^{1,p}(X;T^*X\otimes\ad P),
\\        
\sH &= \Ker d_{A_{\min}}^*\cap W_{A_1}^{-1,2}(X;T^*X\otimes\ad P),
\end{align*}
as the Banach and Hilbert space, respectively.
\end{step}

\begin{step}[Global existence, uniqueness, and convergence]
\label{step:Global_existence_uniqueness_convergence}
We now apply Theorem \ref{thm:Huang_5-1-1_introduction} to obtain the global existence and uniqueness asserted in Item \eqref{item:Global_existence_uniqueness} of of Theorem \ref{mainthm:Yang-Mills_gradient_flow_global_existence_and_convergence_started_near_local_minimum} and convergence asserted in Item \eqref{item:Convergence} of Theorem \ref{mainthm:Yang-Mills_gradient_flow_global_existence_and_convergence_started_near_local_minimum}.
\end{step}

\begin{step}[Convergence rate]
\label{step:Convergence rate}
We apply Theorem \ref{thm:Huang_3-4-8_introduction} to obtain the convergence rate \eqref{eq:Convergence_rate} asserted in Item \eqref{item:Convergence_rate} of Theorem \ref{mainthm:Yang-Mills_gradient_flow_global_existence_and_convergence_started_near_local_minimum}.
\end{step}

\begin{step}[Continuity with respect to initial data]
\label{step:Continuity_solution_initial_data} 
When combined with the global existence conclusion in Item \eqref{item:Global_existence_uniqueness} of Theorem \ref{mainthm:Yang-Mills_gradient_flow_global_existence_and_convergence_started_near_local_minimum}, Theorem \ref{thm:Local_well-posedness_priori_estimates_minimal_lifetimes_Yang-Mills_gradient_flow} yields the continuity with respect to initial data asserted in Item \eqref{item:Continuity_respect_initial_data} of Theorem \ref{mainthm:Yang-Mills_gradient_flow_global_existence_and_convergence_started_near_local_minimum}.
\end{step}  

\begin{step}[Stability]
\label{step:Stability}
We can apply Theorem \ref{thm:Huang_5-1-2_introduction} to yield the stability assertion in Item \eqref{item:Stability} of Theorem \ref{mainthm:Yang-Mills_gradient_flow_global_existence_and_convergence_started_near_local_minimum}.
\end{step}

This completes the proof of Theorem \ref{mainthm:Yang-Mills_gradient_flow_global_existence_and_convergence_started_near_local_minimum}.
\end{proof}

\subsection{Global existence and convergence of Yang--Mills gradient flow on a Coulomb-gauge slice for initial connections with small energy over low-dimensional manifolds}
\label{subsec:Yang-Mills_gradient_flow_global_existence_and_convergence_started_small_energy}
We next complete the

\begin{proof}[Proof of Corollary \ref{maincor:Yang-Mills_gradient_flow_global_existence_and_convergence_started_small_energy}]
If $d=2$ or $3$, the Uhlenbeck Compactness Theorem \ref{thm:Metric_Uhlenbeck_compactness} and Lemma \ref{lem:Finite_open_covering} hold as stated for $p=2$ if $b \in (0,\infty)$ is replaced by a small enough constant $\eps = \eps(g,G) \in (0,1]$. Given $\delta \in (0,1]$ and small enough $\eps=\eps(A_1,g,G,\delta)\in(0,1]$, Theorem \ref{mainthm:Uhlenbeck_Chern_corollary_4-3_prelim} yields a $C^\infty$ flat connection $\Gamma$ on $P$ and a gauge transformation $u\in\Aut^{2,2}(P)$ such that $d_\Gamma^*(u(A_0)-\Gamma)=0$ and $\|u(A_0)-\Gamma\|_{W_{A_1}^{1,2}(X)} < \delta$. Consequently, for $d=2$ or $3$, Theorem \ref{mainthm:Yang-Mills_gradient_flow_global_existence_and_convergence_started_near_local_minimum} now applies to $u(A_0)$ by choosing $\delta=\sigma$.
\end{proof}

\subsection{Almost strong deformation retraction of a neighborhood in the quotient space of Sobolev connections onto the moduli subspace of flat connections}
\label{subsec:Retraction_open_nbhd_onto_moduli_space_flat_connections}
Recall the

\begin{defn}[Strong deformation retraction]
\label{defn:Strong_deformation_retraction}
(See Hatcher \cite[p. 2]{Hatcher}.)  
A continuous map
\[
H: \cX\times [0,1]\to \cX
\]
is a \emph{deformation retraction of a topological space $\cX$ onto a subspace $\cA$} if
\[
H(x,0)=x, \quad H(x,1) \in \cA, \quad \text{and}\quad H(a,1)=a, \quad\text{for all } x \in \cX \text{ and } a \in \cA.
\]
If we add the requirement that
\[
H(a,s)=a, \quad\text{for all } s \in [0, 1] \text{ and } a \in \cA,
\]
then $H$ is called a \emph{strong deformation retraction}.
\end{defn}

We can now proceed to the

\begin{proof}[Proof of Corollary \ref{maincor:Retraction_open_nbhd_onto_moduli_space_flat_connections}]
Let $[A_0] \in \sB_\eps^{1,p}(P)$. Let $\sigma = \sigma(A_1,g,G,p) \in (0,1]$ be as in Theorem \ref{mainthm:Yang-Mills_gradient_flow_global_existence_and_convergence_started_near_local_minimum}, where $A_{\min}$ is now any flat connection on $P$ and thus, by compactness of the moduli space $M(P)$ of flat connections, the constant $\sigma$ is independent of $A_{\min}$. For small enough $\eps=\eps(A_1,g,G,p,\sigma) = \eps(A_1,g,G,p) \in(0,1]$, Theorem \ref{mainthm:Uhlenbeck_Chern_corollary_4-3_prelim} yields a $C^\infty$ flat connection $\Gamma$ on $P$ and a gauge transformation $u\in\Aut^{2,p}(P)$ such that $d_\Gamma^*(u(A_0)-\Gamma)=0$ and $\|u(A_0)-\Gamma\|_{W_{A_1}^{1,2}(X)} < \sigma$. Theorem \ref{mainthm:Yang-Mills_gradient_flow_global_existence_and_convergence_started_near_local_minimum} now applies to $u(A_0)$ and yields a classical solution $A(t)$ for $t\in[0,\infty)$ to Yang--Mills gradient flow on a Coulomb-gauge slice through $\Gamma$ with initial data $A(0) = u(A_0)$ and that converges as $t\to\infty$ to a $W^{1,p}$ flat connection $\Gamma_\infty$ on $P$. By setting
\[
  H([A_0],s) := [A(-\log(1-s)], \quad\text{for all } s \in [0,1],
\]
we obtain the continuous map $H$ in \eqref{eq:Almost_strong_deformation_retraction} with the stated properties.
\end{proof}

\subsection{{\L}ojasiewicz distance inequality for the Yang--Mills energy function}
\label{subsec:Lojasiewicz_distance_inequality_Yang--Mills_energy_function}
Given Theorem \ref{mainthm:Yang-Mills_gradient_flow_global_existence_and_convergence_started_near_local_minimum}, we can specialize Theorem \ref{mainthm:Lojasiewicz_distance_inequality_hilbert_space} to the case of Yang--Mills gradient flow on a Coulomb-gauge slice through a flat connection to complete the

\begin{proof}[Proof of Corollary \ref{maincor:Lojasiewicz_distance_inequality_Yang-Mills_energy_function_slice}]
Recall that
\[
  A_\infty \in A_{\min}+\Ker d_{A_{\min}}^*\cap W_{A_1}^{1,p}(X;T^*X\otimes\ad P) \subset \sA^{1,p}(P)
\]
is a critical point (respectively, local minimum) for $\YM$ in \eqref{eq:Yang-Mills_energy_function} on $\sA^{1,p}(P)$ if and only if $A_\infty$ is a critical point for $\widehat\YM$ in \eqref{eq:Yang-Mills_energy_function_slice} on the Coulomb-gauge slice
\[
   A_{\min}+\Ker d_{A_{\min}}^*\cap W_{A_1}^{1,p}(X;T^*X\otimes\ad P) \subset \sA^{1,p}(P)
\]
by Lemma \ref{lem:Critical_point_Yang-Mills_energy_function_slice}. We shall apply Theorem \ref{mainthm:Lojasiewicz_distance_inequality_hilbert_space} with
\begin{equation}
\label{eq:Lojasiewicz_distance_inequality_yang-mills_banach_spaces}  
\begin{aligned}
\sX &\,= \Ker d_{A_{\min}}^*\cap W_{A_1}^{1,p}(X;T^*X\otimes\ad P),
\\
\sH &\,= \Ker d_{A_{\min}}^*\cap L^2(X;T^*X\otimes\ad P),
\\
\sE &:= \widehat\YM(A_{\min}+\cdot)-\widehat\YM(A_{\min}),
\end{aligned}
\end{equation}
where $\widehat\YM$ is as in \eqref{eq:Yang-Mills_energy_function_slice} (with $A_\infty$ replaced by $A_{\min}$).  
The hypothesis \eqref{eq:Lojasiewicz_gradient_inequality} of Theorem \ref{mainthm:Lojasiewicz_distance_inequality_hilbert_space} that the {\L}ojasiewicz gradient inequality hold for $\sE$ is verified by Theorem \ref{thm:Lojasiewicz-Simon_W-12_gradient_inequality_Yang-Mills_energy_function_slice}, using the continuous Sobolev embedding on the left-hand side of \eqref{eq:Lojasiewicz-Simon_W-12gradient_inequality_Yang-Mills_energy_function_slice},
\[
  L^2(X;T^*X\otimes\ad P) \subset W_{A_1}^{-1,2}(X;T^*X\otimes\ad P).
\]
The hypothesis \eqref{eq:Gradient_flow} of Theorem \ref{mainthm:Lojasiewicz_distance_inequality_hilbert_space} on gradient flow for $\widehat\YM$ near $A_{\min}$ is verified by Theorem \ref{mainthm:Yang-Mills_gradient_flow_global_existence_and_convergence_started_near_local_minimum}. From the first conclusion in Item \eqref{item:Distance_critical_set} in Theorem \ref{mainthm:Lojasiewicz_distance_inequality_hilbert_space}, we obtain \eqref{eq:Lojasiewicz_distance_inequality_critical_set_Yang-Mills_energy_function_slice}.

Since $\widehat\YM \geq 0$ on $A_{\min}+\Ker d_{A_{\min}}^*\cap W_{A_1}^{1,p}(X;T^*X\otimes\ad P)$, we always have
\[
  \Zero\widehat\YM \subset \Crit\widehat\YM.
\]
However, for $\eps=\eps(g,G,[P])\in(0,1]$ as in Theorem \ref{mainthm:L2_energy_gap_Yang-Mills_connections} and small enough $\delta=\delta(\eps)\in(0,\sigma/4]$, the reverse inclusion
\[
  \bB_\delta(A_{\min})\cap\Crit\widehat\YM \subset \Zero\widehat\YM,
\]
also holds by Theorem \ref{mainthm:L2_energy_gap_Yang-Mills_connections}. Hence, the second conclusion in Item \eqref{item:Distance_critical_set} in Theorem \ref{mainthm:Lojasiewicz_distance_inequality_hilbert_space} yields \eqref{eq:Lojasiewicz_distance_inequality_zero_set_Yang-Mills_energy_function_slice}.

If $\widehat\YM$ is Morse--Bott at $A_{\min}$, then $\widehat\YM$ has {\L}ojasiewicz exponent $\theta=1/2$ by Theorem \ref{thm:Optimal_Lojasiewicz-Simon_W-12_gradient_inequality_Yang-Mills_energy_function_slice} and so $\alpha=1/(1-\theta)=2$ in Theorem \ref{mainthm:Lojasiewicz_distance_inequality_hilbert_space} and thus also in \eqref{eq:Lojasiewicz_distance_inequality_critical_set_Yang-Mills_energy_function_slice} and \eqref{eq:Lojasiewicz_distance_inequality_zero_set_Yang-Mills_energy_function_slice}.
\end{proof}

Lastly, we complete the

\begin{proof}[Proof of Corollary \ref{maincor:Lojasiewicz_distance_inequality_Yang-Mills_energy_function}]
For small enough $\eta\in(0,1]$ and any $A \in B_\eta(A_{\min})$, Theorem \ref{thm:Feehan_proposition_3-4-4_Lp} yields a gauge transformation $u\in\Aut^{2,p}(P)$ such that $d_{A_{\min}}^*(u(A)-A_{\min})=0$ and $u(A) \in B_\delta(A_{\min})$. Hence,
\begin{align*}
  \YM(A)-\YM(A_{\min}) &= \YM(u(A))-\YM(A_{\min})
  \\
                       &= \widehat\YM(u(A))-\widehat\YM(A_{\min})
  \\
                       &\geq C\,\mathbf{dist}_{L^2(X)}\left(u(A), \bB_\sigma(A_{\min})\cap \Crit\widehat\YM\right)^\alpha
  \\
                       &= C\,\mathbf{dist}_{L^2(X)}\left(u(A), \bB_\sigma(A_{\min})\cap \Crit\YM\right)^\alpha \quad\text{(by Lemma \ref{lem:Critical_point_Yang-Mills_energy_function_slice})}
  \\
                       &\geq C\dist_{L^2(X)}\left(A, B_\sigma(A_{\min})\cap \Crit\YM\right)^\alpha \quad\text{(by \eqref{eq:Distance_point_subset_affine_space_slice} and \eqref{eq:Distance_point_subset_affine_space})},
\end{align*}
which is \eqref{eq:Lojasiewicz_distance_inequality_critical_set_Yang-Mills_energy_function}. The proof that \eqref{eq:Lojasiewicz_distance_inequality_zero_set_Yang-Mills_energy_function} follows from \eqref{eq:Lojasiewicz_distance_inequality_critical_set_Yang-Mills_energy_function} is identical to the proof that \eqref{eq:Lojasiewicz_distance_inequality_zero_set_Yang-Mills_energy_function_slice} follows from \eqref{eq:Lojasiewicz_distance_inequality_critical_set_Yang-Mills_energy_function_slice}.
\end{proof}

\section{Estimates for distance to moduli subspace of flat connections}
\label{sec:Estimates_distance_moduli_space_flat_connections}
In Section \ref{subsec:General_nonlinear_estimate_distance_moduli_space_flat_connections}, we prove Theorem \ref{mainthm:Uhlenbeck_Chern_corollary_4-3}, while in Section \ref{subsec:Optimal_Lojasiewicz_distance_inequality_Yang-Mills_energy_function_implies_Morse-Bott_property}, we prove Theorem \ref{mainthm:Yang-Mills_energy_function_Lojasiewicz_exponent_one-half_Morse-Bott}.

\subsection{General nonlinear estimate for distance to moduli subspace of flat connections}
\label{subsec:General_nonlinear_estimate_distance_moduli_space_flat_connections}
We complete the 

\begin{proof}[Proof of Theorem \ref{mainthm:Uhlenbeck_Chern_corollary_4-3}]
Let $\delta \in (0,1]$ to be determined below. By applying Theorem \ref{mainthm:Uhlenbeck_Chern_corollary_4-3_prelim} with $\sigma=\delta$ and $A$ obeying $\|F_A\|_{L^p(X)} < \eps$ for $\eps = C(A_1,g,G,p,\delta) \in (0,1]$, we have from \eqref{eq:Uhlenbeck_compactness_bound} that
\begin{equation}
  \label{eq:W1pnorm_uA-Gamma_lt_delta}
  \|u(A)-\Gamma\|_{W_{A_1}^{1,p}(X)} < \delta,
\end{equation}
for some $C^\infty$ flat connection $\Gamma$ on $P$ and $u\in\Aut^{2,p}(P)$ such that $d_\Gamma^*(u(A)-\Gamma)=0$.

Now let $\sigma=\sigma(A_1,g,G,p)\in(0,1]$ and $\delta\in(0,\sigma/4]$ be as in Corollary \ref{maincor:Lojasiewicz_distance_inequality_Yang-Mills_energy_function_slice}. From \eqref{eq:Lojasiewicz_distance_inequality_zero_set_Yang-Mills_energy_function_slice} with $A_{\min}=\Gamma$ and $A$ replaced by $u(A)$, we have
\[
  \YM(A) \geq C\,\mathbf{dist}_{L^2(X)}\left(u(A), \bB_\sigma(\Gamma)\cap \Zero\YM\right)^\alpha, \quad\text{for all } A \in \bB_\delta(\Gamma),
\]
noting that $\widehat\YM(u(A))=\YM(A)$ and $\Zero\widehat\YM=\Zero\YM$ on $\bB_\sigma(\Gamma)$. By shrinking $\sigma\in(0,1]$ if necessary, Corollary \ref{cor:Slice} implies that the map,
\[
  \pi:\bB_\sigma(\Gamma)/\Stab(\Gamma) \to \sB^{1,p}(P),
\]
is a homeomorphism onto an open neighborood of $[\Gamma]$ in $\sB^{1,p}(P)$, recalling that $\Stab(\Gamma)$ may be identified with a subgroup of the compact Lie group $G$. The subspace $M(P)\subset \sB^{1,p}(P)$ is compact and has a finite open cover by such neighborhoods by Theorem \ref{thm:Metric_Uhlenbeck_compactness}, so the closed subset
\[
  \bar\bB_\sigma(\Gamma)\cap \Zero\YM
\]
is compact. Hence, there exists some $W^{1,p}$ flat connection $\Gamma' \in \bar\bB_\sigma(\Gamma)\cap \Zero\YM$ such that 
\[
  \mathbf{dist}_{L^2(X)}\left(u(A), \bB_\sigma(\Gamma)\cap \Zero\YM\right) = \|A-\Gamma'\|_{L^2(X)}.
\]
Combining the preceding observations and recalling the definition \eqref{eq:Yang-Mills_energy_function} of the Yang--Mills energy function $\YM$ gives the inequality
\begin{equation}
\label{eq:L2norm_uA-Gamma_prime_leq_L2norm_FA_lambda}
  \|u(A)-\Gamma'\|_{L^2(X)} \leq C_1\|F_A\|_{L^2(X)}^\lambda,
\end{equation}
for $\lambda = \lambda(A_1,g,G,p) = 2/\alpha \in (0,1]$, where $[\Gamma']\in M(P)$ obeys $d_\Gamma(\Gamma'-\Gamma)=0$ and
\begin{equation}
  \label{eq:W1pnorm_Gamma_prime-Gamma_leq_sigma}
  \|\Gamma'-\Gamma\|_{W_{A_1}^{1,p}(X)} \leq \sigma.
\end{equation}
Because
\[
  u(A)-\Gamma' = u(A)-\Gamma +\Gamma-\Gamma' \quad\text{and}\quad d_\Gamma^*(u(A)-\Gamma) = 0 = d_\Gamma^*(\Gamma'-\Gamma),
\]
then
\begin{equation}
  \label{eq:Coulomb_uA-Gamma_prime}
  d_\Gamma^*(u(A)-\Gamma') = 0.
\end{equation}
Moreover, by \eqref{eq:Donaldson_Kronheimer_2-1-14} the curvature of $u(A) = \Gamma' + (u(A)-\Gamma')$ is given by
\[
  F_{u(A)} = F_{\Gamma'} + d_{\Gamma'}(u(A)-\Gamma') + \frac{1}{2}[u(A)-\Gamma', u(A)-\Gamma'].
\]
We obtain from \eqref{eq:Coulomb_uA-Gamma_prime}, the preceding identity, the following identity from \eqref{eq:Bleecker_corollary_3-1-6_local},
\[
  d_\Gamma(u(A)-\Gamma') = d_{\Gamma'}(u(A)-\Gamma') + [\Gamma-\Gamma', u(A)-\Gamma'],
\]
and $F_{\Gamma'}=0$ that
\begin{equation}
  \label{eq:d+d*_uA-Gamma_prime}
  (d_\Gamma+d_\Gamma^*)(u(A)-\Gamma') = F_{u(A)} - \frac{1}{2}[u(A)-\Gamma', u(A)-\Gamma'] + [\Gamma-\Gamma', u(A)-\Gamma'].
\end{equation}
For $r\in (1,p]$, the $L^r$ estimate for the first-order elliptic operator $d_\Gamma+d_\Gamma^*$ (see \cite[Theorem 14.60]{Feehan_yang_mills_gradient_flow_v4}) gives
\begin{equation}
  \label{eq:d+d*_Lr_elliptic_estimate}
\|u(A)-\Gamma'\|_{W_{A_1}^{1,r}(X)} \leq C_2\left(\|(d_\Gamma+d_\Gamma^*)(u(A)-\Gamma')\|_{L^r(X)} + \|u(A)-\Gamma'\|_{L^r(X)}\right),
\end{equation}
for a constant $C_2 = C_2(A_1,g,G,M(P),r) = C_2(A_1,g,G,[P],r) \in [1,\infty)$. If $r \in (1,d)$, then \cite[Theorem 4.12]{AdamsFournier} gives a continuous Sobolev embedding $W^{1,r}(X) \subset L^s(X)$ for $s := r^* \equiv dr/(d-r) \in (d,\infty)$ with $s > r$ and thus an interpolation inequality by \cite[Equation (7.10)]{GT},
\[
  \|b\|_{L^r(X)} \leq \eta\|b\|_{L^s(X)} + \eta^{-\mu}\|b\|_{L^1(X)}, \quad\text{for all } b \in W_{A_1}^{1,p}(X;T^*X\otimes\ad P),
\]
for any $\eta \in (0,1]$, where $\mu = (1-1/r)/(1/r-1/s)$, and thus
\begin{equation}
  \label{eq:Gilbarg_Trudinger_7-10_plus_embedding}
  \|b\|_{L^r(X)} \leq C_3\eta\|b\|_{W_{A_1}^{1,r}(X)} + \eta^{-\mu}\|b\|_{L^1(X)}, \quad\text{for all } b \in W_{A_1}^{1,p}(X;T^*X\otimes\ad P),
\end{equation}
where $C_3  = C_3(r,g) \in [1,\infty)$ is the norm of the Sobolev embedding $W^{1,r}(X) \subset L^s(X)$. (We apply the Kato Inequality \cite[Equation (6.20)]{FU} to remove the potential dependence on $A_1$ here.) If $r \in [d,\infty)$, then \cite[Theorem 4.12]{AdamsFournier} gives a continuous Sobolev embedding $W^{1,r}(X) \subset L^s(X)$ for any $s \in [1,\infty)$ and thus an interpolation inequality by \cite[Equation (7.10)]{GT} when $r < s < \infty$, and therefore \eqref{eq:Gilbarg_Trudinger_7-10_plus_embedding} holds for this case too for a choice of $s=s(d,r) \in (r,\infty)$. Hence, by applying the interpolation inequality \eqref{eq:Gilbarg_Trudinger_7-10_plus_embedding} in the $L^r$ elliptic estimate \eqref{eq:d+d*_Lr_elliptic_estimate}, we see that
\begin{multline*}
  \|u(A)-\Gamma'\|_{W_{A_1}^{1,r}(X)} \leq C_2\|(d_\Gamma+d_\Gamma^*)(u(A)-\Gamma')\|_{L^r(X)} + C_2C_3\eta\|u(A)-\Gamma'\|_{W_{A_1}^{1,r}(X)}
  \\
  + C_2\eta^{-\mu}\|u(A)-\Gamma'\|_{L^1(X)}.
\end{multline*}
By using rearrangement with  $\eta = 1/(2C_2C_3)$ in the preceding estimate for $u(A)-\Gamma'$, we find that
\begin{equation}
  \label{eq:d+d*_uA-Gamma_prime_Lp_elliptic_estimate_plus_interpolation}
  \|u(A)-\Gamma'\|_{W_{A_1}^{1,r}(X)} \leq C_4\left(\|(d_\Gamma+d_\Gamma^*)(u(A)-\Gamma')\|_{L^r(X)} + \|u(A)-\Gamma'\|_{L^1(X)}\right),
\end{equation}
for a constant $C_4 = C_4(A_1,g,G,[P],r) \in [1,\infty)$. By substituting the identity \eqref{eq:d+d*_uA-Gamma_prime} into the inequality \eqref{eq:d+d*_uA-Gamma_prime_Lp_elliptic_estimate_plus_interpolation}, we obtain
\begin{multline*}
  \|u(A)-\Gamma'\|_{W_{A_1}^{1,r}(X)} \leq C_4\left(\|F_{u(A)}\|_{L^r(X)} + \|[u(A)-\Gamma',u(A)-\Gamma']\|_{L^r(X)} \right.
    \\
    + \left. \|[\Gamma-\Gamma', u(A)-\Gamma']\|_{L^r(X)} + \|u(A)-\Gamma'\|_{L^1(X)}\right).
\end{multline*}
By applying the continuous Sobolev multiplication \eqref{eq:Ls_times_Lq_to_Lr} and embedding \eqref{eq:W1r_to_Ls} and using the fact that $|F_{u(A)}|=|F_A|$, the preceding inequality becomes
\begin{multline*}
  \|u(A)-\Gamma'\|_{W_{A_1}^{1,r}(X)} \leq C_4\left(\|F_A\|_{L^r(X)} + \|u(A)-\Gamma'\|_{L^q(X)}\|u(A)-\Gamma'\|_{W_{A_1}^{1,r}(X)} \right.
    \\
    + \left. \|\Gamma-\Gamma'\|_{L^q(X)} \|u(A)-\Gamma']\|_{W_{A_1}^{1,r}(X)} + \|u(A)-\Gamma'\|_{L^1(X)}\right),
\end{multline*}
for a possibly larger constant $C_4$, where $q = q(d,p) \in (d,\infty)$ is as in the hypotheses of Lemma \ref{lem:Equivalence_Uhlenbeck_metric_convergence} and further restricted so that $q>r$. By the definition of $q$ in Lemma \ref{lem:Equivalence_Uhlenbeck_metric_convergence}, there is a continuous (in fact, compact) Sobolev embedding, $W^{1,p}(X) \subset L^q(X)$, and so the preceding estimate yields
\begin{multline*}
  \|u(A)-\Gamma'\|_{W_{A_1}^{1,r}(X)} \leq C_5\left(\|F_A\|_{L^r(X)} + \|u(A)-\Gamma'\|_{W_{A_1}^{1,p}(X)}\|u(A)-\Gamma'\|_{W_{A_1}^{1,r}(X)} \right.
    \\
    + \left. \|\Gamma-\Gamma'\|_{W_{A_1}^{1,p}(X)} \|u(A)-\Gamma'\|_{W_{A_1}^{1,r}(X)} + \|u(A)-\Gamma'\|_{L^1(X)}\right),
\end{multline*}
for a constant $C_5 = C_5(A_1,g,G,[P],p,r) \in [1,\infty)$. But $\|u(A)-\Gamma\|_{W_{A_1}^{1,p}(X)} < \delta \leq \sigma/4$ by \eqref{eq:W1pnorm_uA-Gamma_lt_delta} and the choice of $\delta$, while $\|\Gamma-\Gamma'\|_{W_{A_1}^{1,p}(X)} \leq \sigma$ by \eqref{eq:W1pnorm_Gamma_prime-Gamma_leq_sigma}, so
\[
  \|u(A)-\Gamma'|_{W_{A_1}^{1,p}(X)} \leq \|u(A)-\Gamma\|_{W_{A_1}^{1,p}(X)} + \|\Gamma-\Gamma'\|_{W_{A_1}^{1,p}(X)} < 5\sigma/4.
\]
Therefore,
\begin{multline*}
  \|u(A)-\Gamma'\|_{W_{A_1}^{1,r}(X)} \leq C_5\left(\|F_A\|_{L^r(X)} + (5\sigma/4)\|u(A)-\Gamma'\|_{W_{A_1}^{1,r}(X)} \right.
    \\
    + \left. \sigma \|u(A)-\Gamma'\|_{W_{A_1}^{1,r}(X)} + \|u(A)-\Gamma'\|_{L^1(X)}\right),
\end{multline*}
For $9\sigma/4 \leq 1/(2C_5)$, we can apply rearrangement to give
\[
\|u(A)-\Gamma'\|_{W_{A_1}^{1,r}(X)} \leq 2C_5\left(\|F_A\|_{L^r(X)} + \|u(A)-\Gamma'\|_{L^1(X)}\right).
\]
By applying the continuous embedding $L^2(X) \subset L^1(X)$ to give
\[
  \|u(A)-\Gamma'\|_{L^1(X)} \leq \kappa\|u(A)-\Gamma'\|_{L^2(X)}
\]
for $\kappa=\kappa(g) \in [1,\infty)$ and the inequality \eqref{eq:L2norm_uA-Gamma_prime_leq_L2norm_FA_lambda}, the preceding estimate becomes
\[
  \|u(A)-\Gamma'\|_{W_{A_1}^{1,r}(X)} \leq C_6\left(\|F_A\|_{L^r(X)} + \|F_A\|_{L^2(X)}^\lambda\right),
\]
for a constant $C_6 = C_6(A_1,g,G,[P],p,r) \in [1,\infty)$. Hence, upon further restricting to $p\geq 2$ and using the continuous embedding, $L^p(X) \subset L^2(X)$, to give $\|F_A\|_{L^2(X)} \leq c\|F_A\|_{L^p(X)}$ for $c=c(g,p) \in [1,\infty)$ and using $\|F_A\|_{L^p(X)} \leq \|F_A\|_{L^p(X)}^\lambda$ when $\|F_A\|_{L^p(X)}\leq 1$ (which is assured by \eqref{eq:Lp_norm_FA_lessthan_epsilon}) and $\lambda \leq 1$, we obtain the simpler bound
\[
  \|u(A)-\Gamma\|_{W_{A_1}^{1,p}(X)} \leq C_6\|F_A\|_{L^p(X)}^\lambda,
\]
for a possibly larger constant $C_6$. This yields \eqref{eq:Uhlenbeck_Chern_corollary_4-3_A_M(P)_W1p_distance_bound} and completes the proof of Theorem \ref{mainthm:Uhlenbeck_Chern_corollary_4-3}.
\end{proof}

\subsection{Optimal {\L}ojasiewicz distance inequality for a Yang--Mills energy function implies its Morse--Bott property}
\label{subsec:Optimal_Lojasiewicz_distance_inequality_Yang-Mills_energy_function_implies_Morse-Bott_property}
We conclude this section with our

\begin{proof}[Proof of Theorem \ref{mainthm:Yang-Mills_energy_function_Lojasiewicz_exponent_one-half_Morse-Bott}]
We begin by lightly modifying Feehan \cite[Section 3, Proof of Theorem 2]{Feehan_lojasiewicz_inequality_ground_state}. Let $\eps=\eps(g,G,p)\in(0,1]$ to be determined. We may assume that $A\neq\Gamma$ without loss of generality. Write
\[
  a := A-\Gamma \in \Ker d_\Gamma^*\cap W_\Gamma^{1,p}(X;T^*X\otimes\ad P)
\]
and note that
\[
\left(\Ker d_\Gamma^*\cap W_\Gamma^{1,2}(X;T^*X\otimes \ad P)\right)^*
=
\Ker d_\Gamma^*\cap W_\Gamma^{-1,2}(X;T^*X\otimes \ad P)
\]
is the continuous dual space of the Hilbert space $\Ker d_\Gamma^*\cap W_\Gamma^{1,2}(X;T^*X\otimes \ad P)$. We have
\[
d_Aa = d_\Gamma a + [a, a] = F_A + \frac{1}{2}[a, a],
\]
using \eqref{eq:Bleecker_corollary_3-1-6_local}, since $F_A = F_{\Gamma+a} = F_\Gamma + d_\Gamma a + \frac{1}{2}[a,a] = d_\Gamma a + \frac{1}{2}[a,a]$ by \eqref{eq:Donaldson_Kronheimer_2-1-14}. Thus, we obtain
\begin{align*}
\|\Pi_\Gamma d_A^*F_A\|_{W_\Gamma^{-1,2}(X)}
&=
\sup_{b \in \Ker d_\Gamma^*\cap W_\Gamma^{1,2}(X) \less \{0\}}
\frac{(\Pi_\Gamma d_A^*F_A, b)_{L^2(X)}}{\|b\|_{W_\Gamma^{1,2}(X)}}
\\
&=
\sup_{b \in \Ker d_\Gamma^*\cap W_\Gamma^{1,2}(X) \less \{0\}}
\frac{(d_A^*F_A, b)_{L^2(X)}}{\|b\|_{W_\Gamma^{1,2}(X)}}
\\  
&\geq \frac{(d_A^*F_A, a)_{L^2(X)}}{\|a\|_{W_\Gamma^{1,2}(X)}}
=
\frac{(F_A, d_Aa)_{L^2(X)}}{\|a\|_{W_\Gamma^{1,2}(X)}}
=
\frac{(F_A, F_A + \frac{1}{2}[a,a])_{L^2(X)}}{\|a\|_{W_\Gamma^{1,2}(X)}},
\end{align*}
and therefore,
\begin{equation}
\label{eq:W12dual_norm_Gradient_Yang-Mills_geq_pre-energy}
\|\Pi_\Gamma d_A^*F_A\|_{W_\Gamma^{-1,2}(X)}
\geq \frac{\|F_A\|_{L^2(X)}^2}{\|a\|_{W_\Gamma^{1,2}(X)}} +
\frac{(F_A,[a,a])_{L^2(X)}}{2\|a\|_{W_\Gamma^{1,2}(X)}}.
\end{equation}
We recall that \cite[Theorem 4.12, Part I (B) and (C)]{AdamsFournier} provides a continuous embedding of Sobolev spaces with norm $\kappa_r = \kappa_r(g) \in [1,\infty)$,
\[
W^{1,2}(X) \subset L^r(X)
\quad\text{for }
\begin{cases}
1 \leq r < \infty, &\text{if } d = 2,
\\
1 \leq r \leq 2^* = 2d/(d-2), &\text{if } d > 2,
\end{cases}
\]
and a continuous embedding, $W^{1,d/2}(X) \subset L^d(X)$, for all $d \geq 2$. When $d > 2$, we claim that
\begin{equation}
\label{eq:L2_a_wedge_a_leq_W12a_times_W1dover2a}
\|[a,a]\|_{L^2(X)} \leq c\kappa_r\kappa_d\|a\|_{W_\Gamma^{1,2}(X)}\|a\|_{W_\Gamma^{1,d/2}(X)},
\end{equation}
for a constant $c = c(d,G) \in [1,\infty)$.

The proof of \eqref{eq:L2_a_wedge_a_leq_W12a_times_W1dover2a} is straightforward. Indeed, writing $1/2 = 1/r + 1/d$ (for $d > 2$ and $r = 2^* = 2d/(d-2) \in (2, \infty)$), we have
\[
\|[a,a]\|_{L^2(X)} \leq c\|a\|_{L^r(X)}\|a\|_{L^d(X)},
\]
for a constant $c = c(d,G) \in [1,\infty)$. Combining the preceding inequality with the continuous embeddings, $W^{1,2}(X) \subset L^r(X)$, when $r = 2^* = 2d/(d-2)$, and $W^{1,d/2}(X) \subset L^d(X)$ yields \eqref{eq:L2_a_wedge_a_leq_W12a_times_W1dover2a}.

Since $\|a\|_{W_\Gamma^{1,p}(X)} < \sigma$ by hypothesis, we obtain, for a constant $z=(g,G,p)\in[1,\infty)$,
\begin{align*}
  \|F_A\|_{L^p(X)} &= \|d_\Gamma a + (1/2)[a,a]\|_{L^p(X)} \leq \|d_\Gamma a\|_{L^p(X)} \|a\|_{W_\Gamma^{1,p}(X)} + z\|a\|_{L^{2p}(X)}^2
  \\
  &\leq z\|a\|_{W_\Gamma^{1,p}(X)}  + z\|a\|_{W_\Gamma^{1,p}(X)}^2 < z(\sigma+\sigma^2),
\end{align*}
using the continuous Sobolev embedding $W^{1,p}(X)\subset L^{2p}(X)$ valid for any $p\geq d/2$. Thus, for small enough $\sigma\in(0,1]$, we have $z(\sigma+\sigma^2) \leq \eps$ and $F_A$ obeys \eqref{eq:Lp_norm_FA_lessthan_epsilon}. The gradient inequality \eqref{eq:Yang-Mills_energy_function_optimal_Lojasiewicz_gradient_inequality} now follows for all $d \geq 2$. Indeed, for $d > 2$,
\begin{align*}
\|\Pi_\Gamma d_A^*F_A\|_{W_\Gamma^{-1,2}(X)}
&\geq
\frac{\|F_A\|_{L^2(X)}^2}{\|a\|_{W_\Gamma^{1,2}(X)}}
- \frac{\|F_A\|_{L^2(X)}\|[a,a]\|_{L^2(X)} }{\|a\|_{W_\Gamma^{1,2}(X)}}
\quad\text{(by \eqref{eq:W12dual_norm_Gradient_Yang-Mills_geq_pre-energy})}
\\
&\geq
\frac{\|F_A\|_{L^2(X)}^2}{\|a\|_{W_\Gamma^{1,2}(X)}}
- c\kappa_r\kappa_d\frac{\|F_A\|_{L^2(X)}\|a\|_{W_\Gamma^{1,2}(X)}\|a\|_{W_\Gamma^{1,d/2}(X)} }{\|a\|_{W_\Gamma^{1,2}(X)}}
\quad\text{(by \eqref{eq:L2_a_wedge_a_leq_W12a_times_W1dover2a})}
\\
&=
\frac{\|F_A\|_{L^2(X)}^2}{\|a\|_{W_\Gamma^{1,2}(X)}}
- c\kappa_r\kappa_d\|F_A\|_{L^2(X)}\|a\|_{W_\Gamma^{1,d/2}(X)}
\\
&\geq
C_2^{-1}\|F_A\|_{L^2(X)} - c\kappa_r\kappa_dC_{d/2}\|F_A\|_{L^2(X)}\|F_A\|_{L^{d/2}(X)}
  \\
  &\qquad\text{(by \eqref{eq:Yang-Mills_energy_function_optimal_Lojasiewicz_distance_inequality}: $L^2$ version and $2\leq d\leq 4$)}
\\
&\geq
(C_2^{-1} - c\kappa_r\kappa_dC_{d/2}\eps)\|F_A\|_{L^2(X)}
\quad\text{(by \eqref{eq:Lp_norm_FA_lessthan_epsilon})}.
\end{align*}
Now choose $\eps = \frac{1}{2}C_2^{-1}/(c\kappa_r\kappa_dC_{d/2})$ to give \eqref{eq:Yang-Mills_energy_function_optimal_Lojasiewicz_gradient_inequality} for $d>2$, noting that because $M_0(P)$ is compact, the explicit dependence of the $W^{-1,2}$ norm on $\Gamma$ may be dropped and recalling that
\[
  \widehat{\YM'}(A) = \Pi_\Gamma d_A^*F_A \quad\text{and}\quad \YM(A) = \frac{1}{2}\|F_A\|_{L^2(X)}^2.
\]
For $d=2$ and $s_0 \in (1,2)$ (we may assume $s_0<2$ without loss of generality), we can instead use $1/2 = 1/r + 1/s_0^*$, where $s_0^* := 2s_0/(2-s_0) \in (2,\infty)$ and $r \in (2,\infty)$, to
give $\|[a,a]\|_{L^2(X)} \leq c\|a\|_{L^r(X)}\|a\|_{L^{s_0^*}(X)}$ and continuous embeddings, $W^{1,s_0}(X;\RR) \subset L^{s_0^*}(X;\RR)$ and $W^{1,2}(X;\RR) \subset L^r(X;\RR)$. Now arguing exactly as in the calculation for $d>2$ gives \eqref{eq:Yang-Mills_energy_function_optimal_Lojasiewicz_gradient_inequality} for $d=2$. 

We now restrict to $2\leq d\leq 4$, choose $\sX = \Ker d_\Gamma^*\cap W_{A_1}^{1,2}(X;T^*X\otimes\ad P)$, so that
\[
  \sX^* = \Ker d_\Gamma^*\cap W_{A_1}^{-1,2}(X;T^*X\otimes\ad P),
\]
and observe that
\[
    \widehat\YM = \YM: \Gamma + \Ker d_\Gamma^*\cap W_{A_1}^{1,2}(X;T^*X\otimes\ad P) \to \RR
\]
is analytic by Feehan and Maridakis \cite[Proposition 3.1.1]{Feehan_Maridakis_Lojasiewicz-Simon_coupled_Yang-Mills} and that 
\[
    \widehat{\YM''}(\Gamma): \Ker d_\Gamma^*\cap W_{A_1}^{1,2}(X;T^*X\otimes\ad P) \to \Ker d_\Gamma^*\cap W_{A_1}^{-1,2}(X;T^*X\otimes\ad P)
\]
is a Fredholm operator with index zero by Feehan and Maridakis \cite[Theorem 3.1.7]{Feehan_Maridakis_Lojasiewicz-Simon_coupled_Yang-Mills}. Hence, by applying Theorem \ref{mainthm:Analytic_function_Lojasiewicz_exponent_one-half_Morse-Bott_Banach} to the functon $\widehat\YM$, we conclude that $\widehat\YM$ is Morse--Bott at $\Gamma$ in the sense of Definition \ref{defn:Morse-Bott_function} \eqref{item:Morse-Bott_point} and hence Morse--Bott in the sense of Definition \ref{defn:Definition_Morse-Bott_quotient_space}.
\end{proof}

\section{Energy gap for Yang--Mills connections}
\label{sec:Energy_gap_Yang-Mills_connections}
Our goal in this section is to prove Theorem \ref{mainthm:L2_energy_gap_Yang-Mills_connections}. Our argument is independent of the proof of Theorem \ref{thm:Ldover2_energy_gap_Yang-Mills_connections} that we gave in \cite{Feehan_yangmillsenergygapflat, Feehan_yangmillsenergygapflat_corrigendum}. While we shall again rely on analyticity of the Yang--Mills energy function, we will not use the {\L}ojasiewicz gradient inequality.

\subsection{Local piecewise real analytic arc-connectedness of semianalytic varieties}
\label{subsec:Local_arc-connectedness_analytic_varieties}
For expositions of the properties of real analytic, semianalytic, and subanalytic sets, we refer to Abhyankar \cite{Abhyankar_local_analytic_geometry}, Benedetti and Risler \cite{Benedetti_Risler_real_algebraic_semi-algebraic_sets}, Bierstone and Milman \cite[Section 2]{BierstoneMilman}, Bochnak, Coste, and Roy \cite{Bochnak_Coste_Roy_real_algebraic_geometry}, Gabri{\`e}lov \cite{Gabrielov_1968}, Goresky and MacPherson \cite{GorMacPh}, Hardt \cite{Hardt_1975}, Hironaka \cite{Hironaka_1973, Hironaka_intro_real-analytic_sets_maps}, Kashiwara and Schapira \cite{Kashiwara_Schapira_sheaves_manifolds}, {\L}ojasiewicz \cite{Lojasiewicz_1964, Lojasiewicz_1965}, Massey and Le \cite{Massey_Le_2007}, Mather \cite{Mather_2012}, Shiota \cite{Shiota_geometry_subanalytic_semialgebraic_sets}, and Whitney \cite{Whitney_1965dct, Whitney_1965am}. Let $S \subset M$ be a subset of a finite-dimensional, real analytic manifold $M$. One calls $S$ a \emph{real analytic variety} (or \emph{set}) \cite[p. 8]{Abhyankar_local_analytic_geometry} if for every point $x_0 \in S$, there are an open neighborhood $U \subset M$ of $x_0$ and finitely many real analytic functions $f_1,\ldots,f_p$ defined on $U$ such that
\[
  S\cap U = \bigcap_{i=1}^p\{x \in U: f_i(x) = 0\}.
\]
One calls $S$ a \emph{semianalytic variety} (or \emph{set}) \cite[Definition 2.1]{BierstoneMilman} if there are finitely many real analytic functions $f_{1,1},\ldots,f_{p,q}$ defined on $U$ such that
\[
  S\cap U = \bigcup_{i=1}^p \bigcap_{j=1}^q S_{ij},
\]
where $S_{ij} = \{x\in U: f_{ij}(x)=0\}$ or $\{x\in U: f_{ij}(x)>0\}$. (See \cite[Section 3]{BierstoneMilman} for the definition of \emph{subanalytic varieties} (or \emph{sets}).)

Recall that a \emph{semianalytic curve} (or \emph{arc}) in a semianalytic set $S$ joining points $x_0$ and $x_1$ is the image of a continuous embedding $\gamma:[0,1]\to S$ with $\gamma(0)=x_0$ and $\gamma(1)=x_1$, such that $\gamma([0,1]) \subset S$ is a semianalytic set; in particular, $\gamma$ is an analytic embedding except at finitely many points \cite[p. 283]{Gabrielov_1968}. We shall need the following result on local piecewise real analytic arc-connectivity of semianalytic varieties.

\begin{prop}[Local piecewise real analytic arc-connectedness of semianalytic varieties]
\label{prop:Local_arc-connectedness_semianalytic_sets}  
Let $M$ be a finite-dimensional, real analytic manifold and $S \subset M$ be a semianalytic variety. If $x_0 \in S$, then there is an open neighborhood $U\subset M$ of $x_0$ such that $S\cap U$ is piecewise real analytic arc connected.
\end{prop}

\begin{proof}
According to Bierstone and Milman \cite[Corollary 2.7]{BierstoneMilman}, the set $S$ is locally connected and so we may choose $U$ such that $S\cap U$ is connected. Moreover, by \cite[Corollary 2.7]{BierstoneMilman}, every connected component of $S$ is semianalytic and thus $S\cap U$ is a connected semianalytic set. According to Gabri{\`e}lov \cite[p. 283]{Gabrielov_1968}, the points $x_1$ and $x_0$ can be joined by a semianalytic curve and so  $S\cap U$ is piecewise real analytic arc connected, that is, there exists a continuous embedding $\gamma:[0,1]\to S\cap U$ such that $\gamma(0)=x_0$ and $\gamma(1)=x_1$ and $\gamma$ is an analytic embedding except at finitely many points.
\end{proof}

Proposition \ref{prop:Local_arc-connectedness_semianalytic_sets} may be proved in other ways and we describe an alternative approach below.

\begin{rmk}[Proof of Proposition \ref{prop:Local_arc-connectedness_semianalytic_sets} via stratification]
\label{rmk:Local_arc-connectedness_semianalytic_sets_via_stratification}  
Because $S \subset M$ is a semianalytic variety, there is a stratification of $S$ in the following sense:
\begin{enumerate}[(\itshape 1\upshape)]
\item $S = \sqcup_{i=0}^\ell S_i\cap U$ (disjoint union) for some integer $\ell\geq 0$,
\item Each $S_i\subset M$ is a semianalytic variety and a real analytic submanifold, and
\item $S_i \cap \bar S_j \neq \emptyset$ if and only if $S_i \subset \bar S_j$ if and only if $i=j$ or $i<j$ and $\dim S_i < \dim S_j$. 
\end{enumerate}
The proof of existence of such a stratification (in fact, the stronger Whitney stratification in the sense of \cite[p. 37]{GorMacPh}) for real analytic varieties is due to Whitney \cite{Whitney_1965am}; see Bierstone and Milman \cite[Proposition 2.10 and Corollary 2.11]{BierstoneMilman} (for semianalytic varieties), Goresky and MacPherson \cite[p. 43]{GorMacPh} (for subanalytic and complex analytic varieties), Kaloshin \cite{Kaloshin_2005} (for analytic and semianalytic varieties), and Sussmann \cite{Sussmann_1990} (for subanalytic varieties).

Given $x_0 \in S$ and the fact that each stratum $S_i$ is a smooth (in fact, real analytic) submanifold of $M$, we may choose an open neighborhood $U \subset M$ of $x_0$ such that $S_i\cap U$ is connected for $i=0,1\ldots,\ell$ and so $S\cap U$ is connected. Consequently, the set $S\cap U$ is $C^0$ arc connected (for example, see Willard \cite[Definition 13.5 and Corollary 31.6]{Willard_general_topology}) and the points $x_0$ and $x_1$ are connected by a $C^0$ arc $\gamma:[0,1]\to S\cap U$ such that $\gamma(0)=x_0$ and $\gamma(1)=x_1$. By approximation, the restriction of $\gamma$ to each stratum $S_i$ may be chosen to be real analytic, so $\gamma$ is piecewise real analytic.
\end{rmk}

\begin{rmk}[Application of embedded resolution of singularities for real analytic varieties and semianalytic sets]
\label{rmk:Resolution_singularities_real_analytic_varieties_semianalytic_sets}
A more modern approach to understanding the local structure of real analytic varieties or semianalytic sets is to apply the Embedded Resolution of Singularities Theorem for real analytic varieties due to Hironaka \cite{Hironaka_1964-I-II} (with a simplified proof by Bierstone and Milman \cite{Bierstone_Milman_1997}) and versions for subanalytic (and thus semianalytic) sets in Hironaka \cite{Hironaka_intro_real-analytic_sets_maps} (see Bierstone and Milman \cite[Theorems 0.1 and 0.2]{BierstoneMilman}. See Feehan \cite{Feehan_lojasiewicz_inequality_all_dimensions} for a description of this approach, its application to a simplified proof of the {\L}ojasiewicz inequalities, and additional references.
\end{rmk}

\subsection{Regularity for weak Yang--Mills connections}
\label{subsec:Regularity_weak_Yang-Mills_connections}
We shall need the

\begin{thm}[Regularity for weak Yang--Mills connections]
\label{thm:Regularity_weakly_Yang-Mills_W1p_connection}
(See Wehrheim \cite[Theorem 9.4 (i)]{Wehrheim_2004}.)  
Let $G$ be a compact Lie group, $P$ be a smooth principal $G$-bundle over a closed, smooth Riemannian manifold $(X,g)$ of dimension $d \geq 2$, and\footnote{We have strengthened Wehrheim's hypothesis on $p$ to ensure that $A$ is in $W^{1,2}\cap L^4$; see Definition \ref{defn:Admissible_Sobolev_exponent_for_Yang-Mills_energy}.}  $p \in (d/2,\infty)$ be as in Definition \ref{defn:Admissible_Sobolev_exponent_for_Yang-Mills_energy}. If $A$ is a $W^{1,p}$ weak Yang--Mills connection on $P$, then there is a $W^{2,p}$ gauge transformation $u \in \Aut^{2,p}(P)$ such that $u(A)$ is $C^\infty$ connection on $P$.
\end{thm}

See also Feehan and Maridakis \cite[Theorem 2.11.1]{Feehan_Maridakis_Lojasiewicz-Simon_coupled_Yang-Mills} or  Uhlenbeck \cite[p. 33]{UhlLp} for regularity results similar to Theorem \ref{thm:Regularity_weakly_Yang-Mills_W1p_connection}. When $A$ is in Coulomb gauge with respect to a smooth reference connection, then Theorem \ref{thm:Regularity_weakly_Yang-Mills_W1p_connection} may be improved.

\begin{thm}[Regularity for weak Yang--Mills connections in Coulomb gauge]
\label{thm:Regularity_weakly_Yang-Mills_W1dover2_connection}
(See Rivi{\`e}re \cite[Theorem VI.7]{Riviere_2015arxiv} for the case $d=2,3,4$ and $p=2$ and Uhlenbeck \cite[Corollary 1.4]{UhlLp} for the case $d\geq 2$ and $p\geq d/2$.)  
Let $G$ be a compact Lie group, $P$ be a smooth principal $G$-bundle over a closed, smooth Riemannian manifold $(X,g)$ of dimension $d \geq 2$, and $p \in (1,\infty)$, and $A_0$ be a $C^\infty$ connection on $P$. Let $A$ be a connection on $P$ that is in\footnote{See Definition \ref{defn:Admissible_Sobolev_exponent_for_Yang-Mills_energy}.} $W^{1,2}$ when $d=2,3,4$ or $W^{1,p}$ for $p\geq d/2$ when $d\geq 8$ or $p\geq 4d/(d+4)$ when $d=5,6,7$. If $A$ is a weak Yang--Mills connection that is in Coulomb gauge with respect to $A_0$, so $d_{A_0}^*(A-A_0) = 0$, then $A$ is $C^\infty$.
\end{thm}

A weaker version of Theorem \ref{thm:Regularity_weakly_Yang-Mills_W1dover2_connection} was proved by the author as \cite[Proposition 3.10]{Feehan_yangmillsenergy_lojasiewicz4d_v1}.

\begin{rmk}[Regularity for weak Yang--Mills connections in Coulomb gauge]
\label{rmk:Regularity_weakly_Yang-Mills_W12_L4_connection}
If $A$ is a weak Yang--Mills connection in Coulomb gauge with respect to $A_0$, then $a = A-A_0 \in W^{1,2}\cap L^4(X;T^*X\otimes\ad P)$ is a weak solution to the second-order quasilinear elliptic partial differential system (see Section \ref{subsec:Yang-Mills_heat_equation}),
\[
  \Delta_{A_0}a + a\times \nabla_{A_0}a + F_{A_0} = 0.
\]
Because $F_{A_0} \in C^\infty(X;\wedge^2(T^*X)\otimes\ad P)$ and, using the continuous Sobolev multiplication, $L^4(X)\times L^2(X) \to L^{4/3}(X)$, we have $a\times \nabla_{A_0}a \in L^{4/3}(X;\wedge^2(T^*X)\otimes\ad P)$. Regularity theory for weak solutions to second-order linear elliptic scalar partial differential equations (for example, Gilbarg and Trudinger \cite[Chapters 8 and 9]{GT}) would suggest that $a \in W^{2,\frac{4}{3}}(X;T^*X\otimes\ad P)$. Elliptic bootstrapping should then imply that $A$ is $C^\infty$.
\end{rmk}  

\subsection{Strong compactness for moduli spaces of Yang--Mills connections}
\label{subsec:Strong_compactness_set_Yang-Mills_connections}
For $p$ as in Definition \ref{defn:Admissible_Sobolev_exponent_for_Yang-Mills_energy} and $r \in [1,p]$, it is convenient to define
\begin{equation}
\label{eq:Quotient_space_W1p_Yang-Mills_connections_Lr_bound_curvature}  
\Crit_b^{1,p}(P,g,r) := \left\{[A] \in \sB^{1,p}(P): \YM'(A) = 0 \text{ \emph{and} } \|F_A\|_{L^r(X)} \leq b\right\}.
\end{equation}
We then have

\begin{thm}[Strong compactness for the moduli space of Yang--Mills connections with a uniform $L^r$ bound on curvature]
\label{thm:Strong_compactness_set_Yang-Mills_connections_uniform_Lp_bound_curvature}
(See Wehrheim \cite[Theorem 10.1]{Wehrheim_2004}.)   
Let $G$ be a compact Lie group, $P$ be a smooth principal $G$-bundle over a smooth Riemannian manifold $(X,g)$ of dimension $d \geq 2$, and $p \in (d/2,\infty)$ be as in Definition \ref{defn:Admissible_Sobolev_exponent_for_Yang-Mills_energy}, and $r>d/2$. If $b \in (0,\infty)$, then $\Crit_b^{1,p}(P,g,r)$ is a compact subset of $\sB^{1,p}(P)$.
\end{thm}

\begin{rmk}[Strong compactness for moduli spaces of Yang--Mills connections]
\label{rmk:Strong_compactness_set_Yang-Mills_connections_uniform_Ldover2_bound_curvature_d_geq_5} 
Theorem \ref{thm:Strong_compactness_set_Yang-Mills_connections_uniform_Lp_bound_curvature} is implied by \cite[Theorem 3.1]{Feehan_yangmillsenergy_lojasiewicz4d_v1} when $d=4$, building on earlier results due to Sedlacek \cite[Theorem 3.1 and Lemma 3.4]{Sedlacek}, Taubes \cite[Proposition 5.1]{TauFrame}, \cite[Proposition 4.4]{TauPath}, and Uhlenbeck \cite[Theorem 1.5 or 3.6]{UhlLp}. For $d \geq 5$, Zhang \cite[Main Theorem]{Zhang_2004cmb} proved that Theorem \ref{thm:Strong_compactness_set_Yang-Mills_connections_uniform_Lp_bound_curvature} holds even when $p=d/2$. 
\end{rmk}

When $b$ in Theorem \ref{thm:Strong_compactness_set_Yang-Mills_connections_uniform_Lp_bound_curvature} is sufficiently small, one obtains the following improvement.

\begin{cor}[Strong compactness for the moduli space of Yang--Mills connections with a uniform small $L^{d/2}$ bound on curvature]
\label{cor:Strong_compactness_set_Yang-Mills_connections_uniform_small_Ldover2_bound_curvature}
Let $G$ be a compact Lie group and $P$ be a smooth principal $G$-bundle over a smooth Riemannian manifold $(X,g)$ of dimension $d \geq 2$. Then there is a constant $\eps=\eps(g,G)\in(0,1]$ with the following significance. If $p \in (d/2,\infty)$ is as in Definition \ref{defn:Admissible_Sobolev_exponent_for_Yang-Mills_energy} and $r=2$ for $2\leq d\leq 4$ or $r=d/2$ for $d \geq 5$, then $\Crit_\eps^{1,p}(P,g,r)$ is a compact subset of $\sB^{1,p}(P)$.
\end{cor}

\begin{proof}
When $d=2,3$ and $r=2$, the conclusion is an immediate consequence of Theorem \ref{thm:Strong_compactness_set_Yang-Mills_connections_uniform_Lp_bound_curvature} and if $d=4$ and $r=2$, the conclusion is an immediate consequence of Feehan \cite[Theorem 3.1]{Feehan_yangmillsenergy_lojasiewicz4d_v1}. If $d\geq 5$ and $r=d/2$, the conclusion is due to Zhang \cite[Main Theorem]{Zhang_2004cmb}.
\end{proof}

\subsection{Existence of local Coulomb gauges}
\label{subsec:Existence_local_Coulomb_gauges}
We recall the original statement of Uhlenbeck's Theorem \cite{UhlLp} on existence of local Coulomb gauges (with a clarification due to Wehrheim \cite{Wehrheim_2004}), together with two extensions proved by the author in
\cite{Feehan_lojasiewicz_inequality_ground_state}.  

\begin{thm}[Existence of a local Coulomb gauge and \apriori estimate for a Sobolev connection with $L^{d/2}$-small curvature]
\label{thm:Uhlenbeck_Lp_1-3}
(See Uhlenbeck \cite[Theorem 1.3 or Theorem 2.1 and Corollary 2.2]{UhlLp}; compare Wehrheim \cite[Theorem 6.1]{Wehrheim_2004}.)
Let $d\geq 2$, and $G$ be a compact Lie group, and $p \in (1,\infty)$ obeying $d/2 \leq p < d$ and $s_0>1$ be constants. Then there are constants, $\eps=\eps(d,G,p,s_0) \in (0,1]$ and $C=C(d,G,p,s_0) \in [1,\infty)$, with the following significance. For $q \in [p,\infty)$, let $A$ be a $W^{1,q}$ connection on $B\times G$ such that
\begin{equation}
\label{eq:Ldover2_ball_curvature_small}
\|F_A\|_{L^{s_0}(B)} \leq \eps,
\end{equation}
where $B \subset \RR^d$ is the unit ball with center at the origin and $s_0=d/2$ when $d\geq 3$ and $s_0 > 1$ when $d=2$. Then there is a $W^{2,q}$ gauge transformation, $u:B\to G$, such that the following holds. If $A = \Theta + a$, where $\Theta$ is the product connection on $B\times G$, and $u(A) = \Theta + u^{-1}au + u^{-1}du$, then
\begin{align*}
d^*(u(A) - \Theta) &= 0 \quad \text{a.e. on } B,
\\
(u(A) - \Theta)(\vec n) &= 0 \quad \text{on } \partial B,
\end{align*}
where $\vec n$ is the outward-pointing unit normal vector field on $\partial B$, and
\begin{equation}
\label{eq:Uhlenbeck_1-3_W1p_norm_connection_one-form_leq_constant_Lp_norm_curvature}
\|u(A) - \Theta\|_{W^{1,p}(B)} \leq C\|F_A\|_{L^p(B)}.
\end{equation}
\end{thm}

\begin{rmk}[Restriction of $p$ to the range $1<p<\infty$]
  \label{rmk:Restriction_p_range}
(See Feehan \cite[Remark 2.4]{Feehan_lojasiewicz_inequality_ground_state}.)  
The restriction $p\in(1,\infty)$ should be included in the statements of \cite[Theorem 1.3 or Theorem 2.1 and Corollary 2.2]{UhlLp} since the bound \eqref{eq:Uhlenbeck_1-3_W1p_norm_connection_one-form_leq_constant_Lp_norm_curvature} ultimately follows from an \apriori $L^p$ estimate for an elliptic system that is apparently only valid when $1<p<\infty$. Wehrheim makes a similar observation in her \cite[Remark 6.2 (d)]{Wehrheim_2004}. This is also the reason that when $d=2$, we require $s_0>1$ in \eqref{eq:Ldover2_ball_curvature_small}.
\end{rmk}

\begin{rmk}[Dependencies of the constants in Theorem \ref{thm:Uhlenbeck_Lp_1-3}]
\label{rmk:Uhlenbeck_Lp_1-3_constant_dependencies}
(See Feehan \cite[Remark 4.2]{Feehan_yangmillsenergygapflat}.)
The statements of \cite[Theorem 1.3 or Theorem 2.1 and Corollary 2.2]{UhlLp} imply that the constants, $\eps$ in \eqref{eq:Ldover2_ball_curvature_small} and $C$ in \eqref{eq:Uhlenbeck_1-3_W1p_norm_connection_one-form_leq_constant_Lp_norm_curvature}, only depend the dimension, $d$. However, their proofs suggest that these constants may also depend on $G$ and $p$ through the appeal to an elliptic estimate for $d+d^*$ in the verification of \cite[Lemma 2.4]{UhlLp} and arguments immediately following.
\end{rmk}

\begin{rmk}[Construction of a $W^{k+1,q}$ transformation to Coulomb gauge]
\label{rmk:Uhlenbeck_theorem_1-3_Wkp}
(See Feehan \cite[Remark 4.3]{Feehan_yangmillsenergygapflat}.)
We note that if $A$ is of class $W^{k,q}$, for an integer $k \geq 1$ and $q \geq 2$, then the gauge transformation, $u$, in Theorem \ref{thm:Uhlenbeck_Lp_1-3} is of class $W^{k+1,q}$; see \cite[page 32]{UhlLp}, the proof of \cite[Lemma 2.7]{UhlLp} via the Implicit Function Theorem for smooth functions on Banach spaces, and our proof of \cite[Theorem 1.1]{FeehanSlice} --- a global version of Theorem \ref{thm:Uhlenbeck_Lp_1-3}.
\end{rmk}

\begin{rmk}[Non-flat Riemannian metrics]
\label{rmk:Non-flat_Riemannian_metrics_local_Coulomb_gauge}
(See Feehan \cite[Remark 2.7]{Feehan_lojasiewicz_inequality_ground_state}.)    
Theorem \ref{thm:Uhlenbeck_Lp_1-3} continues to hold for geodesic unit balls in a manifold $X$ endowed with a non-flat Riemannian metric $g$. The only difference in this more general situation is that the constants $C$ and $\eps$ will depend on bounds on the Riemann curvature tensor, $\Riem$. See Wehrheim \cite[Theorem 6.1]{Wehrheim_2004}.
\end{rmk}

We now recall an extension of Theorem \ref{thm:Uhlenbeck_Lp_1-3} to include the range $1 < p < d/2$.

\begin{cor}[Existence of a local Coulomb gauge and \apriori $W^{1,p}$ estimate for a Sobolev connection with $L^{d/2}$-small curvature when $p < d/2$]
\label{cor:Uhlenbeck_theorem_1-3_p_lessthan_dover2}
(See Feehan \cite[Corollary 2.8]{Feehan_lojasiewicz_inequality_ground_state}.)  
Assume the hypotheses of Theorem \ref{thm:Uhlenbeck_Lp_1-3}, but allow any $p \in (1,\infty)$ obeying $p < d/2$ when $d \geq 3$. Then the estimate \eqref{eq:Uhlenbeck_1-3_W1p_norm_connection_one-form_leq_constant_Lp_norm_curvature} holds for $1 < p < d/2$.
\end{cor}

For completeness, we also recall the following extension of Theorem \ref{thm:Uhlenbeck_Lp_1-3} (and slight improvement of our \cite[Corollary 4.4]{Feehan_yangmillsenergygapflat}) to include the range $d \leq p < \infty$, although this extension will not be needed in this monograph.

\begin{cor}[Existence of a local Coulomb gauge and \apriori $W^{1,p}$ estimate for a Sobolev connection one-form with $L^{\bar p}$-small curvature when $p \geq d$]
\label{cor:Uhlenbeck_theorem_1-3_p_geq_d}
(See Feehan \cite[Corollary 2.9]{Feehan_lojasiewicz_inequality_ground_state}.)    
Assume the hypotheses of Theorem \ref{thm:Uhlenbeck_Lp_1-3}, but consider $d \leq p < \infty$ and strengthen \eqref{eq:Ldover2_ball_curvature_small} to\footnote{In \cite[Corollary 4.4]{Feehan_yangmillsenergygapflat}, we assumed the still stronger condition, $\|F_A\|_{L^p(B)} \leq \eps$. }
\begin{equation}
\label{eq:Lbarp_ball_curvature_small}
\|F_A\|_{L^{\bar p}(B)} \leq \eps,
\end{equation}
where $\bar p = dp(d+p)$ when $p>d$ and $\bar p > d/2$ when $p=d$. Then the estimate \eqref{eq:Uhlenbeck_1-3_W1p_norm_connection_one-form_leq_constant_Lp_norm_curvature} holds for $d \leq p < \infty$ and constant $C = C(d,p,\bar p,G) \in [1,\infty)$.
\end{cor}

\subsection{\Apriori estimates for smooth Yang--Mills connections}
\label{subsec:Apriori_estimate_smooth_Yang--Mills_connection}
Let $B_r(x_0) = \{x\in\RR^d: \|x-x_0\| < r\}$ denote the open ball with center at $x_0 \in \RR^d$ and radius $r>0$ and write $B_r=B_r(0)$ when $x_0$ is the origin. We recall the following \apriori estimate due to Uhlenbeck.

\begin{thm}[\Apriori estimate for the curvature of a smooth Yang--Mills connection]
\label{thm:Uhlenbeck_3-5}
(See Uhlenbeck \cite[Theorem 3.5]{UhlRem}.)
If $d\geq 3$ is an integer and $G$ is a compact Lie group, then there are constants $C=C(d,G) \in [1,\infty)$ and $\eps=\eps(d,G) \in (0,1]$ with the following significance. Let $\rho>0$ be a constant and $A$ be a smooth Yang--Mills connection with respect to the standard Euclidean metric on a smooth principal $G$-bundle over $B_{2\rho}(0)$. If
\begin{equation}
\label{eq:Uhlenbeck_3-5_FA_Ld_over2_small_ball}
\int_{B_{2\rho}(0))}|F_A|^{d/2}\,d\vol \leq \eps,  
\end{equation}
then, for all $B_r(x_0) \subset B_\rho(0)$,
\begin{equation}
\label{eq:Uhlenbeck_3-5_Linfty_norm_FA_leq_constant_L2_norm_FA_ball}
|F_A(x_0)|^2 \leq Cr^{-d}\int_{B_r(x_0)}|F_A|^2\,d\vol.
\end{equation}
\end{thm}

\begin{rmk}[Restriction on the dimension $d$ in Theorem \ref{thm:Uhlenbeck_3-5} to be greater than or equal to three]
\label{rmk:Restriction_dimension_d}  
The restriction $d \geq 3$ in Theorem \ref{thm:Uhlenbeck_3-5} is not explicitly stated by Uhlenbeck in her \cite[Theorem 3.5]{UhlRem} (although it does appear in her \cite[Corollary 2.9]{UhlRem}). However, the condition $d \geq 3$ can be inferred from Uhlenbeck's proof of \cite[Theorem 3.5]{UhlRem}, in particular through her proof of the required \cite[Lemma 3.3]{UhlRem}, where the exponent $\nu=2d/(d-2)$ is undefined when $d=2$. The restriction $d \geq 3$ also appears in Sibner's proof of her \apriori $L^\infty$ estimate for $|F_A|$ in \cite[Proposition 1.1]{Sibner_1984}, where the necessity of the condition appears in her definition \cite[p. 94]{Sibner_1984} of the positive constant $\gamma_1 := (2d-4)/(d^2C_d)$, with $C_d$ denoting a Sobolev embedding constant in dimension $d$. When $d=2$, the proof of \cite[Theorem 4.1]{Smith_1990} due to Smith implies an \apriori $L^p$ estimate for $|F_A|$ (for $1\leq p<\infty$) that is sufficient for the purposes of this monograph; see Feehan \cite[Lemma A.8]{Feehan_yangmillsenergygapflat}.
\end{rmk}  

\begin{rmk}[Non-flat Riemannian metrics]
\label{rmk:Non-flat_Riemannian_metrics_supremum_norm_estimate_curvature_Yang-Mills_connection}  
As Uhlenbeck notes in \cite[Section 3, first paragraph]{UhlRem}, Theorem \ref{thm:Uhlenbeck_3-5} continues to hold for geodesic balls in a manifold $X$ endowed with a non-flat Riemannian metric $g$. The only difference in this more general situation is that the constants $K$ and $\eps$ will depend on bounds on the Riemann curvature tensor $\Riem$ over $B_{2\rho}(x_0)$ and the injectivity radius at $x_0\in X$.
\end{rmk}

\begin{rmk}[Intepretation of the estimates in Theorems \ref{thm:Uhlenbeck_3-5}, \ref{thm:Monotonicity_formula}, and \ref{thm:Uhlenbeck_3-5L2} in terms of Morrey norms]
\label{rmk:Morrey_norms}
The estimate \eqref{eq:Uhlenbeck_3-5_Linfty_norm_FA_leq_constant_L2_norm_FA_ball} provided by Theorem \ref{thm:Uhlenbeck_3-5}, along with those of the forthcoming Theorems \ref{thm:Monotonicity_formula} and \ref{thm:Uhlenbeck_3-5L2}, can be usefully understood in terms of \emph{Morrey spaces} --- see Giaquinta \cite[Chapter III]{Giaquinta_1983}, Smith and Uhlenbeck \cite[Appendix A]{Smith_Uhlenbeck_2022}, Tao and Tian \cite[Section 3]{Tao_Tian_2004}, or Troianiello \cite[Section 1.4.2]{Troianiello} for details and discussion. Following Giaquinta \cite[Definition III.1.1]{Giaquinta_1983}, let $d\geq 2$ be an integer, $\Omega \subset \RR^d$ be a bounded connected open subset and, for $r >0$ and $x_0\in\Omega$, let
\[
  \Omega(x_0,r) := \Omega\cap B(x_0,r) \quad\text{and}\quad \diam\Omega := \sup\{|x-y|: x,y \in \Omega\}.
\]
For constants $p\geq 1$ and $\lambda \geq 0$, we let $L^{p,\lambda}(\Omega)$ denote the \emph{Morrey space} of all functions $u\in L^p(\Omega)$ such that $\|u\|_{L^{p,\lambda}(\Omega)} < \infty$, where
\[
  \|u\|_{L^{p,\lambda}(\Omega)}^p
  :=
  \sup_{\begin{subarray}{c}0<r<\diam\Omega,\\ x\in\Omega\end{subarray}}
  r^{-\lambda} \int_{\Omega(x,r)}|u(y)|^p\,dy.
\]
The expression for $\|u\|_{L^{p,\lambda}(\Omega)}$ defines a norm with respect to which $L^{p,\lambda}(\Omega)$ is a Banach space. When $\lambda = 0$, the Morrey space coincides with $L^p(\Omega)$, while if $\lambda=d$, the Morrey space is equivalent to $L^\infty(\Omega)$ (see Giaquinta \cite[Proposition III.1.1]{Giaquinta_1983}).
\end{rmk}  

\begin{thm}[Monotonicity formula]
\label{thm:Monotonicity_formula}  
(See Price \cite[Theorem 1]{Price_1983}; compare Nakajima \cite[Lemma 3.7]{Nakajima_1987}, \cite[Fact 2.2]{Nakajima_1988}, and Tian \cite[Theorem 2.1.2]{TianGTCalGeom}.)
If $d\geq 4$ is an integer and $\Lambda \in (0,\infty)$ is a constant and $G$ is a compact Lie group, then there is a constant $C=C(d,G,\Lambda) \in [1,\infty)$ with the following significance. Let $\varrho \in (0,1]$ be a constant, and $g$ be a smooth Riemannian metric on $B_\varrho \subset \RR^d$ whose components $g_{ij}=g(\partial/\partial x_i,\partial/\partial x_j)$ for $1\leq i,j\leq d$ obey
\begin{equation}
\label{eq:Riemannian_metric_compare_Euclidean}  
|\delta_{ij}- g_{ij}(x)| \leq \Lambda r^2, \quad \left|\frac{\partial g_{ij}}{\partial x_k}(x)\right| \leq \Lambda r,
\quad \left|\frac{\partial^2 g_{ij}}{\partial x_k\partial x_l}(x)\right| \leq \Lambda, \quad\text{for } r=|x| \in (0,\varrho].
\end{equation}
If $A$ is a smooth Yang--Mills connection with respect to the Riemannian metric $g$ on a smooth principal $G$-bundle over $B_\varrho$, then
\begin{equation}
\label{eq:Monotonicity_formula}
\sigma^{4-d}\int_{B_\sigma}|F_A|^2\,d\vol_g \leq C\rho^{4-d}\int_{B_\rho}|F_A|^2\,d\vol_g, \quad 0 < \sigma \leq \rho \leq \varrho.
\end{equation}
\end{thm}

The following result extends Theorem \ref{thm:Uhlenbeck_3-5} by weakening (when $d\geq 5$) the hypothesis \eqref{eq:Uhlenbeck_3-5_FA_Ld_over2_small_ball} that $F_A$ be $L^{d/2}$ small to a hypothesis that $F_A$ be $L^2$ small.

\begin{thm}[Improved \apriori interior estimate for the curvature of a smooth Yang--Mills connection]
\label{thm:Uhlenbeck_3-5L2}
(See Nakajima \cite[Lemma 3.1]{Nakajima_1988}; compare Meyer and Rivi{\`e}re \cite[Theorem 1.2]{Meyer_Riviere_2003}, Naber and Valtorta \cite[Proposition 2.5]{Naber_Valtorta_2019}, and Tian \cite[Theorem 2.2.1]{TianGTCalGeom}.)
Assume the hypotheses of Theorem \ref{thm:Monotonicity_formula}. Then there are constants $C=C(d,G,\Lambda) \in [1,\infty)$ and $\eps=\eps(d,G,\Lambda) \in (0,1]$ with the following significance. If 
\begin{equation}
\label{eq:Radius_power_FA_L2small}
r^{4-d}\int_{B_r}|F_A|^2\,d\vol_g \leq \eps, 
\end{equation}
for some $r \in (0,\varrho]$, then
\begin{equation}
\label{eq:Supremum_norm_FA_Brover4_leq_constant_r4_L2_norm_FA_Br}
\sup_{B_{r/4}}|F_A| \leq \frac{C}{r^2}\left(r^{4-d}\int_{B_r}|F_A|^2\,d\vol_g\right)^{1/2}.
\end{equation}
\end{thm}

\begin{rmk}[Stationary connections]
\label{rmk:Stationary_connections}
Let $A_0$ be a $C^\infty$ connection on $P$. A connection $A \in A_0 + W^{1,2}\cap L^4(X;T^*X\otimes\ad P)$ on $P$ is called \emph{stationary} if \cite[Equation (4.5.12)]{TianGTCalGeom} if for any geodesic ball $B\subset X$ and vector field $\xi \in C_0^\infty(B,TX)$ with compact support in $B$,
\begin{equation}
\label{eq:Stationary_connection}
\int_X\left(|F_A|^2\divg\xi - 4\sum_{i,j=1}^d \left\langle F_A\left(\nabla_{e_j}\xi,e_j\right), F_A(e_i,e_j)\right\rangle\right)\,d\vol_g = 0,  
\end{equation}
for any local orthonormal frame $\{e_i\}$ for $TX$ over $B$. If $A$ is a Yang--Mills connection on $P$ and $A$ is smooth (see Tian \cite[p. 249]{TianGTCalGeom}) or $d=4$ (see Tao and Tian \cite[p. 558]{Tao_Tian_2004}), then the first variation formula (see Price \cite[p. 146]{Price_1983}) for the Yang--Mills energy function \eqref{eq:Yang-Mills_energy_function}, implies that $A$ is stationary in the sense of \eqref{eq:Stationary_connection}. See also Meyer and Rivi{\`e}re \cite[Definition 1.2]{Meyer_Riviere_2003}. In particular, by Price \cite[Theorem 1 and $1''$]{Price_1983}, a $W^{1,2}\cap L^4$ stationary Yang--Mills connection $A$ obeys the monotonicity formula \eqref{eq:Monotonicity_formula}. If a $W^{1,2}\cap L^4$ stationary Yang--Mills connection $A$ is also \emph{approximable} in the sense of Meyer and Rivi{\`e}re \cite[Equation (1.7)]{Meyer_Riviere_2003} and obeys the hypothesis \eqref{eq:Radius_power_FA_L2small}, then $A$ satisfies the estimate \eqref{eq:Supremum_norm_FA_Brover4_leq_constant_r4_L2_norm_FA_Br} by \cite[Theorem 1.2]{Meyer_Riviere_2003}. Recall that $A$ is an \emph{admissible} Yang--Mills connection over $[-1,2]^d\subset\RR^d$ (see Tian \cite[Section 2.3]{TianGTCalGeom}) if it is a smooth Yang-Mills connection outside a closed subset $S \subset [-1,2]^d$ of finite $(d-4)$-dimensional Hausdorff measure and $\int_{[-1,2]^d}|F_A|^2\,d\vol < \infty$. According to Tao and Tian \cite[Theorem 1.1]{Tao_Tian_2004}, there is a constant $\eps=\eps(d,G)\in(0,1]$ such that if $A$ is an admissible stationary Yang--Mills connection on $[-1,2]^d\times G$ that obeys $\int_{[-1,2]^d}|F_A|^2\,d\vol \leq \eps$, then there is a gauge transformation $u$ of $[0,1]^d\times G$ such that $u(A)$ extends to a smooth connection on $[0,1]^d$ and there are constants $C_j=C_j(d,G)\in[1,\infty)$ such that $\|\nabla^j(u(A)-\Theta)\|_{L^\infty([0,1]^d)} \leq C_j\eps$, for all integers $j\geq 0$. See Smith and Uhlenbeck \cite{Smith_Uhlenbeck_2022} for related results.
\end{rmk}  

By employing a finite cover of $X$ by geodesic balls $B_\rho(x_i)$ of radius $\rho \in (0, \Inj(X,g)/4]$ and applying Theorem \ref{thm:Uhlenbeck_3-5L2} to each ball $B_{2\rho}(x_i)$, we obtain the following global version that extends \cite[Corollary 4.6]{Feehan_yangmillsenergygapflat} by weakening (when $d\geq 5$) the hypothesis that $F_A$ be $L^{d/2}$ small to the hypothesis that $F_A$ be $L^2$ small.

\begin{cor}[\Apriori estimate for the curvature of a Yang--Mills connection over a closed manifold]
\label{cor:Uhlenbeck_3-5_manifoldL2}
Let $(X,g)$ be a closed, smooth Riemannian manifold of dimension $d\geq 3$ and $G$ be a compact Lie group. Then there are constants, $C=C(d,g,G) \in [1,\infty)$ and $\eps=\eps(d,g,G) \in (0,1]$, with the following significance. If $A$ is a smooth Yang--Mills connection with respect to the metric $g$ on a smooth principal $G$-bundle $P$ over $X$ that obeys
\begin{equation}
\label{cor:Uhlenbeck_3-5_FA_L2_small_manifold}
\|F_A\|_{L^2(X)} \leq \eps,
\end{equation}
then
\begin{equation}
\label{cor:Uhlenbeck_3-5_Linfty_norm_FA_leq_constant_L2_norm_FA_manifold}
\|F_A\|_{L^{\infty}(X)} \leq C\|F_A\|_{L^2(X)}.
\end{equation}
\end{cor}

\subsection{Analyticity of critical sets of the Yang--Mills energy function}
\label{subsec:Semianalyticity_critical_set_Yang-Mills_energy_function}
We shall need the

\begin{prop}[Analyticity of the set of Yang--Mills connections with a uniform $L^p$ bound on curvature]
\label{prop:Semianalyticity_set_Yang-Mills_connections_uniform_Lp_bound_curvature}  
Let $G$ be a compact Lie group, $P$ be a smooth principal $G$-bundle over a smooth Riemannian manifold $(X,g)$ of dimension $d \geq 2$, and $p \in (d/2,\infty)$, and $A_1$ and $A_\infty$ be $C^\infty$ connections on $P$, and $A_\infty$ is a Yang--Mills connection. If $\YM'(A)$ is given by \eqref{eq:Derivative_Yang-Mills_energy}, then
\begin{multline}
\label{eq:Coulomb_gauge_Yang-Mills_connections_uniform_Lp_bound_curvature}    
\mathbf{Crit}_b^{1,p}(P,g,A_\infty,\zeta) := A_\infty+\left\{a \in \Ker d_{A_\infty}^*\cap W_{A_1}^{1,p}(X;T^*X\otimes\ad P):  \right.
\\
\left. \YM'(A_\infty+a) = 0 \text{ \emph{and} } \|a\|_{W_{A_1}^{1,p}(X)} < \zeta\right\}
\end{multline}
is a real analytic subvariety of an open ball in Euclidean space, where $\zeta = \zeta(A_\infty,A_1,g,p) \in (0,1]$.
\end{prop}

\begin{rmk}[Global version of Proposition \ref{prop:Semianalyticity_set_Yang-Mills_connections_uniform_Lp_bound_curvature}]
\label{rmk:Semianalyticity_set_Yang-Mills_connections_uniform_Lp_bound_curvature_quotient_space}  
One could formulate a global version of Proposition \ref{prop:Semianalyticity_set_Yang-Mills_connections_uniform_Lp_bound_curvature} by viewing the Yang--Mills energy $\YM$ as a function on the quotient space of connections, $\sB^{1,p}(P)$, or the quotient space of based connections, $\sB(P,x_0)$, but the preceding local statement will be adequate for our application.
\end{rmk}  

\begin{proof}[Proof of Proposition \ref{prop:Semianalyticity_set_Yang-Mills_connections_uniform_Lp_bound_curvature}]
  We begin by choosing a constant $\mu \in (0,\infty)$ that is not in the spectrum of the Hodge Laplace operator \eqref{eq:Lawson_page_93_Hodge_Laplacian}, namely $\Delta_{A_\infty} = d_{A_\infty}^*d_{A_\infty} + d_{A_\infty}d_{A_\infty}^*$, on $L^2(X;T^*X\otimes\ad P)$. (It is well-known that the spectrum of $\Delta_{A_\infty}$ on $L^2(X;T^*X\otimes\ad P)$ is a countable subset of $[0,\infty)$ without accumulation points and consists of eigenvalues $\{\lambda_n\}_{n\in\NN}$ with finite multiplicities equal to $\dim\Ker(\Delta_{A_\infty}-\lambda_n I)$ for $n\in\NN$; see Feehan and Maridakis \cite[Proposition 2.2.3]{Feehan_Maridakis_Lojasiewicz-Simon_coupled_Yang-Mills} for the statement and proof of a more general result.) We now solve for (weakly) Yang--Mills connections, $A = A_\infty+a$ on $P$ with $a \in \Ker d_{A_\infty}^*\cap W_{A_1}^{1,p}(X;T^*X\otimes\ad P)$, that are near $A_\infty$ by considering the quasilinear second-order elliptic equation,
\[
  d_{A_\infty+a}^*F_{A_\infty+a} + d_{A_\infty}d_{A_\infty}^*a = 0 \quad\text{in } W_{A_1}^{-1,p}(X;T^*X\otimes\ad P).
\]
We apply the Method of Kuranishi \cite{Kuranishi} and let
\begin{equation}
\label{eq:Projection_onto_finite_dimensional_eigenspace}  
  \Pi_{A_\infty,\mu}: \Ker d_{A_\infty}^*\cap L^2(X;T^*X\otimes\ad P) \to \Ker d_{A_\infty}^*\cap L^2(X;T^*X\otimes\ad P)
\end{equation}
be the $L^2$-orthogonal projection onto the finite-dimensional vector space
\[
  T_{A_\infty,\mu} := \Ran \Pi_{A_\infty,\mu}\cap \Ker d_{A_\infty}^*\cap W_{A_1}^{1,p}(X;T^*X\otimes\ad P)
\]
of dimension $N=N(A_\infty,g,\mu)$, spanned by the eigenvectors of the unbounded operator
\[
  \Delta_{A_\infty}: \Ker d_{A_\infty}^*\cap L^2(X;T^*X\otimes\ad P) \to \Ker d_{A_\infty}^*\cap L^2(X;T^*X\otimes\ad P)
\]
with eigenvalues less than $\mu$ and denote $\Pi_{A_\infty,\mu}^\perp := \id - \Pi_{A_\infty,\mu}$. Note that $T_{A_\infty,\mu} \subset \Omega^1(X;\ad P)$ due to elliptic regularity for $\Delta_{A_\infty}$; see \cite[Corollary 8.11]{GT} for regularity of solutions to a second-order, linear elliptic equation with $C^\infty$ coefficients and scalar principal symbol. The projection \eqref{eq:Projection_onto_finite_dimensional_eigenspace} extends as a bounded operator to
\[
  \Pi_{A_\infty,\mu}: \Ker d_{A_\infty}^*\cap W_{A_1}^{k,p}(X;T^*X\otimes\ad P) \to \Ker d_{A_\infty}^*\cap W_{A_1}^{k,p}(X;T^*X\otimes\ad P)
\]
for any $k \in \ZZ$ and $p \in (1,\infty)$. We have
\[
T_{A_\infty,\mu}^\perp = \Ran \Pi_{A_\infty,\mu}^\perp \cap \Ker d_{A_\infty}^*\cap W_{A_1}^{1,p}(X;T^*X\otimes\ad P),
\]
so that
\[
  \Ker d_{A_\infty}^*\cap W_{A_1}^{1,p}(X;T^*X\otimes\ad P) = T_{A_\infty,\mu}^\perp \oplus T_{A_\infty,\mu}.
\]
For each $a \in \Ker d_{A_\infty}^*\cap W_{A_1}^{1,p}(X;T^*X\otimes\ad P)$, we write $a = a_\perp+a_\parallel$ with $a_\perp := \Pi_{A_\infty,\mu}^\perp a$ and $a_\parallel := \Pi_{A_\infty,\mu} a$. The bounded linear operator
\begin{multline*}
  \Delta_{A_\infty} : \Ran \Pi_{A_\infty,\mu}^\perp \cap \Ker d_{A_\infty}^*\cap W_{A_1}^{1,p}(X;T^*X\otimes\ad P)
  \\
  \to \Ran \Pi_{A_\infty,\mu}^\perp \cap \Ker d_{A_\infty}^*\cap W_{A_1}^{-1,p}(X;T^*X\otimes\ad P)
\end{multline*}
is bijective and is thus an isomorphism of Banach spaces by the Open Mapping Theorem. For $\delta = \delta(A_\infty,A_1,g,p) \in (0,1]$ small enough and
\[
  U_\delta(A_\infty) := \left\{\alpha \in T_{A_\infty,\mu}: \|\alpha\|_{W_{A_1}^{1,p}(X)} < \delta\right\},
\]
and each $a_\parallel \in U_\delta(A_\infty)$, we may solve the analytic, infinite-dimensional \emph{Yang--Mills Kuranishi equation} for $a_\perp$, 
\begin{equation}
\label{eq:Yang-Mills_Kuranishi}
  \Upsilon(a_\perp,a_\parallel) := \Pi_{A_\infty,\mu}^\perp d_{A_\infty+a_\perp+a_\parallel}^*F_{A_\infty+a_\perp+a_\parallel} = 0 \quad\text{in } T_{A_\infty,\mu}^\perp,
\end{equation}
by applying the Analytic Implicit Function Theorem on Banach spaces (for example, \cite[Theorem F.1]{Feehan_Maridakis_Lojasiewicz-Simon_coupled_Yang-Mills}) and therefore define an analytic map,
\[
  \Phi: U_\delta(A_\infty) \ni a_\parallel \mapsto a_\perp := \Phi(a_\parallel) \in T_{A_\infty,\mu}^\perp.
\]
The resulting solution, $A = A_\infty+\Phi(a_\parallel)+a_\parallel$, to \eqref{eq:Yang-Mills_Kuranishi} is a Yang--Mills connection if $a_\parallel$ then solves the analytic, finite-dimensional \emph{Yang--Mills balancing equation},
\begin{equation}
\label{eq:Yang-Mills_balancing}
\chi(a_\parallel) := \Pi_{A_\infty,\mu}d_{A_\infty+\Phi(a_\parallel)+a_\parallel}^* F_{A_\infty+\Phi(a_\parallel)+a_\parallel} = 0 \quad\text{in } T_{A_\infty,\mu}.
\end{equation}
Hence, for small enough $\zeta = \zeta(A_\infty,A_1,g,p) \in (0,1]$, the set \eqref{eq:Coulomb_gauge_Yang-Mills_connections_uniform_Lp_bound_curvature} may be identified with
\[
\left\{\alpha \in U_\delta(A_\infty): \chi(\Phi(\alpha)) = 0 \text{ and } \|\Phi(\alpha)\|_{W_{A_1}^{1,p}(X)} < \zeta\right\}.
\]
Since $U_\delta(A_\infty)$ is an open ball in Euclidean space (of real dimension $N$) and $\Phi$ and $\chi$ are analytic maps, then the preceding set and thus \eqref{eq:Coulomb_gauge_Yang-Mills_connections_uniform_Lp_bound_curvature} are real analytic subvarieties of an open ball in Euclidean space.
\end{proof}  

\subsection{A proof of the Yang--Mills energy gap}
\label{subsec:Proof_energy_gap_Yang-Mills_connections}
It remains to complete the

\begin{proof}[Proof of Theorem \ref{mainthm:L2_energy_gap_Yang-Mills_connections}]
Choose $p=(d+1)/2$. According to Corollary \ref{cor:Uhlenbeck_3-5_manifoldL2}, there are constants $\eps=\eps(d,g,G)\in(0,1]$ and $C=C(d,g,G)\in[1,\infty)$ such that if $A$ is any smooth Yang--Mills connection obeying \eqref{eq:Curvature_L2_small}, that is,
\[
  \|F_A\|_{L^2(X)} \leq \eps,
\]
then by \eqref{cor:Uhlenbeck_3-5_Linfty_norm_FA_leq_constant_L2_norm_FA_manifold} one has
\[
  \|F_A\|_{L^{\infty}(X)} \leq C\|F_A\|_{L^2(X)} \leq C\eps.
\]
Hence,
\[
  \|F_A\|_{L^p(X)} \leq b := C\eps\Vol_g(X)^{1/p},
\]
and consequently, 
\[
  \Crit_\eps^{1,p}(P,g,2) = \left\{[A] \in \sB^{1,p}(P): \YM'(A) = 0 \text{ and } \|F_A\|_{L^2(X)} \leq \eps\right\} \subset \Crit_b^{1,p}(P,g,p),
\]
where $\Crit_b^{1,p}(P,g,r)$ is as in \eqref{eq:Quotient_space_W1p_Yang-Mills_connections_Lr_bound_curvature}. According to Theorem \ref{thm:Strong_compactness_set_Yang-Mills_connections_uniform_Lp_bound_curvature}, the set $\Crit_b^{1,p}(P,g,p)$ is a compact subspace of $\sB^{1,p}(P)$. For small enough $\zeta = \zeta(A_1,g,P,g,p) \in (0,1]$, Corollary \ref{cor:Slice} implies that each point $[A_\infty] \in \Crit_b^{1,p}(P,g,p)$ has an open neighborhood in $\sB^{1,p}(P)$ that is the image under the quotient map $\pi:\sA^{1,p}(P) \to \sB^{1,p}(P)$ of an open ball \eqref{eq:W1p_connections_Coulomb_gauge_ball} of the form
\[
\bB_\zeta(A_\infty) = A_\infty+\left\{a \in \Ker d_{A_\infty}^*\cap W_{A_1}^{1,p}(X;T^*X\otimes\ad P): \|a\|_{W_{A_1}^{1,p}(X)} < \zeta\right\}.
\]
By compactness, the set $\Crit_b^{1,p}(P,g,p)$ is covered by $N$ open neighborhoods of the form
\[
  \Crit_b^{1,p}(P,g,p)\cap\pi(\bB_\zeta(A_\infty)) = \pi\left(\mathbf{Crit}_b^{1,p}(P,g,A_\infty,\zeta)\right),
\]
where $\mathbf{Crit}_b^{1,p}(P,g,A_\infty,\zeta)$ is as in \eqref{eq:Coulomb_gauge_Yang-Mills_connections_uniform_Lp_bound_curvature}, for some positive integer $N=N(A_1,b,g,P,p)$. 

By Proposition \ref{prop:Semianalyticity_set_Yang-Mills_connections_uniform_Lp_bound_curvature}, each set $\mathbf{Crit}_b^{1,p}(P,g,A_\infty,\zeta)$ is an analytic subvariety of an open ball in Euclidean space. Proposition \ref{prop:Local_arc-connectedness_semianalytic_sets} ensures that we may choose the radius $\zeta=\zeta(A_1,b,g,P,p)\in(0,1]$ small enough that the sets $\mathbf{Crit}_b^{1,p}(P,g,A_\infty,\zeta)$ are piecewise-$C^1$ arc connected. Consequently, the Yang--Mills energy function is constant on each set $\mathbf{Crit}_b^{1,p}(P,g,A_\infty,\zeta)$. Therefore, $\mathbf{Crit}_b^{1,p}(P,g,A_\infty,\zeta)$ has finitely many piecewise-$C^1$ arc-connected components $C_i^\lambda$, labeled by $i\in\{1,\ldots,M\}$ and energy levels $\lambda\in \{\lambda_0,\lambda_1,\ldots,\lambda_N\}$ with $\lambda_j<\lambda_{j+1}$, for $0\leq j\leq N-1$, and $\lambda_0=0$, for integers $M\geq 1$ and $N\geq 0$. For each $j\in\{0,\ldots,N-1\}$, one has $\YM(A)=\lambda_j$ for all $A \in C_i^{\lambda_j}$ and $i\in\{1,\ldots,M\}$. If $A$ obeys $\YM(A) < \lambda_1$, then $\YM(A)=\lambda_0=0$ and $A$ is a flat connection. Recalling the definition \eqref{eq:Yang-Mills_energy_function} of $\YM$ and choosing $\eps=\sqrt{\lambda_1}$ completes the proof of Theorem \ref{mainthm:L2_energy_gap_Yang-Mills_connections}.
\end{proof}

\appendix

\section{Flat SU(2) connections over a torus and Uhlenbeck's Corollary 4.3}
\label{sec:Counterexample_corollary_4-3_Uhlenbeck_1985}
In this Appendix, we describe a counterexample (see Example \ref{exmp:Two_torus}) to the estimates stated in \cite[Corollary 4.3]{UhlChern} and \cite[Theorem 5.1]{Feehan_yangmillsenergygapflat}, based on an observation due to Mrowka \cite{Mrowka_7-30-2018}. While Example \ref{exmp:Two_torus} uses the moduli space of flat $\SU(2)$ connections over the torus $\TT^2$ to illustrate the issues in a simple setting, one should be able to construct other counterexamples using moduli spaces $M(X,G)$ of flat $G$-connections over higher-dimensional manifolds $X$ and higher-dimensional Lie groups $G$. For example, the structure of the moduli space of flat $\SU(2)$-connections over $\TT^3$, including its singularities, is discussed by Donaldson \cite[Chapter 2, Appendix A, and pp. 107--108]{DonFloer} and in more generality, when $X$ is a circle bundle over a closed, orientable Riemann surface, by Morgan, Mrowka, and Ruberman \cite[Chapter 13]{MMR} and Taubes \cite[Chapter 11]{TauL2}. See also Gompf and Mrowka \cite{GompfMrowka} for further discussion and application of examples of this kind. An early example of cubic singularities of the character variety $\Hom(\pi_1(X),G)/G$, when $X$ is a closed three-dimensional manifold, is due to Goldman and Millson \cite[Section 9]{Goldman_Millson_1988}. See also Saveliev \cite{Saveliev_2002} and references therein for additional examples. For further explorations of singularities of representation varieties, we refer the reader to Goldman and Millson \cite{Goldman_Millson_1987} and Spinaci \cite{Spinaci_2014} and the references cited therein.

Regarding the forthcoming Example \ref{exmp:Two_torus}, one might ask the

\begin{ques}
\label{ques:Improve_Mrowka_example}  
Can the forthcoming estimate \eqref{eq:W1p_distance_A_to_product_SU2_connection_torus_Lp_norm_curvature_power_half} for $A_t$ be improved in the sense of replacing $\|F_{A_t}\|_{L^p(\TT^2)}^{1/2}$ by $\|F_{A_t}\|_{L^p(\TT^2)}$ through finding
\begin{enumerate}
\item
\label{ques:Improve_Mrowka_example_flat_connection}  
A flat connection $\Gamma$ such that $\|A_t-\Gamma\|_{W^{1,p}(\TT^2)} \leq \|A_t\|_{W^{1,p}(\TT^2)}$, or
\item
\label{ques:Improve_Mrowka_example_gauge_transformation} 
A gauge transformation $u$ such that $\|u(A_t)\|_{W^{1,p}(\TT^2)} \leq \|A_t\|_{W^{1,p}(\TT^2)}$.
\end{enumerate}
\end{ques}

We explain in Section \ref{subsec:Counterexample_corollary_4-3_Uhlenbeck_1985} that Strategy \eqref{ques:Improve_Mrowka_example_flat_connection} cannot be used to improve \eqref{eq:W1p_distance_A_to_product_SU2_connection_torus_Lp_norm_curvature_power_half} using results from Section \ref{subsec:Moduli_spaces_flat_connections_Riemann_surfaces}, which builds in turn on results in Section \ref{subsec:Structure_character_variety_closed_Rieman_surface_stratified_space}. We also explain that Strategy \eqref{ques:Improve_Mrowka_example_gauge_transformation} cannot be used to improve \eqref{eq:W1p_distance_A_to_product_SU2_connection_torus_Lp_norm_curvature_power_half} using results from Section \ref{subsec:Local_minimizing_property_Coulomb_gauge_condition}.

\subsection{Moduli space of flat SU(2) connections over the two-torus and the estimate in Uhlenbeck's Corollary 4.3}
\label{subsec:Counterexample_corollary_4-3_Uhlenbeck_1985}
The following example is due to Mrowka  \cite{Mrowka_7-30-2018} and illustrates the fact that in Theorem \ref{mainthm:Uhlenbeck_Chern_corollary_4-3}, the constant $\lambda\in(0,1]$ in the {\L}ojasiewicz distance inequality \eqref{eq:Uhlenbeck_Chern_corollary_4-3_A_M(P)_W1p_distance_bound} may be strictly less than one. In Example \ref{exmp:Two_torus}, like in Nishinou \cite{Nishinou_2007}, we only consider the case of the two-dimensional torus, $\TT^2$, but the example extends to tori $\TT^d$ of any dimension $d\geq 2$. 

\begin{exmp}[Estimate for distance to moduli subspace of flat $\SU(2)$ connections over a two-dimensional torus]
\label{exmp:Two_torus}
In the notation of Theorems \ref{mainthm:Uhlenbeck_Chern_corollary_4-3_prelim} and \ref{mainthm:Uhlenbeck_Chern_corollary_4-3}, choose
\[
  G = \SU(2), \quad X = \TT^2 = \RR^2/\ZZ^2, \quad P = \TT^2\times\SU(2),
\]
identify connection one-forms on $P$ with $\su(2)$-valued one-forms on $\TT^2$, where $\su(2)$ denotes the Lie algebra of $\SU(2)$, and equip $\TT^2$ with its flat Riemannian metric. For a pair of matrices $\xi,\eta \in \su(2)$, consider the connection one-form 
\[
  A = \xi\otimes dx + \eta\otimes dy \in \Omega^1(\TT^2;\su(2)).
\]
From \eqref{eq:Bleecker_theorem_2-2-11}, we have
\[
  F_A = dA + \frac{1}{2}[A, A] = \frac{1}{2}[\xi,\eta]dx\wedge dy  \in \Omega^2(\TT^2;\su(2)),
\]
and thus $F_A = 0 \iff [\xi,\eta] = 0$. Using\footnote{From \cite[Equation (6.2)]{Warner} when $X$ has dimension $d$.} $d^* = (-1)^{d(p+1)+1}\star\,d\star$ on $\Omega^p(X;\RR)$, we note that
\[
  d^*A = -\star\,d\star A = -\star\,d(\xi\otimes dy + \eta\otimes dx) = 0,
\]
since $d^2x = 0 = d^2y$, and thus $A$ is in Coulomb gauge with respect to the product connection $\Theta$ on $P$. Recall that $\su(2)$ has basis
\begin{equation}
  \label{eq:su2_basis}
  I = \begin{pmatrix}0&i\\i&0\end{pmatrix}, \quad J = \begin{pmatrix}0&-1\\1&0\end{pmatrix}, \quad K = \begin{pmatrix}i&0\\0&-i\end{pmatrix},
\end{equation}
with relations $[I,J]=2K$, and $[J,K]=2I$, and $[K,I]=2J$. For the Lie algebra
$\su(2)$, one can take $B(\xi,\eta) = \tr(\xi\eta)$ to be the Killing form, giving $B(I,I) = B(J,J) = B(K,K) = -2$, and choose $\langle\xi,\eta\rangle := -\frac{1}{2}B(\xi,\eta)$ to be an $\Ad\SU(2)$-invariant inner product on $\su(2)$, with respect to which the basis $\{I,J,K\}$ is orthonormal.

If $\xi = tI$ and $\eta = tJ$, for a constant $t\in\RR$, and we write $A_t$ for the resulting one-parameter family of connections, then
\[
  F_{A_t} = \frac{1}{2}t^2[I,J]dx\wedge dy = t^2K dx\wedge dy,
\]
and so $|A_t| \propto |t|$ and $|F_{A_t}| \propto |t|^2$. Consequently, for any $p\in(1,\infty)$,
\begin{equation}
\label{eq:W1p_distance_A_to_product_SU2_connection_torus_Lp_norm_curvature_power_half}  
  \|A_t\|_{W^{1,p}(\TT^2)} \leq C\|F_{A_t}\|_{L^p(\TT^2)}^{1/2}, \quad\text{for all } t \in \RR,
\end{equation}
where $C=C(p)\in[1,\infty)$ is a constant\footnote{See Morgan, Mrowka, and Ruberman \cite[Lemma 13.3.2 and Remark 13.3.3]{MMR} for related calculations.}. \qed
\end{exmp}

For the family $A$ of connections parameterized by $\su(2)\times\su(2)$ in Example \ref{exmp:Two_torus}, we also have $dA = 0$ and so
\[
  A \in \bH_\Theta^1(\TT^2;\su(2)) = \Ker(d+d^*)\cap \Omega^1(\TT^2;\su(2)) \cong H^1(\TT^2;\RR)\otimes\su(2) \cong \RR^2\otimes\su(2),
\]
where $\bH_\Theta^1(\TT^2;\su(2))$ is the Zariski tangent space at $\Theta$ to $M(\TT^2,\SU(2))$ and, by dimension-counting every element of $\bH_\Theta^1(\TT^2;\su(2))$ has this form. Furthermore, $[\Theta]$ is not a regular point of $M(\TT^2,\SU(2)$ in the sense of \eqref{eq:Regular_flat_connection} because
\[
  \bH_\Theta^2(\TT^2;\ad P) = \Ker(d+d^*)\cap \Omega^2(\TT^2;\ad P) \cong H^2(\TT^2;\RR)\otimes\su(2) \cong \su(2).
\]
As an aside, we note that the virtual dimension $s$ of $M(\TT^2,\SU(2))$ is equal to zero since
\begin{multline*}
  s := \Ind\left(d+d^*:\Omega^1(\TT^2;\su(2)) \to \Omega^2(\TT^2;\su(2))\oplus \Omega^0(\TT^2;\su(2))\right)
  \\
  = \dim \bH_\Theta^1(\TT^2;\ad P) - \dim \bH_\Theta^2(\TT^2;\ad P) - \dim \bH_\Theta^0(\TT^2;\ad P) = 6-3-3 = 0.
\end{multline*} 
We now discuss the two approaches to potentially improving \eqref{eq:W1p_distance_A_to_product_SU2_connection_torus_Lp_norm_curvature_power_half} described in Question \ref{ques:Improve_Mrowka_example}.

\subsubsection{Replacement of the product connection $\Theta$ by a flat connection $\Gamma$ that is closer to $A_t$.}
Theorem \ref{thm:Stratified-space_structure_moduli_space_SU(2)_connections_torus} describes the stratified-space structure of the moduli space of flat $\SU(2)$ connections over $\TT^2$ as the two-dimensional pillowcase (see Figure \ref{fig:Pillow}), where $[\Theta]$ represents one corner of the pillowcase. The connections $A_t$ in Example \ref{exmp:Two_torus} are closest to $\Theta$, with $\|A_t-\Gamma\|_{W^{1,p}(\TT^2)} \geq \|A_t\|_{W^{1,p}(\TT^2)}$ for any $[\Gamma] \in M(\TT^2,\SU(2))$ obeying $d^*\Gamma=0$, and so \eqref{eq:W1p_distance_A_to_product_SU2_connection_torus_Lp_norm_curvature_power_half} cannot be improved as suggested in Part \eqref{ques:Improve_Mrowka_example_flat_connection} of Question \ref{ques:Improve_Mrowka_example}. Indeed, the parameterization \eqref{eq:Pillow_case_parameterization} of the pillowcase $\Hom(\pi_1(\TT^2),\SU(2))/\SU(2)$ and Theorem \ref{thm:Stratified-space_structure_moduli_space_SU(2)_connections_torus} show that the family of flat connections
\begin{align*}
  \Gamma(\alpha,\beta) &:= \begin{pmatrix}i\alpha & 0 \\ 0 &-i\alpha\end{pmatrix}dx
                                                             + \begin{pmatrix}i\beta & 0 \\ 0 &-i\beta\end{pmatrix}dy
  \\
  &\,= \alpha K dx + \beta K dy \in \Ker d^*\cap\Omega^1(\TT^2;\su(2)), \quad\text{for all } (\alpha,\beta)\in [0,\pi]\times[0,2\pi]
\end{align*}
is a parameterization of $M(\TT^2,\SU(2))$. But then
\begin{align*}
  |A_t-\Gamma(\alpha,\beta)| &= |(tI-\alpha K)dx + (tJ-\beta K)dy|
  \\
  &= \left(t^2 + \alpha^2 + \beta^2\right)^{1/2},
\end{align*}
and so $|A_t-\Gamma(\alpha,\beta)| \geq |A_t|$, with equality if and only $(\alpha,\beta)=(0,0)$ and $\Gamma(0,0)=\Theta$.

\subsubsection{Replacement of the connection $A_t$ by a gauge-transformed connection $u(A_t)$.}
Corollary \ref{cor:Local_minimizing_property_Coulomb_gauge_condition} implies that  $\|u(A_t)\|_{W^{1,p}(\TT^2)} \geq \|A_t\|_{W^{1,p}(\TT^2)}$ for any $u\in\Aut^{2,p}(P)$ and $t\in (-\delta,\delta)$, for small enough $\delta\in (0,1]$, and so \eqref{eq:W1p_distance_A_to_product_SU2_connection_torus_Lp_norm_curvature_power_half} cannot be improved as suggested in Part \eqref{ques:Improve_Mrowka_example_gauge_transformation}  of Question \ref{ques:Improve_Mrowka_example}.

\subsection{Morse--Bott property of the Yang--Mills energy function at regular flat connections}
\label{subsec:Regular_flat_connections_and_Morse-Bott_property_Yang-Mills_energy_function}
If $A$ is a $W^{1,p}$ flat connection on $P$, so $F_A=0$, then $\YM'(A) \equiv 0$ by \eqref{eq:Derivative_Yang-Mills_energy} and $[A]$ is a critical point of $\YM:\sB^{1,p}(P) \to \RR$ in the sense of Definition \ref{defn:Definition_critical_point_quotient_space}, so one has the (trivial) inclusion,
\begin{equation}
\label{eq:Moduli_space_flat_connections_subset_critical_points}
M(P) \subset \Crit\YM,
\end{equation}
where $\Crit\YM$ denotes the set of critical points of $\YM:\sB^{1,p}(P) \to \RR$.

Conversely, suppose $[A] \in \Crit\YM$ and that $A$ is $C^\infty$ (an assumption that involves no loss of generality by Theorem \ref{rmk:Regularity_weakly_Yang-Mills_W12_L4_connection}). The Bianchi Identity \cite[Equation (2.1.21)]{DK} implies that $d_AF_A = 0$, so $F_A \in \Ker d_A \cap L^p(X;\wedge^2(T^*X)\otimes\ad P)$ and if $A$ is a \emph{regular point} of the curvature section,
  \[
    F:\sB^{1,p}(P) \ni [A] \mapsto [F_A] \in \sA^{1,p}(P)\times_{\Aut^{2,p}(P)} L^p(X;\wedge^2(T^*X)\otimes\ad P),
  \]
in the sense that 
\begin{multline}
\label{eq:Regular_flat_connection}  
\Ker \left\{d_A:L^p(X;\wedge^2(T^*X)\otimes\ad P) \to W_{A_1}^{-1,p}(X;\wedge^3(T^*X)\otimes\ad P)\right\}
\\
= \Ran \left\{d_A:W_{A_1}^{1,p}(X;T^*X\otimes\ad P) \to L^p(X;\wedge^2(T^*X)\otimes\ad P)\right\},
\end{multline}
then \eqref{eq:Derivative_Yang-Mills_energy} implies that $F_A=0$ and $[A] \in M(P)$. Of course, in the absence of an assumption that $A$ is regular in the preceding sense, then $A$ is (by definition) a \emph{Yang--Mills connection} as in \eqref{eq:Yang-Mills_equation},
\[
d_A^*F_A = 0,
\]
and of course need not be flat. However, if we require in addition to $A$ obeying \eqref{eq:Yang-Mills_equation} that it also obeys \eqref{eq:Lojasiewicz-Simon_gradient_inequality_Yang-Mills_neighborhood} (with $A_\infty=\Gamma$),
\[
\|A-\Gamma\|_{W_{A_1}^{1,p}(X)} < \sigma,
\]
for $\sigma = \sigma(A_1,g,G,p,\Gamma) \in (0,1]$, then $A$ is necessarily flat by the {\L}ojasiewicz gradient inequality \eqref{eq:Lojasiewicz-Simon_W-12gradient_inequality_Yang-Mills_energy_function} and so we have the partial reverse inclusion,
\begin{equation}
\label{eq:Critical_points_near_flat_connection_subset_moduli_space_flat_connections}
\Crit\YM \cap \left\{[A] \in \sB(P): \|A-\Gamma\|_{W_{A_1}^{1,p}(X)} < \sigma\right\} \subset M(P).
\end{equation}
Of course, Theorems \ref{thm:Ldover2_energy_gap_Yang-Mills_connections} and \ref{mainthm:L2_energy_gap_Yang-Mills_connections} provide considerably more refined inclusions than \eqref{eq:Critical_points_near_flat_connection_subset_moduli_space_flat_connections} but we shall not need them in this discussion.

If $\Gamma$ is a flat connection on $P$, then its exterior covariant derivative defines an elliptic complex,
\[
\cdots \Omega^i(X;\ad P) \xrightarrow{d_\Gamma} \Omega^{i+1}(X;\ad P) \xrightarrow{d_\Gamma} \Omega^{i+2}(X;\ad P) \cdots
\]
for $i \geq 0$, since $d_\Gamma\circ d_\Gamma = F_\Gamma = 0$. By analogy with their definitions based on the elliptic deformation complex for an anti-self-dual connection (see Atiyah, Hitchin, and Singer \cite[Proof of Theorem 6.1]{AHS} or Donaldson and Kronheimer \cite[Section 4.2.5]{DK}) on a principal $G$-bundle $P$ over a four-dimensional Riemannian manifold, one can define cohomology groups,
\begin{equation}
\label{eq:DeRham_cohomology_group_flat_connection}  
H_\Gamma^i(X;\ad P)
:=
\frac{\Ker\left(d_\Gamma:\Omega^i(X;\ad P) \to \Omega^{i+1}(X;\ad P)\right)}
{\Ran\left(d_\Gamma:\Omega^{i-1}(X;\ad P) \to \Omega^i(X;\ad P)\right)}, \quad i \geq 0,
\end{equation}
and their harmonic representatives,
\begin{equation}
\label{eq:DeRham_cohomology_group_flat_connection_harmonic} 
\bH_\Gamma^i(X;\ad P)
:=
\Ker \left(d_\Gamma + d_\Gamma^*:\Omega^i(X;\ad P) \to \Omega^{i+1}(X;\ad P) \oplus \Omega^{i-1}(X;\ad P) \right),
\quad i \geq 0.
\end{equation}
In particular, $\bH_\Gamma^1(X;\ad P)$ is the \emph{Zariski tangent space} to $M(P)$ at $[\Gamma]$, the \emph{Kuranishi obstruction space} is $\bH_\Gamma^2(X;\ad P)$, and $\bH_\Gamma^0(X;\ad P)$ is the tangent space to $\Stab(\Gamma)$ in \eqref{eq:Stabilizer} at the identity. See Ho, Wilkin, and Wu \cite[Section 2.1]{Ho_Wilkin_Wu_2019} for further discussion of these cohomology groups. The definition \eqref{eq:Regular_flat_connection} of $[\Gamma]$ being a regular point of $M(P)$ is thus equivalent to the condition
\begin{equation}
\label{eq:Regular_flat_connection_H2_zero}
\bH_\Gamma^2(X;\ad P) = (0).  
\end{equation}
Following Ho, Wilkin, and Wu \cite[Section 2.3]{Ho_Wilkin_Wu_2019}, we call
\begin{equation}
\label{eq:Kuranishi_map_Ho_Wilkin_Wu_section_2-3}  
  \kappa_\Gamma:W_{A_1}^{1,p}(X;T^*X\otimes \ad P) \ni a \mapsto a + \frac{1}{2}d_\Gamma^*G_\Gamma[a,a] \in W_{A_1}^{1,p}(X;T^*X\otimes \ad P)
\end{equation}
the \emph{Kuranishi deformation map} defined by a $C^\infty$ flat connection $\Gamma$ on $P$, where $G_\Gamma$ is the \emph{Green's operator} on $\Omega^2(X;\ad P)$ for the Hodge Laplace operator $\Delta_\Gamma=d_\Gamma^*d_\Gamma+d_\Gamma d_\Gamma^*$ on $\Omega^2(X;\ad P)$ given by \eqref{eq:Lawson_page_93_Hodge_Laplacian}.

\begin{prop}[Smooth manifold structure of the moduli space of flat connections near regular points]
\label{prop:Smooth_manifold_structure_moduli_space_flat_connections_near_regular_points}
(See Ho, Wilkin, and Wu \cite[Proposition 2.4]{Ho_Wilkin_Wu_2019}.)
Let $(X,g)$ be a closed, smooth Riemannian manifold of dimension $d\geq 2$, and $G$ be a compact Lie group, $P$ be a smooth principal $G$-bundle over $X$, and $p \in (d/2,\infty)$ be admissible in the sense of Definition \ref{defn:Admissible_Sobolev_exponent_for_Yang-Mills_energy}, and $A_1$ be a $C^\infty$ connection on $P$. If $\Gamma$ is a $C^\infty$ flat connection on $P$ that is a \emph{regular point} of $M(P)$ in the sense that $\bH_\Gamma^2(X;\ad P)=(0)$, then $\kappa_\Gamma$ in \eqref{eq:Kuranishi_map_Ho_Wilkin_Wu_section_2-3} gives a diffeomorphism from an open neighborhood of $\Gamma$ in
\[
  \left\{A \in \Gamma+\Ker d_\Gamma^*\cap W_{A_1}^{1,p}(X;T^*X\otimes\ad P): F_A=0\right\}
\]
onto an open neighborhood of the origin in $\bH_\Gamma^1(X;\ad P)$.
\end{prop}

Ho, Wilkin, and Wu assume that $G$ is a complex reductive Lie group, as they do throughout their article, but the proof of \cite[Proposition 2.4]{Ho_Wilkin_Wu_2019} does not change under the assumption that $G$ is a compact, semisimple Lie group.

Morgan, Mrowka, and Ruberman provide a more general version \cite[Theorem 12.1.1]{MMR} of Proposition \ref{prop:Smooth_manifold_structure_moduli_space_flat_connections_near_regular_points}, allowing for consideration of flat connections $\Gamma$ with non-zero $\bH_\Gamma^2(X;\ad P)$. While those authors assume that $X$ has dimension $d=3$ and that $G=\SU(2)$, their proof extends almost without change to the case of arbitrary $d\geq 2$ and compact Lie groups $G$, using an argument like that in the proof of our Proposition \ref{prop:Semianalyticity_set_Yang-Mills_connections_uniform_Lp_bound_curvature}.

\begin{thm}[Local Kuranishi models for the moduli space of flat connections]
\label{thm:Local_Kuranishi_model_moduli_space_flat_connections}
(See Morgan, Mrowka, and Ruberman \cite[Theorem 12.1.1]{MMR}.)
Let $(X,g)$ be a closed, smooth Riemannian manifold of dimension $d\geq 2$, and $G$ be a compact Lie group, $P$ be a smooth principal $G$-bundle over $X$, and $p \in (d/2,\infty)$ be admissible in the sense of Definition \ref{defn:Admissible_Sobolev_exponent_for_Yang-Mills_energy}, and $A_1$ be a $C^\infty$ connection on $P$. If $\Gamma$ is a $C^\infty$ flat connection on $P$, then there are
\begin{enumerate}
\item A $\Stab(\Gamma)$-invariant neighborhood $V_\Gamma$ of the origin in $\bH_\Gamma^1(X;\ad P)$;
\item An $\Aut^{2,p}(P)$-invariant neighborhood $U_\Gamma$ of $\Gamma$ in $\sA^{1,p}(P)$;
\item A $\Stab(\Gamma)$-equivariant, real-analytic embedding,
  \[
    \varphi_\Gamma: V_\Gamma \to U_\Gamma\cap\left(\Gamma+\Ker d_\Gamma^*\cap W_{A_1}^{1,p}(X;T^*X\otimes\ad P)\right),
  \]
whose differential at the origin is inclusion of $\bH_\Gamma^1(X;\ad P)$ into $\Ker d_\Gamma^*\cap W_{A_1}^{1,p}(X;T^*X\otimes\ad P)$; and
\item A $\Stab(\Gamma)$-equivariant map,
  \[
    \Phi_\Gamma: V_\Gamma \to \bH_\Gamma^2(X;\ad P),
  \]
so that $\varphi_\Gamma$ is a homeomorphism of $\Phi_\Gamma^{-1}(0)$ onto the space of flat connections in $\varphi_\Gamma(V_\Gamma)$.
\end{enumerate}
\end{thm}

The map $\Phi_\Gamma$ is called the \emph{Kuranishi obstruction map} in \cite{MMR}. We have the following refinement of our \cite[Lemma 4.2]{Feehan_lojasiewicz_inequality_ground_state}.

\begin{lem}[Morse--Bott property of the Yang--Mills energy function on a Coulomb-gauge slice at regular flat connections]
\label{lem:Morse-Bott_property_Yang-Mills_energy_near_flat_connection}
Let $(X,g)$ be a closed, smooth Riemannian manifold of dimension $d\geq 2$, and $G$ be a compact Lie group, $P$ be a smooth principal $G$-bundle over $X$, and $p \in (d/2,\infty)$ be admissible in the sense of Definition \ref{defn:Admissible_Sobolev_exponent_for_Yang-Mills_energy}. If $\Gamma$ is a $C^\infty$ flat connection on $P$ that is a \emph{regular point} of $M(P)$ in the sense that $\bH_\Gamma^2(X;\ad P)=(0)$, then the Yang--Mills energy function $\YM:\sB^{1,p}(P) \to \RR$ is Morse--Bott at $[\Gamma]$ in the sense of Definition \ref{defn:Definition_Morse-Bott_quotient_space}.
\end{lem}

\begin{proof}
By analogy with the definition of $\mathbf{Crit}^{1,p}\YM$ in Definition \ref{defn:Definition_Morse-Bott_quotient_space}, set
\[
  \bM(P) := \left\{A \in \Gamma+\Ker d_\Gamma^*\cap W_{A_1}^{1,p}(X;T^*X\otimes\ad P): F_A=0\right\}
\]
and recall that we had defined
\[
  \mathbf{Ker}^{1,p}\YM''(\Gamma) = \Ker\YM''(\Gamma)\cap \Ker d_\Gamma^*\cap W_{A_1}^{1,p}(X;T^*X\otimes\ad P). 
\]
We recall from \eqref{eq:Moduli_space_flat_connections_subset_critical_points} and \eqref{eq:Critical_points_near_flat_connection_subset_moduli_space_flat_connections} that there is a constant $\eps=\eps(A_1,g,G,p,\Gamma) \in (0,1]$ such that
\begin{equation}
\label{eq:Flat_connections_equivalent_to_critical_points_Yang-Mills_energy}
  \bM(P)\cap \bB_\eps(\Gamma) = \mathbf{Crit}^{1,p}\YM \cap\, \bB_\eps(\Gamma),
\end{equation}
where from \eqref{eq:W1p_connections_Coulomb_gauge_ball},
\[
\bB_\eps(\Gamma) = \Gamma + \left\{a \in \Ker d_\Gamma^*\cap W_{A_1}^{1,p}(X;T^*X\otimes\ad P): \|a\|_{W_{A_1}^{1,p}(X)} < \eps\right\}.
\]
If $[\Gamma] \in \sB^{1,p}(P)$ is a regular point, so $\bH_\Gamma^2(X;\ad P)=(0)$, then after possibly decreasing $\eps \in (0,1]$, we have that $\bM(P)\cap \bB_\eps(\Gamma)$ is an embedded smooth submanifold of $\bB_\eps(\Gamma)$ of dimension $\dim\bH_\Gamma^1(X;\ad P)$ by Proposition \ref{prop:Smooth_manifold_structure_moduli_space_flat_connections_near_regular_points}.

The tangent space to $\bM(P)$ at $\Gamma$ is given by
\[
T_{\Gamma} \bM(P) = \Ker \left(d_\Gamma+d_\Gamma^*\right)\cap W_{A_1}^{1,p}(X;T^*X\otimes\ad P) = \bH_\Gamma^1(X;\ad P).
\]
On the other hand, because $F_\Gamma=0$ we have by \eqref{eq:Hessian_Yang-Mills_energy_function} that
\[
\YM''(\Gamma)(a,b) = (d_\Gamma a,d_\Gamma b)_{L^2(X)} = (d_\Gamma^*d_\Gamma a,b)_{L^2(X)}, \quad\text{for all } a, b \in W_{A_1}^{1,p}(X;T^*X\otimes\ad P),
\]
and so
\[
  \YM''(\Gamma) = d_\Gamma^*d_\Gamma: W_{A_1}^{1,p}(X;T^*X\otimes\ad P) \to W_{A_1}^{-1,p}(X;T^*X\otimes\ad P), 
\]
which in turn restricts to
\[
  \YM''(\Gamma) = d_\Gamma^*d_\Gamma: \Ker d_\Gamma^*\cap W_{A_1}^{1,p}(X;T^*X\otimes\ad P) \to \Ker d_\Gamma^*\cap W_{A_1}^{-1,p}(X;T^*X\otimes\ad P).
\]
But then
\begin{align*}
  \mathbf{Ker}^{1,p}\YM''(\Gamma) &= \Ker d_\Gamma^*d_\Gamma \cap \Ker d_\Gamma^*\cap W_{A_1}^{1,p}(X;T^*X\otimes\ad P)
  \\
                                  &= \Ker d_\Gamma \cap \Ker d_\Gamma^*\cap W_{A_1}^{1,p}(X;T^*X\otimes\ad P),
\end{align*}
that is,
\[
  \mathbf{Ker}^{1,p}\YM''(\Gamma) = T_\Gamma \bM(P),
\]
and thus $\YM:\sB^{1,p}(P) \to \RR$ is Morse--Bott at $[\Gamma]$ by Definition \ref{defn:Definition_Morse-Bott_quotient_space}.
\end{proof}

\subsection{Stratified-space structure of the SU(2)-character variety for a closed Riemann surface}
\label{subsec:Structure_character_variety_closed_Rieman_surface_stratified_space}
For a finitely presented group $\pi$, the quotient
\[
  \Hom(\pi, G)/G  
\]
is call the \emph{character variety}. The group $G$ acts on $\Hom(\pi, G)$ by conjugation:
\begin{equation}
\label{eq:Action_G_by_conjugation}  
  \Hom(\pi, G) \ni \rho \mapsto g\rho g^{-1} \in \Hom(\pi, G).
\end{equation}
The space $\Hom(\pi, G)$ has the structure of an analytic subvariety of $G^n$, where $n$ is the number of generators of $\pi$, and an algebraic subvariety if $G$ is algebraic; see Goldman and Millson \cite{Goldman_Millson_1988}, Lubotzky and Magid \cite{Lubotzky_Magid_varieties_representations_finitely_generated_groups}, Simpson \cite{Simpson_1994part1, Simpson_1994part2}, or Saveliev \cite[Section 14.1]{Saveliev_lectures_topology_3-manifolds} for the case $\pi=\pi_1(\Sigma)$, where $\Sigma$ is a closed, connected, orientable Riemann surface, and $G=\SU(2)$. As noted in \cite[p. 184]{MMR}, $\Hom(\pi_1(\Sigma), \SU(2))$ is a Whitney-stratified space and consequently, it is a smoothly-stratified space with local cone-bundle neighborhoods. 
We abbreviate the $G$-character variety of the fundamental group $\pi_1(Y)$ of a topological space $Y$ by
\[
  \chi(Y,G) := \Hom(\pi_1(Y), G)/G.
\]
As in Goldman \cite[p. 203]{Goldman_1984}, we let $\fg_{\Ad\rho}$ denote the $\pi$-module defined by the composition of the representations $\rho:\pi\to G$ and $\Ad:G\to \Aut(\fg)$. Recall from \cite[Section 1.4]{Goldman_1984} that $H^1(\pi;\fg_{\Ad\rho})$ is the \emph{Zariski tangent space} to $\Hom(\pi, G)/G$ at $\rho \in \Hom(\pi, G)$. The cohomology groups $H^i(\pi;\fg_{\Ad\rho})$ for $i=0,2$ are also defined and interpreted in Goldman \cite[Section 1.4]{Goldman_1984}. 

In the case of a Riemann surface $\Sigma$ and semisimple $G$, by Poincar{\'e} duality one has \cite[p. 206]{Goldman_1984}
\begin{equation}
\label{eq:Goldman_page_206}
  H^2(\pi_1(\Sigma);\fg_{\Ad\rho}) \cong (H^0(\pi_1(\Sigma);\fg_{\Ad\rho}))^*.
\end{equation}
We let $C_G(H) := \{g\in G:gh=hg, \text{for all } h\in H\}$ denote the \emph{centralizer} of a subgroup $H\subset G$. From \cite[p. 207]{Goldman_1984}, one sees that
\begin{equation}
\label{eq:Goldman_page_207}
  H^2(\pi_1(\Sigma);\fg_{\Ad\rho}) = \dim H^0(\pi_1(\Sigma);\fg_{\Ad\rho}) = \dim C_G(\rho),
\end{equation}
where \cite[p. 204]{Goldman_1984} $C_G(\rho)$ is the centralizer of $\rho(\pi_1(\Sigma)) \subset G$.

The different types of reducible representations of $\pi_1(\Sigma)$ in $\SU(2)$ are described by Saveliev in \cite[Section 14.2]{Saveliev_lectures_topology_3-manifolds}. (See also \cite[Proposition 2.1]{Hedden_Herald_Kirk_2014} for explicit representations.) A representation $\rho \in \Hom(\pi_1(\Sigma), \SU(2))$ is reducible if and only if it factors through a copy
of $\U(1)$ in $\SU(2)$. Among the reducible representations, Saveliev distinguishes the following three classes:
\begin{enumerate}
\item The trivial representation $\theta$ defined by the formula $\theta(g)=1$ for all $g\in\pi_1(\Sigma)$;
\item The central representations that factor through the center $\ZZ/2\ZZ=\{\pm 1\}$ of $\SU(2)$ but are different from $\theta$; and
\item The other reducible representations, which factor through $\U(1)$ but are not central.
\end{enumerate}
If a reducible representation $\rho$ belongs to one of the first two classes, its stabilizer is the entire group $\SU(2)$; otherwise, $\Stab(\rho) = \U(1)$.

We let $\chi^H(\Sigma, \SU(2)) \subset \chi(\Sigma, \SU(2))$ denote the smooth stratum defined by the set of $\rho\in \Hom(\pi_1(\Sigma), \SU(2))$ that factor through a subgroup $H\subset \SU(2)$ and, for clarity, let $\chi^{\irr}(\Sigma, \SU(2)) = \chi^{\SU(2)}(\Sigma, \SU(2))$ denote the stratum defined by irreducible representations. 

\begin{thm}[Stratum of the character variety defined by irreducible representations]
\label{thm:Saveliev_corollary_14-3} 
(See Saveliev \cite[Corollary 14.3]{Saveliev_lectures_topology_3-manifolds}.)  
Let $\Sigma$ be a closed, orientable Riemann surface of genus $h\geq 1$. If $h\geq 2$, then $\chi^{\irr}(\Sigma, \SU(2))$ is a smooth open manifold of dimension $6h-6$ and is empty if $h=1$.
\end{thm}

We shall need the dimensions of the Zariski tangent spaces for points in other strata of $\chi(\Sigma, \SU(2))$. The local structure of $\chi(\Sigma, G)$ for more general Lie groups is examined by Goldman \cite{Goldman_1984, Goldman_1985}. 

\begin{thm}[Zariski tangent spaces for points in the character variety]
\label{thm:Goldman_1984_page_204} 
(See Goldman \cite[Proposition, p. 204]{Goldman_1984}, Walker \cite[Proposition 1.3]{Walker_extension_casson_invariant}.)  
If $G$ is a reductive Lie group and $\Sigma$ is a closed, orientable Riemann surface of genus $h\geq 1$, then
\[
  \dim H^1(\pi_1(\Sigma);\fg_{\Ad\rho}) = (2h-2)\dim G + 2\dim C_G(\rho).
\]
\end{thm}

Recall that semisimple algebraic Lie groups are reductive and, in particular, $\SU(2)$ is reductive \cite[Section 11]{Borel_linear_algebraic_groups}. For $G=\SU(2)$, we have:
\begin{enumerate}
\item If $\rho$ is irreducible, so $\rho(\pi_1(\Sigma))=\SU(2)$, then $C_G(\rho) = \ZZ/2\ZZ$ and $\dim C_G(\rho)=0$;
\item If $\rho$ is reducible with $\rho(\pi_1(\Sigma))=\U(1)$, then $C_G(\rho) = \U(1)$ and $\dim C_G(\rho)=1$;
\item If $\rho$ is reducible with $\rho(\pi_1(\Sigma))=\ZZ/2\ZZ$, then $C_G(\rho) = \SU(2)$ and $\dim C_G(\rho)=3$.    
\end{enumerate}
From Theorem \ref{thm:Goldman_1984_page_204}, we obtain:
\begin{align*}
  \dim H^1(\pi_1(\Sigma);\su(2)_{\Ad\rho}) &= 6h-6, \quad\text{for all } \rho \in \chi^{\irr}(\Sigma, \SU(2)),
  \\
  \dim H^1(\pi_1(\Sigma);\su(2)_{\Ad\rho}) &= 6h-4, \quad\text{for all } \rho \in \chi^{\U(1)}(\Sigma, \SU(2)),
  \\
  \dim H^1(\pi_1(\Sigma);\su(2)_{\Ad\rho}) &= 6h, \quad\text{for all } \rho \in \chi^{\ZZ/2\ZZ}(\Sigma, \SU(2)).
\end{align*}
(See \cite[Example, p. 161]{Saveliev_lectures_topology_3-manifolds} for the calculation when $\rho \in \chi^{\irr}(\Sigma, \SU(2))$.)

When $h=1$ and $\Sigma=\TT^2$, Theorem \ref{thm:Saveliev_corollary_14-3} implies that the stratum $\chi^{\irr}(\Sigma, \SU(2))$ is empty while
\begin{align}
\label{eq:Torus_character_variety_dimH1_at_U1_reducible}    
  \dim H^1(\pi_1(\TT^2);\su(2)_{\Ad\rho}) &= 2, \quad\text{for all } \rho \in \chi^{\U(1)}(\TT^2, \SU(2)),
  \\
\label{eq:Torus_character_variety_dimH1_at_Z2_reducible}  
  \dim H^1(\pi_1(\TT^2);\su(2)_{\Ad\rho}) &= 6, \quad\text{for all } \rho \in \chi^{\ZZ/2\ZZ}(\Sigma, \SU(2)).
\end{align}
(This is the content of \cite[Exercises 14.1 and 15.1]{Saveliev_lectures_topology_3-manifolds}; the dimensions of these Zariski tangent spaces can be computed directly using the methods described in \cite[Chapter 15]{Saveliev_lectures_topology_3-manifolds}.) Moreover, by \eqref{eq:Goldman_page_207} we obtain
\begin{align}
\label{eq:Torus_character_variety_dimH2_H0_at_U1_reducible}  
  \dim H^2(\pi_1(\TT^2);\su(2)_{\Ad\rho}) = \dim H^0(\pi_1(\TT^2);\su(2)_{\Ad\rho}) &= \dim C_G(\rho) = 1,
                                            \\
                                            &\quad\text{for all } \rho \in \chi^{\U(1)}(\TT^2, \SU(2)), \notag
  \\
\label{eq:Torus_character_variety_dimH2_H0_at_Z2_reducible}    
  \dim H^2(\pi_1(\TT^2);\su(2)_{\Ad\rho}) = \dim H^0(\pi_1(\TT^2);\su(2)_{\Ad\rho}) &= \dim C_G(\rho) = 3,
  \\
  &\quad\text{for all } \rho \in \chi^{\ZZ/2\ZZ}(\Sigma, \SU(2)). \notag
\end{align}
The stratified-space structure of the character variety $\chi(\TT^2, \SU(2))$ as a \emph{pillowcase} is summarized in Figure \ref{fig:Pillow}. Following Hedden, Herald, and Kirk \cite[Sections 3.1 and 3.2]{Hedden_Herald_Kirk_2014} and Kirk \cite[Section 1.2]{Kirk_1993}, let $\mu, \gamma \in \pi_1(\TT^2)$ denote choices of generators, so $\pi_1(\TT^2)=\ZZ\mu\oplus\ZZ\gamma$. To any point $(\alpha,\beta) \in \RR^2$, one can assign the conjugacy class in $\chi(\TT^2, \SU(2))$ of the representation
\begin{equation}
  \label{eq:Pillow_case_parameterization}
  \mu \mapsto \begin{pmatrix}e^{i\alpha} & 0 \\ 0 &e^{-i\alpha}\end{pmatrix},
  \quad \gamma \mapsto \begin{pmatrix}e^{i\beta} & 0 \\ 0 &e^{-i\beta}\end{pmatrix}.
\end{equation}
The induced map $\RR^2\to \chi(\TT^2, \SU(2))$ factors through the branched cover
\[
  \RR^2 \to \RR^2/(\ZZ^2\rtimes \ZZ/2\ZZ),
\]
where $\ZZ^2\rtimes \ZZ/2\ZZ$ acts on $\RR^2$ by
\[
  (m,n)\cdot (x,y) = (x+2\pi m, y+2\pi n), \quad \tau\cdot(x,y) = (-x,-y),
\]
for all $(m,n)\in\ZZ^2$ and $(x,y)\in\RR^2$, where $\tau\in\ZZ/2\ZZ$ is the generator. Equivalently, one can view the pillowcase as the quotient of $S^1\times S^1$ by the action of $\ZZ/2\ZZ$ as $\tau\cdot(z,w) = (\bar z,\bar w)$, regarding the unit circle $S^1$ as a subset of the complex plane \cite[Exercise 14.1]{Saveliev_lectures_topology_3-manifolds}.

\begin{figure}
\label{fig:Pillow}
\centering
\def\svgwidth{3.5in}
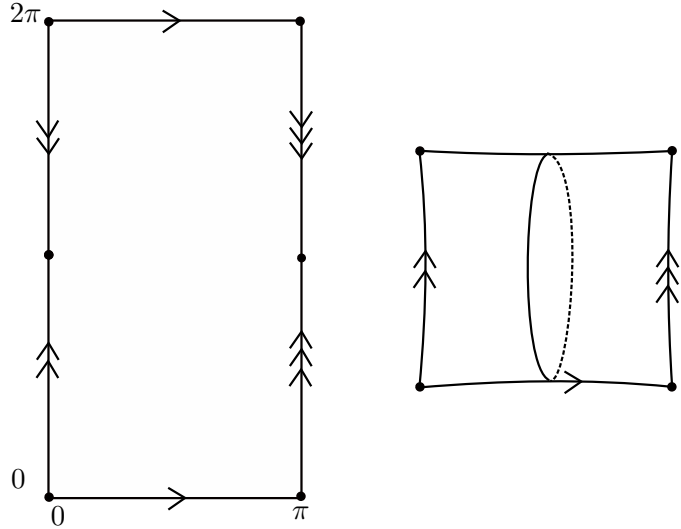
\caption{The left-hand panel describes a fundamental domain for the action of the semidirect product $\ZZ^2\rtimes \ZZ/2\ZZ$ on $\RR^2$ while on the right-hand panel shows the ``pillowcase'' obtained by performing the identifications described on the left. The pillowcase is homeomorphic to the two-dimensional sphere. (This is \cite[Figure 2]{Hedden_Herald_Kirk_2014} due to Hedden, Herald, and Kirk and used by permission of those authors.)}
\end{figure}

We summarize the preceding discussion of the structure of $\chi(\TT^2, \SU(2))$ in the

\begin{thm}[Stratified-space structure of the $\SU(2)$ character variety for the two-dimensional torus]
\label{thm:Stratified-space_structure_SU(2)_character_variety_torus}
The character variety $\chi(\TT^2, \SU(2))$ has the following properties:
\begin{enumerate}  
\item The space $\chi(\TT^2, \SU(2))$ is homeomorphic to $S^2$.
\item The stratum $\chi^{\irr}(\TT^2, \SU(2))$ is empty.
\item The stratum $\chi^{\U(1)}(\TT^2, \SU(2))$ is diffeomorphic to the complement in $S^2$ of four points and the Zariski tangent space to any point $[\rho] \in \chi^{\U(1)}(\TT^2, \SU(2))$ has dimension two.
\item The stratum $\chi^{\ZZ/2\ZZ}(\TT^2, \SU(2))$ comprises four distinct points and the Zariski tangent space to any point $[\rho] \in \chi^{\ZZ/2\ZZ}(\TT^2, \SU(2))$ has dimension six.  
\end{enumerate}
\end{thm}

We note that the character variety $\chi(\TT^2, \SU(2))$ is an analytic subvariety of $\SU(2)\times\SU(2)$ given by the quotient of zero set of the map
\[
  \SU(2)\times\SU(2) \ni (g_1,g_2) \mapsto [g_1,g_2] \in \SU(2)
\]
by the action of $\SU(2)$ by conjugation \cite[Theorem 14.2]{Saveliev_lectures_topology_3-manifolds}.

\subsection{Stratified-space structure of the moduli space of flat SU(2) connections over a closed Riemann surface}
\label{subsec:Moduli_spaces_flat_connections_Riemann_surfaces}
We recall that if $X$ is a connected manifold and $A$ is a smooth connection on a principal $G$-bundle $P$ over $X$, then by \cite[Lemma 4.2.8]{DK} there is an isomorphism of Lie groups,
\begin{equation}
\label{eq:Stabilizer_in_gauge_group_isomorphic_centralizer_holonomy_group}
  \Stab(A) \cong C_G(\Hol(A)),
\end{equation}
where $\Hol(A)$ is the holonomy group of the connection $A$ in $G$ (with respect to a choice of basepoint $x\in X$). When $G=\SU(2)$, the only possibilities for $\Stab(A)$ are $\ZZ/2\ZZ$ (when $\Hol(A)=\SU(2)$), $\U(1)$ (when $\Hol(A)=\U(1)$), and $\SU(2)$ (when $\Hol(A)=\ZZ/2\ZZ$). Let $M(X,G)$ denote the moduli space of flat connections on the product bundle $P=X\times G$. If $\pi_1(G)=\{1\}$, then any principal $G$-bundle over a closed, connected, orientable Riemann surface $\Sigma$ is isomorphic to $\Sigma\times G$ (see Audin \cite[p. 148]{Audin_torus_actions_symplectic_manifolds} and so we fix $P=\Sigma\times G$. (In particular, any principal $\SU(2)$-bundle over $\Sigma$ is isomorphic to $\Sigma\times\SU(2)$.) 

Let $G$ be a connected Lie group and $X$ be a closed, smooth manifold of dimension $d\geq 2$. We recall from a discussion of Goldman \cite[Section 1.8]{Goldman_1984} of the homeomorphism (see, for example, \cite[Proposition 1.2.6]{Kobayashi})
\begin{equation}
\label{eq:Holonomy_representation}
M(X,G) \ni [\Gamma] \mapsto [\rho] \in \Hom(\pi_1(X), G)/G  
\end{equation}
that there is a canonical isomorphism,
\begin{equation}
\label{eq:Equivalence_Zariski_tangent_space_flat_connections_character_variety}  
  H_\Gamma^1(X;\ad P) \cong H^1(\pi_1(X);\fg_{\Ad\rho}).
\end{equation}
In the context of gauge theory, when $\Sigma$ is a Riemann surface, Poincar{\'e} duality \cite[Lemma 2.1]{Ho_Wilkin_Wu_2019} also gives
\begin{equation}
  \label{eq:Ho_Wilkin_Wu_2-1}
  H_\Gamma^2(\Sigma;\ad P) \cong (H_\Gamma^0(\Sigma;\ad P))^*.
\end{equation}
This observation is used in Morgan, Mrowka, and Ruberman \cite[Section 13.2.1, p. 189]{MMR}. Consequently, by \eqref{eq:Stabilizer_in_gauge_group_isomorphic_centralizer_holonomy_group} one has 
\begin{equation}
  \label{eq:Ho_Wilkin_Wu_2-1_dimensions}
  \dim H_\Gamma^2(\Sigma;\ad P) = \dim H_\Gamma^0(\Sigma;\ad P) = \dim\Stab(\Gamma).
\end{equation}
The cohomology groups $H^2(\pi_1(X);\fg_{\Ad\rho})$ and $H_\Gamma^2(X;\ad P)$ are carefully compared by Ho, Wilkin, and Wu in \cite[Section 4]{Ho_Wilkin_Wu_2019} when $G$ is a complex reductive Lie group, though much of the discussion is valid for compact, semisimple Lie groups. In particular, when $\Sigma$ is a Riemann surface then (see \cite[Paragraph following Equation (4.1)]{Ho_Wilkin_Wu_2019})
\begin{equation}
\label{eq:Equivalence_obstruction_space_flat_connections_character_variety}  
  H_\Gamma^2(\Sigma;\ad P) \cong H^2(\pi_1(\Sigma);\fg_{\Ad\rho}).
\end{equation}
We specialize again to the case $G=\SU(2)$ in order to compare the moduli space $M(\Sigma,\SU(2)$ of flat $\SU(2)$ connections over $\Sigma$ and the character variety $\chi(\Sigma,\SU(2)) = \Hom(\pi_1(\Sigma), \SU(2))/\SU(2)$. By analogy with the stratification of the character variety $\chi(\Sigma,\SU(2))$, we let $M^H(\Sigma,\SU(2)) \subset M(\Sigma,\SU(2))$ denote the subspace of points $[\Gamma]$ with holonomy groups $\Hol(\Gamma) \cong H$ and stabilizer $\Stab(\Gamma) \cong G_G(\Hol(\Gamma))$, where $H=\ZZ/2\ZZ$, $\U(1)$, or $\SU(2)$ and for clarity, we again write $M^{\irr}(\Sigma,\SU(2)) = M^{\SU(2)}(\Sigma,\SU(2))$. Ho, Wilkin, and Wu \cite[Section 4, last paragraph]{Ho_Wilkin_Wu_2019} note that the map \eqref{eq:Holonomy_representation} gives a diffeomorphism:
\[
  M^{\irr}(\Sigma,\SU(2)) \cong \chi^{\irr}(\Sigma,\SU(2)).
\]
For the other strata of $M(\Sigma,\SU(2))$, at least when $\Sigma$ has genus one, the preceding discussion yields the following consequence of Theorem \ref{thm:Stratified-space_structure_SU(2)_character_variety_torus}.

\begin{thm}[Stratified-space structure of the moduli space of flat $\SU(2)$ connections over the two-dimensional torus]
\label{thm:Stratified-space_structure_moduli_space_SU(2)_connections_torus}
The moduli space $M(\TT^2,\SU(2))$ has the following properties:
\begin{enumerate}  
\item The space $M(\TT^2,\SU(2))$ is homeomorphic to $S^2$.
\item The stratum $M^{\irr}(\TT^2,\SU(2))$ is empty.
\item The stratum $M^{\U(1)}(\TT^2,\SU(2))$ is diffeomorphic to the complement in $S^2$ of four points and the Zariski tangent space to any point $[\Gamma] \in M^{\U(1)}(\TT^2, \SU(2))$ has dimension two.
\item The stratum $M^{\ZZ/2\ZZ}(\TT^2, \SU(2))$ comprises four distinct points and the Zariski tangent space to any point $[\Gamma] \in M^{\ZZ/2\ZZ}(\TT^2,\SU(2))$ has dimension six.  
\end{enumerate}
\end{thm}

See Nishinou \cite[Theorem 2.1]{Nishinou_2007} (and Friedman \cite[Theorem 8.25]{FriedmanBundleBook}) for another  approach to Theorem \ref{thm:Stratified-space_structure_SU(2)_character_variety_torus} using the relationships between rank-two, semistable Hermitian vector bundles, flat connections, and unitary representations of $\pi_1(\Sigma)$. From Lemma \ref{lem:Morse-Bott_property_Yang-Mills_energy_near_flat_connection} and Theorem \ref{thm:Stratified-space_structure_moduli_space_SU(2)_connections_torus} we obtain the following

\begin{cor}[Morse--Bott property of the Yang--Mills energy function along strata of the moduli space of flat $\SU(2)$ connections over a closed Riemann surface]
\label{cor:Morse-Bott_Yang--Mills_energy_function_moduli_space_SU(2)_connections_torus}
If $\Sigma$ is a closed, connected, orientable Riemann surface of genus $h\geq 1$, then the Yang--Mills energy functions $\YM$ and $\widehat{\YM}$ are Morse--Bott along $M^{\irr}(\Sigma,\SU(2))$. If $h=1$, then $\YM$ and $\widehat{\YM}$ are Morse--Bott along the two-dimensional smooth stratum $M^{\U(1)}(\Sigma,\SU(2))$ but not Morse--Bott at any one of the four points in the zero-dimensional stratum $M^{\ZZ/2\ZZ}(\Sigma,\SU(2))$.
\end{cor}

The Kuranishi model \cite[Proposition 13.2.3]{MMR} for an open neighborhood of $[\Theta]$ in the quotient space $\sB(\TT^2,\SU(2))$ (see also Theorem \ref{thm:Local_Kuranishi_model_moduli_space_flat_connections}) shows that $M(\TT^2,\SU(2))$ is an analytic subvariety of $\sB(\TT^2,\SU(2))$ with quadratic singularities. For any Riemann surface $\Sigma$, one has the following special case of a result due to Herald.

\begin{thm}[Stratified-space structure of the moduli space of flat $\SU(2)$ connections over a closed Riemann surface]
\label{thm:Herald_1994_theorem_10}  
(See Herald \cite[Theorem 10]{Herald_1994}.)
Let $\Sigma$ be a closed, connected, orientable Riemann surface of genus $h\geq 1$. Then $M(\Sigma,\SU(2))$ is a real analytic variety with smooth strata
\[
  M^{\irr}(\Sigma,\SU(2)), \quad M^{\U(1)}(\Sigma,\SU(2)), \quad\text{and}\quad M^{\ZZ/2\ZZ}(\Sigma,\SU(2))
\]
of dimensions $6h-6$, $2h$, and zero, respectively.
\end{thm}

Herald's proof of Theorem \ref{thm:Herald_1994_theorem_10} relies on key results due to Goldman \cite[Theorem 3]{Goldman_1985} and Morgan, Mrowka, and Ruberman \cite[Proposition 13.2.3]{MMR}. In \cite{Huebschmann_1994, Huebschmann_1995, Huebschmann_1996, Huebschmann_1998}, Huebschmann provides a comprehensive analysis of the local structure of moduli spaces of Yang--Mills connections over closed Riemann surfaces that generalizes \cite[Theorem 10]{Herald_1994}.

\subsection{Local distance-minimizing property for a Coulomb gauge condition}
\label{subsec:Local_minimizing_property_Coulomb_gauge_condition}
We will prove a generalization of the following well-known exercise in texts on partial differential equations: If $\Omega\subset\RR^n$ is a bounded domain with $C^2$ boundary $\partial\Omega$, then $\Omega$ obeys the \emph{exterior sphere condition} in the sense that for each $y_0\in\partial\Omega$, there is an open ball $B_\delta(z_0)\subset\RR^n\less\bar\Omega$, depending on $y_0$, such that $\bar B_\delta(z_0)\cap\partial\Omega=\{y_0\}$. See, for example, Fornaro, Metafune, and Priola \cite[Proposition B.2]{Fornaro_Metafune_Priola_2004}, Gilbarg and Trudinger \cite[Paragraph preceding Theorem 6.13]{GT}, McOwen \cite[Exercise 4.3.2]{McOwen_partial_differential_equations}, or Taheri \cite[Exercise 10.33.1]{Taheri_function_spaces_partial_differential_equations}. Fornaro et al. use a uniform version \cite[Proposition B.2]{Fornaro_Metafune_Priola_2004} of the exterior sphere condition to prove basic properties \cite[Proposition B.3]{Fornaro_Metafune_Priola_2004} of the nearest-point projection operator for the boundary $\partial\Omega$ of an unbounded domain $\Omega$. See Gilbarg and Trudinger \cite[Section 14.6]{GT} for results of this kind for bounded domains in Euclidean space and Simon \cite[Section 2.12.3, Theorem 1]{Simon_1996} for a definitive result on properties of the nearest-point projection operator for a closed $C^k$ submanifold (with $k\geq 2$, $k=\infty$, or $k=\omega$) of Euclidean space.

\begin{lem}[Sphere condition for a $C^2$ submanifold of a Banach space]
\label{lem:Sphere_condition_smooth_submanifold}
Let $\sX$ be a Banach space that is continuously embedded in and a dense subspace of a Hilbert space $\sH$ and let $\sY$ be a Banach space that is continuously embedded in $\sX$. Let $\sS\subset\sX$ be a $C^2$ submanifold modeled on $\sY$. If $y_0\in \sS$, then there is a positive constant $\delta$ such that if $z_0=y_0+\delta\eta$, where $\eta \in T_{y_0}^\perp\sS$ (the orthogonal complement of the tangent space $T_{y_0}\sS$) and $\|\eta\|_\sX = 1$, then $\partial B_\delta(z_0)\cap \sS = \{y_0\}$ and $B_\delta(z_0) \subset \sX \less \sS$, where $B_r(x) := \{w \in \sX: \|w-x\|_\sX < r\}$ for $r\in(0,\infty)$.
\end{lem}

\begin{proof}
Let $\pi:\sX\to\sY$ denote orthogonal projection, denote $\pi^\perp := \id_\sX - \pi$, and let $\sY^\perp$ denote the orthogonal complement of $\sY$. Observe that $\sX=\sY\oplus\sY^\perp$ as a direct sum of Banach spaces. We may assume without loss of generality that $y_0=\bzero$, the zero element of $\sX$, and note that $T_\bzero\sS = \sY$ and $T_\bzero^\perp\sS = \sY^\perp$.
  
We first consider the model case where $\sS=\sY$. Let $\delta$ be any positive constant and let $z_0 = \delta\eta$. Let $\bar B_\delta(z_0) := \{w \in \sX: \|w-z_0\|_\sX \leq \delta\}$ denote the closure of $B_\delta(z_0)$ in $\sX$ and $\bar{\bar{B}}_\delta(z_0)$ denote the closure of $B_\delta(z_0)$ in $\sH$. If $x, y \in \bar{\bar{B}}_\delta(z_0)$, then there are sequences $\{x_n\}_{n\in\NN}$, $\{y_n\}_{n\in\NN}$ in $B_\delta(z_0)$ such that $x_n\to x$ and $y_n\to y$ in $\sH$ as $n \to \infty$. For all $t\in[0,1]$, we have $t(x_n-z_0)+(1-t)(y_n-z_0) \in B_\delta(z_0)$ and so, by taking limits in $\sH$ as $n\to\infty$, we obtain $t(x-z_0)+(1-t)(y-z_0) \in \bar{\bar{B}}_\delta(z_0)$. Hence, $\bar{\bar{B}}_\delta(z_0)$ is a non-empty, closed, convex subset of the Hilbert space, $\sH$, and thus contains a unique element of smallest $\sH$ norm. But $\|z_0\|_\sX = \delta$, so $\bzero \in \bar B_\delta(z_0) \subset \bar{\bar{B}}_\delta(z_0)$, and $\bzero$ must be that unique element of smallest $\sH$ norm. Therefore
\[
  \sY \cap \bar{B}_\delta(z_0) = \{\bzero\}
\]
and so $B_\delta(z_0) \subset \sX \less \sY$.

For the general case where $\sS$ is a noncompact $C^2$ submanifold, we extend the argument described by Fornaro, Metafune, and Priola \cite[Proposition B.2]{Fornaro_Metafune_Priola_2004}. Because $\sS$ is a $C^2$ submanifold of $\sX$ modeled on $\sY$, there are an open neighborhood $\sU \subset \sX$ of the origin and a $C^2$ embedding $\varphi: \sU \to \sX$ such that $\varphi(\bzero)=\bzero$ and $\varphi(\sU\cap\sS) = \varphi(\sU)\cap\sY$ is a relatively open neighborhood of the origin in $\sY$. Denote $\varphi^\parallel := \pi\circ\varphi$ and $\varphi^\perp := \pi^\perp\circ\varphi$. There is a constant $\eps=\eps(\sU)\in(0,1]$ such that, for any $w^\parallel \in \sU\cap \sY$,
\begin{equation}
  \label{eq:Line_through_wparallel_direction_zeta}
  x_\zeta(t) := w^\parallel + t\zeta \in \sU, \quad\text{for all } t\in [0,\eps] \text{ and } \zeta \in \sY^\perp \text{ with } \|\zeta\|_\sX = 1.
\end{equation}
The Taylor Formula gives
\[
  \varphi^\perp(x_\zeta(t)) = tD\varphi^\perp(\bzero)\zeta + R_\zeta(t),
\]
with $\|R_\zeta(t)\|_\sX \leq Ct^2$, for some constant $C=C(\varphi)\in[1,\infty)$. For convenience, we abbreviate $L := D\varphi^\perp(\bzero) \in \sL(\sY^\perp)$. Note that $L$ is invertible since $D\varphi(\bzero)$ is invertible by virtue of the fact that $\varphi$ is a local diffeomorphism and $D\varphi(\bzero) = D\varphi^\parallel(\bzero)\oplus D\varphi^\perp(\bzero)$ on $\sX=\sY\oplus\sY^\perp$. Hence,
\[
\|\zeta\|_\sX = \|\zeta\|_{\sY^\perp} = \|L^{-1}L\zeta\|_{\sY^\perp} \leq \|L^{-1}\|_{\sL(\sY^\perp)}\|L\zeta\|_{\sY^\perp} = K\|L\zeta\|_\sX, \quad\text{for all } \zeta \in \sY^\perp,
\]
where $K := \|L^{-1}\|_{\sL(\sY^\perp)}$. Thus
\begin{align*}
  \|\varphi^\perp(x_\zeta(t))\|_\sX &= \|tD\varphi^\perp(\bzero)\zeta + R_\zeta(t)\|_\sX
  \\
                                    &\geq t\|D\varphi^\perp(\bzero)\zeta\|_\sX - \|R_\zeta(t)\|_\sX
  \\
  &\geq K^{-1}t - Ct^2.
\end{align*}
Now $K^{-1}t - Ct^2 = t(K^{-1}-Ct) > 0$ if and only if $0<t<4\delta:=1/(CK)$. In particular,
\[
  \varphi^\perp(x_\zeta(t)) \neq \bzero, \quad\text{for all } t \in (0, 4\delta).
\]
By construction of $\varphi$, we have $x\in\sU\cap\sS \iff \varphi^\parallel(x) \in \varphi(\sU)\cap\sY \iff \varphi^\perp(x) = \bzero \in \sY^\perp$.
Consequently,
\begin{equation}
  \label{eq:Point_on_line_through_wparallel_direction_zeta_notin_submanifold}
  x_\zeta(t) \notin \sS, \quad\text{for all } t \in (0, 4\delta).
\end{equation}
Suppose $w = w^\parallel + w^\perp \in B_\delta(z_0)$, where $w^\parallel := \pi w$ and $w^\perp := \pi^\perp w$, and we set $\zeta := w^\perp/\|w^\perp\|_\sX$, noting that $w^\perp\neq \bzero$ since $B_\delta(z_0) \subset \sX\less\sY$ by construction. If we now choose
\[
  x_\zeta(t) := w^\parallel + t\zeta, \quad\text{for all } t \in (0,4\delta),
\]
as in \eqref{eq:Line_through_wparallel_direction_zeta} then $w = x_\zeta(t)$ if and only if $t = \|w^\perp\|_\sX$. But $0 < \|w^\perp\|_\sX = \|\pi^\perp w\|_\sX \leq \|w\|_\sX < 2\delta$ and so $w = x_\zeta(t) \notin \sS$ by \eqref{eq:Point_on_line_through_wparallel_direction_zeta_notin_submanifold}. Because $w \in B_\delta(z_0)$ was arbitrary, we see that $B_\delta(z_0) \subset \sX\less\sS$. To finish the proof, it remains to observe that $\bzero \in \partial B_\delta(z_0)$ and $\sS \cap \partial B_\delta(z_0) = \{\bzero\}$
\end{proof}

The conclusion of Lemma \ref{lem:Sphere_condition_smooth_submanifold} may be usefully rephrased in the folliowing

\begin{cor}[Local minimizing property for distance to a $C^2$ submanifold of a Banach space]
\label{cor:Sphere_condition_smooth_submanifold}
Continue the hypotheses of Lemma \ref{lem:Sphere_condition_smooth_submanifold}. Then
\[
  \|y-z_0\|_\sX \geq \|y_0-z_0\|_\sX, \quad\text{for all } y \in \sY,
\]  
with equality if and only if $y=y_0$.
\end{cor}

We now apply Lemma \ref{lem:Sphere_condition_smooth_submanifold} and Corollary \ref{cor:Sphere_condition_smooth_submanifold} to the orbit of a connection $A_0$ in $\sA^{1,p}(P)$.

\begin{cor}[Local distance-minimizing property for a Coulomb gauge condition]
\label{cor:Local_minimizing_property_Coulomb_gauge_condition}  
Let $(X,g)$ be a closed, connected, smooth Riemannian manifold of dimension $d \geq 2$, and $G$ be a compact Lie group, and $P$ be a smooth principal $G$-bundle over $X$. If $A_0$ is a $C^\infty$ connection on $P$ and $p\in(d/2,\infty)$, then there is a constant $\delta=\delta(A_0,g,G,p)\in(0,1]$ with the following significance. If $d_{A_0}^*(A-A_0)=0$ and $\|A-A_0\|_{W_{A_0}^{1,p}(X)}<\delta$, then
\begin{equation}
  \label{eq:Local_distance-minimizing_property_Coulomb_gauge_condition}
  \|u(A)-A_0\|_{W_{A_0}^{1,p}(X)} \geq \|A-A_0\|_{W_{A_0}^{1,p}(X)}, \quad\text{for all } u \in \Aut^{2,p}(P),
\end{equation}
and equality holds if and only if $u \in\Stab(A_0)$.
\end{cor}

\begin{proof}
Observe that $\Aut^{2,p}(P)/\Stab(A_0)$ acts freely on $\sA^{1,p}(P)$, where $\Stab(A_0) \subset G$ is as defined in \eqref{eq:Stabilizer}. The smooth map
\begin{equation}
\label{eq:Gauge_orbit}  
  \Aut^{2,p}(P)/\Stab(A_0) \ni u \mapsto u(A_0) \in \sA^{1,p}(P)
\end{equation}
has derivative given by $\xi \mapsto d_{A_0}\xi$ at the identity in $\Aut^{2,p}(P)/\Stab(A_0)$, where
\[
  \xi \in T_{\id}\left(\Aut^{2,p}(P)/\Stab(A_0)\right) = \left(\Ker\Delta_{A_0}\right)^\perp\cap W_{A_0}^{2,p}(X;\ad P) = T_{\id}\Aut_0^{2,p}(P),
\]
and where $\Aut_0^{2,p}(P) \subset \Aut^{2,p}(P)$ is as in \eqref{eq:W2q_gauge_transformations_orthogonal_stabilizer} and locally diffeomorphic to a neighborhood of the identity $\mathbf{1}$ in $\Aut^{2,p}(P)/\Stab(A_0)$. In particular, the map \eqref{eq:Gauge_orbit} is an immersion and thus an embedding on a small enough open neighborhood $\sG_0 \subset \Aut^{2,p}(P)/\Stab(A_0)$ of the identity.

\begin{figure}
  	\label{fig:gaugeorbitimmersion}
	\centering
	\includegraphics[width=0.7\linewidth]{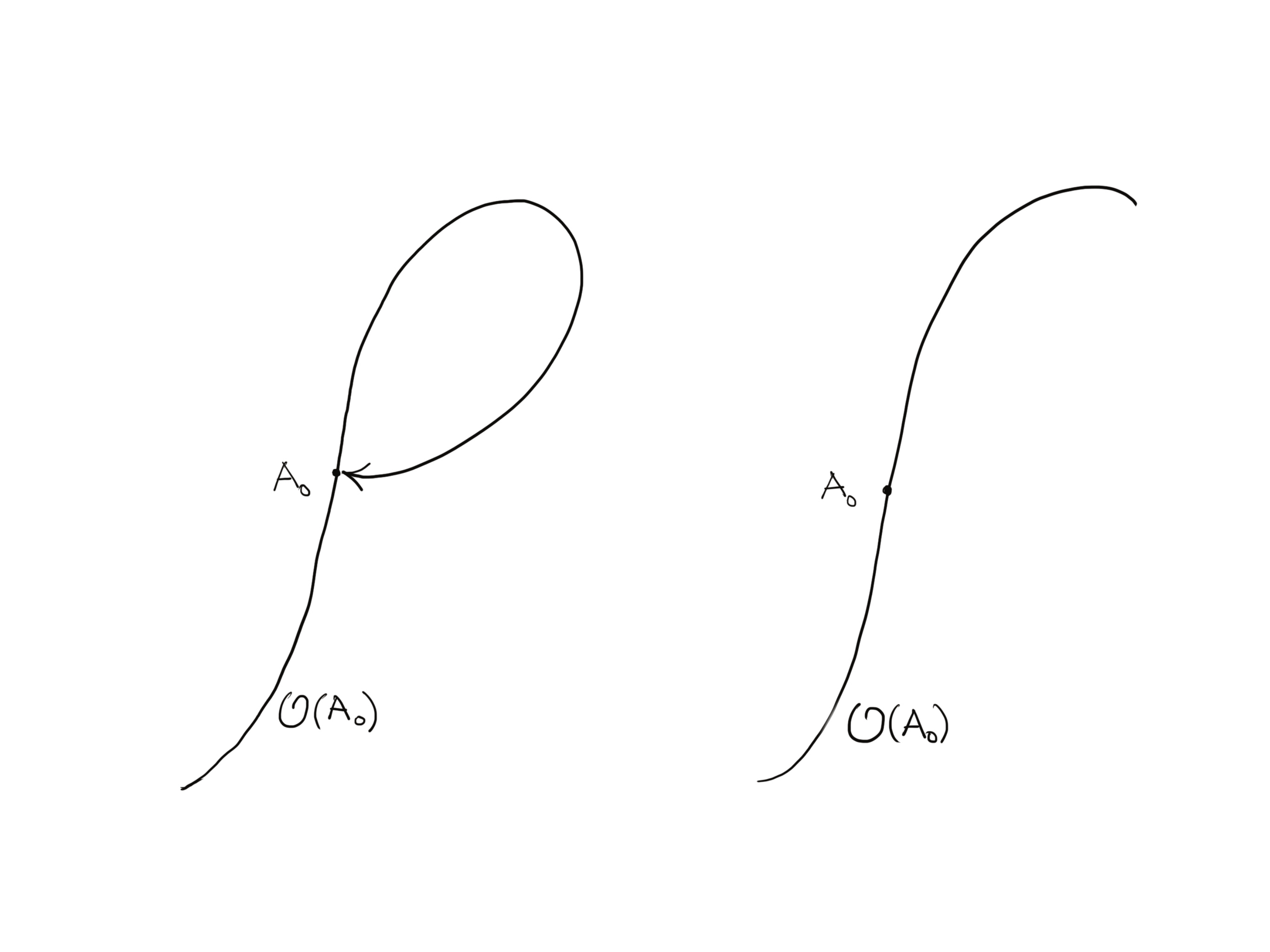}
	\caption[Example of an immersion and an embedding of $\Aut^{2,p}(P)/\Stab(A_0)$ in $\sA^{1,p}(P)$ as the orbit $\sO(A_0) = \{u(A_0): u \in \Aut^{2,p}(P)\}$.]{Example of an immersion and an embedding of $\Aut^{2,p}(P)/\Stab(A_0)$ in $\sA^{1,p}(P)$ as the orbit $\sO(A_0) = \{u(A_0): u \in \Aut^{2,p}(P)\}$.}
\end{figure}

We wish to also show that the geometry suggested in the left-hand panel of Figure \ref{fig:gaugeorbitimmersion}, based on the classic example described by Lee \cite[Example 4.19]{Lee_john_smooth_manifolds} of an immersion that is not an embedding, does not occur for the map \eqref{eq:Gauge_orbit}. If we can choose $\eps=\eps(A_0,g,G,p)\in(0,1]$ small enough that $\|u(A_0) - A_0\|\geq\eps$ for all $u \in \Aut^{2,p}(P)-\sG_0$, then we are done and the geometry is as indicated in the right-hand panel of Figure \ref{fig:gaugeorbitimmersion}. Otherwise, there is a sequence $\{u_n\}_{n\in\NN} \subset \Aut^{2,p}(P)$ such that $u_n(A_0)\to A_0$ in $W_{A_0}^{1,p}(X;T^*X\otimes\ad P)$ as $n\to\infty$. But then Corollary \ref{cor:Freed_Uhlenbeck_3-2_W1q} implies that $u_n \to \mathbf{1}$ in $\Aut^{2,p}(P)/\Stab(A_0)$ as $n\to\infty$. Thus $u_n \in \sG_0$ for all large enough $n$, where the map \eqref{eq:Gauge_orbit} is a local embedding and the orbit is embedded as suggested in the right-hand panel of Figure \ref{fig:gaugeorbitimmersion}.
  
It remains to apply Corollary \ref{cor:Sphere_condition_smooth_submanifold} with $\sS$, $\sX$, and $\sY$ defined by Corollary \ref{cor:Freed_Uhlenbeck_3-2_W1q}:
\begin{align*}
  \sX &= W_{A_0}^{1,p}(X;T^*X\otimes\ad P),
  \\
  \sY &= \Ran d_{A_0}\cap W_{A_0}^{1,p}(X;T^*X\otimes\ad P),        
  \\
  \sY^\perp &= \Ker d_{A_0}^*\cap W_{A_0}^{1,p}(X;T^*X\otimes\ad P),
  \\
  \sS &= \Aut_0^{2,p}(P).
\end{align*}
The conclusion \eqref{eq:Local_distance-minimizing_property_Coulomb_gauge_condition} when $u\in \Aut^{2,p}(P)/\Stab(A_0)$ follows by choosing $y_0 = A_0$ and $z_0 = A$ in Corollary \ref{cor:Sphere_condition_smooth_submanifold}. Finally, observe that if $w \in \Stab(A_0)$, then
\[
  \|wu(A)-A_0\|_{W_{A_0}^{1,p}(X)} = \|u(A)-w^{-1}(A_0)\|_{W_{w^{-1}(A_0)}^{1,p}(X)} = \|u(A)-A_0\|_{W_{A_0}^{1,p}(X)},
\]
and so the conclusion \eqref{eq:Local_distance-minimizing_property_Coulomb_gauge_condition} follows for all $u\in \Aut^{2,p}(P)$.
\end{proof}

%
%
\bibliography{master,mfpde}

\def\cprime{$'$} \def\cprime{$'$}
  \def\ocirc#1{\ifmmode\setbox0=\hbox{$#1$}\dimen0=\ht0 \advance\dimen0
  by1pt\rlap{\hbox to\wd0{\hss\raise\dimen0
  \hbox{\hskip.2em$\scriptscriptstyle\circ$}\hss}}#1\else {\accent"17 #1}\fi}
  \def\cprime{$'$} \def\cprime{$'$} \def\cprime{$'$} \def\cprime{$'$}
  \def\cprime{$'$} \def\polhk#1{\setbox0=\hbox{#1}{\ooalign{\hidewidth
  \lower1.5ex\hbox{`}\hidewidth\crcr\unhbox0}}} \def\cprime{$'$}
  \def\cprime{$'$} \def\cprime{$'$}
  \def\lfhook#1{\setbox0=\hbox{#1}{\ooalign{\hidewidth
  \lower1.5ex\hbox{'}\hidewidth\crcr\unhbox0}}} \def\cprime{$'$}
  \def\cprime{$'$} \def\cprime{$'$} \def\cprime{$'$} \def\cprime{$'$}
\providecommand{\bysame}{\leavevmode\hbox to3em{\hrulefill}\thinspace}
\providecommand{\MR}{\relax\ifhmode\unskip\space\fi MR }
\providecommand{\MRhref}[2]{%
  \href{http://www.ams.org/mathscinet-getitem?mr=#1}{#2}
}
\providecommand{\href}[2]{#2}
\begin{thebibliography}{100}

\bibitem{Abhyankar_local_analytic_geometry}
Shreeram~S. Abhyankar, \emph{Local analytic geometry}, World Scientific
  Publishing Co., Inc., River Edge, NJ, 2001, Reprint of the 1964 original.
  \MR{1824736}

\bibitem{AdamsFournier}
Robert~A. Adams and John J.~F. Fournier, \emph{Sobolev spaces}, second ed.,
  Elsevier/Academic Press, Amsterdam, 2003. \MR{2424078 (2009e:46025)}

\bibitem{AHS}
Michael~F. Atiyah, Nigel~J. Hitchin, and Isadore~M. Singer, \emph{Self-duality
  in four-dimensional {R}iemannian geometry}, Proc. Roy. Soc. London Ser. A
  \textbf{362} (1978), no.~1711, 425--461. \MR{506229 (80d:53023)}

\bibitem{Audin_torus_actions_symplectic_manifolds}
Mich\`ele Audin, \emph{Torus actions on symplectic manifolds}, revised ed.,
  Progress in Mathematics, vol.~93, Birkh\"{a}user Verlag, Basel, 2004.
  \MR{2091310}

\bibitem{Benedetti_Risler_real_algebraic_semi-algebraic_sets}
Riccardo Benedetti and Jean-Jacques Risler, \emph{Real algebraic and
  semi-algebraic sets}, Actualit\'{e}s Math\'{e}matiques. [Current Mathematical
  Topics], Hermann, Paris, 1990. \MR{1070358}

\bibitem{BierstoneMilman}
Edward Bierstone and Pierre~D. Milman, \emph{Semianalytic and subanalytic
  sets}, Inst. Hautes \'Etudes Sci. Publ. Math. (1988), no.~67, 5--42.
  \MR{972342 (89k:32011)}

\bibitem{Bierstone_Milman_1997}
Edward Bierstone and Pierre~D. Milman, \emph{Canonical desingularization in
  characteristic zero by blowing up the maximum strata of a local invariant},
  Invent. Math. \textbf{128} (1997), no.~2, 207--302. \MR{1440306}

\bibitem{Bleecker_1981}
David~D. Bleecker, \emph{Gauge theory and variational principles}, Global
  Analysis Pure and Applied Series A, vol.~1, Addison-Wesley Publishing Co.,
  Reading, Mass., 1981. \MR{643361 (83h:53049)}

\bibitem{Bochnak_Coste_Roy_real_algebraic_geometry}
Jacek Bochnak, Michel Coste, and Marie-Fran\c{c}oise Roy, \emph{Real algebraic
  geometry}, Ergebnisse der Mathematik und ihrer Grenzgebiete (3) [Results in
  Mathematics and Related Areas (3)], vol.~36, Springer-Verlag, Berlin, 1998,
  Translated from the 1987 French original, Revised by the authors.
  \MR{1659509}

\bibitem{Borel_linear_algebraic_groups}
Armand Borel, \emph{Linear algebraic groups}, second ed., Graduate Texts in
  Mathematics, vol. 126, Springer-Verlag, New York, 1991. \MR{1102012}

\bibitem{Bourguignon_Lawson_1981}
Jean-Pierre Bourguignon and H.~Blaine Lawson, Jr., \emph{Stability and
  isolation phenomena for {Y}ang--{M}ills fields}, Comm. Math. Phys.
  \textbf{79} (1981), 189--230. \MR{612248 (82g:58026)}

\bibitem{Bourguignon_Lawson_Simons_1979}
Jean-Pierre Bourguignon, H.~Blaine Lawson, Jr., and James Simons,
  \emph{Stability and gap phenomena for {Y}ang--{M}ills fields}, Proc. Nat.
  Acad. Sci. U.S.A. \textbf{76} (1979), 1550--1553. \MR{526178 (80h:53028)}

\bibitem{Brezis}
Haim Br{\'e}zis, \emph{Functional analysis, {S}obolev spaces and partial
  differential equations}, Universitext, Springer, New York, 2011. \MR{2759829
  (2012a:35002)}

\bibitem{BrockertomDieck}
Theodor Br{\"o}cker and Tammo tom Dieck, \emph{Representations of compact {L}ie
  groups}, Graduate Texts in Mathematics, vol.~98, Springer, New York, 1995.
  \MR{1410059 (97i:22005)}

\bibitem{Bryant_simons_collaboration_special_holonomy}
Robert~L. Bryant, \emph{Simons collaboration on special holonomy in geometry,
  analysis, and physics}, 2016--2024, \url{https://sites.duke.edu/scshgap}.

\bibitem{Chavel}
Isaac Chavel, \emph{Eigenvalues in {R}iemannian geometry}, Pure and Applied
  Mathematics, vol. 115, Academic Press, Inc., Orlando, FL, 1984, Including a
  chapter by Burton Randol, With an appendix by Jozef Dodziuk. \MR{768584}

\bibitem{Chen_Sun_2020}
Xuemiao Chen and Song Sun, \emph{Singularities of {H}ermitian-{Y}ang--{M}ills
  connections and {H}arder-{N}arasimhan-{S}eshadri filtrations}, Duke Math. J.
  \textbf{169} (2020), no.~14, 2629--2695. \MR{4149506}

\bibitem{Chen_Wentworth_2021arxiv}
Xuemiao Chen and Richard~A. Wentworth, \emph{Compactness for
  $\omega$-{Y}ang--{M}ills connections}, arXiv:2106.09131.

\bibitem{Chen_Shen_1995}
Yun-Mei Chen and Chu-Li Shen, \emph{Evolution problem of {Y}ang--{M}ills flow
  over {$4$}-dimensional manifold}, Variational methods in nonlinear analysis
  ({E}rice, 1992), Gordon and Breach, Basel, 1995, pp.~63--66. \MR{1451148}

\bibitem{Chen_Shen_1993}
Yun-Mei Chen and Chun-Li Shen, \emph{Evolution of {Y}ang--{M}ills connections},
  Differential geometry ({S}hanghai, 1991), World Sci. Publ., River Edge, NJ,
  1993, pp.~33--41. \MR{1341596}

\bibitem{Chen_Shen_1994}
Yun-Mei Chen and Chun-Li Shen, \emph{Monotonicity formula and small action
  regularity for {Y}ang--{M}ills flows in higher dimensions}, Calc. Var.
  Partial Differential Equations \textbf{2} (1994), no.~4, 389--403.
  \MR{1383915}

\bibitem{Chen_Shen_Zhou_2002}
Yun-Mei Chen, Chun-Li Shen, and Qing Zhou, \emph{Asymptotic behavior of
  {Y}ang--{M}ills flow in higher dimensions}, Differential geometry and related
  topics, World Sci. Publ., River Edge, NJ, 2002, pp.~16--38. \MR{2066798}

\bibitem{DeTurck_1983}
Dennis~M. DeTurck, \emph{Deforming metrics in the direction of their {R}icci
  tensors}, J. Differential Geom. \textbf{18} (1983), 157--162. \MR{697987
  (85j:53050)}

\bibitem{DonASD}
Simon~K. Donaldson, \emph{Anti self-dual {Y}ang--{M}ills connections over
  complex algebraic surfaces and stable vector bundles}, Proc. London Math.
  Soc. (3) \textbf{50} (1985), 1--26. \MR{765366 (86h:58038)}

\bibitem{DonCompact}
Simon~K. Donaldson, \emph{Compactification and completion of {Y}ang--{M}ills
  moduli spaces}, Differential geometry ({P}e\~n\'\i scola, 1988), Lecture
  Notes in Math., vol. 1410, Springer, Berlin, 1989, pp.~145--160. \MR{1034277}

\bibitem{DonFloer}
Simon~K. Donaldson, \emph{Floer homology groups in {Y}ang--{M}ills theory},
  Cambridge Tracts in Mathematics, vol. 147, Cambridge University Press,
  Cambridge, 2002, With the assistance of M. Furuta and D. Kotschick.
  \MR{1883043}

\bibitem{DK}
Simon~K. Donaldson and Peter~B. Kronheimer, \emph{The geometry of
  four-manifolds}, Oxford Mathematical Monographs, The Clarendon Press, Oxford
  University Press, New York, 1990, Oxford Science Publications. \MR{1079726}

\bibitem{Federer}
Herbert Federer, \emph{Geometric measure theory}, Die Grundlehren der
  mathematischen Wissenschaften, vol. Band 153, Springer-Verlag New York, Inc.,
  New York, 1969. \MR{257325}

\bibitem{Feehan_yangmillsenergygapflat_corrigendum}
Paul M.~N. Feehan, \emph{Corrigendeum to ``{E}nergy gap for {Y}ang--{M}ills
  connections, {II}: {A}rbitrary closed {R}iemannian manifolds''}, preprint,
  July 15, 2019.

\bibitem{Feehan_yangmillsenergy_lojasiewicz4d_v1}
Paul M.~N. Feehan, \emph{Discreteness for energies of {Y}ang--{M}ills
  connections over four-dimensional manifolds}, arXiv:1505.06995v1.

\bibitem{Feehan_yang_mills_gradient_flow_v4}
Paul M.~N. Feehan, \emph{Global existence and convergence of solutions to
  gradient systems and applications to {Y}ang--{M}ills gradient flow},
  arXiv:1409.1525v4, xx+475 pages.

\bibitem{Feehan_lojasiewicz_inequality_ground_state_v1}
Paul M.~N. Feehan, \emph{Optimal {{\L}}ojasiewicz--{S}imon inequalities and
  {M}orse--{B}ott {Y}ang--{M}ills energy functions}, arXiv:1706.09349v1.

\bibitem{FeehanSlice}
Paul M.~N. Feehan, \emph{Critical-exponent {S}obolev norms and the slice
  theorem for the quotient space of connections}, Pacific J. Math. \textbf{200}
  (2001), no.~1, 71--118, arXiv:dg-ga/9711004. \MR{1863408}

\bibitem{Feehan_yangmillsenergygapflat}
Paul M.~N. Feehan, \emph{Energy gap for {Y}ang--{M}ills connections, {II}:
  {A}rbitrary closed {R}iemannian manifolds}, Adv. Math. \textbf{312} (2017),
  547--587, arXiv:1502.00668. \MR{3635819}

\bibitem{Feehan_lojasiewicz_inequality_all_dimensions}
Paul M.~N. Feehan, \emph{Resolution of singularities and geometric proofs of
  the {{\L}}ojasiewicz inequalities}, Geom. Topol. \textbf{23} (2019), no.~7,
  3273--3313, arXiv:1708.09775. \MR{4046966}

\bibitem{Feehan_lojasiewicz_inequality_all_dimensions_morse-bott}
Paul M.~N. Feehan, \emph{On the {M}orse--{B}ott property of analytic functions
  on {B}anach spaces with {{\L}}ojasiewicz exponent one half}, Calc. Var.
  Partial Differential Equations \textbf{59} (2020), no.~2, Paper No. 87, 50,
  arXiv:1803.11319. \MR{4087392}

\bibitem{Feehan_lojasiewicz_inequality_ground_state}
Paul M.~N. Feehan, \emph{Optimal {{\L}}ojasiewicz--{S}imon inequalities and
  {M}orse--{B}ott {Y}ang--{M}ills energy functions}, Adv. Calc. Var.
  \textbf{15} (2022), no.~4, 635--671, arXiv:1706.09349. \MR{4489597}

\bibitem{Feehan_Leness_introduction_virtual_morse_theory_so3_monopoles}
Paul M.~N. Feehan and Thomas~G. Leness, \emph{Virtual {M}orse--{B}ott index,
  moduli spaces of pairs, and applications to topology of four-manifolds},
  Memoirs of the American Mathematical Society, American Mathematical Society,
  Providence, RI, xx+366 pages, arXiv:2010.15789, to appear,
  \url{https://www.ams.org/cgi-bin/mstrack/accepted_papers/memo}.

\bibitem{Feehan_Maridakis_Lojasiewicz-Simon_coupled_Yang-Mills}
Paul M.~N. Feehan and Manousos Maridakis, \emph{{\L}ojasiewicz--{S}imon
  gradient inequalities for coupled {Y}ang--{M}ills energy functionals}, Mem.
  Amer. Math. Soc. \textbf{267} (2020), no.~1302, xiii+138, arXiv:1510.03815.
  \MR{4199212}

\bibitem{Feehan_Maridakis_Lojasiewicz-Simon_Banach}
Paul M.~N. Feehan and Manousos Maridakis, \emph{{\L}ojasiewicz-{S}imon gradient
  inequalities for analytic and {M}orse-{B}ott functions on {B}anach spaces},
  J. Reine Angew. Math. \textbf{765} (2020), 35--67, arXiv:1510.03817.
  \MR{4129355}

\bibitem{Floer}
Andreas Floer, \emph{An instanton-invariant for {$3$}-manifolds}, Comm. Math.
  Phys. \textbf{118} (1988), no.~2, 215--240. \MR{956166}

\bibitem{Fornaro_Metafune_Priola_2004}
Simona Fornaro, Giorgio Metafune, and Enrico Priola, \emph{Gradient estimates
  for {D}irichlet parabolic problems in unbounded domains}, J. Differential
  Equations \textbf{205} (2004), no.~2, 329--353. \MR{2092861}

\bibitem{Fredrickson_et_al_analytic_geometric_aspects_gauge_theory}
Laura Fredrickson, Rafe Mazzeo, Tomasz Mrowka, Laura Schaposnik, and Thomas
  Walpuski, \emph{Analytic and geometric aspects of gauge theory}, Simons
  Laufer Mathematical Sciences Institute (MSRI/SLMath), Fall 2022,
  \url{https://www.slmath.org/programs/340}.

\bibitem{FU}
Daniel~S. Freed and Karen~K. Uhlenbeck, \emph{Instantons and four-manifolds},
  second ed., Mathematical Sciences Research Institute Publications, vol.~1,
  Springer, New York, 1991. \MR{1081321 (91i:57019)}

\bibitem{FriedmanBundleBook}
Robert Friedman, \emph{Algebraic surfaces and holomorphic vector bundles},
  Universitext, Springer--Verlag, New York, 1998. \MR{1600388}

\bibitem{FrM}
Robert Friedman and John~W. Morgan, \emph{Smooth four-manifolds and complex
  surfaces}, Ergebnisse der Mathematik und ihrer Grenzgebiete (3) [Results in
  Mathematics and Related Areas (3)], vol.~27, Springer--Verlag, Berlin, 1994.
  \MR{1288304}

\bibitem{Fukaya_9-25-2018}
Kenji Fukaya, \emph{personal communication}, September 25, 2018.

\bibitem{Fukaya_1998}
Kenji Fukaya, \emph{Anti-self-dual equation on {$4$}-manifolds with degenerate
  metric}, Geom. Funct. Anal. \textbf{8} (1998), no.~3, 466--528. \MR{1631255}

\bibitem{Gabrielov_1968}
Andrei~M. Gabri\`elov, \emph{Projections of semianalytic sets}, Functional
  Anal. Appl. \textbf{2} (1968), no.~4, 282--291. \MR{0245831}

\bibitem{Gagliardo_Uhlenbeck_2012}
Michael Gagliardo and Karen Uhlenbeck, \emph{Geometric aspects of the
  {K}apustin-{W}itten equations}, J. Fixed Point Theory Appl. \textbf{11}
  (2012), no.~2, 185--198. \MR{3000667}

\bibitem{Gamelin_Greene_introduction_topology}
Theodore~W. Gamelin and Robert~Everist Greene, \emph{Introduction to topology},
  second ed., Dover Publications, Inc., Mineola, NY, 1999. \MR{1700747}

\bibitem{Gerhardt_2010}
Claus Gerhardt, \emph{An energy gap for {Y}ang--{M}ills connections}, Comm.
  Math. Phys. \textbf{298} (2010), no.~2, 515--522. \MR{2669447}

\bibitem{Giaquinta_1983}
Mariano Giaquinta, \emph{Multiple integrals in the calculus of variations and
  nonlinear elliptic systems}, Annals of Mathematics Studies, vol. 105,
  Princeton University Press, Princeton, NJ, 1983. \MR{717034}

\bibitem{GT}
David Gilbarg and Neil~S. Trudinger, \emph{Elliptic partial differential
  equations of second order}, second ed., Grundlehren der Mathematischen
  Wissenschaften [Fundamental Principles of Mathematical Sciences], vol. 224,
  Springer--Verlag, Berlin, 1983. \MR{737190}

\bibitem{Goldman_1984}
William~M. Goldman, \emph{The symplectic nature of fundamental groups of
  surfaces}, Adv. in Math. \textbf{54} (1984), no.~2, 200--225. \MR{762512}

\bibitem{Goldman_1985}
William~M. Goldman, \emph{Representations of fundamental groups of surfaces},
  Geometry and topology ({C}ollege {P}ark, {M}d., 1983/84), Lecture Notes in
  Math., vol. 1167, Springer, Berlin, 1985, pp.~95--117. \MR{827264}

\bibitem{Goldman_Millson_1987}
William~M. Goldman and John~J. Millson, \emph{Deformations of flat bundles over
  {K}\"ahler manifolds}, Geometry and topology ({A}thens, {G}a., 1985), Lecture
  Notes in Pure and Appl. Math., vol. 105, Dekker, New York, 1987,
  pp.~129--145. \MR{873290}

\bibitem{Goldman_Millson_1988}
William~M. Goldman and John~J. Millson, \emph{The deformation theory of
  representations of fundamental groups of compact {K}\"ahler manifolds}, Inst.
  Hautes \'Etudes Sci. Publ. Math. (1988), no.~67, 43--96. \MR{972343}

\bibitem{GompfMrowka}
Robert~E. Gompf and Tomasz~S. Mrowka, \emph{Irreducible {$4$}-manifolds need
  not be complex}, Ann. of Math. (2) \textbf{138} (1993), no.~1, 61--111.
  \MR{1230927}

\bibitem{GorMacPh}
Mark Goresky and Robert MacPherson, \emph{Stratified {M}orse theory},
  Ergebnisse der Mathematik und ihrer Grenzgebiete (3) [Results in Mathematics
  and Related Areas (3)], vol.~14, Springer--Verlag, Berlin, 1988. \MR{932724
  (90d:57039)}

\bibitem{Hardt_1975}
Robert~M. Hardt, \emph{Stratification of real analytic mappings and images},
  Invent. Math. \textbf{28} (1975), 193--208. \MR{372237}

\bibitem{Hatcher}
Allen~E. Hatcher, \emph{Algebraic topology}, Cambridge University Press,
  Cambridge, 2002. \MR{1867354}

\bibitem{Hedden_Herald_Kirk_2014}
Matthew Hedden, Christopher~M. Herald, and Paul Kirk, \emph{The pillowcase and
  perturbations of traceless representations of knot groups}, Geom. Topol.
  \textbf{18} (2014), no.~1, 211--287, arXiv:1301.0164. \MR{3158776}

\bibitem{Herald_1994}
Christopher~M. Herald, \emph{Legendrian cobordism and {C}hern-{S}imons theory
  on {$3$}-manifolds with boundary}, Comm. Anal. Geom. \textbf{2} (1994),
  no.~3, 337--413. \MR{1305710}

\bibitem{Hilgert_Neeb_structure_geometry_lie_groups}
Joachim Hilgert and Karl-Hermann Neeb, \emph{Structure and geometry of {L}ie
  groups}, Springer Monographs in Mathematics, Springer, New York, 2012.
  \MR{3025417}

\bibitem{Hironaka_1964-I-II}
Heisuke Hironaka, \emph{Resolution of singularities of an algebraic variety
  over a field of characteristic zero. {I}, {II}}, Ann. of Math. (2) {\bf 79}
  (1964), 109--203; ibid. (2) \textbf{79} (1964), 205--326. \MR{0199184}

\bibitem{Hironaka_intro_real-analytic_sets_maps}
Heisuke Hironaka, \emph{Introduction to real-analytic sets and real-analytic
  maps}, Istituto Matematico ``L. Tonelli'' dell'Universit\`a di Pisa, Pisa,
  1973, Quaderni dei Gruppi di Ricerca Matematica del Consiglio Nazionale delle
  Ricerche. \MR{0477121}

\bibitem{Hironaka_1973}
Heisuke Hironaka, \emph{Subanalytic sets}, Number theory, algebraic geometry
  and commutative algebra, in honor of {Y}asuo {A}kizuki, Kinokuniya, Tokyo,
  1973, pp.~453--493. \MR{0377101}

\bibitem{Ho_Wilkin_Wu_2019}
Nan-Kuo Ho, Graeme Wilkin, and Siye Wu, \emph{Conditions of smoothness of
  moduli spaces of flat connections and of character varieties}, Math. Z.
  \textbf{293} (2019), no.~1-2, 1--23, arXiv:1610.09987. \MR{4002269}

\bibitem{Hong_Tian_2004}
Min-Chun Hong and Gang Tian, \emph{Asymptotical behaviour of the
  {Y}ang--{M}ills flow and singular {Y}ang--{M}ills connections}, Math. Ann.
  \textbf{330} (2004), no.~3, 441--472. \MR{2099188}

\bibitem{Hormander_1958}
Lars H{\"o}rmander, \emph{On the division of distributions by polynomials},
  Ark. Mat. \textbf{3} (1958), 555--568. \MR{0124734}

\bibitem{Huang_2006}
Sen-Zhong Huang, \emph{Gradient inequalities}, Mathematical Surveys and
  Monographs, vol. 126, American Mathematical Society, Providence, RI, 2006.
  \MR{2226672 (2007b:35035)}

\bibitem{Huang_2017}
Teng Huang, \emph{A proof of energy gap for {Y}ang--{M}ills connections}, C. R.
  Math. Acad. Sci. Paris \textbf{355} (2017), no.~8, 910--913,
  arXiv:1704.02772. \MR{3693514}

\bibitem{Huebschmann_1994}
Johannes Huebschmann, \emph{Holonomies of {Y}ang--{M}ills connections for
  bundles on a surface with disconnected structure group}, Math. Proc.
  Cambridge Philos. Soc. \textbf{116} (1994), no.~2, 375--384. \MR{1281554}

\bibitem{Huebschmann_1995}
Johannes Huebschmann, \emph{The singularities of {Y}ang--{M}ills connections
  for bundles on a surface. {I}. {T}he local model}, Math. Z. \textbf{220}
  (1995), no.~4, 595--609. \MR{1363857}

\bibitem{Huebschmann_1996}
Johannes Huebschmann, \emph{The singularities of {Y}ang--{M}ills connections
  for bundles on a surface. {II}. {T}he stratification}, Math. Z. \textbf{221}
  (1996), no.~1, 83--92. \MR{1369463}

\bibitem{Huebschmann_1998}
Johannes Huebschmann, \emph{Smooth structures on certain moduli spaces for
  bundles on a surface}, J. Pure Appl. Algebra \textbf{126} (1998), no.~1-3,
  183--221. \MR{1600534}

\bibitem{Jacob_2016}
Adam Jacob, \emph{The {Y}ang--{M}ills flow and the {A}tiyah-{B}ott formula on
  compact {K}\"{a}hler manifolds}, Amer. J. Math. \textbf{138} (2016), no.~2,
  329--365, arXiv:1109.1550. \MR{3483467}

\bibitem{Kaloshin_2005}
Vadim~Yu. Kaloshin, \emph{A geometric proof of the existence of {W}hitney
  stratifications}, Mosc. Math. J. \textbf{5} (2005), no.~1, 125--133.
  \MR{2153470}

\bibitem{Kashiwara_Schapira_sheaves_manifolds}
Masaki Kashiwara and Pierre Schapira, \emph{Sheaves on manifolds}, Grundlehren
  der mathematischen Wissenschaften [Fundamental Principles of Mathematical
  Sciences], vol. 292, Springer-Verlag, Berlin, 1994, With a chapter in French
  by Christian Houzel, Corrected reprint of the 1990 original. \MR{1299726}

\bibitem{Kirk_1993}
Paul~A. Kirk, \emph{{${\rm SU}(2)$} representation varieties of
  {$3$}-manifolds, gauge theory invariants, and surgery on knots}, Proceedings
  of {GARC} {W}orkshop on {G}eometry and {T}opology '93 ({S}eoul, 1993),
  Lecture Notes Ser., vol.~18, Seoul Nat. Univ., Seoul, 1993, pp.~137--176.
  \MR{1270935}

\bibitem{Knapp_1986}
Anthony~W. Knapp, \emph{Lie groups beyond an introduction}, second ed.,
  Progress in Mathematics, vol. 140, Birkh\"auser Boston, Inc., Boston, MA,
  2002. \MR{1920389 (2003c:22001)}

\bibitem{Knopf_Sesum_2019}
Dan Knopf and Nata\v{s}a \v{S}e\v{s}um, \emph{Dynamic instability of
  {$\Bbb{CP}^N$} under {R}icci flow}, J. Geom. Anal. \textbf{29} (2019), no.~1,
  902--916. \MR{3897037}

\bibitem{Kobayashi}
Shoshichi Kobayashi, \emph{Differential geometry of complex vector bundles},
  Publications of the Mathematical Society of Japan, vol.~15, Princeton
  University Press, Princeton, NJ, 1987, Kan{\^o} Memorial Lectures, 5.
  \MR{909698 (89e:53100)}

\bibitem{Kobayashi_differential_geometry_complex_vector_bundles}
Shoshichi Kobayashi, \emph{Differential geometry of complex vector bundles},
  Princeton Legacy Library, Princeton University Press, Princeton, NJ, [2014],
  Reprint of the 1987 edition [ MR0909698]. \MR{3643615}

\bibitem{Kobayashi_Nomizu_v1}
Shoshichi Kobayashi and Katsumi Nomizu, \emph{Foundations of differential
  geometry. {V}ol. {I}}, Wiley Classics Library, John Wiley \& Sons, Inc., New
  York, 1996, Reprint of the 1963 original, A Wiley-Interscience Publication.
  \MR{1393940}

\bibitem{Koiso_1987}
Norihito Koiso, \emph{Yang-{M}ills connections and moduli space}, Osaka J.
  Math. \textbf{24} (1987), no.~1, 147--171. \MR{881753}

\bibitem{Kozono_Maeda_Naito_1995}
Hideo Kozono, Yoshiaki Maeda, and Hisashi Naito, \emph{Global solution for the
  {Y}ang--{M}ills gradient flow on {$4$}-manifolds}, Nagoya Math. J.
  \textbf{139} (1995), 93--128. \MR{1355271 (97a:58038)}

\bibitem{KMStructure}
Peter~B. Kronheimer and Tomasz~S. Mrowka, \emph{Embedded surfaces and the
  structure of {D}onaldson's polynomial invariants}, J. Differential Geom.
  \textbf{41} (1995), 573--734. \MR{1338483 (96e:57019)}

\bibitem{KMBook}
Peter~B. Kronheimer and Tomasz~S. Mrowka, \emph{Monopoles and three-manifolds},
  Cambridge University Press, Cambridge, 2007. \MR{2388043 (2009f:57049)}

\bibitem{Kuranishi}
Masatake Kuranishi, \emph{New proof for the existence of locally complete
  families of complex structures}, Proc. {C}onf. {C}omplex {A}nalysis
  ({M}inneapolis, 1964) (A.~Aeppli, E.~Calabi, and H.~R{\"o}hrl, eds.),
  Springer, Berlin, 1965, pp.~142--154. \MR{0176496 (31 \#768)}

\bibitem{Lawson}
H.~Blaine Lawson, Jr., \emph{The theory of gauge fields in four dimensions},
  CBMS Regional Conference Series in Mathematics, vol.~58, Published for the
  Conference Board of the Mathematical Sciences, Washington, DC; by the
  American Mathematical Society, Providence, RI, 1985. \MR{799712}

\bibitem{Lee_john_topological_manifolds}
John~M. Lee, \emph{Introduction to topological manifolds}, second ed., Graduate
  Texts in Mathematics, vol. 202, Springer, New York, 2011. \MR{2766102}

\bibitem{Lee_john_smooth_manifolds}
John~M. Lee, \emph{Introduction to smooth manifolds}, second ed., Graduate
  Texts in Mathematics, vol. 218, Springer, New York, 2013. \MR{2954043}

\bibitem{Lojasiewicz_1959}
Stanis{\l}aw {\L}ojasiewicz, \emph{Sur le probl\`eme de la division}, Studia
  Math. \textbf{18} (1959), 87--136. \MR{0107168 (21 \#5893)}

\bibitem{Lojasiewicz_1961}
Stanis{\l}aw {\L}ojasiewicz, \emph{Sur le probl\`eme de la division}, Rozprawy
  Mat. \textbf{22} (1961), 1--57. \MR{0126072}

\bibitem{Lojasiewicz_1963}
Stanis{\l}aw {\L}ojasiewicz, \emph{Une propri\'et\'e topologique des
  sous-ensembles analytiques r\'eels}, Les \'{E}quations aux {D}\'eriv\'ees
  {P}artielles ({P}aris, 1962), \'Editions du Centre National de la Recherche
  Scientifique, Paris, 1963, pp.~87--89. \MR{0160856 (28 \#4066)}

\bibitem{Lojasiewicz_1964}
Stanis{\l}aw {\L}ojasiewicz, \emph{Triangulation of semi-analytic sets}, Ann.
  Scuola Norm. Sup. Pisa (3) \textbf{18} (1964), 449--474. \MR{0173265}

\bibitem{Lojasiewicz_1965}
Stanis{\l}aw {\L}ojasiewicz, \emph{Ensembles semi-analytiques},  (1965), Publ.
  Inst. Hautes Etudes Sci., Bures-sur-Yvette. LaTeX version by M. Coste, August
  29, 2006 based on mimeographed course notes by S. {\L}ojasiewicz, available
  at \url{perso.univ-rennes1.fr/michel.coste/Lojasiewicz.pdf}.

\bibitem{Lojasiewicz_1984}
Stanis{\l}aw {\L}ojasiewicz, \emph{Sur les trajectoires du gradient d'une
  fonction analytique}, Geometry seminars, 1982--1983 ({B}ologna, 1982/1983),
  Univ. Stud. Bologna, Bologna, 1984, pp.~115--117. \MR{771152 (86m:58023)}

\bibitem{Lotay_Oliveira_2022}
Jason~D. Lotay and Gon\c{c}alo Oliveira, \emph{Examples of deformed {$\rm
  G_2$}-instantons/{D}onaldson-{T}homas connections}, Ann. Inst. Fourier
  (Grenoble) \textbf{72} (2022), no.~1, 339--366. \MR{4448598}

\bibitem{Lubotzky_Magid_varieties_representations_finitely_generated_groups}
Alexander Lubotzky and Andy~R. Magid, \emph{Varieties of representations of
  finitely generated groups}, Mem. Amer. Math. Soc. \textbf{58} (1985),
  no.~336. \MR{818915}

\bibitem{Massey_Le_2007}
David~B. Massey and D\~ung~Tr\'ang L{\^e}, \emph{Notes on real and complex
  analytic and semianalytic singularities}, Singularities in geometry and
  topology, World Sci. Publ., Hackensack, NJ, 2007, pp.~81--126. \MR{2311485
  (2008h:32048)}

\bibitem{Mather_2012}
John~N. Mather, \emph{Notes on topological stability}, Bull. Amer. Math. Soc.
  (N.S.) \textbf{49} (2012), no.~4, 475--506. \MR{2958928}

\bibitem{Mazzeo_Witten_2020}
Rafe Mazzeo and Edward Witten, \emph{The {KW} equations and the {N}ahm pole
  boundary condition with knots}, Comm. Anal. Geom. \textbf{28} (2020), no.~4,
  871--942. \MR{4165312}

\bibitem{McOwen_partial_differential_equations}
Robert~C. McOwen, \emph{Partial differential equations: Methods and
  applications}, second ed., Prentice Hall, 2003.

\bibitem{Meyer_Riviere_2003}
Yves Meyer and Tristan Rivi\`ere, \emph{A partial regularity result for a class
  of stationary {Y}ang--{M}ills fields in high dimension}, Rev. Mat.
  Iberoamericana \textbf{19} (2003), no.~1, 195--219. \MR{1993420}

\bibitem{MilnorStasheff}
John~W. Milnor and James~D. Stasheff, \emph{Characteristic classes}, Princeton
  University Press, Princeton, N. J.; University of Tokyo Press, Tokyo, 1974,
  Annals of Mathematics Studies, No. 76. \MR{0440554}

\bibitem{MMR}
John~W. Morgan, Tomasz~S. Mrowka, and Daniel Ruberman, \emph{The {$L^2$}-moduli
  space and a vanishing theorem for {D}onaldson polynomial invariants},
  Monographs in Geometry and Topology, vol.~2, International Press, Cambridge,
  MA, 1994. \MR{1287851 (95h:57039)}

\bibitem{Mrowka_7-30-2018}
Tomasz~S. Mrowka, \emph{personal communication}, July 30, 2018.

\bibitem{Munkres_topology_second_edition}
James~R. Munkres, \emph{Topology}, second ed., Prentice Hall, Inc., Upper
  Saddle River, NJ, 2000. \MR{3728284}

\bibitem{Naber_Valtorta_2019}
Aaron Naber and Daniele Valtorta, \emph{Energy identity for stationary {Y}ang
  {M}ills}, Invent. Math. \textbf{216} (2019), no.~3, 847--925,
  arXiv:1610.02898. \MR{3955711}

\bibitem{Naito_1994}
Hisashi Naito, \emph{Finite time blowing-up for the {Y}ang--{M}ills gradient
  flow in higher dimensions}, Hokkaido Math. J. \textbf{23} (1994), 451--464.
  \MR{1299637 (95i:58054)}

\bibitem{Nakajima_1987}
Hiraku Nakajima, \emph{Removable singularities for {Y}ang--{M}ills connections
  in higher dimensions}, J. Fac. Sci. Univ. Tokyo Sect. IA Math. \textbf{34}
  (1987), no.~2, 299--307. \MR{914024}

\bibitem{Nakajima_1988}
Hiraku Nakajima, \emph{Compactness of the moduli space of {Y}ang--{M}ills
  connections in higher dimensions}, J. Math. Soc. Japan \textbf{40} (1988),
  no.~3, 383--392. \MR{945342}

\bibitem{Nishinou_2007}
Takeo Nishinou, \emph{Global gauge fixing for connections with small curvature
  on {$T^2$}}, Internat. J. Math. \textbf{18} (2007), no.~2, 165--177.
  \MR{2307419}

\bibitem{Ok_real_analysis_economic_applications}
Efe~A. Ok, \emph{Real analysis with economic applications}, Princeton
  University Press, Princeton, NJ, 2007. \MR{2275400}

\bibitem{ParkerGauge}
Thomas~H. Parker, \emph{Gauge theories on four-dimensional {R}iemannian
  manifolds}, Comm. Math. Phys. \textbf{85} (1982), 563--602. \MR{677998
  (84b:58036)}

\bibitem{Price_1983}
Peter Price, \emph{A monotonicity formula for {Y}ang--{M}ills fields},
  Manuscripta Math. \textbf{43} (1983), no.~2-3, 131--166. \MR{707042}

\bibitem{Reed_Simon_v4}
Michael Reed and Barry Simon, \emph{Methods of modern mathematical physics.
  {IV}. {A}nalysis of operators}, Academic Press [Harcourt Brace Jovanovich,
  Publishers], New York-London, 1978. \MR{493421}

\bibitem{Riviere_2015arxiv}
Tristan Rivi{\`e}re, \emph{The variations of {Y}ang--{M}ills {L}agrangian},
  mini-course lecture notes, ninth summer school in Differential Geometry,
  Korean Institute for Advanced Studies, June 23--27, 2014, arXiv:1506.04554.

\bibitem{Rade_1992}
Johan R\r{a}de, \emph{On the {Y}ang--{M}ills heat equation in two and three
  dimensions}, J. Reine Angew. Math. \textbf{431} (1992), 123--163. \MR{1179335
  (94a:58041)}

\bibitem{Rudin}
Walter Rudin, \emph{Functional analysis}, second ed., International Series in
  Pure and Applied Mathematics, McGraw-Hill, Inc., New York, 1991. \MR{1157815}

\bibitem{Saveliev_2002}
Nikolai Saveliev, \emph{Representation spaces of {S}eifert fibered homology
  spheres}, Topology Appl. \textbf{126} (2002), no.~1-2, 49--61. \MR{1934252}

\bibitem{Saveliev_lectures_topology_3-manifolds}
Nikolai Saveliev, \emph{Lectures on the topology of 3-manifolds}, revised ed.,
  De Gruyter Textbook, Walter de Gruyter \& Co., Berlin, 2012, An introduction
  to the Casson invariant. \MR{2893651}

\bibitem{Schlatter_1997}
Andreas~E. Schlatter, \emph{Long-time behaviour of the {Y}ang--{M}ills flow in
  four dimensions}, Ann. Global Anal. Geom. \textbf{15} (1997), no.~1, 1--25.
  \MR{1443269}

\bibitem{Sedlacek}
Steven~B. Sedlacek, \emph{A direct method for minimizing the {Y}ang--{M}ills
  functional over {$4$}-manifolds}, Comm. Math. Phys. \textbf{86} (1982),
  515--527. \MR{679200 (84e:81049)}

\bibitem{Sell_You_2002}
George~R. Sell and Yuncheng You, \emph{Dynamics of evolutionary equations},
  Applied Mathematical Sciences, vol. 143, Springer, New York, 2002.
  \MR{1873467 (2003f:37001b)}

\bibitem{Shiota_geometry_subanalytic_semialgebraic_sets}
Masahiro Shiota, \emph{Geometry of subanalytic and semialgebraic sets},
  Progress in Mathematics, vol. 150, Birkh\"auser Boston, Inc., Boston, MA,
  1997. \MR{1463945}

\bibitem{Sibley_2015}
Benjamin Sibley, \emph{Asymptotics of the {Y}ang--{M}ills flow for holomorphic
  vector bundles over {K}\"{a}hler manifolds: the canonical structure of the
  limit}, J. Reine Angew. Math. \textbf{706} (2015), 123--191,
  arXiv:1206.5491v3. \MR{3393366}

\bibitem{Sibley_Wentworth_2015}
Benjamin Sibley and Richard~A. Wentworth, \emph{Analytic cycles, {B}ott-{C}hern
  forms, and singular sets for the {Y}ang--{M}ills flow on {K}\"{a}hler
  manifolds}, Adv. Math. \textbf{279} (2015), 501--531. \MR{3345190}

\bibitem{Sibner_1984}
Lesley~M. Sibner, \emph{Removable singularities of {Y}ang--{M}ills fields in
  {${\bf R}^{3}$}}, Compositio Math. \textbf{53} (1984), 91--104. \MR{762308
  (86c:58151)}

\bibitem{Simon_1983}
Leon Simon, \emph{Asymptotics for a class of nonlinear evolution equations,
  with applications to geometric problems}, Ann. of Math. (2) \textbf{118}
  (1983), 525--571. \MR{727703 (85b:58121)}

\bibitem{Simon_1996}
Leon Simon, \emph{Theorems on regularity and singularity of energy minimizing
  maps}, Lectures in Mathematics ETH Z\"urich, Birkh{\"a}user, Basel, 1996.
  \MR{1399562 (98c:58042)}

\bibitem{Simpson_1994part1}
Carlos~T. Simpson, \emph{Moduli of representations of the fundamental group of
  a smooth projective variety. {I}}, Inst. Hautes \'Etudes Sci. Publ. Math.
  (1994), no.~79, 47--129. \MR{1307297}

\bibitem{Simpson_1994part2}
Carlos~T. Simpson, \emph{Moduli of representations of the fundamental group of
  a smooth projective variety. {II}}, Inst. Hautes \'Etudes Sci. Publ. Math.
  (1994), no.~80, 5--79 (1995). \MR{1320603}

\bibitem{Smith_Uhlenbeck_2022}
Penny Smith and Karen Uhlenbeck, \emph{Removeability of a codimension four
  singular set for solutions of a {Y}ang--{M}ills--{H}iggs equation with small
  energy}, Surveys in differential geometry 2019. {D}ifferential geometry,
  {C}alabi-{Y}au theory, and general relativity. {P}art 2, Surv. Differ. Geom.,
  vol.~24, Int. Press, Boston, MA, 2022, arXiv:1811.03135, pp.~257--291.
  \MR{4479723}

\bibitem{Smith_1990}
Penny~D. Smith, \emph{Removable singularities for the {Y}ang--{M}ills-{H}iggs
  equations in two dimensions}, Ann. Inst. H. Poincar\'e Anal. Non Lin\'eaire
  \textbf{7} (1990), no.~6, 561--588. \MR{1079572}

\bibitem{Spinaci_2014}
Marco Spinaci, \emph{Deformations of twisted harmonic maps and variation of the
  energy}, Math. Z. \textbf{278} (2014), no.~3-4, 617--648. \MR{3278886}

\bibitem{Struwe_1994}
Michael Struwe, \emph{The {Y}ang--{M}ills flow in four dimensions}, Calc. Var.
  Partial Differential Equations \textbf{2} (1994), 123--150. \MR{1385523
  (97i:58034)}

\bibitem{Sussmann_1990}
Hector~J. Sussmann, \emph{Real analytic desingularization and subanalytic sets:
  an elementary approach}, Trans. Amer. Math. Soc. \textbf{317} (1990), no.~2,
  417--461. \MR{943608}

\bibitem{Taheri_function_spaces_partial_differential_equations}
Ali Taheri, \emph{Function spaces and partial differential equations. {V}ol.
  2}, Oxford Lecture Series in Mathematics and its Applications, vol.~41,
  Oxford University Press, Oxford, 2015, Contemporary analysis. \MR{3443825}

\bibitem{Tanaka_2012}
Yuuji Tanaka, \emph{A construction of {$Spin(7)$}-instantons}, Ann. Global
  Anal. Geom. \textbf{42} (2012), no.~4, 495--521, arXiv:1201.3150.
  \MR{2995202}

\bibitem{Tanaka_2013}
Yuuji Tanaka, \emph{A weak compactness theorem of the {D}onaldson-{T}homas
  instantons on compact {K}\"{a}hler threefolds}, J. Math. Anal. Appl.
  \textbf{408} (2013), no.~1, 27--34. \MR{3079943}

\bibitem{Tanaka_2014}
Yuuji Tanaka, \emph{A removal singularity theorem of the {D}onaldson-{T}homas
  instanton on compact {K}\"{a}hler threefolds}, J. Math. Anal. Appl.
  \textbf{411} (2014), no.~1, 422--428. \MR{3118496}

\bibitem{Tanaka_2016}
Yuuji Tanaka, \emph{On the moduli space of {D}onaldson-{T}homas instantons},
  Extracta Math. \textbf{31} (2016), no.~1, 89--107. \MR{3585951}

\bibitem{Tanaka_2019}
Yuuji Tanaka, \emph{A perturbation and generic smoothness of the
  {V}afa-{W}itten moduli spaces on closed symplectic four-manifolds}, Glasg.
  Math. J. \textbf{61} (2019), no.~2, 471--486. \MR{3928649}

\bibitem{Tao_Tian_2004}
Terence Tao and Gang Tian, \emph{A singularity removal theorem for
  {Y}ang--{M}ills fields in higher dimensions}, J. Amer. Math. Soc. \textbf{17}
  (2004), no.~3, 557--593. \MR{2053951}

\bibitem{TauPath}
Clifford~Henry Taubes, \emph{Path-connected {Y}ang--{M}ills moduli spaces}, J.
  Differential Geom. \textbf{19} (1984), 337--392. \MR{755230 (85m:58049)}

\bibitem{TauFrame}
Clifford~Henry Taubes, \emph{A framework for {M}orse theory for the
  {Y}ang--{M}ills functional}, Invent. Math. \textbf{94} (1988), 327--402.
  \MR{958836 (90a:58035)}

\bibitem{TauL2}
Clifford~Henry Taubes, \emph{{$L^2$} moduli spaces on 4-manifolds with
  cylindrical ends}, Monographs in Geometry and Topology, I, International
  Press, Cambridge, MA, 1993. \MR{1287854 (96b:58018)}

\bibitem{Taubes_2013cjm}
Clifford~Henry Taubes, \emph{{${\rm PSL}(2;\Bbb C)$} connections on 3-manifolds
  with {${\rm L}^2$} bounds on curvature}, Camb. J. Math. \textbf{1} (2013),
  no.~2, 239--397. \MR{3272050}

\bibitem{TianGTCalGeom}
Gang Tian, \emph{Gauge theory and calibrated geometry. {I}}, Ann. of Math. (2)
  \textbf{151} (2000), no.~1, 193--268. \MR{1745014}

\bibitem{Tian_Yang_2002}
Gang Tian and Baozhong Yang, \emph{Compactification of the moduli spaces of
  vortices and coupled vortices}, J. Reine Angew. Math. \textbf{553} (2002),
  17--41, arXiv.math/0203078. \MR{1944806}

\bibitem{Troianiello}
Giovanni~Maria Troianiello, \emph{Elliptic differential equations and obstacle
  problems}, The University Series in Mathematics, Plenum Press, New York,
  1987. \MR{1094820}

\bibitem{Uhlenbeck_3-1-2019}
Karen~K. Uhlenbeck, \emph{personal communication}, March 1, 2019.

\bibitem{UhlLp}
Karen~K. Uhlenbeck, \emph{Connections with {$L\sp{p}$} bounds on curvature},
  Comm. Math. Phys. \textbf{83} (1982), no.~1, 31--42. \MR{648356}

\bibitem{UhlRem}
Karen~K. Uhlenbeck, \emph{Removable singularities in {Y}ang--{M}ills fields},
  Comm. Math. Phys. \textbf{83} (1982), no.~1, 11--29. \MR{648355}

\bibitem{UhlChern}
Karen~K. Uhlenbeck, \emph{The {C}hern classes of {S}obolev connections}, Comm.
  Math. Phys. \textbf{101} (1985), 449--457. \MR{815194 (87f:58028)}

\bibitem{Waldron_2023}
Alex Waldron, \emph{Uhlenbeck compactness for {Y}ang--{M}ills flow in higher
  dimensions}, Calc. Var. Partial Differential Equations \textbf{62} (2023),
  no.~5, Paper No. 165. \MR{4597629}

\bibitem{Walker_extension_casson_invariant}
Kevin Walker, \emph{An extension of {C}asson's invariant}, Annals of
  Mathematics Studies, vol. 126, Princeton University Press, Princeton, NJ,
  1992. \MR{1154798}

\bibitem{WalpuskiThesis}
Thomas Walpuski, \emph{Gauge theory on ${G}_2$-manifolds}, {Ph.D}. thesis,
  Imperial College London, United Kingdom, 2013,
  \url{http://hdl.handle.net/10044/1/14365}.

\bibitem{Walpuski_2017}
Thomas Walpuski, \emph{{$G_2$}-instantons, associative submanifolds and
  {F}ueter sections}, Comm. Anal. Geom. \textbf{25} (2017), no.~4, 847--893.
  \MR{3731643}

\bibitem{Warner}
Frank~W. Warner, \emph{Foundations of differentiable manifolds and {L}ie
  groups}, Graduate Texts in Mathematics, vol.~94, Springer, New York, 1983.
  \MR{722297 (84k:58001)}

\bibitem{Wehrheim_2004}
Katrin Wehrheim, \emph{Uhlenbeck compactness}, EMS Series of Lectures in
  Mathematics, European Mathematical Society (EMS), Z\"urich, 2004. \MR{2030823
  (2004m:53045)}

\bibitem{Whitney_1965dct}
Hassler Whitney, \emph{Local properties of analytic varieties}, Differential
  and {C}ombinatorial {T}opology ({A} {S}ymposium in {H}onor of {M}arston
  {M}orse), Princeton Univ. Press, Princeton, NJ, 1965, pp.~205--244.
  \MR{188486}

\bibitem{Whitney_1965am}
Hassler Whitney, \emph{Tangents to an analytic variety}, Ann. of Math. (2)
  \textbf{81} (1965), 496--549. \MR{192520}

\bibitem{Willard_general_topology}
Stephen Willard, \emph{General topology}, Dover Publications, Inc., Mineola,
  NY, 2004, Reprint of the 1970 original [Addison-Wesley, Reading, MA;
  MR0264581]. \MR{2048350}

\bibitem{Yosida}
Kosaku Yosida, \emph{Functional analysis}, Classics in Mathematics,
  Springer-Verlag, Berlin, 1995, Reprint of the sixth (1980) edition.
  \MR{1336382}

\bibitem{Zhang_2004cmb}
Xi~Zhang, \emph{A compactness theorem for {Y}ang--{M}ills connections}, Canad.
  Math. Bull. \textbf{47} (2004), no.~4, 624--634. \MR{2099759}

\end{thebibliography}
\bibliographystyle{amsplain-nodash}

\end{document}